\documentclass{article}
\usepackage[%
journal=MSL,    
lang=american,   
]{ems-journal}

\usepackage{macro_MSL}

\numberwithin{equation}{section}

\begin{document}

\title{Spectral Estimators for Structured Generalized Linear Models\\via Approximate Message Passing}
\titlemark{Spectral Estimators for Structured Generalized Linear Models via Approximate Message Passing}



\emsauthor{1}{
	\givenname{Yihan}
	\surname{Zhang}
	\mrid{1502212}
	\orcid{0000-0002-6465-6258}}{Y.~Zhang}
\emsauthor{2}{
	\givenname{Hong Chang}
	\surname{Ji}
	\mrid{1390777}
	\orcid{0009-0004-8183-2876}}{H.C.~Ji}
\emsauthor{3}{
	\givenname{Ramji}
	\surname{Venkataramanan}
	\mrid{813310}
	\orcid{0000-0001-7915-5432}}{R.~Venkataramanan}
\emsauthor{4}{
	\givenname{Marco}
	\surname{Mondelli}
	\mrid{1103285}
	\orcid{0000-0002-3242-7020}}{M.~Mondelli}

\Emsaffil{1}{
	\department{School of Mathematics}
	\organisation{University of Bristol}
	\rorid{0524sp257}
	\address{Fry Building, Woodland Road}
	\zip{BS8 1UG}
	\city{Bristol}
	\country{United Kingdom}
	\affemail{yihan.zhang@bristol.ac.uk}}
\Emsaffil{2}{
	\department{Department of Mathematics}
	\organisation{University of Wisconsin-Madison}
	\rorid{01y2jtd41}
	\address{480 Lincoln Dr}
	\zip{WI 53706-1325}
	\city{Madison}
	\country{United States}
	\affemail{hji56@wisc.edu}}
%
\Emsaffil{3}{
	\department{Department of Engineering}
	\organisation{University of Cambridge}
	\rorid{013meh722}
	\address{Trumpington Street}
	\zip{CB2 1PZ}
	\city{Cambridge}
	\country{United Kingdom}
	\affemail{rv285@cam.ac.uk}}
\Emsaffil{4}{
	\organisation{Institute of Science and Technology Austria}
	\rorid{03gnh5541}
	\address{Am Campus 1}
	\zip{3400}
	\city{Klosterneuburg}
	\country{Austria}
	\affemail{marco.mondelli@ist.ac.at}}

\classification[62E20]{62J12}


\begin{abstract}
We consider the problem of parameter estimation in a high-dimensional generalized linear model. Spectral methods obtained via the principal eigenvector of a suitable data-dependent matrix provide a simple yet surprisingly effective solution. However, despite their wide use, a rigorous performance characterization, as well as a principled way to preprocess the data, are available only for unstructured (i.i.d.\ Gaussian and Haar orthogonal) designs. In contrast, real-world data matrices are highly structured and exhibit non-trivial correlations. To address the problem, we consider correlated Gaussian designs capturing the anisotropic nature of the features via a covariance matrix $\Sigma$. Our main result is a precise asymptotic characterization of the performance of spectral estimators. This allows us to identify the optimal preprocessing that minimizes the number of samples needed for parameter estimation. Surprisingly, such preprocessing is universal across a broad set of designs, which partly addresses a conjecture on optimal spectral estimators for rotationally invariant models. Our principled approach vastly improves upon previous heuristic methods, including for designs common in computational imaging and genetics. The proposed methodology, based on approximate message passing, is broadly applicable and opens the way to the precise characterization of spiked matrices and of the corresponding spectral methods in a variety of settings.
\end{abstract}

\maketitle


\newcounter{asmpctr} 
\setcounter{asmpctr}{\value{enumi}}

\section{Introduction}

This paper considers the prototypical problem of learning a parameter vector from observations obtained via a generalized linear model (GLM) \cite{McCullagh_GLM}:
\begin{align}
    y_i &= q\paren{\inprod{x_i}{\beta^*}, \eps_i} , \quad 1\le i\le n , 
    \label{eqn:model-main}
\end{align}
where $ \beta^*\in\bbR^d $ consists of (unknown) regression coefficients. 
The statistician wishes to 
estimate $ \beta^* $ based on the {observations} $ y = (y_i)_{i = 1}^n\in\bbR^n $ and the covariate vectors 
$x_1, \ldots, x_n \in\bbR^{d} $. 
The vector $ \eps = (\eps_i)_{i = 1}^n\in\bbR^n $ contains (unknown) i.i.d.\ random variables accounting for noise in the measurements.
The (known) {link function} $ q\colon\bbR^2\to\bbR $ is 
applied 
element-wise, i.e., $ q(g, \eps) = (q(g_1, \eps_1), \cdots, q(g_n, \eps_n)) $ for any $ g, \eps \in \bbR^n $. 
The nonlinearity $q$ generalizes linear regression ($q(g, \eps) = g + \eps$) and incorporates various 
problems in statistics, machine learning,  signal processing and computational biology, e.g., 
phase retrieval ($q(g, \eps) = |g| + \eps$) \cite{
Fannjiang_numerics_phase_retr}, $1$-bit compressed sensing ($ q(g, \eps) = \sgn(g) + \eps $) \cite{Boufounos_1bCS}, and logistic regression \cite{Sur_logistic}. 

For estimation in GLMs, several works have considered methods based on convex  programming,  
e.g.\ \cite{candes_convex_prog,
Waldspurger_phase, Thrampoulidis_lifting}. 
However, these methods often become computationally infeasible as $d$ grows. 
Thus, 
fast iterative methods including 
alternating minimization \cite{Netrapalli_alt_min}, approximate message passing \cite{
rangan_GAMP}, Wirtinger flow \cite{candes_wirtinger}, iterative projections \cite{Li_itr_proj}, and the Kaczmarz method \cite{wei_kaczmarz} has been developed. 
Due to their iterative nature, to converge to an informative solution, these procedures require a ``warm start'', i.e., a vector $ \wh{\beta}\in\bbR^d $ whose ``overlap'' $ |\langle{\wh{\beta}},{\beta^*}\rangle|/(\|\wh{\beta}\|_2 \|\beta^*\|_2) $
with 
$ \beta^* $ is non-vanishing for large $d$.  
In this paper, we focus on \emph{spectral estimators} \cite{ChenChiFanMa_spec_book}, which provide a simple yet effective approach for estimating $\beta^*$, and serve as a warm start for the local methods above. 
Spectral estimators have been applied in a range of problems including 
polynomial learning \cite{ChenMeka_poly}, 
estimation from mixed linear regression \cite{Yi_mixed_linreg_ICML} 
and
ranking \cite{Chen_ranking}. 
For the GLM in \Cref{eqn:model-main}, the spectral estimator processes the observations via a function $ \cT\colon\bbR\to\bbR $ and outputs the principal eigenvector of the 
matrix
\begin{align}
    D &\coloneqq \sum_{i = 1}^n x_i x_i^\top \cT(y_i) \in \bbR^{d\times d}  . 
\label{eqn:def_D_intro}
\end{align}
%
%
%
To understand the accuracy of spectral estimators, it is crucial to:
\emph{(i)} characterize their performance (e.g., in terms of limiting overlap), 
and \emph{(ii)} design the 
preprocessing function $\cT$ that minimizes the sample complexity, i.e., the number $n$ of observations required to attain a desired limiting overlap. 
This work gives precise answers to both these questions, providing solid performance guarantees as well as a principled basis for optimizing spectral estimators used in practical applications.


A line of work \cite{Netrapalli_alt_min,candes_wirtinger,Chen_Candes_trim}
has bounded the sample complexity of spectral estimators obtained from \Cref{eqn:def_D_intro} 
for i.i.d.\ Gaussian designs via matrix concentration inequalities. However, these bounds require the number $n$ of observations to  substantially exceed the parameter dimension $d$, and they
are not sharp enough to optimize $\mathcal T$.
Using tools from random matrix theory, the works \cite{LuLi,mondelli-montanari-2018-fundamental} obtained tight results in the proportional regime where $n,d\to\infty$ and $ n/d\to\delta $ for a fixed constant $\delta\in(0,\infty)$ (called the ``aspect ratio''). Specifically, a \emph{phase transition} phenomenon is established: if $\delta$ exceeds a critical value (referred to as the ``spectral threshold''), then \emph{(i)} a spectral gap emerges between the first two eigenvalues of $D$, and \emph{(ii)} the spectral estimator attains non-vanishing correlation with $ \beta^* $. For $\delta$ below this critical value, there is no outlier to the right of the spectrum of $D$, and  the spectral estimator is asymptotically independent of $\beta^*$. 
This precise characterization 
allows to derive the optimal preprocessing function that minimizes the spectral threshold \cite{mondelli-montanari-2018-fundamental} and also  that maximizes the overlap for a given $\delta$ \cite{lal-opt-spec}. These results are 
extended by \cite{dudeja-2020-rigorous-analysis,Ma_optspec_haar}   to cover  a sub-sampled Haar design, 
consisting of a subset of columns from a uniformly random orthogonal matrix.

The line of work above crucially relies on the design matrix $X = \matrix{x_1, & \cdots, & x_n}^\top$ being unstructured, namely 
i.i.d.\ Gaussian or rotationally-invariant with unit singular values. In contrast, design matrices occurring in practice are highly structured and their entries exhibit significant correlations (e.g., in computational genomics \cite{lonsdale2013genotype} and 
imaging \cite{candes_CDP}). 
In this paper, we capture the correlation and heterogeneity of the data via general (\emph{correlated}) Gaussian designs. 
Specifically, each covariate $ x_i $ is an i.i.d.\ $d$-dimensional 
zero-mean Gaussian vector with an arbitrary positive definite covariance matrix $ \Sigma/n \in\bbR^{d\times d} $. 
The covariance matrix $ \Sigma $ captures 
correlations between 
covariates and the 
heterogeneity in their variances.
General Gaussian designs (e.g., with Toeplitz or circulant covariance structures) have been widely adopted in high-dimensional regression models \cite{javanmard_montanari_hypo_test, javanmard_montanari_confidence,
javanmard_montanari_debias_lasso, zhang_zhang_confidence, van2014asymptotically, wainwright2009sharp}. 
However, existing results largely focus on (penalized) maximum-likelihood estimators for linear and logistic models
\cite{celentano2021cad,CelentanoMontanariWei, Sur_logistic,ZhaoSurCandes,Sur_thesis}. 
An asymptotic theory of spectral estimators for  GLMs with general Gaussian designs has been lacking. One significant challenge is that current techniques for i.i.d.\ and Haar designs all crucially depend on their right rotational invariance, which fails to hold for correlated covariates.

\subsection{Main results}
\label{sec:contribution}

Our main contribution is to give 
a \emph{precise asymptotic characterization of the overlap between the leading eigenvector of $D$ and the unknown parameter 
$\beta^*$}, as well as the locations of the top two eigenvalues of $D$, provided a criticality condition holds. This is the content of \Cref{thm:main}, which is informally stated  below. 



\begin{theorem*}[Informal version of \Cref{thm:main}]
Consider the GLM in \Cref{eqn:model-main} under a general Gaussian design with covariance 
$\Sigma/n$. 
Assume $ n,d\to\infty $ with $ n/d\to\delta\in(0,\infty) $. 
Let $\ol{\Sigma}$ be a random variable whose law is the limiting eigenvalue distribution of $\Sigma$. 
Fix $ \cT\colon\bbR\to\bbR $ and let $ \beta^{\spec} $ denote the leading eigenvector of the matrix $D$ 
defined in \Cref{eqn:def_D_intro}. 
Then, there exist computable scalars $ F(\delta, \ol{\Sigma}, \cT), \lambda_1(\delta, \ol{\Sigma}, \cT),\lambda_2(\delta, \ol{\Sigma},\cT), \eta(\delta, \ol{\Sigma}, \cT) $ such that the following holds.
If $ F(\delta, \ol{\Sigma}, \cT) > 0 $, then:
\begin{enumerate}
    \item The limits of the top two eigenvalues of $D$ equal $ \lambda_1(\delta, \ol{\Sigma}, \cT) > \lambda_2(\delta, \ol{\Sigma}, \cT) $, respectively; and 

    \item $ \frac{\abs{\inprod{\beta^{\spec}}{\beta^*}}}{\normtwo{\beta^{\spec}} \normtwo{\beta^*}} \to \eta(\delta, \ol{\Sigma}, \cT) > 0 $. 
\end{enumerate}
\end{theorem*}

The performance characterization of spectral estimators provided by our main result opens the way to their principled optimization. In \Cref{subsec:opt}, we optimize the preprocessing $\mathcal T$ towards minimizing the spectral threshold.  A remarkable feature of the optimal preprocessing is that it depends on the covariance matrix $\Sigma$ of the 
design only through its normalized trace. 
In other words, it is \emph{universally optimal over any covariance structure} with fixed trace. 
An important practical implication is that to apply the optimal spectral estimator, only the normalized trace $ \frac{1}{d} \tr(\Sigma)$ needs to be estimated, instead of the whole matrix $\Sigma$. 
In the proportional regime, 
the scalar $ \frac{1}{d} \tr(\Sigma) $ can be estimated consistently using a simple plug-in estimator involving the sample covariance matrix. In contrast, consistent estimation of $\Sigma$ typically requires a sample size larger than that needed by the spectral estimator itself,
see \Cref{rk:mild_depend_on_Sigma} for details. %
Our result on the optimal spectral threshold also resolves in part a conjecture in \cite{maillard2020construction} on optimal spectral methods for rotationally invariant designs; see \Cref{subsec:optrot}. 

The criticality condition $ F(\delta, \ol{\Sigma}, \cT) > 0 $ does not depend on the data and can be easily checked numerically. 
Whenever the condition holds, our results  imply that 
\emph{(i)} the top eigenvalue is detached from the bulk of the spectrum of $D$, hence constituting an outlier, and
\emph{(ii)} the spectral estimator attains strictly positive asymptotic overlap.
We conjecture that $ F(\delta, \ol{\Sigma}, \cT) > 0 $ is in fact necessary to achieve positive overlap, see \Cref{rk:phase_trans}.

\subsection{Technical ideas} Our goal is to characterize 
top eigenvector and top two eigenvalues of  the matrix $D$ in \Cref{eqn:def_D_intro},  which can be expressed as $X^\top T X$, with $ X = \matrix{x_1, & \cdots, & x_n}^\top \in\bbR^{n\times d} $ and $T = \diag(\cT(y)) \in \bbR^{n\times n} $. 
From the 
analysis 
for i.i.d.\ Gaussian designs \cite{LuLi,mondelli-montanari-2018-fundamental}, we expect that the dependence between $T$ and $X$ will, under a suitable criticality condition, lead to an outlier eigenvalue in the spectrum of $D$, and when this happens, 
the corresponding eigenvector (i.e., the spectral estimator) has non-zero overlap with 
$\beta^*$. Note that
\begin{align}
    D=X^\top T X = \Sigma^{1/2} \wt{X}^\top T \wt{X} \Sigma^{1/2},
    \label{eq:D_alt}
\end{align} where $\wt{X}\in\bbR^{n\times d} $ has i.i.d.\ $ \cN(0,1/n) $ entries. If $T$ were independent of $X$, then $D$ would be a 
\emph{spiked separable covariance matrix} recently studied in \cite{DingYang_SpikedSeparableCov}. However, in the GLM setting, $y$ (and, thus, $T$) depends on $X$ via the 1-dimensional projection $ X\beta^* $, so 
results from \cite{DingYang_SpikedSeparableCov} cannot be applied. Indeed, to the best of our knowledge, there is no off-the-shelf result in random matrix theory giving 
spectral information on $D$.  Existing techniques for i.i.d.\ Gaussian designs \cite{LuLi,mondelli-montanari-2018-fundamental} also seem difficult to extrapolate as $X$ is not isotropic. 
\label{par:spikesep}


To overcome these difficulties, we propose a novel proof strategy using the theory of approximate message passing (AMP). 
 Specifically, AMP refers to a family of iterative algorithms that are specified by a sequence of `denoising' functions.
A key feature of AMP is the presence of a memory term, which debiases the iterates, ensuring that their joint empirical distribution is asymptotically Gaussian. This in turn allows to track their covariance structure via a low-dimensional recursion known as \emph{state evolution} \cite{BM-MPCS-2011,bolthausen2014iterative}. 
Our key idea is to simulate a power iteration using AMP: via a judicious choice of denoisers, we ensure that the  AMP recursion, once executed for a sufficiently large number of steps, approximates an eigenequation of $D$. Then, we leverage state evolution to:
\begin{itemize}
    \item identify 
     the location of the outliers in the spectrum of $D$, by controlling the $\ell_2$-norm of the iterates of AMP, and 
    \item establish the limiting correlation between the top eigenvector of $D$ and $\beta^*$, by characterizing the inner product of the iterates with the parameter vector $\beta^*$. 
\end{itemize}



The idea of using AMP to simulate an algorithm whose output is aligned with the estimator of interest has been used to characterize the asymptotic performance in many settings \cite{Donoho_Montanari_M_estimation, Bu_SLOPE,Bu_SLOPE_AoS,Rush_LASSO,LiWei_ell1, Sur_logistic}. We highlight that,  
for the study of spectral estimators for GLMs, all previous works using AMP as a proof technique \cite{mondelli2021optimalcombination,mondelli-2021-amp-spec-glm,mixed-zmv-arxiv} require precise knowledge of when a spectral gap emerges. For the settings considered in those works,  complete characterizations of the spectrum are available via known results from random matrix theory.  
This is however not the case for our setting with a correlated Gaussian design. 
In this work, we exploit random matrix theory tools  for studying the right edge of the bulk.
The fundamental novelty of our approach is that the more challenging task of locating the spike is accomplished by AMP.

\subsection{Related work}
\label{sec:related-work}

\paragraph{Spectral methods.}
Spectral methods 
find applications in various domains across statistics and data science \cite{ChenChiFanMa_spec_book} and, as discussed earlier,  
the spectrally-initialized optimization paradigm is widely employed for estimation from GLMs and their variants. 
Beyond GLMs, other applications include
community detection \cite{abbe2017community},
clustering \cite{ng2001spectral}, angular synchronization in cryo-EM \cite{singer2011angular}, inference of low-rank matrices \cite{montanari2017estimation} and tensors \cite{MR_tensor_PCA}.

\paragraph{Approximate message passing.}
Approximate message passing algorithms were first proposed for linear regression \cite{Kabashima_2003,DMM09,krzakala2012}, and 
have since been applied to several statistical estimation 
problems, including parameter recovery in a GLM  \cite{BKMMZ_PNAS,
rangan_GAMP,
Sur_logistic,rangan2019vector, venkataramanan2021estimation}; 
see the review \cite{amp-tutorial} and references therein. In this paper, AMP is used solely as a tool for analyzing spectral estimators. 
Following \cite{mondelli-2021-amp-spec-glm,CelentanoMontanariWu_GFOM, MontanariWu_GFOM}, we expect that our results can be used to analyze general first order iterative methods (including AMP itself) with spectral initialization. 
An alternative 
way to initialize first order methods is via random initialization.
A recent line of work \cite{LiWei_nonasymp,LiFanWei_Z2} analyzes AMP with spectral and random initializations in the context of symmetric rank-$1$ matrix estimation, by 
 establishing a non-asymptotic state evolution result. A different non-asymptotic analysis of AMP, leveraging a leave-one-out approach, was recently put forward  in \cite{bao2023leave}. 

\paragraph{Random matrix theory.}
The separable covariance matrix model \cite{PaulSilv,CouilletHachem,Yang2019} and its spiked counterpart \cite{DingYang_SpikedSeparableCov,DingYang_TW} are related to the matrix $D$ that we study, but as discussed earlier, the results in these papers cannot be applied to GLMs with correlated designs. 
A related (and more general) model is considered in \cite{liao2021hessian}, where potential outlier eigenvalues/eigenvectors are identified via a deterministic equivalent of the resolvent. 
However, \cite{liao2021hessian} provides no explicit condition under which these outliers indeed emerge. 
In comparison, our result locates both the right edge of the spectral bulk and the outlier eigenvalue, yielding an almost sure characterization. 
Our approach has the advantage of rendering itself ready for initializing iterative procedures. 

\paragraph{{Label transformation and generative exponent.}}
The preprocessing function in our work corresponds to the label transformation technique used in \cite{ChenMeka_poly,Damian_etal}. In fact, the thresholding filter in \cite{ChenMeka_poly} is the same as the subset scheme proposed in \cite{Wang_subset}, and it is a special case of the spectral estimators considered in our paper with preprocessing function $ \cT^{\textnormal{subset}}(y) = \indicator{\abs{y} \ge K_{\textnormal{subset}}} $. 
Damian et al.\ \cite{Damian_etal} extend the analysis to tensor estimators that provide weak recovery guarantees for $n$ super-linear in $d$ (i.e., not in the proportional regime considered in this work). This handles settings where the optimal spectral threshold identified in \Cref{prop:opt_thr} (or \cite[Theorem 2]{mondelli-montanari-2018-fundamental}) is infinity, which points to the need of having $n$ that grows faster than $d$.
The insight of \cite{Damian_etal} is that applying the optimal preprocessing function $ \cT^* $ lowers the information exponent of $q$ to its generative exponent, which equals the information exponent of the functional composition of $ q $ with $ \cT^* $.

\section{Preliminaries}
\label{sec:prelim}

\subsection{Generalized linear models with general Gaussian designs}
\label{sec:model}

Recall that the goal is to estimate the  parameter vector $\beta^*\in \mathbb R^d$ from  observations obtained via the model in \Cref{eqn:model-main}. We write $ y = q(X \beta^*, \eps) \in \mathbb R^n$ for the observation vector, with the link function $q$ acting component-wise on its inputs. We make the following assumptions on the model: 

\begin{enumerate}[label=(A\arabic*)]
\setcounter{enumi}{\value{asmpctr}}

    \item \label[asmp]{asmp:signal_prior} 
    $ \beta^*\sim P^{\ot d} $,\footnote{For a tuple of distributions $ P_1, \cdots, P_k $, $P_1\ot \cdots \ot P_k$ denotes the product distribution with $P_i$ being its $i$-th marginal. 
    If all $ P_i $'s are equal to $P$, we use the notation $ P^{\ot k} $. } where $P$ is a 
    distribution on $\bbR$ with mean $0$ and variance $1$. 

    \item \label[asmp]{asmp:sigma} For $ 1\le i\le n $, $ x_i\iid\cN(0_d, \Sigma/n) $ independent of $ \beta^* $, where $ \Sigma\in\bbR^{d\times d} $ 
is deterministic and strictly positive definite with empirical spectral distribution\footnote{
The empirical spectral distribution of a $p\times p$ matrix is a probability measure that assigns weight $1/p$ to a Dirac mass supported at each of the eigenvalues. 
} converging weakly to the law of a random variable $ \ol{\Sigma} $ compactly supported on $ (0,\infty) $. 
    The spectral norm $\normtwo{\Sigma}$ is uniformly bounded over $d$ and, 
    for all $ \varsigma>0 $, there exists $d_0\in\bbN$ such that for all $d\ge d_0$, 
    \begin{align}
        \supp(\mu_{\Sigma}) \subset \supp(\ol{\mu}_{\Sigma}) + [-\varsigma, \varsigma] , 
        \label{eqn:sigma_no_outlier}
    \end{align}
    where $ \mu_{\Sigma} $ and $ \ol{\mu}_{\Sigma} $ denote respectively the empirical and limiting spectral distributions of $ \Sigma $, $\supp$ denotes their support and `$+$' denotes the Minkowski sum.

    \item \label[asmp]{asmp:noise} $ \eps = (\eps_1, \cdots, \eps_n)\in\bbR^n $ is independent of $ (\beta^*, X) $ and has empirical distribution converging in probability in Wasserstein-$2$ distance to $ P_\eps $ which is a distribution on $ \bbR $ with bounded second moment. 

    \item \label[asmp]{asmp:proportional} We work in the proportional regime where $n, d\to\infty$ with $ n/d\to\delta $ for some $ \delta\in(0,\infty) $. 

\setcounter{asmpctr}{\value{enumi}}
\end{enumerate}

\Cref{asmp:signal_prior} specifies an i.i.d.\ prior distribution on the unknown parameter $\beta^*$. 
We remark that our analysis carries over to $ \beta^* \sim \unif(\sqrt{d}\,\bbS^{d-1}) $ (where $ \bbS^{d-1} $ denotes the unit sphere in dimension $d$), giving 
the same results as for $ P = \cN(0,1) $. 
Spectral estimators are unable to exploit any prior structure in the parameter vector, since the eigenvectors of the spectral matrix are not a priori guaranteed to obey structures (e.g., binary, sparse or conic)  that may be enjoyed by the parameter.
In fact, our results are universal with respect to $P$. 
We leave it for future work to perform parameter estimation with prior information taken into account. 
{Furthermore, our results can be extended to the setting where $ \beta^* $ has non-i.i.d.\mbox{\ }prior and in particular can align with eigenvectors of $ \Sigma $. See \mbox{\Cref{rk:noniid}} for the required modifications for such adaptation. }

The \emph{general} Gaussian design in \Cref{asmp:sigma} constitutes the major challenge of this work. 
    We highlight that no distributional assumption is imposed on the matrix $\Sigma$: 
this in particular means that $X$ is only \emph{left} rotationally invariant in law. 
As such, the model falls out of the bi-rotationally invariant ensemble which has recently attracted a flurry of research \cite{fan2020approximate, venkataramanan2021estimation, WangZhongFan_Universality, maillard2020construction, cademartori2023nonasymptotic}.
The requirement of strict positive definiteness of $\Sigma$ could 
be relaxed to positive semidefiniteness with the modification in the proof that 
$\Sigma^{-1}$ is replaced with the pseudoinverse $\Sigma^+$  and $\ol{\Sigma}$ is replaced with a proper mixture of $\delta_0$ (where $ \delta_\lambda $ is the Dirac delta measure at $ \lambda\in\bbR $) and a certain absolutely continuous (with respect to the Lebesgue measure) probability measure. 
The assumption on uniform boundedness of $\normtwo{\Sigma}$ is 
 technical and is satisfied by many natural covariance structures used in practice, such as Toeplitz or circulant. 
The condition \Cref{eqn:sigma_no_outlier} excludes outlier eigenvalues from the spectrum of $\Sigma$. 
Otherwise, it is known that spikes in $\Sigma$ will result in spikes in 
$ D $ 
 \cite{DingYang_SpikedSeparableCov,BBCF,DingJi_SpikedMultiplicative}. These additional spikes are undesirable from an inference perspective, as they may be confused with the one contributed by the unknown parameter $\beta^*$.

The proportionality between parameter dimension $d$ and sample size $n$ in \Cref{asmp:proportional} is a natural 
 scaling since 
the spectral estimator starts being 
correlated with $\beta^*$ in this regime. 

\subsection{Spectral estimator}
\label{sec:spec_estimator}
The spectral estimator is defined as
\begin{align}
    \beta^{\spec}(y, X) &\coloneqq v_1\paren{D } \in \bbS^{d-1} , \label{eqn:def_xspec} 
\end{align}
where 
$ v_1(\cdot) $ denotes the principal eigenvector. 
We also define random variables
\begin{align}
(\ol{G}, \ol{\eps}) &\sim \cN\paren{0, \frac{1}{\delta}\expt{\ol{\Sigma}}} \ot P_\eps , \quad 
\ol{Y} = q(\ol{G}, \ol{\eps}) , \label{eqn:rand-var}
\end{align}
{and an auxiliary function $ \cF_a \colon \bbR \to \bbR $ (for any $ a > \sup\supp(\cT(\ol{Y})) $):} 
\begin{align}
    \cF_a(\cdot) &\coloneqq \frac{\cT(\cdot)}{a - \cT(\cdot)} . \label{eqn:cF_a}
\end{align}

We make the following assumption on the preprocessing function: 
\begin{enumerate}[label=(A\arabic*)]
\setcounter{enumi}{\value{asmpctr}}

    \item \label[asmp]{asmp:preprocessor}
    $ \cT\colon\bbR\to\bbR $ is bounded and satisfies: 
    \begin{align}
        \sup_{y\in\supp(\ol{Y})} \cT(y) > 0. 
        \label{eqn:asmp_T} 
    \end{align}
    Furthermore, $\cT$ is pseudo-Lipschitz of finite order, i.e., there exist  $j$ and $L$ such that \ 
    \begin{align}
        \abs{\cT(x) - \cT(y)} &\le L \abs{x - y} \paren{1 + \abs{x}^{j-1} + \abs{y}^{j-1}}, \qquad \mbox{for all }x, y. 
        \label{eq:PLdef0}
    \end{align}

\setcounter{asmpctr}{\value{enumi}}
\end{enumerate}

The condition in \Cref{eqn:asmp_T} is rather mild: it is satisfied by the optimal preprocessing function (see \Cref{prop:opt_thr}), and it is also required by prior work for $\Sigma = I_d$ \cite{mondelli-montanari-2018-fundamental,lal-opt-spec}. 

Finally, we single out two technical conditions that guarantee the well-posedness of the auxiliary quantities appearing in the statement of our main result, \Cref{thm:main}. 


\begin{enumerate}[label=(A\arabic*)]
\setcounter{enumi}{\value{asmpctr}}

    \item \label[asmp]{asmp:sigma_technical} 
    For any $x\ne0$, let 
    \begin{align}
        \vartheta \coloneqq \begin{cases}
            x \cdot (\sup\supp(\ol{\Sigma})) , & x>0 \\
            x \cdot (\inf\supp(\ol{\Sigma})) , & x<0 
        \end{cases} , \notag
    \end{align}
    where we use $ \supp(\ol{\Sigma}) $ to denote the support of the density function of $\ol{\Sigma}$. 
    Then for any $x\ne0$, the random variable $ \ol{\Sigma} $ satisfies
    \begin{align}
        \lim_{\gamma\searrow \vartheta} \expt{\frac{\ol{\Sigma}}{\gamma - x \ol{\Sigma}}}
        \stackrel{(a)}{=} \lim_{\gamma\searrow \vartheta} \expt{\frac{\ol{\Sigma}^2}{\paren{\gamma - x \ol{\Sigma}}^2}}
        \stackrel{(b)}{=} \lim_{\gamma\searrow \vartheta} \expt{\frac{\ol{\Sigma}^3}{\paren{\gamma - x \ol{\Sigma}}^2}}
        \stackrel{(c)}{=} \infty . \label{eqn:asmp_gamma_lim} 
    \end{align}



    \item \label[asmp]{asmp:preprocessor_technical} 
    The function $ \cT $ satisfies
    \begin{align}
        \lim_{a\searrow \, \sup\supp(\cT(\ol{Y}))} \expt{\cF_a(\ol{Y})} \stackrel{(d)}{=}
        \lim_{a\searrow \,  \sup\supp(\cT(\ol{Y}))} \expt{\ol{G}^2 \cF_a(\ol{Y})} \stackrel{(e)}{=} \infty . \label{eqn:asmp_a_lim}
    \end{align}

\setcounter{asmpctr}{\value{enumi}}
\end{enumerate}

{\mbox{\Cref{asmp:sigma_technical,asmp:preprocessor_technical}} are mild and common in related work. 
In fact, \mbox{\Cref{asmp:preprocessor_technical}} appeared in a similar form in \mbox{\cite[(A.5)]{LuLi}}, and it is common in the random matrix theory literature as well (see, e.g., Assumption 6 ``Thickness of the bulk edge'' on page 129 of \mbox{\cite{Donoho_Gavish_Romanov}}). 
\mbox{\Cref{asmp:sigma_technical}} is similar to \mbox{\Cref{asmp:preprocessor_technical}}, but instead imposed on $ \ol{\Sigma} $. Note that \mbox{\Cref{asmp:sigma_technical,asmp:preprocessor_technical}} do not a priori impose any dependence of $ \cT $ on $q$ or $ \ol{\Sigma} $.}
We remark that these two conditions can be removed, at the price of a slightly more involved definition of such auxiliary quantities; see \Cref{rk:remove}.  


\section{Main results}
\label{sec:main-result}

Our main contribution, 
 \Cref{thm:main}, gives a precise asymptotic characterization of the overlap between the leading eigenvector of $D$ and the unknown parameter, provided a criticality condition is satisfied. This condition ensures that $D$ has a spectral gap in the high-dimensional limit. \Cref{thm:main} also gives 
 asymptotic formulas for the location of the right edge of the bulk and for  the (right) outlier eigenvalue of $D$. 



To state the results, 
 we require some 
definitions. 
For $ a\in(\sup\supp(\cT(\ol{Y})), \infty) $, let
\begin{align}
    s(a) \coloneqq \begin{cases}
        (\sup\supp(\ol{\Sigma})) \expt{\cF_a(\ol{Y})} , & \expt{\cF_a(\ol{Y})}>0 \\
        (\inf\supp(\ol{\Sigma})) \expt{\cF_a(\ol{Y})} , & \expt{\cF_a(\ol{Y})}<0 \\
        0 , & \expt{\cF_a(\ol{Y})}=0
    \end{cases} . \label{eqn:def-s(a)}
\end{align}
Note that $ s(a) $ also depends on $ \ol{\Sigma} $ and $\cT$. 
%
For $ a > \sup\supp(\cT(\ol{Y})) $, define the function
\begin{align}
    \phi(a) &= \frac{a}{\expt{\ol{\Sigma}} } \expt{ \ol{G}^2 \cF_a(\ol{Y}) } \expt{\frac{\ol{\Sigma}^2}{\gamma(a) - \expt{\cF_a(\ol{Y})} \ol{\Sigma}}} , \qquad
    \psi(a) = a \gamma(a) , \label{eqn:def-phi-psi} 
\end{align}
where $ \gamma(a) $ is an implicit function of $a$ 
given by the unique solution in $ \paren{s(a) , \infty} $ to 
\begin{align}
    1 = \frac{1}{\delta} \expt{\frac{\ol{\Sigma}}{\gamma(a) - \expt{\cF_a(\ol{Y})}\ol{\Sigma}}} . \label{eqn:def-gamma-fn} 
\end{align}
To see existence and uniqueness of the solution, note that for any $ a > \sup\supp(\cT(\ol{Y})) $ s.t.\ $ \expt{\cF_a(\ol{Y})} \ne 0 $, $ \frac{1}{\delta} \expt{\frac{\ol{\Sigma}}{\gamma - \expt{\cF_a(\ol{Y})}\ol{\Sigma}}} $ is a strictly decreasing (since $ \ol{\Sigma} $ is strictly positive) function of $ \gamma $ which approaches $ \infty $ as $ \gamma\searrow s(a) $ (see \emph{(a)} in \Cref{eqn:asmp_gamma_lim}) and approaches $0$ as $ \gamma\nearrow\infty $. 
If $ \expt{\cF_a(\ol{Y})} = 0 $, the solution $ \gamma(a) = \frac{1}{\delta} \expt{\ol{\Sigma}} > 0 $ is obviously unique. 

Next, using $ \psi $ and $ \phi $, we define two parameters $ a^*, a^\circ $ that govern the validity of our spectral characterization. 
It can be shown (see \Cref{lem:a0}) that $\psi$ is differentiable and has at least one critical point. 
Let $ a^\circ > \sup\supp(\cT(\ol{Y})) $ be the largest solution to 
\begin{align}
    \psi'(a^\circ) &= 0 . 
    \label{eqn:def_a_circ} 
\end{align}
We then define $ \zeta \colon (\sup\supp(\cT(\ol{Y})), \infty)\to\bbR $ 
by flattening $\psi$ to the left of $ a^\circ $: 
\begin{align}
    \zeta(a) &\coloneqq \psi(\max\{a, a^\circ\}) . \label{eqn:def-zeta} 
\end{align}
Finally, let $ a^* $ be the largest solution in $ (\sup\supp(\cT(\ol{Y})), \infty) $ to the following equation:
\begin{align}
    \zeta(a^*) = \phi(a^*) . 
    \label{eqn:def_a*}
\end{align}
\Cref{prop:exist_a*} shows that such a solution must exist. 
The functions $\phi, \psi, \zeta$ are plotted in \Cref{fig:phi_psi_zeta} for two examples of covariance matrix $\Sigma$. 

\begin{figure}[t]
    \centering
    \begin{subfigure}{0.4\linewidth}
        \centering
        \includegraphics[width = .9\linewidth]{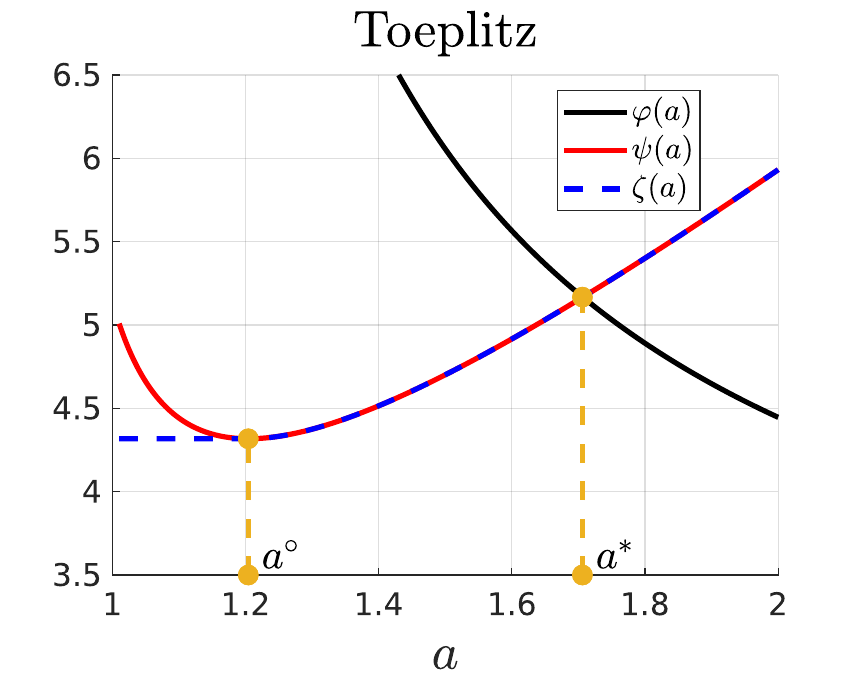}
        \caption{The Toeplitz case with $ \delta = 0.2 $. }
        \label{fig:phi_psi_zeta_toeplitz}
    \end{subfigure}
    \begin{subfigure}{0.4\linewidth}
        \centering
        \includegraphics[width = .9\linewidth]{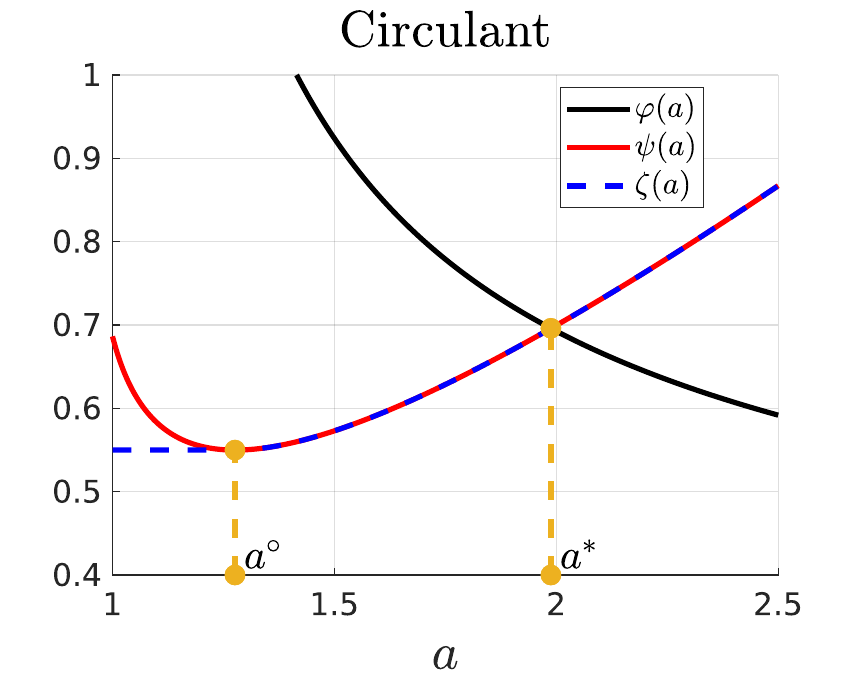}
        \caption{The circulant case with $ \delta = 1.5 $. }
        \label{fig:phi_psi_zeta_circulant}
    \end{subfigure}
    \caption{Plots of the functions $ \phi, \psi, \zeta \colon (\sup\supp(\cT^*(\ol{Y})), \infty) \to \infty $ defined in 
     \Cref{eqn:def-phi-psi,eqn:def-zeta} 
    with $ \cT^* $ obtained by truncating the optimal preprocessing 
    and $ \ol{\Sigma} $ given by the Toeplitz or circulant matrices
    (see \Cref{sec:experiments_syn} for details). 
    }
    \label{fig:phi_psi_zeta}
\end{figure}

Then, the limits of the top two eigenvalues of $D$ are given by  
\begin{align}
    \lambda_1 \coloneqq \zeta(a^*) , \quad 
    \lambda_2 \coloneqq \zeta(a^\circ) , 
    \label{eqn:def_lambda1_lambda2}
\end{align}
and the asymptotic overlap admits the following explicit expression: 
\begin{align}
    \eta &\coloneqq \paren{ \frac{(1 - w_2) \expt{\frac{\ol{\Sigma}}{\gamma(a^*) - \expt{\cF_{a^*}(\ol{Y})} \ol{\Sigma}}}^2}{(1 - w_2) \expt{\frac{\ol{\Sigma}^2}{\paren{ \gamma(a^*) - \expt{\cF_{a^*}(\ol{Y})} \ol{\Sigma} }^2}} + w_1 \expt{\frac{\ol{\Sigma}}{\paren{ \gamma(a^*) - \expt{\cF_{a^*}(\ol{Y})} \ol{\Sigma} }^2}}} }^{1/2} , \label{eqn:def-eta} 
\end{align}
where the function $ \cF_{a^*}(\cdot) $ is defined in \mbox{\Cref{eqn:cF_a}} and the ancillary parameters $ w_1, w_2 $ are given by: 
\begin{align}
    w_1 &\coloneqq \frac{1}{\delta \expt{\ol{\Sigma}}} \expt{\paren{\frac{\delta}{\expt{\ol{\Sigma}}} \ol{G}^2 - 1} \cF_{a^*}(\ol{Y})^2} \expt{\frac{\ol{\Sigma}^2}{\gamma(a^*) - \expt{ \cF_{a^*}(\ol{Y}) } \ol{\Sigma}}}^2 \notag \\
    &\phantom{=}~ + \frac{1}{\delta} \expt{\cF_{a^*}(\ol{Y})^2} \expt{\frac{\ol{\Sigma}^3}{\paren{ \gamma(a^*) - \expt{ \cF_{a^*}(\ol{Y}) } \ol{\Sigma} }^2}} , \label{eqn:def_x1_main} \\
    w_2 &\coloneqq \frac{1}{\delta} \expt{\cF_{a^*}(\ol{Y})^2} \expt{\frac{\ol{\Sigma}^2}{\paren{ \gamma(a^*) - \expt{\cF_{a^*}(\ol{Y})} \ol{\Sigma} }^2}} . \label{eqn:def_x2_main}
\end{align}
We note that, given $ a^* > a^\circ $, $ \eta $ is well-defined as the fraction under the square root is strictly positive. 
This is because \emph{(i)} all three expectations in \Cref{eqn:def-eta} are positive ($ \ol{\Sigma} > 0 $ in \Cref{asmp:sigma} and $ \gamma(a^*) > s(a^*) $); \emph{(ii)} $ w_1 > 0 $ (see \Cref{prop:x1>0}); \emph{(iii)} $ 1-w_2>0 $ if $ a^* > a^\circ $ (see \Cref{itm:threshold4} of \Cref{prop:equiv-threshold}). 

We are now ready to state our main result, whose proof is given in \Cref{sec:heuristics}, with several details deferred to 
\Cref{sec:pf-AMP-all}.

\begin{theorem}[Performance characterization of spectral estimator]
\label{thm:main}
Consider the setting of \Cref{sec:prelim} and let \Cref{asmp:signal_prior,asmp:sigma,asmp:noise,asmp:proportional,asmp:preprocessor,asmp:sigma_technical,asmp:preprocessor_technical} hold. 
Suppose $ a^* > a^\circ $. 
Then, the top two eigenvalues $ \lambda_1(D), \lambda_2(D) $ of $D$ satisfy\footnote{For a symmetric matrix $ M\in\bbR^{p\times p} $, we write its (real) eigenvalues as $ \lambda_1(M) \ge \cdots \ge \lambda_p(M) $ and the associated eigenvectors (normalized to have unit $\ell_2$-norm) as $ v_1(M), \cdots, v_p(M)\in\bbS^{p-1} $.}
\begin{align}
    \plim_{d\to\infty} \lambda_1(D) &= \lambda_1 , \qquad 
    \lim_{d\to\infty} \lambda_2(D) = \lambda_2 \quad \text{almost surely} , 
    \label{eqn:main_thm_eigval}
\end{align}
and $ \lambda_1 > \lambda_2 $, where $\plim$ denotes the limit in probability. 
Furthermore, the limiting overlap between the top eigenvector $ v_1(D) $ and $ \beta^* $ equals
\begin{align}
    \plim_{d\to\infty} \frac{\abs{\inprod{v_1(D)}{\beta^*}}}{\normtwo{\beta^*}} &= \eta > 0 . 
    \label{eqn:main_thm_overlap}
\end{align}
\end{theorem}

\begin{remark}[Uniqueness of $ a^* $]
\label{rk:unique_a*}
Recall that the parameter $ a^* $ is 
the largest solution in $ (\sup\supp(\cT(\ol{Y})), \infty) $ to \Cref{eqn:def_a*}. 
With additional assumptions, we can show that \Cref{eqn:def_a*} admits a unique solution; see \Cref{prop:unique} for details. 
We expect that the additional assumptions can be removed and the solution to \Cref{eqn:def_a*} in $ (\sup\supp(\cT(\ol{Y})), \infty) $ always exists and is unique. 
\end{remark}


\begin{remark}[Consistency with isotropic covariance]
We note that, by setting $\Sigma=I_d$, we recover the existing result on i.i.d.\ Gaussian designs (i.e., Lemma 2 in \cite{mondelli-montanari-2018-fundamental}).
\end{remark}

\begin{remark}[Removing \Cref{asmp:sigma_technical,asmp:preprocessor_technical}]
\label{rk:remove} 
\Cref{asmp:sigma_technical} requires $ \law(\ol{\Sigma}) $ to have sufficiently slow decay on both the left and right edges, and \Cref{asmp:preprocessor_technical} requires such behaviour on the right edge of $ \law(\cT(\ol{Y})) $. 
However, both assumptions 
can be removed at the cost of a vanishing perturbation of $ \ol{\Sigma}, \cT $ around their edges in the definitions of $ \lambda_1, \lambda_2, \eta $ in \Cref{eqn:def_lambda1_lambda2,eqn:def-eta}.  
The perturbed quantities, denoted by $ \lambda_1', \lambda_2', \eta' $, are guaranteed to satisfy both assumptions. 
Hence, \Cref{thm:main} ensures that the high-dimensional limits of the top two eigenvalues and of the overlap for the perturbed matrix $D'$ are given by $ \lambda_1',\lambda_2',\eta' $, respectively. 
An application of the Davis--Kahan theorem \cite{Davis_Kahan} shows that, as the perturbation vanishes, the top two eigenvalues and overlap obtained with $D'$ coincide with those given by the unperturbed matrix $D$. 
Furthermore, since $ \lambda_1', \lambda_2', \eta' $ are continuous with respect to the perturbation, their limits as the perturbation vanishes exist. 
Therefore, the latter limits must equal the high-dimensional limits of the top two eigenvalues and overlap given by the original $D$. The formal argument is deferred to \Cref{sec:remove_asmp_sigma_tech}, and  by a similar argument, \Cref{asmp:sigma_technical,asmp:preprocessor_technical} in \Cref{prop:opt_thr} below can be removed as well. 
\end{remark}

\begin{remark}[Phase transition]
\label{rk:phase_trans}
Our characterization of the outlier eigenvalue and the overlap is valid given an explicit and checkable condition {$ a^* > a^\circ $} not depending on the data $ (y, X) $. 
Informally, it guarantees that the aspect ratio $\delta$ exceeds a certain threshold which leads to a spike in $D$. 
We conjecture that this condition is in fact necessary, in the sense that otherwise the spectral estimator fails to achieve a positive limiting overlap and the top eigenvalue sticks to the bulk of the spectrum of $D$. 
It is easy to verify that $ \lambda_1 = \lambda_2 $ and $ \eta = 0 $ precisely when $ a^* = a^\circ $, indicating a continuous phase transition at the conjectured threshold. 
\end{remark}

\begin{remark}[{Non-i.i.d.\mbox{\ }prior}]
\label{rk:noniid}
We expect that all results in the paper can be extended to the more general setting in which $ \beta^*$ may be asymptotically aligned with the eigenvectors of $\Sigma$.
We describe the required modifications below.
Suppose that $ \beta^* \in \bbR^d $ is such that $ \plim\limits_{d\to\infty} d^{-1} \normtwo{\beta^*}^2 = 1 $ and the following empirical probability measure admits a weak limit in probability: 
\begin{align}
    \varrho &\coloneqq \sum_{i = 1}^d \frac{\inprod{\beta^*}{v_i(\Sigma)}^2}{\normtwo{\beta^*}^2} \delta_{\lambda_i(\Sigma)}
    \xRightarrow[d\to\infty]{\bbP} \ol{\varrho} . \notag 
\end{align}
The measure $ \varrho $ records the alignment between $ \beta^* $ and each of the eigenvectors of $ \Sigma $. 
Note that in the special case where $ \beta^* $ has i.i.d.\ entries (with mean $0$ and variance $1$) considered in the paper, $ \ol{\varrho} $ equals the law of $ \ol{\Sigma} $. 
This can be seen by examining the convergence of the Stieltjes transform of $ \varrho $.  
In the state evolution analysis, whenever the limit of $ d^{-1} \expt{ \inprod{\beta^*}{f(\Sigma) \beta^*}} $ (for some $ f\colon\bbR \to \bbR $ that applies to $ \Sigma $ according to functional calculus) needs to be computed, we would write 
\begin{align}
    \lim_{d\to\infty} \frac{1}{d} \expt{ \inprod{\beta^*}{f(\Sigma) \beta^*} }
    &= \lim_{d\to\infty} \expt{ \frac{\normtwo{\beta^*}^2}{d} \sum_{i = 1}^d \frac{\inprod{\beta^*}{v_i(\Sigma)}^2}{\normtwo{\beta^*}^2} f(\lambda_i(\Sigma)) } \notag \\
    &= \lim_{d\to\infty} \expt{ \frac{\normtwo{\beta^*}^2}{d} \int f(\lambda) \,\varrho(\dd\lambda) } \notag \\
    &= \int f(\lambda) \,\ol{\varrho}(\dd\lambda) . \notag 
\end{align}
We expect that such modifications can allow us to obtain analogous results for general $ \beta^* $ potentially correlated with eigenvectors of $ \Sigma $. 
\end{remark}

\begin{remark}[{Spectrally initialized AMP}]
\label{rk:spec_init}
Whenever $\delta$ is large enough so that our spectral estimator becomes effective, one can analyze spectrally initialized Bayes-AMP by running AMP with linear denoisers as designed in the paper for a large (but constant with respect to $n, d$) number of steps and then transitioning to Bayes-optimal denoisers. This corresponds to the strategy pursued in \cite{mondelli-2021-amp-spec-glm}. However, we note that Bayes-AMP in general requires the knowledge of $ \Sigma $ which is not assumed to be available in our paper for the design of spectral estimators.
\end{remark}

\subsection{Optimal spectral methods for general Gaussian designs}\label{subsec:opt}

\Cref{thm:main} holds for an arbitrary function $\cT$ subject to mild regularity conditions. This enables the optimization of $\cT$ to minimize the spectral threshold, i.e., the smallest $\delta$ s.t.\ $ a^* > a^\circ $. The result on the optimization of the pre-processing function is stated below and proved in \Cref{sec:opt_thr}.


\begin{theorem}[Optimal spectral threshold]
\label{prop:opt_thr}
Consider the setting of \Cref{sec:prelim}, let \Cref{asmp:signal_prior,asmp:sigma,asmp:noise,asmp:proportional,asmp:sigma_technical} hold, and let $ \sT $ be the set of functions $ \cT\colon\bbR\to\bbR $ satisfying \Cref{asmp:preprocessor,asmp:preprocessor_technical}. 
Then the following two statements hold.
\begin{enumerate}
    \item \label{itm:opt_thr_1} There exists $ \cT\in\sT $ such that $ a^* > a^\circ $ holds if 
    \begin{align}
        \delta &>  \Delta(\delta) := \frac{\expt{\ol{\Sigma}}^2}{\expt{\ol{\Sigma}^2}}
        \paren{ \int_{\supp(\ol{Y})} \frac{\expt{ p(y \,|\, \ol{G}) \paren{\frac{\delta}{\expt{\ol{\Sigma}}} \ol{G}^2 - 1} }^2}{\expt{p(y\,|\,\ol{G})}} \,\diff y }^{-1} , \label{eqn:opt_thr_fp} 
    \end{align}
    with $ p( \ol{Y} \,|\, \ol{G}) $ the conditional density of $\ol{Y} = q(\ol{G}, \ol{\eps})$ given $\ol{G}$, determined via the joint distribution in \Cref{eqn:rand-var}.
    In this case, if 
    \begin{align}
        \cT^*(y) = 1 - \paren{ \sqrt{\frac{\Delta(\delta)}{\delta}} \frac{ \expt{p(y\,|\,\ol{G}) \paren{ \frac{\delta}{\expt{\ol{\Sigma}}} \ol{G}^2 }} }{ \expt{p(y\,|\,\ol{G})} } + 1 - \sqrt{\frac{\Delta(\delta)}{\delta}} }^{-1} 
        \label{eqn:opt_T_main}
    \end{align} 
    is pseudo-Lipschitz of finite order, 
    then the spectral estimator using the preprocessing function $ \cT^* $ achieves strictly positive limiting overlap. 

    \item \label{itm:opt_thr_2} Conversely, suppose that the function $ \phi$ defined in \Cref{eqn:def-phi-psi} is strictly decreasing for every $\cT\in\sT$. If there exists $ \cT\in\sT $ such that $ a^* > a^\circ $, then $\delta$ satisfies \Cref{eqn:opt_thr_fp}.
\end{enumerate}
\end{theorem}

\begin{remark}[Mild dependence of $\cT^*$ on $ \ol{\Sigma} $]
\label{rk:mild_depend_on_Sigma}
The optimal function $ \cT^* $ in \Cref{eqn:opt_T_main} depends on $ \ol{\Sigma} $ only through its first moment, or equivalently it depends on $ \Sigma $ only through its normalized trace. We highlight that approximating $ \frac{1}{d}\tr(\Sigma) $ from the data is significantly easier than approximating the whole matrix $\Sigma$. In fact, $ \frac{1}{d}\tr(\Sigma) $ can be estimated consistently via the plugin estimator $ \frac{1}{d} \tr(X^\top X) $. Specifically, achieving a root mean square error of $ \varsigma $ only requires $ n = \cO(\varsigma^{-2}) $, which is trivially satisfied by \Cref{asmp:proportional}. In contrast, the sample complexity needed to estimate $\Sigma$ with sufficient accuracy may be larger than that required by the spectral estimator itself. Specifically, 
achieving an error of $ \varsigma $ in spectral norm  for the estimation of $ \Sigma $ via the sufficient statistic $ X^\top X $ requires $ n = \Theta(d\varsigma^{-2}) $; see \cite[Exercise VI.15]{WuPolyanskiy_book}, \cite[Section 24.2]{Wu_LectureNote_old}. Note that, to estimate $\Sigma$, $n$ scales linearly with $d$ and the proportionality constant may be larger than the critical value of $\delta$ in the right-hand side of \Cref{eqn:opt_thr_fp}; instead, to estimate $ \frac{1}{d}\tr(\Sigma) $, $n$ does not depend on $d$ and, hence, the estimate is consistent for all $\delta>0$. 
\end{remark}


\begin{remark}[Sufficient condition for $\cT^*$ being pseudo-Lipschitz]
\label{rk:ratio_PL}
The assumption in \Cref{prop:opt_thr} that $\cT^*$ is pseudo-Lipschitz of finite order is satisfied by models that contain an additive component of Gaussian noise (regardless of the variance of the Gaussian noise). This requirement is mild, and common in the related literature, see e.g.\  \cite{BKMMZ_PNAS}.
Specifically, consider the GLM $ y = \wt{q}(X\beta^*, \eps') + \eps'' $, where  $ \wt{q}(X \beta^*, \eps') $ satisfies \Cref{asmp:signal_prior,asmp:sigma,asmp:noise,asmp:proportional,asmp:sigma_technical} and is independent of $ \eps''\sim\cN(0_n, \sigma^2 I_n) $ (for some $\sigma>0$). 
Then, one can verify that 
$ \expt{p(y\,|\,\ol{G}) \paren{ \frac{\delta}{\expt{\ol{\Sigma}}} \ol{G}^2 }} / \expt{p(y\,|\,\ol{G})}$, and hence $ \cT^*(y) $,  
is pseudo-Lipschitz of finite order. 
\end{remark}

\begin{remark}[Monotonicity of $\phi$]
The second part of \Cref{prop:opt_thr} assumes the monotonicity of $\phi$. 
One readily checks that this holds when $ \ol{\Sigma} = 1 $ (i.e., 
$\Sigma = I_d$).
Furthermore, in \Cref{sec:pf_properties_phi_psi}, we prove that $\varphi$ is strictly decreasing for non-negative $\cT$ (\Cref{prop:properties_phi_psi}) and give numerical evidence that the same result holds for general $\cT$ (\Cref{rk:mono-phi}).
\end{remark}

\subsection{Optimal spectral methods for rotationally invariant designs}\label{subsec:optrot}

\Cref{eqn:opt_thr_fp} can be interpreted as giving the optimal spectral threshold, i.e., the minimal $\delta$ above which positive overlap is achievable by some spectral estimator. 
Furthermore, this threshold is attained by $ \cT^* $ in \Cref{eqn:opt_T_main}. As $\delta$ gets close to the spectral threshold $\Delta(\delta)$, $ \cT^* $ approaches the following 
function (obtained by replacing $ \sqrt{\Delta(\delta)/\delta} $ in $ \cT^* $ with $1$):
\begin{align}
    \cT^\star(y) &= 1 - \frac{\expt{p(y \mid \ol{G})}}{\expt{p(y\,|\,\ol{G}) \paren{ \frac{\delta}{\expt{\ol{\Sigma}}} \ol{G}^2 }}}. \label{eqn:opt_T_conj}
\end{align}
When $\Sigma=I_d$, $\cT^\star$ minimizes the spectral threshold \cite{mondelli-montanari-2018-fundamental} and maximizes the 
overlap for any $\delta$ above that threshold \cite{lal-opt-spec}. Supported by evidence from statistical physics, \cite[Conjecture 2]{maillard2020construction} conjectures the optimality to hold for the more general ensemble of right rotationally invariant designs. 
Although our design $X$ is only left rotationally invariant, if the unknown parameter is Gaussian ($ \beta^* \sim \cN(0_d,I_d) $) or uniform on the sphere ($ \beta^*\sim\unif(\sqrt{d}\,\bbS^{d-1}) $), the model in \Cref{eqn:model-main} is equivalent to one with a design that is also right rotationally invariant. 
Therefore, \emph{\Cref{prop:opt_thr} proves \cite[Conjecture 2]{maillard2020construction} for a class of spectral distributions of $X$} -- specifically, those given by the multiplicative free convolution of the Marchenko-Pastur law with a measure compactly supported on $(0, \infty)$.
Formally, with the following two assumptions in place of \Cref{asmp:signal_prior,asmp:sigma}, \Cref{prop:opt_thr} implies \Cref{cor:conj}. 
\begin{enumerate}[label=(A\arabic*)]
\setcounter{enumi}{\value{asmpctr}}
    
    \item \label[asmp]{asmp:right_inv_prior} $ \beta^*\sim\unif(\sqrt{d}\,\bbS^{d-1}) $ or $ \beta^*\sim\cN(0_d,I_d) $. 
    
    \item \label[asmp]{asmp:right_inv_design} $ X = \matrix{x_1 \cdots x_n}^\top\in\bbR^{n\times d} $ can be written as $ X = B Q^\top $, with the rows of $B\in\bbR^{n\times d}$ satisfying \Cref{asmp:sigma} and $ Q \sim\haar(\bbO(d)) $ independent of everything else, where $\bbO(d)$ is the orthogonal group in dimension $d$. 
    
\setcounter{asmpctr}{\value{enumi}}
\end{enumerate}

\begin{corollary}
\label{cor:conj}
Consider the setting of \Cref{sec:prelim} and let \Cref{asmp:right_inv_prior,asmp:right_inv_design,asmp:noise,asmp:proportional,asmp:sigma_technical} hold.
Then, the conclusions of \Cref{prop:opt_thr} hold. 
\end{corollary}

\begin{proof}
By \Cref{asmp:right_inv_design}, 
$ X = \wt{B} \Sigma^{1/2} Q^\top $, 
where $\wt{B} \in \bbR^{n \times d}$ has i.i.d.\ $\cN(0, 1/n)$ entries and $\Sigma \in \bbR^{d \times d}$ is a covariance matrix satisfying \Cref{asmp:sigma}. 
Let
\begin{align}
    D &= X^\top \diag(\cT(q(X \beta^*, \eps))) X
    = Q \Sigma^{1/2} \wt{B}^\top \diag(\cT(q(\wt{B} \Sigma^{1/2} Q^\top \beta^*, \eps))) \wt{B} \Sigma^{1/2} Q^\top  , \notag \\
    \wh{D} &= \Sigma^{1/2} \wt{B}^\top \diag(\cT(q(\wt{B} \Sigma^{1/2} Q^\top \beta^*, \eps))) \wt{B} \Sigma^{1/2}  , \notag \\
    \wt{D} &= \Sigma^{1/2} \wt{B}^\top \diag(\cT(q(\wt{B} \Sigma^{1/2} \beta^*, \eps))) \wt{B} \Sigma^{1/2} . \notag 
\end{align}
Then, we have
\begin{align}
    \frac{\abs{\inprod{v_1(D)}{\beta^*}}}{\normtwo{\beta^*}}
    &= \frac{\abs{\inprod{Q v_1(\wh{D})}{\beta^*}}}{\normtwo{\beta^*}}
    = \frac{\abs{\inprod{v_1(\wh{D})}{Q^\top \beta^*}}}{\normtwo{Q^\top \beta^*}}
    \eqqlaw \frac{\abs{\inprod{v_1(\wt{D})}{\beta^*}}}{\normtwo{\beta^*}} . \label{eqn:rightinv} 
\end{align}
The first equality uses that, if $ (\lambda, v)$ is an eigenpair of $D$, 
then $ (\lambda, Q v) $ is an eigenpair of $ Q D Q^\top $ for $ Q\in\bbO(d) $. 
The second equality holds as $Q$ is orthogonal. 
The third passage follows since by \Cref{asmp:right_inv_prior}, $ \beta^* \eqqlaw Q^\top \beta^* $ for $ Q\sim\haar(\bbO(d)) $ independent of $ \beta^* $. 
Now \Cref{prop:opt_thr} applies to the rightmost side of \Cref{eqn:rightinv},
which completes the proof. 
\end{proof}

\section{Numerical experiments}
\label{sec:experiments}
{We consider noiseless phase retrieval ($y_i=|\langle x_i, \beta^*\rangle|$) in \mbox{\Cref{sec:phase_retrieval}} and Poisson regression ($ y_i \sim \pois(\inprod{x_i}{\beta^*}^2) $) in \mbox{\Cref{sec:poisson_regression}}, and evaluate the performance of the spectral estimator with different preprocessing functions.
In all plots, `sim' and `thy' in legends denote simulation results and theoretical predictions, respectively. }

\subsection{Phase retrieval}
\label{sec:phase_retrieval}

\subsubsection{Synthetic data}
\label{sec:experiments_syn}
%
For all the synthetic experiments, we take the parameter $ \beta^* \sim \unif(\sqrt{d}\,\bbS^{d-1}) $ and $d=2000$. We plot the overlap between the spectral estimator and  $ \beta^*$, as a function of the aspect ratio $\delta$. 
Each value is computed from $10$ i.i.d.\ trials, 
the error bar is at $1$ standard deviation, and the corresponding theoretical predictions are continuous lines with the same color. 
We consider three types of covariance matrix $ \Sigma$: \emph{(i)} Toeplitz covariance, $ \Sigma_{i,j}=\rho^{|i - j|} $ for $ 1\le i,j\le d $ with $ \rho = 0.9 $, as considered in \cite[Section 4]{zhang_zhang_confidence} and \cite[Section 5.3]{javanmard_montanari_debias_lasso}. 
\emph{(ii)} Circulant covariance, $\Sigma_{i,j}=c_0$ for $i=j$, $\Sigma_{i,j}=c_1$ for $i+1\le j\le i+\ell$ and $i+d - \ell \le j\le i+d-1$, $\Sigma_{i,j}=0$ otherwise, with $ c_0 = 1, c_1 = 0.1, \ell = 17 $, as 
 considered in \cite[Section F]{javanmard_montanari_hypo_test} and \cite[Section 5.1]{javanmard_montanari_confidence}. 
    \emph{(iii)} Identity covariance, $ \Sigma = I_d $. 
%

\begin{figure}[t]
    \centering
    \begin{subfigure}{0.583\linewidth}
        \centering
        \includegraphics[width = .92\linewidth]{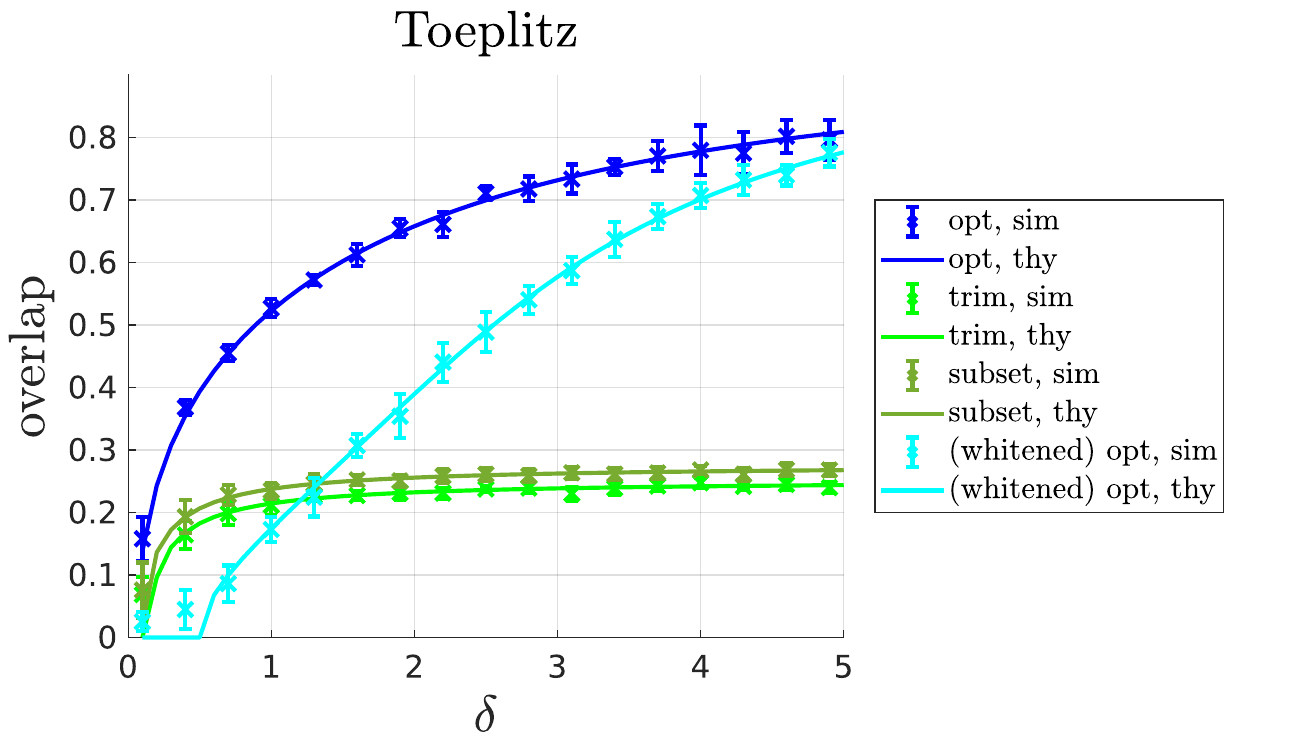}
        \label{fig:toeplitz_all_sub}
    \end{subfigure}
    ~
    \begin{subfigure}{0.416\linewidth}
        \centering
        \includegraphics[width = .92\linewidth]{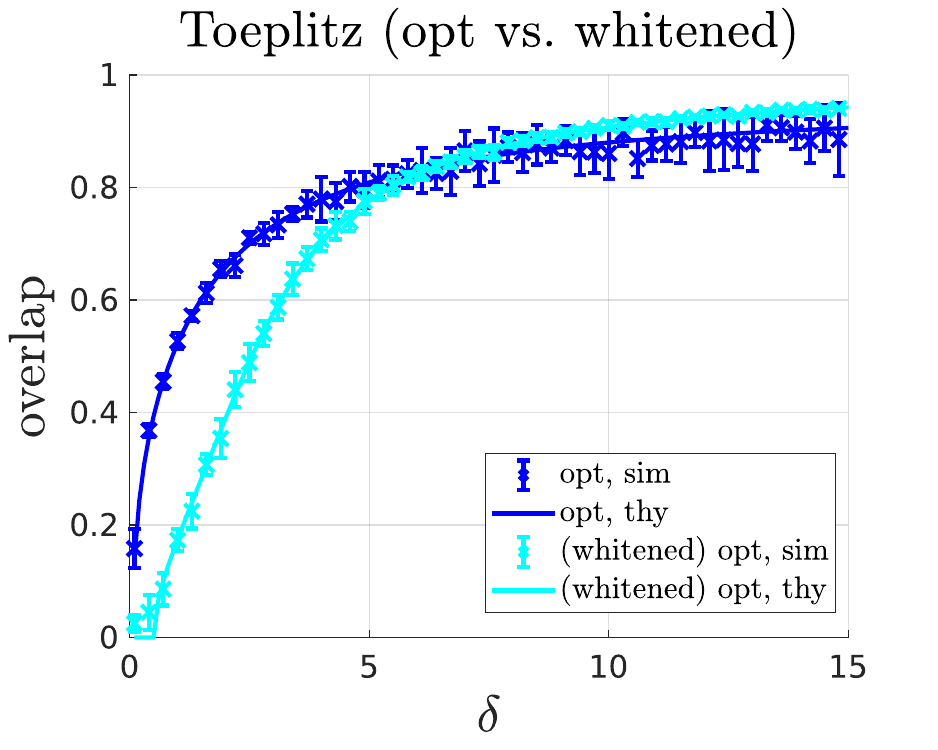}
        \label{fig:toeplitz_opt_vs_known}
    \end{subfigure}
      \begin{subfigure}{0.583\linewidth}
        \centering
        \includegraphics[width = .92\linewidth]{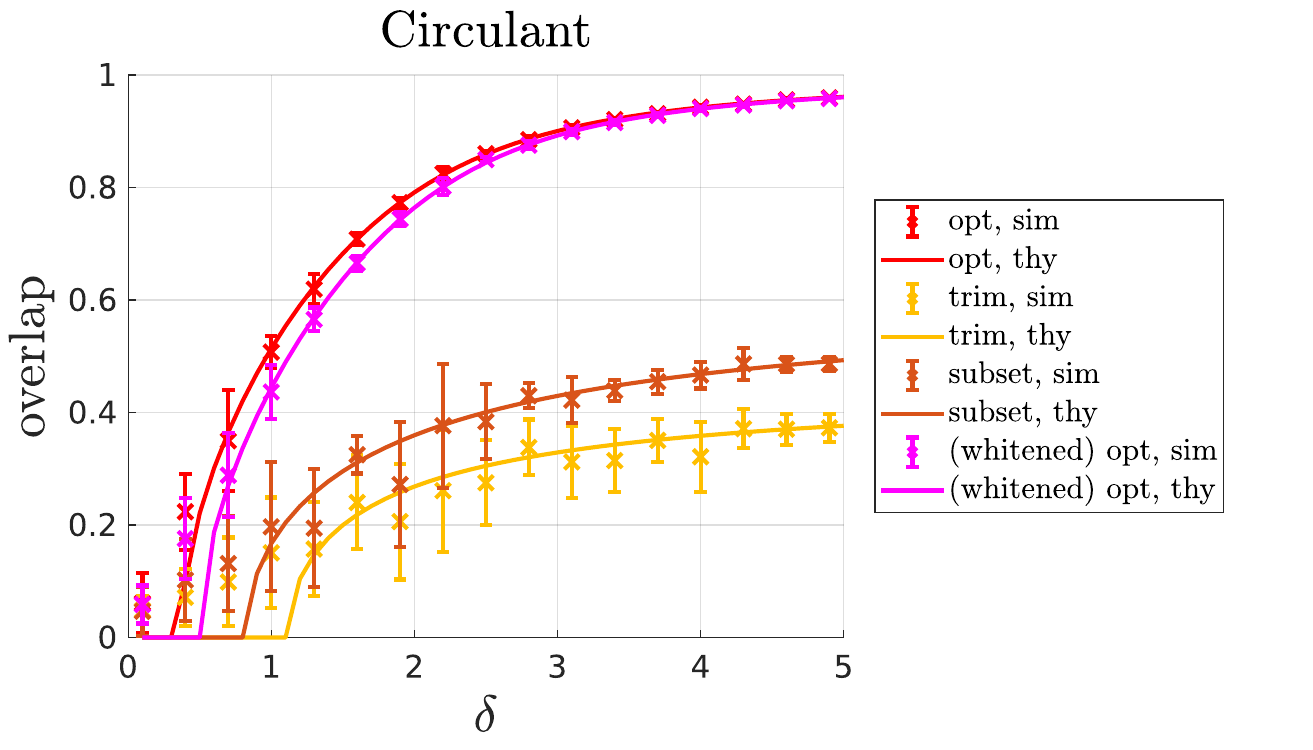}
        \label{fig:circulant_all_sub}
    \end{subfigure}
    ~
    \begin{subfigure}{0.416\linewidth}
        \centering
        \includegraphics[width = .92\linewidth]{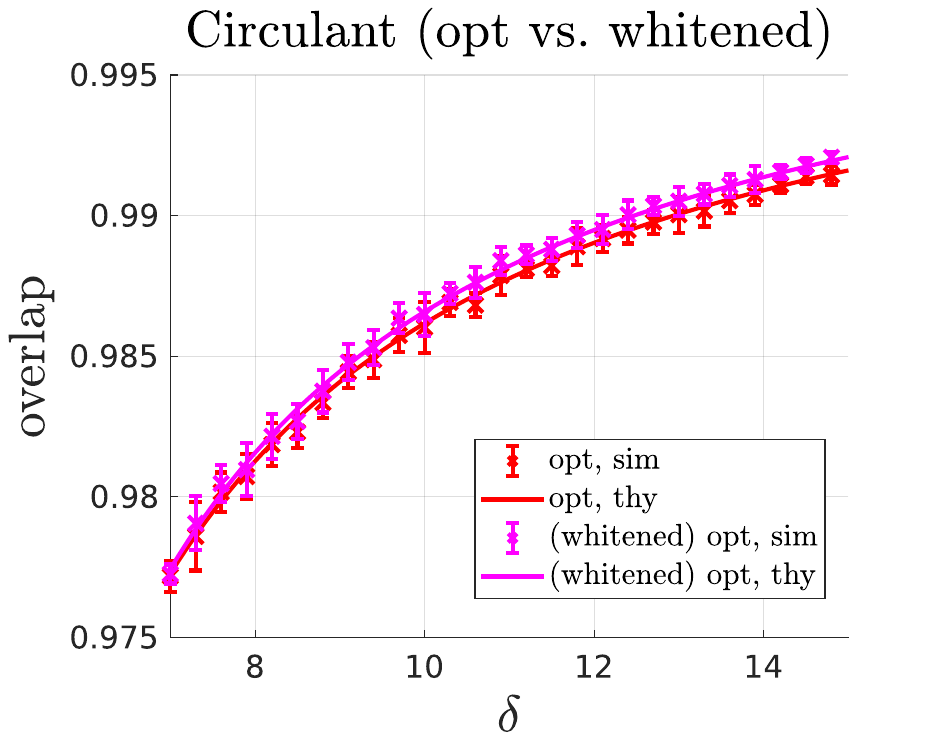}
        \label{fig:circulant_opt_vs_known}
    \end{subfigure}
    \caption{Overlap of spectral estimators with different preprocessing functions for noiseless phase retrieval when the covariate vectors are independent zero-mean Gaussians with Toeplitz (top row) and circulant (bottom row) covariance. 
    }
    \label{fig:toeplitz_all}
\end{figure}

We compare spectral estimators using different 
preprocessing functions: 
\begin{enumerate}
    \item[(i)] The optimal choice in \Cref{eqn:opt_T_main} with truncation, i.e., $$\cT^*(y) = \max\brace{ 1 - \expt{\ol{\Sigma}}/(\delta y^2), -K_* },$$ with $ K_* = 10 $. 
  The truncation 
  ensures that the preprocessing is bounded as required by our theory and, by taking $K_*$ sufficiently large, it does not affect performance. 
\item[(ii)] The trimming scheme \cite{Chen_Candes_trim}, i.e.,  $$\cT^{\mathrm{trim}}(y) = \delta y^2/\expt{\ol{\Sigma}} \indicator{\sqrt{\delta} \abs{y}/\sqrt{\expt{\ol{\Sigma}}} \le K_{\mathrm{trim}}}, $$ with $ K_{\mathrm{trim}} = \sqrt{7} $. 
\item[(iii)] The subset scheme \cite{Wang_subset}, i.e., $$\cT^{\mathrm{subset}}(y) = \indicator{ \sqrt{\delta} \abs{y}/\sqrt{\expt{\ol{\Sigma}}} \ge K_{\mathrm{subset}} }, $$ with $ K_{\mathrm{subset}} = \sqrt{2} $. The values of both $ K_{\mathrm{trim}} $ and $ K_{\mathrm{subset}} $ are taken from \cite[Section 7.1]{mondelli-montanari-2018-fundamental} where they are optimized to yield the smallest spectral threshold for $ \Sigma = I_d $. \item[(iv)] The identity function with truncation, i.e., $$\cT^{\mathrm{id}}(y) \coloneqq \min\brace{ \max\brace{ \sqrt{\delta} y/\sqrt{\expt{\ol{\Sigma}}} , -K_{\mathrm{id}} } , K_{\mathrm{id}} } ,$$
    with $ K_{\mathrm{id}} $ equal to $ 3.5 $ and $ 3 $ for circulant and Toeplitz covariances, respectively. 
    Empirically, the performance under these choices of $ K_{\mathrm{id}} $ does not differ much from avoiding the truncation, i.e., $ K_{\mathrm{{id}}} = \infty $. 
\end{enumerate}

We also compare the performance with a \emph{whitened} spectral estimator, which requires knowledge of the covariance $ \Sigma $. The whitened spectral estimator is given by 
\begin{align}
    \beta^{\spec}_\known \coloneqq \Sigma^{-1/2} v_1(D_\known),
    \label{eqn:whiten_spec_main}
\end{align}
where $ D_\known \coloneqq (X\Sigma^{-1/2})^\top \diag(\cT(y)) (X\Sigma^{-1/2}) $. 
This estimator uses $\Sigma$ to whiten $X$ and computes the principal eigenvector of $D_\known$ obtained via the decorrelated covariates $ X\Sigma^{-1/2} $. 
As the eigenvector can be thought of as an estimate of $ \Sigma^{1/2} \beta^* $, it is left-multiplied by $ \Sigma^{-1/2} $ to produce an estimate of $ \beta^* $.  Formal results and proofs on $ \beta^{\spec}_\known $ are deferred to \Cref{sec:spec_known}. 
%

\Cref{fig:toeplitz_all} shows that our proposed optimal spectral estimator significantly outperforms the trimming/subset schemes for both Toeplitz (top) and circular (bottom) covariances.  Furthermore, in a large interval of $\delta$, the performance of the whitened spectral estimator in \Cref{eqn:whiten_spec_main} (which requires $\Sigma$)  is significantly worse than that of the standard spectral estimator (which does not require $ \Sigma $),  even though optimal preprocessing functions are employed for both.

  In \Cref{fig:toeplitz_vs_circulant_vs_identity}, the plots for Toeplitz, circulant and identity covariance are superimposed. 
    An interesting observation is that there is no universally best covariance structure, even if the optimal  preprocessing function with respect to the corresponding covariance  is adopted. 
    

\begin{figure}[t]
    \centering
    \includegraphics[width = \linewidth]{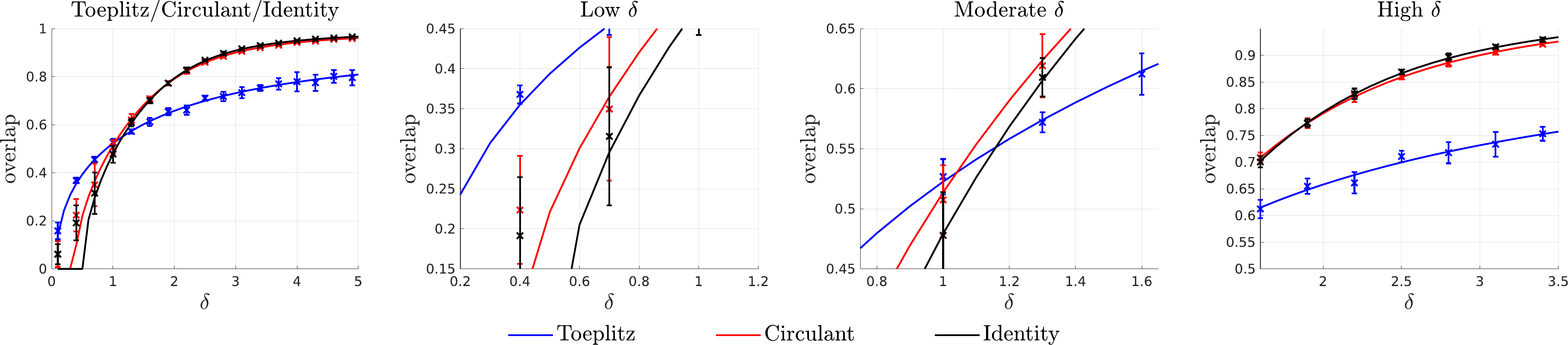}
    \caption{Overlap of spectral estimators with optimal preprocessing function 
     for noiseless phase retrieval when the covariate vectors are independent zero-mean Gaussians with Toeplitz/circulant/identity covariance. 
    The right three panels respectively zoom into regimes where $ \delta $ takes low, moderate and high values to demonstrate that in this particular setting, any one of the three types of covariance structures can attain the highest overlap. }
    \label{fig:toeplitz_vs_circulant_vs_identity}
\end{figure}

\subsubsection{Real data}
\label{sec:experiments_real}

We also demonstrate the advantage of the optimal preprocessing given by our theory for datasets popular in quantitative genetics and computational imaging. 

Specifically, the design matrices for the first two plots of \Cref{fig:plot_real} are obtained from two GTEx datasets ``skin sun exposed lower leg'' ($56200\times701$) and ``muscle skeletal'' ($56200\times803$) \cite{lonsdale2013genotype}. 
These matrices record gene counts and therefore contain non-negative entries. 
We preprocess them as follows: \emph{(i)} remove all-$0$ rows,
    \emph{(ii)} build a matrix by sequentially including each row only if it has an overlap smaller than $0.3$ with all existing rows, and  
    \emph{(iii)} center and normalize each row such that it has zero mean and unit variance. All these operations are typical in genetic studies, see e.g.\ the widely used toolset PLINK \cite{chang2015second}. The unknown parameter vector is given by $ \beta^* \sim \unif(\sqrt{d}\,\bbS^{d-1}) $ for $ d \in \{ 701, 803\}$.
For each $ \delta $, the design matrix is formed by the first $ \floor{d\delta} $ rows of the above preprocessed matrix. 
The value of overlap for each $\delta$ is computed from $100$ i.i.d.\ trials where the randomness is only over $ \beta^* $, and the error bar is reported at $1$ standard deviation. 
The truncation levels for different preprocessing functions are chosen as follows: 
for $ \cT^* $, we set $ K_* = 100 $; 
for $ \cT^{\mathrm{trim}} $ and $ \cT^{\mathrm{subset}} $, for each $\delta$, we choose $ K_{\mathrm{trim}} $ and $ K_{\mathrm{subset}} $ in $ \brace{0.25i : 1\le i\le 40} $ to maximize the respective overlaps (averaged over $100$ trials);
for $ \cT^{\mathrm{id}} $, we do not truncate, i.e., $ K_{\mathrm{id}} = \infty $. Despite the advantage due to the adaptive choice of the truncation level for the trimming/subset scheme, the preprocessing we propose still performs vastly better than all alternatives.

The design matrices for the last two plots of \Cref{fig:plot_real} follow a coded diffraction pattern \cite{candes_CDP}, i.e., $X$ is obtained by stacking in its rows the matrices $F D_1 S$, $F D_2 S$, $\ldots$, $F D_\delta S$. Here, $\delta\in\bbZ_{\ge1}$, $ F\in\bbR^{d\times d} $ is a Discrete Fourier Transform matrix, $S\in\bbR^{d\times d}$ is  diagonal containing i.i.d.\ uniformly random signs, and $ D_1, D_2, \cdots , D_\delta\in\bbC^{d\times d} $ are diagonal with elements following one of these two distributions: \emph{(i)} uniform modulation, $ (D_\ell)_{i,i}\iid\unif([-10, 10]) $, and      
    \emph{(ii)} octanary modulation \cite[Equation (1.9)]{candes_CDP}, $ (D_\ell)_{i,i}\iid\law(\ol{D}) $ with $\ol{D} = \ol{D}_1\ol{D}_2$, $\law(\ol{D}_1) = \frac{1}{4} \paren{ \delta_{1} + \delta_{-1} + \delta_{-\ii} + \delta_{\ii} }$ and $\law(\ol{D}_2) = \frac{4}{5} \delta_{1/\sqrt{2}} + \frac{1}{5} \delta_{\sqrt{3}}$.
For fractional $ \delta\in(0,\infty) $, we first construct a matrix of size $ \ceil{\delta}d\times d $ and then randomly subsample $ \floor{\delta d} - \floor{\delta}d $ rows from the last block $ F D_{\ceil{\delta}} S $ to obtain a design matrix of size $ \floor{\delta d}\times d $.

The parameter $\beta^*$ in the last two plots of \Cref{fig:plot_real} is a $ 75\times64 $ RGB image of the painting ``Girl with a Pearl Earring''. 
The $3$ color bands give $3$ matrices in $ [0, 256]^{75\times64} $. 
The parameter vectors $ \beta_{\mathrm{R}}^*, \beta_{\mathrm{G}}^*, \beta_{\mathrm{B}}^*\in\bbS^{d-1} $ (with $ d = 75\times64 = 4800 $) are then obtained by vectorizing, centering, and normalizing each  of these matrices.
For each $\delta$, we have $5$ i.i.d.\ trials where the randomness is only over $X$. 
In each trial, we compute $3$ spectral estimators using the same $X$ and observations $y_{\mathrm{R}},y_{\mathrm{G}},y_{\mathrm{B}}\in\bbR^n$ generated from $\beta_{\mathrm{R}}^*, \beta_{\mathrm{G}}^*, \beta_{\mathrm{B}}^*$ respectively. 
We report the mean of $5\times3 = 15$ overlaps for each $\delta$ with error bar at $1$ standard deviation. 
The truncation levels for different preprocessing functions are $ K_* = 10, K_{\mathrm{trim}} = \sqrt{7}, K_{\mathrm{subset}} = \sqrt{2}, K_{\mathrm{id}} = \infty $. For all datasets, our proposed preprocessing (optimal in red) 
    outperforms previous heuristic choices (trimming \cite{Chen_Candes_trim} in black, subset \cite{Wang_subset} in blue, and identity in green).

\begin{figure}
    \centering
    \includegraphics[width = \linewidth]{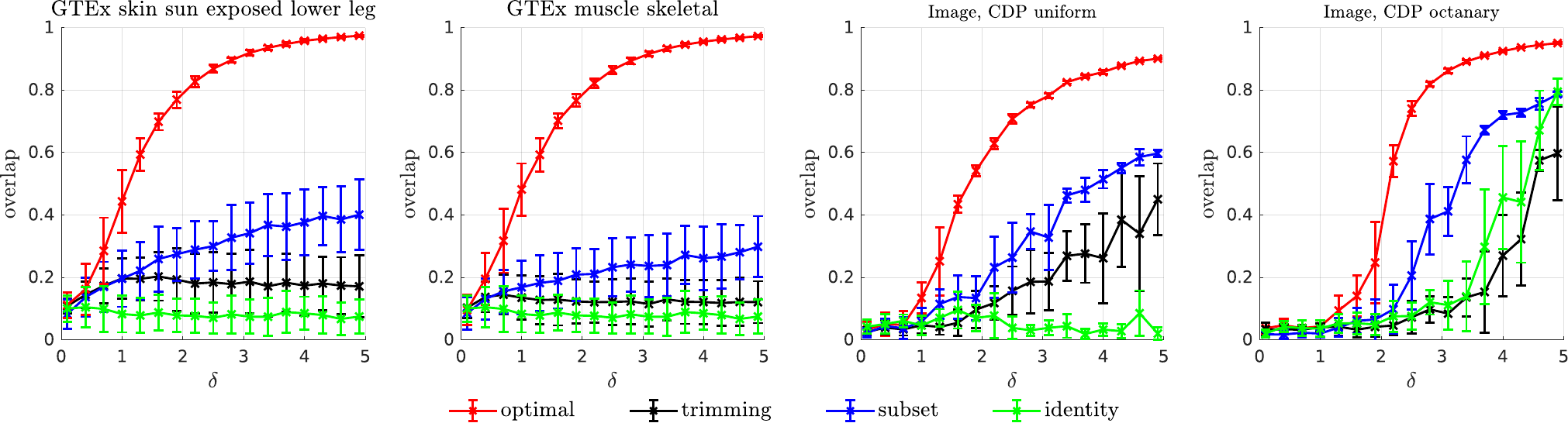}
    \caption{Overlap of spectral estimators for noiseless phase retrieval when the design matrix is obtained from two Genotype-Tissue Expression (GTEx) datasets (first two plots) and two coded diffraction patterns (CDP) (last two plots). 
    }
    \label{fig:plot_real}
\end{figure}

\subsection{Poisson regression}
\label{sec:poisson_regression}

{We also consider the Poisson regression model $ y_i \sim \pois(\inprod{x_i}{\beta^*}^2) $ where $ \beta^* \sim \unif(\sqrt{d} \,\bbS^{d - 1}) $ and $ d = 2000 $. 
The covariance matrix $ \Sigma $ is taken to be Toeplitz or circulant with the same parameters as in \mbox{\Cref{sec:experiments_syn}}. 
We again consider $3$ preprocessing functions: the optimal one $ \cT^*(y) = \frac{y - \expt{\ol{\Sigma}} / \delta}{y + 1/2} $, the trimming function $ \cT^{\textnormal{trim}}(y) = y \indicator{ \abs{y} \le K_{\textnormal{trim}, \delta} } $, and the subset function $ \cT^{\textnormal{subset}}(y) = \indicator{\abs{y} \ge K_{\textnormal{subset}, \delta}} $. 
For each $ \delta $, $ K_{\textnormal{trim}, \delta}, K_{\textnormal{subset}, \delta} $ are optimized over $ [0.5, 50], [0.5, 20] $, respectively, so as to maximize the overlap. 
Note that since $ y_i $ is $ \bbZ_{\ge0} $-valued, it suffices to consider $ K_{\textnormal{trim}, \delta}, K_{\textnormal{subset}, \delta} $ of the form $ K + 1/2 $ for an integer $K$. 
The numerical results are shown in \mbox{\Cref{fig:poisson}}. }

\begin{figure}[htbp]
    \centering
    \begin{subfigure}[b]{0.49\textwidth}
        \centering
        \includegraphics[width=\linewidth]{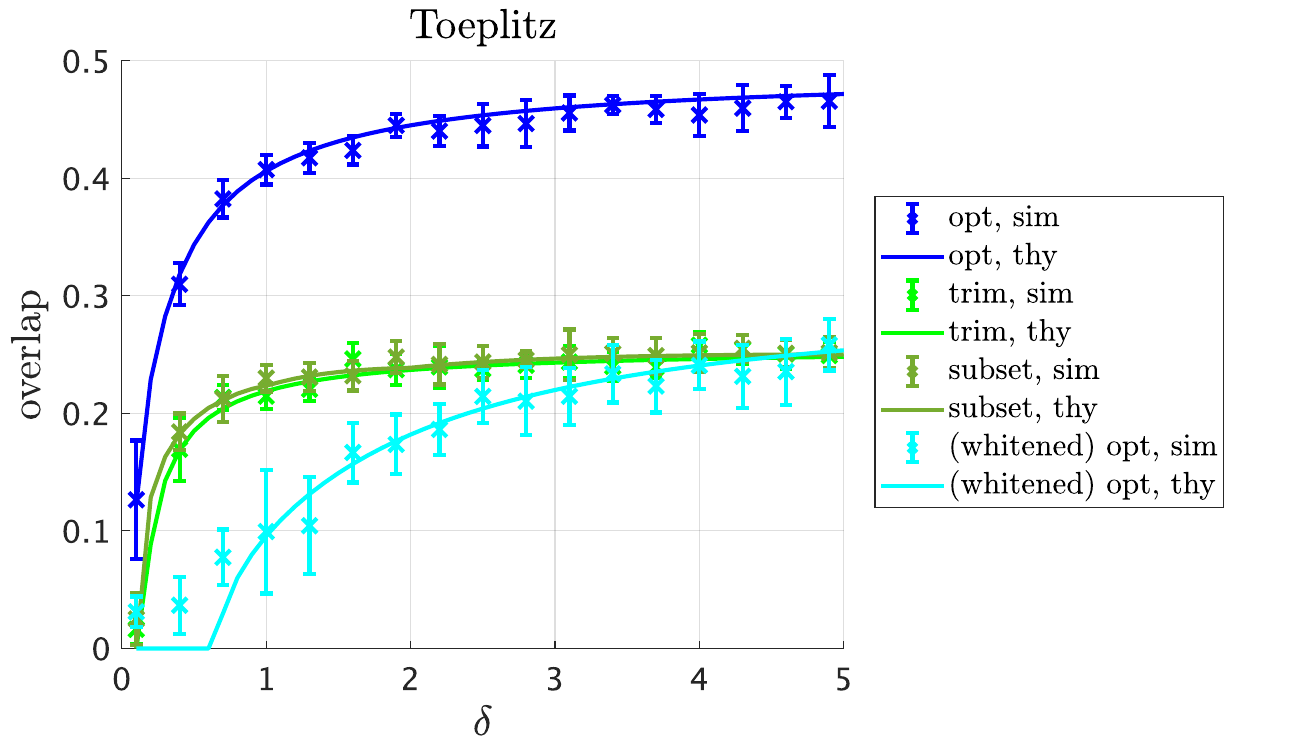}
    \end{subfigure}
    \begin{subfigure}[b]{0.49\textwidth}
        \centering
        \includegraphics[width=\linewidth]{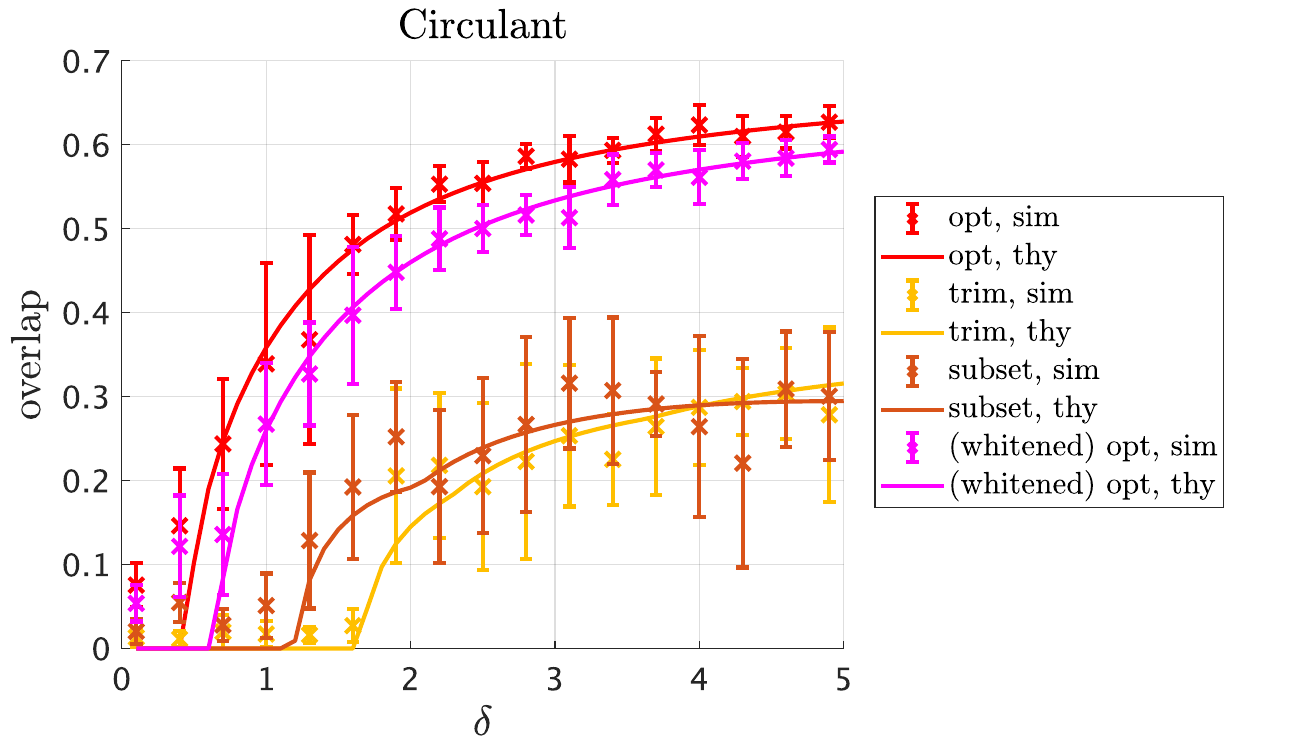}
    \end{subfigure}
    \caption{{Overlap of spectral estimators with different preprocessing functions for Poisson regression with correlated Gaussian design with Toeplitz (left panel) and circulant (right panel) covariance.} }
    \label{fig:poisson}
\end{figure}

\section{Proof of \Cref{thm:main}} 
\label{sec:heuristics}

\subsection{Overview of the argument}
\label{sec:pf_technique}
The outlier location and asymptotic overlap in \Cref{thm:main} are derived using 
a variant of AMP for GLMs, known as \emph{generalized approximate message passing (GAMP)}  \cite{rangan_GAMP}, \cite[Section 4]{amp-tutorial}. 
An instance of GAMP is specified by two sequences of denoising functions, $ (g_t)_{t\ge0}$ and $ (f_{t+1})_{t\ge0} $.
Starting with  initialization $ \wt{u}^{-1} = 0_n\in\bbR^n $ and some $\wt{v}^0 \in \bbR^d$, for $t \ge 0$ the GAMP iterates are computed as:  
\begin{align}
\begin{aligned}
u^t &= \wt{X} \wt{v}^t - b_t \wt{u}^{t-1} , \quad 
\wt{u}^t = g_t(u^t; y) , \quad 
c_t = \frac{1}{n} \div g_t(u^t; y)
= \frac{1}{n} \sum_{i = 1}^n \frac{\partial g_t(u^t; y)_i}{\partial u^t_i} , \\
v^{t+1} &\hspace{-.25em}=\hspace{-.25em} \wt{X}^\top \wt{u}^t \hspace{-.25em}-\hspace{-.25em} c_t \wt{v}^t , \,\, 
\wt{v}^{t+1} \hspace{-.25em}=\hspace{-.25em} f_{t+1}(v^{t+1}) , \,\, 
b_{t+1} \hspace{-.25em}=\hspace{-.25em} \frac{1}{n} \div f_{t+1}(v^{t+1})\hspace{-.25em}
=\hspace{-.25em} \frac{1}{n} \sum_{i = 1}^d \frac{\partial f_{t+1}(v^{t+1})_i}{\partial v^{t+1}_i} , 
\end{aligned}
\label{eqn:GAMP_nonsep}
\end{align}
where we recall $ \wt{X} = X \Sigma^{-1/2} $. To handle $\Sigma\neq I_d$, the denoising functions $g_t\colon\bbR^n\times\bbR^n\to\bbR^n $ and $ f_{t+1}\colon\bbR^d\to\bbR^d $ need to be \emph{non-separable}, i.e., they cannot be decomposed in terms of functions acting component-wise on the vector inputs. 

AMP algorithms come with an associated deterministic scalar recursion called \emph{state evolution} which describes the limiting distribution (as $d\to\infty$) of the AMP iterates $ u^t\in\bbR^n $ and $ v^{t+1}\in\bbR^d $ using a collection of Gaussian vectors. 
The covariance structure of these Gaussians admits a succinct representation which can be recursively tracked via the state evolution. The state evolution result for GAMP with non-separable denoisers is not immediately available -- we prove it by reducing such a GAMP to a general family of abstract AMP algorithms introduced in \cite{GraphAMP} for which a state evolution has been established. This is detailed in \Cref{sec:GAMP_nonsep}.
{We note that state evolution results for an abstract AMP similar to those in {\cite{GraphAMP}} can also be found in {\cite{Su_etal}}. }

The key idea is to design a GAMP algorithm that simulates the power iteration $ v^{t+1} = D v^t/\normtwo{D v^t} $, via a careful choice of denoising functions $g_t$ and $f_{t+1}$. To this end, we set 
\begin{align}
g_t(u^t; y) &= Fu^t , \quad t\ge0 , \label{eqn:gt}
\end{align}
where $ F = \diag(\cF(y))\in\bbR^{n\times n} $, and  the functions $ \mathcal F, (f_{t+1})_{t\ge0}$ are specified later. With this choice for $g_t$, we have
\begin{align}
c_t &= \frac{1}{n} \sum_{i = 1}^n \cF(y_i) \xrightarrow{n\to\infty} \expt{\cF(\ol{Y})} \eqqcolon c , \quad t\ge0 , \label{eqn:ct}
\end{align}
where we recall that $\ol{Y} $ is defined in \Cref{eqn:rand-var}. 
Thus, the GAMP iteration, with $ c_t $ replaced with its high-dimensional limit, becomes
\begin{align}
u^t &= \wt{X} f_t(v^t) - b_t F u^{t-1} , \quad 
v^{t+1} = \wt{X}^\top F u^t - c f_t(v^t) . \notag 
\end{align}
We show in \Cref{sec:pf-AMP} that $ u^t, v^{t+1}, b_t, f_{t+1} $ converge in probability as $t\to\infty$, i.e., there exist $ u\in\bbR^n, v\in\bbR^d, b\in\bbR $ and $ f\colon\bbR^d\times\bbR^d\to\bbR^d $ such that 
\begin{align}
\begin{aligned}
\lim_{t\to\infty} \lim_{n\to\infty} \frac{1}{\sqrt{n}} \normtwo{u^t - u} &= 0 , \quad 
\lim_{t\to\infty} \lim_{d\to\infty} \frac{1}{\sqrt{d}} \normtwo{v^{t+1} - v} = 0 , \notag \\
\lim_{t\to\infty} \lim_{d\to\infty} \abs{b_t - b} &= 0 , \quad 
\lim_{t\to\infty} \lim_{d\to\infty} \frac{1}{\sqrt{d}} \normtwo{f_{t+1}(v^{t+1}) - f(v)} = 0 . 
\end{aligned}
\end{align}
Thus, we obtain
\begin{align}
u &= \wt{X} f(v) - b F u , \quad 
v = \wt{X}^\top F u - c f(v) . \notag 
\end{align}
The first equation for $u$ implies
\begin{align}
u &= \paren{I_n + b F}^{-1} \wt{X} f(v) . \notag 
\end{align}
Substituting this into the equation for $v$ and multiplying both sides by $ \Sigma^{1/2} $, we have 
\begin{align}
\Sigma^{1/2} \paren{v + c f(v)} = \Sigma^{1/2} \wt{X}^\top F \paren{I_n + b F}^{-1} \wt{X} \Sigma^{1/2} \Sigma^{-1/2} f(v) . \label{eqn:to-be-eigen}
\end{align}
At this point, we consider the following choice of $ \cF $ and $ f $: 
\begin{align}
\cF(\cdot) &= \frac{\cT(\cdot)}{a - b \cT(\cdot)},\qquad f(v) = (\gamma I_d - c \Sigma)^{-1} \Sigma v,  \label{eqn:cF}
\end{align}
for some $ a,\gamma\in\bbR $ to be specified. Then, \Cref{eqn:to-be-eigen} becomes 
\begin{align}
\Sigma^{-1/2} f(v) = \frac{1}{a\gamma} \Sigma^{1/2} \wt{X}^\top T \wt{X} \Sigma^{1/2} \Sigma^{-1/2} f(v)= \frac{1}{a\gamma} D \Sigma^{-1/2} f(v) , \notag 
\end{align}
which is an eigenequation of $D$ with eigenvalue $ a \gamma \eqqcolon \lambda_1 $ and eigenvector (possibly scaled by a constant)
$\Sigma^{-1/2} f(v) = \Sigma^{-1/2} (\gamma I_d - c \Sigma)^{-1} \Sigma v$.
Assuming a spectral gap, we expect that $ \lambda_1 $ equals the limiting value of $ \lambda_1(D) $ and $\Sigma^{-1/2} f(v)$ is asymptotically aligned with $ v_1(D) $. 

It remains to pick $a, \gamma$ which are in principle free parameters. Our choice is motivated by the fixed points of state evolution characterized in \Cref{sec:fp-se}, and it simplifies the derivations. Specifically, the limiting Onsager coefficient is given by 
\begin{align}
b &= \frac{1}{n} \sum_{i = 1}^d \frac{\partial f(v)_i}{\partial v_i}
= \frac{1}{n} \sum_{i = 1}^d ((\gamma I_d - c \Sigma)^{-1} \Sigma)_{i,i}
\xrightarrow{n\to\infty} \frac{1}{\delta} \expt{\frac{\ol{\Sigma}}{\gamma - c\ol{\Sigma}}} . \notag 
\end{align}
Then, we choose $ (a, \gamma)$ to satisfy 
\begin{align}
\lim_{t\to\infty} \lim_{d\to\infty} \frac{1}{d} \normtwo{f_{t+1}(v^{t+1})}^2 = 1, \qquad  b = 1.\label{eqn:vtilde_1}
\end{align}
The constraint on $\normtwo{f_{t+1}(v^{t+1})}^2$ normalizes the GAMP iterate so that, as $t$ grows, its norm does not blow up or vanish. Using state evolution and the characterization of its fixed points, we can show that the conditions \Cref{eqn:vtilde_1} can be written as
\begin{align}
\begin{aligned}
1 &= \frac{1}{\expt{\ol{\Sigma}}} \expt{\paren{\frac{\delta}{\expt{\ol{\Sigma}}} \ol{G}^2 - 1} \frac{\cT(\ol{Y})}{a - \cT(\ol{Y})}} \expt{\frac{\ol{\Sigma}^2}{\gamma - \expt{\frac{\cT(\ol{Y})}{a - \cT(\ol{Y})}} \ol{\Sigma}}} , \\
1 &= \frac{1}{\delta} \expt{\frac{\ol{\Sigma}}{\gamma - \expt{\frac{\cT(\ol{Y})}{a - \cT(\ol{Y})}}\ol{\Sigma}}} . 
\end{aligned}
\label{eqn:fp-a-gamma}
\end{align}
\Cref{prop:equiv-def-a-gamma} shows that in the presence of a spectral gap, \Cref{eqn:fp-a-gamma} is equivalent to $\zeta(a)=\varphi(a)$, with $\varphi, \zeta$ defined in \Cref{eqn:def-phi-psi,eqn:def-zeta}. Thus, from  \Cref{eqn:def_a*}, we have that  $(a, \gamma)=(a^*, \gamma(a^*))$. 


With the above choice of denoisers, the GAMP iteration can be expressed as
\begin{align}
    \wh{v}^{t+1} &= \frac{D}{a^* \gamma(a^*)} \wh{v}^t + \wh{e}^t , \label{eqn:almost_pow_itr_main} 
\end{align}
for some auxiliary iterate $ \wh{v}^{t+1} $ and error term $ \wh{e}^t $. We show in \Cref{sec:pf_bound_error_t_t'} that $ \wh{e}^t $ asymptotically vanishes as $ t $ grows. Now, if $ \wh{e}^t $ is zero, \Cref{eqn:almost_pow_itr_main} is exactly a power iteration for $M \coloneqq (a^* \gamma(a^*))^{-1} D$. The convergence of this power iteration to the leading eigenvector of $M$ (or, equivalently, of $D$) 
crucially relies on the existence of a spectral gap, i.e., on the fact that $\lim_{d\to\infty} \lambda_1(D)>\lim_{d\to\infty} \lambda_2(D)$.

To pinpoint when a spectral gap exists, we establish the limiting value of $ \lambda_2(D) $. In \Cref{sec:bulk}, we prove that $ \lambda_2(D) $ converges to $ \lambda_2 \coloneqq a^\circ \gamma(a^\circ) $, where $ a^\circ $ is given as in \Cref{eqn:def_a_circ}. This is obtained by interlacing the eigenvalues of $D$ with those of a ``decoupled'' matrix $ \wh{D} $ in which $X$ is replaced with an i.i.d.\ copy $ \wh{X} $ independent of $T$. The support of the limiting spectral distribution of 
$ \wh{D} $ is characterized in  \cite[Section 3]{CouilletHachem}, when $T$ is positive semi-definite. By extending this analysis, we deduce the desired characterization of $ \lambda_2 $. One technical challenge is that, when $T$ is not positive semi-definite, the roles of $\Sigma$ and $T$ are not interchangeable in determining $\lambda_{2}$, whereas in \cite{CouilletHachem} this symmetry simplifies the arguments. 


Given the normalization in \Cref{eqn:vtilde_1}, 
the largest eigenvalue of $M$ converges to $1$ and, thus, $\lim_{d\to\infty}\lambda_1(D)=\lambda_1:=a^* \gamma(a^*)$. Hence, the criticality condition for the existence of a spectral gap reads $ a^* \gamma(a^*) > a^\circ \gamma(a^\circ) $. 
This 
is equivalent to $ a^* > a^\circ $, as adopted in \Cref{thm:main}, by the monotonicity properties of the function $ \psi(a) = a \gamma(a) $ in \Cref{eqn:def-phi-psi} (see \Cref{lem:a0=sup}). 

To formalize 
the above reasoning, 
assume $ a^* > a^\circ $ and execute \Cref{eqn:almost_pow_itr_main} for $t'$ steps to amplify the spectral gap: 
\begin{align}
    \wh{v}^{t+t'} &\approx M^{t'} \wh{v}^t , 
    \label{eqn:pow_itr_amp-main}
\end{align}
where the error terms can be neglected by taking $t$ sufficiently large (and also much larger than $t'$).
Now, we look at the rescaled norms $ \normtwo{\cdot}/\sqrt{d} $ of both sides of \Cref{eqn:pow_itr_amp-main}. Due to the GAMP state evolution, the rescaled norm of the left-hand side $ \normtwo{\wh{v}^{t+t'}}/\sqrt{d} $ can be accurately determined in the high-dimensional limit. 
Furthermore, it converges to an explicit strictly positive constant in the large $t$ limit, by convergence of state evolution. Thus, inspecting the right-hand side of \Cref{eqn:pow_itr_amp-main} allows us to conclude that $ \lambda_1(M) $ must be $1$ in the high-dimensional limit. 
Indeed, if that's not the case, $ \normtwo{M^{t'} \wh{v}^t}/\sqrt{d} $ would be either amplified or shrunk geometrically as $ t' $ grows, violating the equality in \Cref{eqn:pow_itr_amp-main}.
%
At this point, we 
have $ \lim\limits_{d\to\infty} \lambda_1(D) = \lambda_1 $, $ \lim\limits_{d\to\infty} \lambda_2(D) = \lambda_2 $ and that 
$ \wh{v}^t $ 
is asymptotically 
aligned 
with the top eigenvector $v_1(D)$, provided $ a^* > a^\circ $. 
Then, the limiting overlap between $ \beta^* $ and $v_1(D)$ is the same as that between $ \beta^* $ and $ \wh{v}^t $, 
which is again derived using state evolution. 

The rest of this section is organized as follows: \Cref{sec:GAMP_nonsep} presents the state evolution of GAMP with non-separable denoisers, \Cref{sec:fp-se} establishes its fixed points when GAMP simulates a power iteration, \Cref{sec:bulk} characterizes the right edge of the bulk of $D$, and \Cref{sec:pf-AMP} puts everything together concluding the proof of \Cref{thm:main}.

\subsection{State evolution of GAMP with non-separable denoisers}
\label{sec:GAMP_nonsep}

To precisely state the state evolution result for GAMP, we require the notion of pseudo-Lipschitz functions with matrix-valued inputs and outputs. 
\begin{definition}[Pseudo-Lipschitz functions]
\label{def:PL_fn}
A function $ h\colon\bbR^{k\times m}\to\bbR^{\ell\times m} $ is called \emph{pseudo-Lipschitz of order $j$} if there exists $L$ such that 
\begin{align}
    \frac{1}{\sqrt{\ell}} \normf{h(x) - h(y)} &\le \frac{L}{\sqrt{k}} \normf{x - y} \brack{ 1 + \paren{\frac{1}{\sqrt{k}} \normf{x}}^{j-1} + \paren{\frac{1}{\sqrt{k}} \normf{y}}^{j-1} } , \label{eqn:def_PL} 
\end{align}
for every $ x, y\in\bbR^{k\times m} $. 
\end{definition}

We will consider sequences of functions $h_i\colon\bbR^{k_i\times m} \to \bbR^{\ell_i\times m}$ indexed by $ i\to\infty $ though the index $i$ is often not written explicitly. 
A sequence of functions $ (h_i\colon\bbR^{k_i\times m} \to \bbR^{\ell_i\times m})_{i\ge1} $ is called \emph{uniformly pseudo-Lipschitz of order $j$} if there exists a constant $L$ such that for every $i\ge1$, \Cref{eqn:def_PL} holds. 
Note that $L$ is a constant as $i\to\infty$.

Define the random vectors
\begin{align}
\begin{aligned}
\mathfrak{B}^* &\sim P^{\ot d} , \quad \wt{\mathfrak{B}}^* = \Sigma^{1/2} \mathfrak{B}^*, \quad 
(G, \eps) \sim\cN\paren{0_n, \frac{1}{\delta}\expt{\ol{\Sigma}} I_n}\ot P_\eps^{\ot n} , \quad Y = q(G, \eps) . 
\end{aligned}
\label{eqn:rand-vec}
\end{align}
If $ \beta^* \sim \unif(\sqrt{d}\,\bbS^{d-1}) $, $ P $ should be taken to be $\cN(0,1)$. \label{page:spherical_prior}

We further impose the following assumptions which guarantee the existence and finiteness of various state evolution parameters. 

\begin{enumerate}[label=(A\arabic*)]
\setcounter{enumi}{\value{asmpctr}}

    \item \label[asmp]{asmp:init} 
    The initializer $ \wt{v}^0 \in\bbR^d $ is independent of $\wt{X}$. 
    Furthermore, 
    \begin{align}
        \plim_{d\to\infty} \frac{1}{\sqrt{d}} \normtwo{\wt{v}^0} \label{eqn:asmp_init_norm} 
    \end{align}
    exists and is finite. 
    There exists a uniformly pseudo-Lipschitz function $ f_0\colon\bbR^d\to\bbR^d $ of order $1$ such that 
    \begin{align}
        \lim_{d\to\infty} \frac{1}{d} \expt{\inprod{f_0(\wt{\mathfrak{B}}^*)}{f_0(\wt{\mathfrak{B}}^*)}} &\le \plim_{d\to\infty} \frac{1}{d} \normtwo{\wt{v}^0}^2 , \notag 
    \end{align}
    and for every uniformly pseudo-Lipschitz $ h\colon\bbR^d\to\bbR^d $ of finite order, 
    \begin{align}
    \plim_{d\to\infty} \frac{1}{d} \inprod{\wt{v}^0}{h(\wt{\beta}^*)} &= \lim_{d\to\infty} \frac{1}{d} \expt{\inprod{f_0(\wt{\mathfrak{B}}^*)}{h(\wt{\mathfrak{B}}^*)}} ; \label{eqn:asmp_init_corr_signal} 
    \end{align}
    in particular, limits on both sides of the above two displayed equations exist and are finite. 
    Here, we have set $ \wt{\beta}^* = \Sigma^{1/2} \beta^* $ and  we recall that $ \beta^* \sim P^{\ot d} $ from \Cref{asmp:signal_prior}. 
    Let $ \wt{\chi} \in \bbR , \wt{\sigma}_V \in \bbR_{\ge0} $. 
    For any $ t\ge0 $, 
    \begin{align}
        \lim_{d\to\infty} \frac{1}{d} \expt{\inprod{f_0(\wt{\mathfrak{B}}^*)}{f_{t+1}\paren{ \wt{\chi} \wt{\mathfrak{B}}^* + \wt{\sigma}_V \wt{W}_V }}} \notag 
    \end{align}
    exists and is finite, where $ \wt{W}_{V} \sim \cN(0_d, I_d) $ is independent of $ \wt{\mathfrak{B}}^* $.

\setcounter{asmpctr}{\value{enumi}}
\end{enumerate}

\begin{enumerate}[label=(A\arabic*)]
\setcounter{enumi}{\value{asmpctr}}
    \item \label[asmp]{asmp:SE} 
    Let $ \wt{\nu} \in \bbR $, and $ T \in \bbR^{2\times 2} $ be positive definite. 
    For $ s,t\ge0 $, 
    \begin{align}
        \lim_{d\to\infty} \frac{1}{d} \expt{\inprod{f_{s+1}(\wt{\nu} \wt{\mathfrak{B}}^* + \wt{N})}{f_{t+1}(\wt{\nu} \wt{\mathfrak{B}}^* + \wt{N}')}} \notag 
    \end{align}
    exists and is finite, where $ (\wt{\mathfrak{B}}^*, (\wt{N}, \wt{N}')) \sim \cN(0_d, \Sigma) \ot \cN(0_{2d}, T \ot I_d) $. 
    Let $ \wt{\mu} \in \bbR_{\ge0} $, and $ S \in \bbR^{2\times2} $ be positive definite. 
    For any $ s,t\ge0 $, 
    \begin{align}
        & \lim_{n\to\infty} \frac{1}{n} \expt{\inprod{g_s(\wt{G} + \wt{M}; Y)}{g_t(\wt{G} + \wt{M}'; Y)}} , \notag \\
        & \lim_{n\to\infty} \frac{1}{n} \expt{ \left.\paren{ \div_g g_t(u, q(g, e)) }\right|_{u = \wt{G} + \wt{M}, g = \wt{G}, e = \eps} } 
        \notag 
    \end{align}
    exist and are finite, where $ (\wt{G}, \eps, \wt{M}, \wt{M}') \sim \cN(0_n, \wt{\mu}^2 I_n ) \ot P_\eps^{\ot n} \ot \cN(0_{2n}, S \ot I_n) $ and $ Y = q(\wt{G}, \eps) $. 

\setcounter{asmpctr}{\value{enumi}}
\end{enumerate}

The state evolution result  -- formally stated below --  asserts that, for each $t\ge0$, in the large $n$ limit, the joint distributions of  the AMP iterates $ (\wt{\beta}^*, v^1, v^2, \cdots, v^{t+1}) $  and $ (g=X\beta^*, u^0, u^1, \cdots, u^t) $  converge to the laws of $ (\wt{\mathfrak{B}}^*, V_1, V_2, \cdots, V_{t+1}) $ and $ (G, U_0, U_1, \cdots, U_t) $, respectively. For $t \ge 0$,  the random vectors $ U_t\in\bbR^n $ and $ V_{t+1}\in\bbR^d $ are defined as:
\begin{align}
U_t = \mu_t G + \sigma_{U, t} W_{U, t}, \qquad 
   V_{t+1} = \chi_{t+1} \wt{\mathfrak{B}}^* + \sigma_{V, t+1} W_{V, t+1} , \label{eqn:def-Ut-Vt+1}
\end{align}
where $W_{U, t} \sim\cN(0_n, I_n)$ is independent of $(G, \eps)$, and $ W_{V,t+1}\sim\cN(0_d, I_d) $ is independent of $ \wt{\mathfrak{B}}^*$. The constants $\mu_t, \sigma_{U, t}, \chi_{t+1}, \sigma_{V, t+1} $ are recursively defined, starting from 
\begin{align}
\mu_0 &= \frac{\delta}{\expt{\ol{\Sigma}}} \lim_{n\to\infty} \frac{1}{n} \expt{\inprod{\wt{\mathfrak{B}}^*}{f_0(\wt{\mathfrak{B}}^*)}} , \quad \sigma_{U,0}^2 = \plim_{n\to\infty} \frac{1}{n} \inprod{\wt{v}^0}{\wt{v}^0} - \frac{\expt{\ol{\Sigma}}}{\delta} \mu_0^2. \label{eqn:mu0}
\end{align}
For $t \ge 0$, we have
\begin{equation}
\begin{aligned}
    \chi_{t+1} &= \frac{\delta}{\expt{\ol{\Sigma}}} \lim_{n\to\infty} \frac{1}{n} \expt{\inprod{G}{g_t(U_t; Y)}} - \mu_t \lim_{n\to\infty} \frac{1}{n} \expt{\div_{U_t} g_t(U_t; Y)}, \\
    \sigma_{V, t+1}^2 & = \lim_{n\to\infty} \frac{1}{n} \expt{\inprod{g_t(U_t; Y)}{g_t(U_t; Y)}},  
\end{aligned}
\label{eqn:chi-alt}
\end{equation}
and 
\begin{equation}
\begin{aligned}
\mu_{t+1} &= \frac{\delta}{\expt{\ol{\Sigma}}} \lim_{n\to\infty}\frac{1}{n} \expt{\inprod{\wt{\mathfrak{B}}^*}{f_t(V_{t+1})}} , \\
 \sigma_{U, t+1}^2 & = \lim_{n\to\infty}\frac{1}{n} \expt{\inprod{f_t(V_{t+1})}{f_t(V_{t+1})}} - \frac{\expt{\ol{\Sigma}}}{\delta} \mu_{t+1}^2.  
\end{aligned}
\label{eqn:mut}
\end{equation}

The two sequences of random vectors $ (W_{U,t})_{t\ge0} $ and $ (W_{V,t+1})_{t\ge0} $, are each jointly Gaussian with the following  laws:
\begin{align}
\matrix{\sigma_{U, 0} W_{U, 0} \\ \sigma_{U, 1} W_{U, 1} \\ \vdots \\ \sigma_{U, t} W_{U, t}}
\sim \cN\paren{ 0_{(t+1)n}, \Phi_t \ot I_n } , \quad 
\matrix{\sigma_{V, 1} W_{V, 1} \\ \sigma_{V, 2} W_{V, 2} \\ \vdots \\ \sigma_{V, t+1} W_{V, t+1}}
\sim \cN\paren{ 0_{(t+1)d}, \Psi_t \ot I_d } , \label{eqn:joint_WUt_WVt+1} 
\end{align}
where $ \Phi_t , \Psi_t \in \bbR^{(t+1)\times(t+1)} $ are matrices with entries: 
\begin{align}
(\Phi_t)_{1,1} &\coloneqq \plim_{n\to\infty} \frac{1}{n} \inprod{\wt{v}^0}{\wt{v}^0} - \frac{\expt{\ol{\Sigma}}}{\delta} \mu_0^2 , \notag \\
(\Phi_t)_{1,s+1} &\coloneqq \lim_{n\to\infty} \frac{1}{n} \expt{\inprod{f_0(\wt{\mathfrak{B}}^*) - \mu_0 \wt{\mathfrak{B}}^*}{f_{s}(V_{s}) - \mu_{s} \wt{\mathfrak{B}}^*}} , \quad \mbox{ for }\,\, 1\le s\le t , \label{eqn:Phit1} \\
(\Phi_t)_{r+1,s+1} &\coloneqq \lim_{n\to\infty} \frac{1}{n} \expt{\inprod{f_{r}(V_{r}) - \mu_{r} \wt{\mathfrak{B}}^*}{f_{s}(V_{s}) - \mu_{s} \wt{\mathfrak{B}}^*}} , \quad \mbox{ for }\,\, 1\le r,s\le t , \label{eqn:Phit2} \\
(\Psi_t)_{r+1,s+1} &\coloneqq \lim_{n\to\infty} \frac{1}{n} \expt{\inprod{g_{r}(U_{r}; Y)}{g_{s}(U_{s}; Y)}} , \quad \mbox{ for }\,\, 0\le r,s\le t . \label{eqn:Psit}
\end{align}
Note that for $ r = s $, $ (\Psi_t)_{r+1,r+1} = \sigma_{V,r+1}^2 $ is consistent with \Cref{eqn:chi-alt} and 
\begin{align}
&    (\Phi_t)_{r+1, r+1} = \lim_{n\to\infty} \frac{1}{n} \expt{\inprod{f_{r}(V_{r}) - \mu_{r} \wt{\mathfrak{B}}^*}{f_{r}(V_{r}) - \mu_{r} \wt{\mathfrak{B}}^*}} \notag \\
    &= \lim_{n\to\infty} \frac{1}{n} \expt{\inprod{f_{r}(V_{r})}{f_{r}(V_{r})}} 
    - 2 \mu_r \lim_{n\to\infty} \frac{1}{n} \expt{\inprod{f_{r}(V_{r})}{\wt{\mathfrak{B}}^*}} \notag  \\
   &\hspace{15em}
    + \mu_r^2 \lim_{n\to\infty} \frac{1}{n} \expt{\inprod{\wt{\mathfrak{B}}^*}{\wt{\mathfrak{B}}^*}} \notag \\
    &= \lim_{n\to\infty} \frac{1}{n} \expt{\inprod{f_{r}(V_{r})}{f_{r}(V_{r})}} 
    - 2 \mu_r^2 \frac{\expt{\ol{\Sigma}}}{\delta} + \mu_r^2 \frac{\expt{\ol{\Sigma}}}{\delta} 
    = \sigma_{U,r}^2 \notag 
\end{align}
is consistent with \Cref{eqn:mut}, where 
the last line above follows from the definition of $\mu_r$ in \Cref{eqn:mut}. 
As $ (G, W_{U,0}, \cdots, W_{U,t}) $ are jointly Gaussian by \Cref{eqn:joint_WUt_WVt+1}, their covariance structure (and therefore that of $ (G, U_0, \cdots, U_t) $ in view of \Cref{eqn:def-Ut-Vt+1}) is completely determined by the constants defined in 
 \Cref{eqn:mu0,eqn:chi-alt,eqn:mut,eqn:Phit1,eqn:Phit2}. 
Similarly $ (V_1 - \chi_1 \wt{\mathfrak{B}}^*, \cdots, V_{t+1} - \chi_{t+1} \wt{\mathfrak{B}}^*) = (\sigma_{V,1} W_{V,1}, \cdots, \sigma_{V,t+1} W_{V,t+1}) $ are jointly Gaussian by \Cref{eqn:def-Ut-Vt+1,eqn:joint_WUt_WVt+1}, hence the covariance structure of $ (\wt{\mathfrak{B}}^*, V_1, \cdots, V_{t+1}) $ is completely determined by the constants in \Cref{eqn:chi-alt,eqn:Psit}. 

We are now ready to present the state evolution result. Its proof, deferred to \Cref{sec:pf_SE},  reduces the GAMP iteration in \Cref{eqn:GAMP_nonsep} to a family of abstract AMP algorithms introduced in \cite{GraphAMP} for which a general state evolution result has been established. 
In the abstract AMP algorithm, iterates are associated with the edges of a given directed graph, and the denoising functions are allowed to be non-separable, as needed in our case.

\begin{proposition}[State evolution]
\label{prop:SE}
Consider the GLM in \Cref{sec:model} subject to \Cref{asmp:signal_prior,asmp:sigma,asmp:noise,asmp:proportional} and the GAMP iteration in \Cref{eqn:GAMP_nonsep}. 
Let initializers $ \wt{u}^{-1} = 0_n $ and $ \wt{v}^0\in\bbR^d $ satisfy \Cref{asmp:init}. 
For every $ t\ge0 $, 
let $ (g_t\colon\bbR^{2n}\to\bbR^n)_{n\ge1} $ and $ (f_{t+1}\colon\bbR^d\to\bbR^d)_{d\ge1} $ be uniformly pseudo-Lipschitz  functions of finite constant order subject to \Cref{asmp:SE}. 
For any $ t\ge0 $, let $ (h_1\colon\bbR^{n(t+2)} \to \bbR)_{n\ge1} $ and $ (h_2\colon\bbR^{d(t+2)} \to \bbR)_{d\ge1} $ be two sequences of uniformly pseudo-Lipschitz  test functions of finite order.
Then, 
\begin{align}
\begin{aligned}
    \plim_{n\to\infty} h_1(g, u^0, u^1, \cdots, u^t) - \expt{ h_1(G, U_0, U_1, \cdots, U_t) } &= 0 , \\
    \plim_{d\to\infty} h_2(\wt{\beta}^*, v^1, v^2, \cdots, v^{t+1}) - \expt{ h_2(\wt{\mathfrak{B}}^*, V_1, V_2, \cdots, V_{t+1}) } &= 0 , 
\end{aligned}
\label{eqn:SE_GAMP_nonsep}
\end{align}
where $ (U_t, V_{t+1})_{t\ge0} $ are given in \Cref{eqn:def-Ut-Vt+1}. 
\end{proposition}

\subsection{GAMP as a power method and its fixed points}
\label{sec:fp-se}
We now formalize the argument in \Cref{sec:pf_technique}.  Recall the definition of $ \ol{Y} $ in \Cref{eqn:rand-var} and $ s(\cdot) $ in \Cref{eqn:def-s(a)}. 
Let 
\begin{align}
    \domain \coloneqq \brace{ (a, \gamma) : a>\sup\supp(\cT(\ol{Y})), \gamma>s(a) }
    \label{eqn:def-domain}
\end{align}
and $ (a^*, \gamma(a^*)) \in \domain $ be defined through \Cref{eqn:def_a*}, where the largest solution $ a^* $ is taken. For convenience, for the rest of the paper, we will use the shorthand
\begin{align}
    \gamma^* \coloneqq \gamma(a^*) , \quad 
    \gamma^\circ \coloneqq \gamma(a^\circ) . 
    \label{eqn:def_gamma0*} 
\end{align} 
If $ a^* > a^\circ $ (where $a^\circ$ is defined in \Cref{eqn:def_a_circ}), \Cref{prop:equiv-def-a-gamma} shows that this pair of equations is equivalent to \Cref{eqn:fp-a-gamma}. 
Furthermore, let 
\begin{align}
\begin{aligned}
    \cF_a(\cdot) &= \frac{\cT(\cdot)}{a - \cT(\cdot)} , \quad a > \sup\supp(\cT(\ol{Y})) , \\
    F &= \diag(\cF_{a^*}(y)) , \quad 
    c = \expt{\cF_{a^*}(\ol{Y})} . 
\end{aligned}
\label{eqn:my-denoisers-1} 
\end{align}
Let us initialize the iteration in \Cref{eqn:GAMP_nonsep} with $ \wt{u}^{-1} = 0_n $ and $ \wt{v}^0\in\bbR^d $ defined in \Cref{eqn:artificialGAMP_init}, and for subsequent iterates, set 
\begin{align}
g_t(u^t; y) &= Fu^t , \qquad 
f_{t+1}(v^{t+1}) = (\gamma_{t+1} I_d - c \Sigma)^{-1} \Sigma v^{t+1} , \quad t\ge0 . \label{eqn:my-denoisers-2} 
\end{align}
Recall from \Cref{asmp:preprocessor} that $\cT\colon\bbR\to\bbR$ is bounded and pseudo-Lipschitz of finite order. 
Since $ a^* > \sup\supp(\cT(\ol{Y})) $, $\cF_{a^*}\colon\bbR\to\bbR$ is also bounded and pseudo-Lipschitz of finite order. 
Therefore, for every $ t\ge0 $, $ (g_t \colon \bbR^n \times \bbR^n \to \bbR^n)_{n\ge1} $ is a sequence of uniformly pseudo-Lipschitz functions of finite order in both arguments. 
The parameter $ \gamma_{t+1}\in(s(a^*), \infty) $ is s.t.\ 
\begin{align}
\plim_{d\to\infty} \frac{1}{d} \normtwo{f_{t+1}(v^{t+1})}^2 & = \lim_{d\to\infty} \frac{1}{d} \expt{\normtwo{f_{t+1}(V_{t+1})}^2} = 1 \label{eqn:SE_before_gamma}
\end{align}
for $ t\ge0 $. The first equality above follows from the state evolution result in \Cref{prop:SE}. 
For notational convenience, let
\begin{align}
B_{t+1} &\coloneqq (\gamma_{t+1} I_d - c \Sigma)^{-1} \Sigma . \label{eqn:def-mtx-B}
\end{align}
Since $ \gamma_{t+1} > s(a^*) $ and $ \normtwo{\Sigma} $ is uniformly bounded by \Cref{asmp:sigma}, $ \normtwo{B_{t+1}} $ is uniformly bounded. 
Therefore for every $ t\ge0 $, $ (f_{t+1}\colon\bbR^d\to\bbR^d)_{d\ge1} $ is a sequence of pseudo-Lipschitz functions of order $1$. 

With the above definitions, the Onsager coefficients become 
\begin{align}
    c_t &= \frac{1}{n} \tr(F) , \quad 
    b_{t+1} = \frac{d}{n} \tr(B_{t+1}) , 
    \label{eqn:my_Onsager_coeff}
\end{align}
for every $ t\ge0 $. 
Furthermore, the state evolution in \Cref{eqn:mut,eqn:chi-alt} specializes to the following recursion
\begin{align}
\begin{aligned}
\mu_t &= \frac{\delta}{\expt{\ol{\Sigma}}} \lim_{n\to\infty}\frac{1}{n} \expt{(\wt{\mathfrak{B}}^*)^\top B_t V_t} , \\
\sigma_{U, t}^2 &= \lim_{n\to\infty}\frac{1}{n} \expt{V_t^\top B_t^\top B_t V_t} - \frac{\expt{\ol{\Sigma}}}{\delta} \mu_t^2 , \\
\chi_{t+1} &= \frac{\delta}{\expt{\ol{\Sigma}}} \lim_{n\to\infty} \frac{1}{n} \expt{G^\top \diag(\cF_{a^*}(Y)) U_t} - \mu_t \expt{\cF_{a^*}(\ol{Y})} , \\
\sigma_{V, t+1}^2 &= \lim_{n\to\infty} \frac{1}{n} \expt{U_t^\top \diag(\cF_{a^*}(Y))^2 U_t} . 
\end{aligned}
\label{eqn:SE_before}
\end{align}

Let
\begin{align}
    z_1 \coloneqq \expt{\frac{\ol{\Sigma}^3}{\paren{ \gamma^* - \expt{\cF_{a^*}(\ol{Y})} \ol{\Sigma} }^2}} , \quad 
    z_2 \coloneqq \expt{\frac{\ol{\Sigma}^2}{\paren{ \gamma^* - \expt{\cF_{a^*}(\ol{Y})} \ol{\Sigma} }^2}} . 
    \label{eqn:def_z1_z2_main}
\end{align}
Note that $ z_1, z_2 > 0 $. Recalling $ w_1, w_2 $ from \Cref{eqn:def_x1_main,eqn:def_x2_main}, 
define 
\begin{align}
    \chi &= \sqrt{\frac{1-w_2}{(1-w_2) z_1 + w_1 z_2}} , \quad 
    \sigma_V = \sqrt{\frac{w_1}{(1-w_2) z_1 + w_1 z_2}} , \label{eqn:def_chi_sigmaV} \\
    \mu &= \frac{1}{\expt{\ol{\Sigma}}} \expt{\frac{\ol{\Sigma}^2}{ \gamma^* - \expt{\cF_{a^*}(\ol{Y})} \ol{\Sigma} }} \sqrt{\frac{1-w_2}{(1-w_2) z_1 + w_1 z_2}} , \label{eqn:def_mu} \\
    \sigma_U &= \sqrt{\frac{1/\delta}{(1-w_2) z_1 + w_1 z_2}} \left( \expt{\frac{\ol{\Sigma}^3}{\paren{ \gamma^* - \expt{\cF_{a^*}(\ol{Y})} \ol{\Sigma} }^2}} - \frac{1}{\expt{\ol{\Sigma}}} \expt{\frac{\ol{\Sigma}^2}{\gamma^* - \expt{\cF_{a^*}(\ol{Y})} \ol{\Sigma}}}^2 \right. \notag \\
    &\phantom{=}~+ \left. \frac{1}{\expt{\ol{\Sigma}}^2} \expt{\frac{\ol{\Sigma}^2}{\paren{ \gamma^* - \expt{\cF_{a^*}(\ol{Y})} \ol{\Sigma} }^2}} \expt{ \ol{G}^2 \cF_{a^*}(\ol{Y})^2} \expt{\frac{\ol{\Sigma}^2}{\gamma^* - \expt{\cF_{a^*}(\ol{Y})} \ol{\Sigma}}}^2 \right)^{1/2} . \label{eqn:def_sigmaU} 
\end{align}
Note that all these quantities are well-defined provided $ a^* > a^\circ $. 
Indeed, $ w_1>0 $ and $ 1-w_2>0 $ under the latter condition. 
Also, the second factor in the definition of $\sigma_{U}$ is positive since the sum of the first two terms is non-negative by Cauchy-Schwarz and the third term is positive. 
Define also $ \gamma^\sharp $ as the unique solution in $ \paren{s(a^*), \infty} $ to
\begin{align}
    1 &= \frac{1}{\delta} \expt{\cF_{a^*}(\ol{Y})^2} \expt{\frac{\ol{\Sigma}^2}{\paren{ \gamma^\sharp - \expt{\cF_{a^*}(\ol{Y})} \ol{\Sigma} }^2}} . \label{eqn:def_gamma_sharp} 
\end{align}
The well-posedness of $ \gamma^\sharp$ follows the same reasoning after \Cref{eqn:def-gamma-fn}. 

We now characterize the fixed points of state evolution and show that the recursion can be initialized precisely at the fixed point. The proof of the next two lemmas are obtained via a series of manipulations which are deferred to \Cref{subsec:pflem:FP_SE,subsec:pflem:SE_stay}.

\begin{lemma}[Fixed points of state evolution]
\label{lem:FP_SE}
The quintuple $ (\mu_t, \sigma_{U,t}, \chi_{t+1}, \sigma_{V,t+1}, \gamma_{t+1}) $ in the recursion given by \Cref{eqn:SE_before,eqn:SE_before_gamma} has $3$ fixed points $ \mathsf{FP}_+, \mathsf{FP}_-, \mathsf{FP}_0 \in\bbR^5 $: 
\begin{gather}
    \mathsf{FP}_+ = (\mu, \sigma_U, \chi, \sigma_V, \gamma^*) , \quad 
    \mathsf{FP}_- = (-\mu, \sigma_U, -\chi, \sigma_V, \gamma^*) , \notag \\
    \mathsf{FP}_0 = \paren{ 0, \frac{1}{\sqrt{\delta}}, 0, \expt{\frac{\ol{\Sigma}^2}{\paren{ \gamma^\sharp - \expt{\cF_{a^*}(\ol{Y})} \ol{\Sigma} }^2}}^{-1/2}, \gamma^\sharp } , \notag 
\end{gather}
where the parameters on the right are given in \Cref{eqn:def_chi_sigmaV,eqn:def_mu,eqn:def_sigmaU,eqn:def_gamma_sharp,eqn:def_gamma0*}.
\end{lemma}

We initialize the AMP iteration with 
\begin{align}
    \wt{u}^{-1} = 0_n , \quad 
    \wt{v}^0 \coloneqq \mu\, \wt{\beta}^* + \sqrt{1 - \mu^2 \expt{\ol{\Sigma}}}\, w \in \bbR^d
    \label{eqn:artificialGAMP_init}
\end{align}
where we have set $\wt{\beta}^*=\Sigma^{1/2}\beta^*$, $w\sim\cN(0_d, I_d) $ is independent of everything else and $\mu$ is given in \Cref{eqn:def_mu}. 
This choice is valid since from the proof of \Cref{lem:SE_stay} one can deduce that $ 1 - \mu^2 \expt{\ol{\Sigma}} > 0 $. 
The scaling ensures that 
$\plim_{d\to\infty} \normtwo{\wt{v}^0}^2/d = 1$
almost surely. 
According to \Cref{eqn:mu0}, \Cref{eqn:artificialGAMP_init} gives that the state evolution parameters are initialized as
\begin{align}
\begin{aligned}
\mu_0 &= \frac{\delta}{\expt{\ol{\Sigma}}} \lim_{n\to\infty} \frac{\mu}{n} \expt{\inprod{\wt{\mathfrak{B}}^*}{\wt{\mathfrak{B}}^*}} 
= \mu , \\
\sigma_{U,0}^2 &= \plim_{n\to\infty} \frac{1}{n} \inprod{\wt{v}^0}{\wt{v}^0} - \frac{\expt{\ol{\Sigma}}}{\delta} \mu_0^2 
= \frac{1}{\delta} - \frac{\expt{\ol{\Sigma}}}{\delta} \mu^2 .
\end{aligned}
\label{eqn:SE_init} 
\end{align}

\begin{lemma}[State evolution stays put]
\label{lem:SE_stay}
Initialized with \Cref{eqn:SE_init}, the parameters $ (\mu_t, \sigma_{U,t}, \allowbreak \chi_{t+1}, \sigma_{V, t+1})_{t\ge0} $ of the state evolution recursion in \Cref{eqn:SE_before,eqn:SE_before_gamma} stay at the initialization, that is, for every $ t\ge0 $:
\begin{align}
    \mu_t &= \mu , \quad 
    \sigma_{U, t} = \sigma_U , \quad 
    \chi_{t+1} = \chi , \quad 
    \sigma_{V, t+1} = \sigma_V , \quad 
    \gamma_{t+1} = \gamma^* , \notag 
\end{align}
where the right-hand sides are defined in \Cref{eqn:def_chi_sigmaV,eqn:def_mu,eqn:def_sigmaU,eqn:def_gamma0*}. 
\end{lemma}

\subsection{Right edge of the bulk of $D$}
\label{sec:bulk}
Let
\begin{align}
\wh{D} &= \Sigma^{1/2} \wh{X}^\top T \wh{X} \Sigma^{1/2} ,\quad \mbox{with }  \ T = \diag(\cT(y)) = \diag(\cT(q(\wt{X} \Sigma^{1/2} \beta^*, \eps))) , \label{eqn:def_D_hat} 
\end{align}
and $ \wh{X}\in\bbR^{n\times d} $ has i.i.d.\ $ \cN(0,1/n) $ entries, independent of $T$.
One should think of $ \wh{D} $ as a ``decoupled'' version of $ D $ in the sense that $ \wh{X} $ and $T$ are independent and no outlier eigenvalue is expected to show up in the spectrum of $\wh{D}$. 
This is to be contrasted with $ D = \Sigma^{1/2} \wt{X}^\top T \wt{X} \Sigma^{1/2} $ in which $T$ depends on $\wt{X}$ (see \Cref{eqn:def_D_hat}),
and the top eigenvalue of $D$ will be detached from the bulk of the spectrum provided that 
$ a^* > a^\circ $. 

Given the above intuition, one expects that the behaviour of the right edge of the bulk of $\wh{D}$ resembles that of $D$. 
This is made formal in the following lemma, which is proved in \Cref{sec:pf_bulk}. 
The idea is to first show that $ \lambda_3(\wh{D}) \le \lambda_2(D) \le \lambda_1(\wh{D}) $ using the variational representation of eigenvalues, and then use \cite{FanSunWang-mixed,Zhang_Thesis_RMT} to show that both $ \lambda_1(\wh{D})$ and $ \lambda_3(\wh{D}) $ converge to the right edge of the bulk of $\wh{D}$. 
We comment on the second step. 
Building on the almost sure weak convergence result of the empirical spectral distribution of $\wh{D}$ \cite[Theorem 1.2.1]{Zhang_Thesis_RMT}, it was proved in \cite[Theorem 1]{PaulSilv} that almost surely there exists no eigenvalue outside the support of the limiting spectral distribution, and \cite[Section 3]{CouilletHachem} further characterized the support. 
However, both \cite{PaulSilv,CouilletHachem} assumed a positive semidefinite $T$ which corresponds to $\cT\ge0$. Thus, we build on \cite[Theorem 1.2.1]{Zhang_Thesis_RMT} and use a recent strong asymptotic freeness result of GOE and deterministic matrices \cite[Theorem 4.3]{FanSunWang-mixed} which guarantees the absence of eigenvalues outside the support  of the limiting spectral distribution. 
Of particular benefit to our purposes is that neither \cite[Theorem 1.2.1]{Zhang_Thesis_RMT} nor \cite[Theorem 4.3]{FanSunWang-mixed} requires $T$ to be PSD. 

\begin{lemma}
\label{lem:bulk}
Consider the matrices $D$ and $ \wh{D} $ in \Cref{eqn:def_D_intro,eqn:def_D_hat}, respectively. 
Denote by $ \ol{\mu}_{\wh{D}} $ the limiting spectral distribution of $\wh{D}$. 
Then, we have 
\begin{align}
\lim_{d\to\infty} \lambda_2(D) &= \sup\supp(\ol{\mu}_{\wh{D}}) \quad \text{almost surely} . \label{eqn:bulk}
\end{align}
\end{lemma}

Next, we characterize the right edge of the support of $ \ol{\mu}_{\wh{D}} $. The detailed proof of the lemma below is given in \Cref{sec:bulk_pf}, and it generalizes the analysis in \cite[Section 3]{CouilletHachem}, showing that the same characterization of the support therein also holds for a possibly non-positive $T$ (or equivalently $\caT$). The critical obstacle for non-positive $T$ is that the Stieltjes-like transform $z\mapsto\E \left[\frac{\caT(\ol{Y})}{\caT(\ol{Y})-z}\right]$ no longer maps the complex upper-half plane into itself, rendering parts of \cite{CouilletHachem} using this property unusable. We treat this problem by considering meromorphic generalizations of various concepts in \cite{CouilletHachem} (e.g., \Cref{prop:Pick} in \Cref{sec:bulk_pf} plays the role of Proposition 1.2 in \cite{CouilletHachem}). 

\begin{lemma}
\label{lem:bulk_nonpositive}
Let $ a^\circ > \sup\supp(\cT(\ol{Y})) $ be the largest critical point of $\psi$. 
Then, we have
\begin{align}
    \sup\supp(\ol{\mu}_{\wh{D}}) = \psi(a^\circ) . 
    \label{eqn:bulk-selfcons}
\end{align}
\end{lemma}

\subsection{Concluding the proof of \Cref{thm:main}}
\label{sec:pf-AMP}
In this final section, we show the following lemma. 
\begin{lemma}
\label{lem:key}
Consider $ v^t $ obtained from the GAMP iteration in \Cref{eqn:GAMP_nonsep} with denoisers in \Cref{eqn:my-denoisers-1,eqn:my-denoisers-2} and initializers in \Cref{eqn:artificialGAMP_init}.
Let $ \wh{v}^t \coloneqq \Sigma^{-1/2} (\gamma^* I_d - c \Sigma)^{-1} \Sigma v^t$. 
If $ a^* > a^\circ $,
\begin{align}
    \lim_{t\to\infty} \plim_{d\to\infty} \frac{\inprod{\wh{v}^t}{v_1(D)}^2}{ \normtwo{\wh{v}^t}^2 } = 1 ,\qquad \plim_{d\to\infty} \lambda_1(D) = \lambda_1 >\lambda_2, \label{eqn:overlap_vhat_v1} 
\end{align}
where $ \lambda_1 ,\lambda_2 $ are defined in \Cref{eqn:def_lambda1_lambda2}. 
\end{lemma}
Then,  \Cref{eqn:overlap_vhat_v1} directly gives the first part of \Cref{eqn:main_thm_eigval} in \Cref{thm:main}; the second part follows from \Cref{lem:bulk,lem:bulk_nonpositive}; and the expression in \Cref{eqn:main_thm_overlap} for the overlap is a consequence of state evolution, whose proof is given at the end of this section. 

\begin{proof}[Proof of \Cref{lem:key}.]
Recall the following definitions: $ B_{t+1} $ in \Cref{eqn:def-mtx-B}, $ f_{t+1}(v^{t+1}) = B_{t+1} v^{t+1} $ (see \Cref{eqn:my-denoisers-2}) and $ c = \expt{\cF_{a^*}(\ol{Y})} $ (see \Cref{eqn:my-denoisers-1}). 
Let 
\begin{align}
    b \coloneqq \frac{1}{\delta} \expt{\frac{\ol{\Sigma}}{\gamma^* - c \ol{\Sigma}}} , \quad
    B \coloneqq (\gamma^* I_d - c \Sigma)^{-1} \Sigma , \label{eqn:def_B}
\end{align}
be the fixed points of $ b_{t+1}, B_{t+1} $, respectively, where $ \gamma^* $ (together with $a^*$) satisfies \Cref{eqn:fp-a-gamma}. 
Note that $ b = 1 $ by \Cref{eqn:def-gamma-fn}. 
For $ t\ge1 $, define
\begin{align}
e_1^t &\coloneqq u^t - u^{t-1} \in\bbR^n , \quad 
e_2^t \coloneqq v^{t+1} - v^t \in\bbR^d . \label{eqn:e1_e2}
\end{align}
The GAMP iteration in \Cref{eqn:GAMP_nonsep} can be written as
\begin{align}
u^t &= \wt{X} B_t v^t - b_t F u^{t-1} , \quad 
v^{t+1} = \wt{X}^\top F u^t - c_t B_t v^t . \label{eqn:my_ut_vt+1}
\end{align}
Using the first equation in the second, we get
\begin{align}
v^{t+1} &= (\wt{X}^\top F \wt{X} - c_t I_d) B_t v^t - b_t \wt{X}^\top F^2 u^{t-1} . \label{eqn:vt+1} 
\end{align}
Using the definition of $ e_1^t $ in the iteration for $ u^t $, we have
\begin{align}
    u^{t-1} &= \wt{X} B_t v^t - b_t F u^{t-1} - e_1^t . \notag
\end{align}
Solving for $ u^{t-1} $ yields:
\begin{align}
u^{t-1} &= (b_t F + I_n)^{-1} \wt{X} B_t v^t - (b_t F + I_n)^{-1} e_1^t . \notag 
\end{align}
Then, we can eliminate $ u^{t-1} $ in the iteration for $ v^{t+1} $ by substituting the right-hand side above in \Cref{eqn:vt+1} and, after some manipulations, we obtain
\begin{align}
v^{t+1} 
&= \brack{\wt{X}^\top F (b_t F + I_n)^{-1} \wt{X} - c_t I_d} B_t v^t + b_t \wt{X}^\top F^2 (b_t F + I_n)^{-1} e_1^t . \notag 
\end{align}
We expand $b_{t}$ and $B_{t}$ respectively around their fixed points $b$ and $B$ to write
\begin{align}
v^{t+1} 
&= \brack{\wt{X}^\top F (b F + I_n)^{-1} \wt{X} - c I_d} B v^t \notag \\
&\phantom{=} + (b - b_t) \wt{X}^\top F (b_t F + I_n)^{-1} (b F + I_n)^{-1} \wt{X} B_t v^t \notag \\
&\phantom{=} + (\gamma^* - \gamma_t) \wt{X}^\top F (b F + I_n)^{-1} \wt{X} (\gamma_t I_d - c \Sigma)^{-1} (\gamma^* I_d - c \Sigma)^{-1} \Sigma v^t \notag \\
&\phantom{=} + (c_t - c) B_t v^t + c (\gamma_t - \gamma^*) (\gamma_t I_d - c \Sigma)^{-1} (\gamma^* I_d - c \Sigma)^{-1} \Sigma v^t \notag \\
&\phantom{=} + b_t \wt{X}^\top F^2 (b_t F + I_n)^{-1} e_1^t . \notag 
\end{align}
Using the definition of $ e_2^t $, we further have
\begin{align}
\begin{aligned}
(I_d + c B) v^{t+1} &= \wt{X}^\top F (b F + I_n)^{-1} \wt{X} B v^t + c B e_2^t \\
&\phantom{=} + (b - b_t) \wt{X}^\top F (b_t F + I_n)^{-1} (b F + I_n)^{-1} \wt{X} B_t v^t \\
&\phantom{=} + (\gamma^* - \gamma_t) \wt{X}^\top F (b F + I_n)^{-1} \wt{X} (\gamma_t I_d - c \Sigma)^{-1} (\gamma^* I_d - c \Sigma)^{-1} \Sigma v^t \\
&\phantom{=} + (c_t - c) B_t v^t + c (\gamma_t - \gamma^*) (\gamma_t I_d - c \Sigma)^{-1} (\gamma^* I_d - c \Sigma)^{-1} \Sigma v^t \\
&\phantom{=} + b_t \wt{X}^\top F^2 (b_t F + I_n)^{-1} e_1^t . 
\end{aligned}
\label{eqn:almost_pow_itr}
\end{align}
Define $ e^t\in\bbR^d $ by \
\begin{align}
\begin{aligned}
e^t &\coloneqq c \Sigma^{1/2} B e_2^t 
+ (b - b_t) \Sigma^{1/2} \wt{X}^\top F (b_t F + I_n)^{-1} (b F + I_n)^{-1} \wt{X} B_t v^t \\
&\phantom{=} + (\gamma^* - \gamma_t) \Sigma^{1/2} \wt{X}^\top F (b F + I_n)^{-1} \wt{X} (\gamma_t I_d - c \Sigma)^{-1} (\gamma^* I_d - c \Sigma)^{-1} \Sigma v^t \\
&\phantom{=} + (c_t - c) \Sigma^{1/2} B_t v^t + c (\gamma_t - \gamma^*) \Sigma^{1/2} (\gamma_t I_d - c \Sigma)^{-1} (\gamma^* I_d - c \Sigma)^{-1} \Sigma v^t \\
&\phantom{=} + b_t \Sigma^{1/2} \wt{X}^\top F^2 (b_t F + I_n)^{-1} e_1^t . 
\end{aligned}
\label{eqn:def_et}
\end{align}
Multiplying both sides of \Cref{eqn:almost_pow_itr} by $ \Sigma^{1/2} $, we arrive at
\begin{align}
\Sigma^{1/2} (I_d + c B) v^{t+1} &= \Sigma^{1/2} \wt{X}^\top F (b F + I_n)^{-1} \wt{X} B v^t + e^t . \notag 
\end{align}
By the definition of $D$ (see \Cref{eqn:def_D_intro}) and the choice of $\cF_{a^*}$ (see \Cref{eqn:my-denoisers-1}), we note that $ \Sigma^{1/2} \wt{X}^\top F (b F + I_n)^{-1} \wt{X} \Sigma^{1/2} = \frac{1}{a^*} D $ (recall from \Cref{eqn:def_B} that $ b = 1 $). 
Also, by the definition of $B$ (see \Cref{eqn:def_B}), we have the identity
\begin{align}
    \frac{1}{\gamma^*} \Sigma^{1/2} (I_d + c B) = \Sigma^{-1/2} B, 
    \label{eqn:def-B-tilde} 
\end{align}
both sides of which we define to be $ \wt{B}\in\bbR^{d\times d} $. 
Using the above observations and letting 
\begin{align}
    \wh{v}^{t+1} \coloneqq \wt{B} v^{t+1} \in\bbR^d , \label{eqn:def_vhat}
\end{align}
we obtain 
\begin{align}
\wh{v}^{t+1} &= M \wh{v}^t + \frac{1}{\gamma^*} e^t , \quad 
\text{where} \  M \coloneqq \frac{D}{\lambda_1} , \ 
\lambda_1 \coloneqq a^*\gamma^* , 
\label{eqn:pow_itr}
\end{align}
which takes the form of a power iteration with an error term. 

It is now convenient to
 shift the spectrum of $M$ to the right so that all of its eigenvalues are positive. 
Specifically, choose $ \ell>0 $ to be a sufficiently large constant. 
By \Cref{eqn:bound_eigval_D}, it suffices to take 
$    \ell = \const_D + 1 > \normtwo{D}+1 $, 
where the constant $ \const_D \in (0,\infty) $ is defined in \Cref{eqn:bound_||D||}.
Adding $ \frac{\ell}{\lambda_1} \wh{v}^{t+1} $ on both sides of \Cref{eqn:pow_itr} and using the definitions of $ \wh{v}^t $ in \Cref{eqn:def_vhat} and $ e_2^t $ in \Cref{eqn:e1_e2}, we have
\begin{align}
    \paren{ 1 + \frac{\ell}{\lambda_1} } \wh{v}^{t+1}
    &= \frac{D + \ell I_d}{\lambda_1} \, \wh{v}^t + \frac{\ell}{\lambda_1} \wt{B} e_2^t + \frac{1}{\gamma^*} e^t . \notag 
\end{align}
Using the following notation:
\begin{align}
    \wh{M} &\coloneqq \frac{D + \ell I_d}{\lambda_1 + \ell}  , \quad 
    \wh{e}^t \coloneqq \frac{\ell}{\lambda_1 + \ell} \wt{B} e_2^t + \frac{a^*}{\lambda_1 + \ell} e^t , \label{eqn:def-Mhat-ehat} 
\end{align}
we write the iteration as
\begin{align}
    \wh{v}^{t+1} &= \wh{M} \wh{v}^t + \wh{e}^t . \label{eqn:pow_itr_shifted}
\end{align}
By construction, $ \wh{M} $ is strictly positive definite, and 
all results concerning the spectral properties of $ \wh{M} $ can be easily translated to those of $ M $ by cancelling the shift $\ell$. 

Suppose that the iteration in \Cref{eqn:pow_itr_shifted} has been run for a certain large constant $t>0$ steps. 
We further run it for an additional $ t' $ steps for some large constant $t'>0$. 
By unrolling the iteration down to time $t$, we obtain
\begin{align}
\wh{v}^{t+t'} &= \wh{M}^{t'} \wh{v}^t + \wh{e}^{t,t'} , \label{eqn:pow_itr_t'}
\end{align}
where 
\begin{align}
\wh{e}^{t,t'} &\coloneqq \sum_{s = 1}^{t'} \wh{M}^{t' - s} \wh{e}^{t + s - 1} . \label{eqn:def_hat_e_t_t'} 
\end{align}
Taking the normalized squared norm $ \frac{1}{d} \normtwo{\cdot}^2 $ on both sides of \Cref{eqn:pow_itr_t'} and sending first $d$ then $t$ and finally $t'$ to infinity, we get the left-hand side 
\begin{align}
&\phantom{=}~\lim_{t'\to\infty} \lim_{t\to\infty} \plim_{d\to\infty} \frac{1}{d} \normtwo{\wh{v}^{t+t'}}^2 
= \lim_{t'\to\infty} \lim_{t\to\infty} \plim_{d\to\infty} \frac{1}{d} \normtwo{\wt{B} v^{t+t'}}^2 \notag \\
&= \lim_{t'\to\infty} \lim_{t\to\infty} \plim_{d\to\infty} \frac{1}{d} \normtwo{\Sigma^{-1/2} (\gamma^* I_d - c \Sigma)^{-1} \Sigma v^{t+t'}}^2 \notag \\
&= \lim_{t'\to\infty} \lim_{t\to\infty} \lim_{d\to\infty} \frac{1}{d} \expt{\normtwo{ \Sigma^{-1/2} (\gamma^* I_d - c \Sigma)^{-1} \Sigma V_{t+t'} }^2} \notag \\
&= \lim_{t'\to\infty} \lim_{t\to\infty} \lim_{d\to\infty} \frac{1}{d} \expt{\normtwo{ \Sigma^{-1/2} (\gamma^* I_d - c \Sigma)^{-1} \Sigma \wt{\mathfrak{B}}^* }^2} \chi_{t+t'}^2  \notag \\
 &\phantom{=\lim_{t'\to\infty} \lim_{t\to\infty} \lim_{d\to\infty}}~ 
+ \frac{1}{d} \expt{\normtwo{ \Sigma^{-1/2} (\gamma^* I_d - c \Sigma)^{-1} \Sigma W_{V, t+t'} }^2} \sigma_{V,t+t'}^2 \notag \\
&= \lim_{t'\to\infty} \lim_{t\to\infty} \lim_{d\to\infty} \frac{1}{d} \expt{ {\mathfrak{B}^*}^\top \Sigma^{1/2} \Sigma (\gamma^* I_d - c \Sigma)^{-1} \Sigma^{-1} (\gamma^* I_d - c \Sigma)^{-1} \Sigma \Sigma^{1/2} \mathfrak{B}^* } \chi_{t+t'}^2 \notag \\
&\phantom{=\lim_{t'\to\infty} \lim_{t\to\infty} \lim_{d\to\infty}}~ + \frac{1}{d} \expt{ W_{V, t+t'}^\top \Sigma (\gamma^* I_d - c \Sigma)^{-1} \Sigma^{-1} (\gamma^* I_d - c \Sigma)^{-1} \Sigma W_{V, t+t'} } \sigma_{V,t+t'}^2 \notag \\
&= \lim_{t'\to\infty} \lim_{t\to\infty} \expt{\frac{\ol{\Sigma}^2}{(\gamma^* - c \ol{\Sigma})^2}} \chi_{t+t'}^2 + \expt{\frac{\ol{\Sigma}}{(\gamma^* - c \ol{\Sigma})^2}} \sigma_{V, t+t'}^2 \notag \\
&= \expt{\frac{\ol{\Sigma}^2}{(\gamma^* - c \ol{\Sigma})^2}} \chi^2 + \expt{\frac{\ol{\Sigma}}{(\gamma^* - c \ol{\Sigma})^2}} \sigma_V^2 \eqqcolon \nu^2 , \label{eqn:LHS_result} 
\end{align}
where we use the state evolution result (\Cref{prop:SE}) in the third equality. 
Taking $ \frac{1}{d} \normtwo{\cdot}^2 $ and the same sequential limits on the right-hand side, we have:
\begin{align}
    \lim_{t'\to\infty} \lim_{t\to\infty} \plim_{d\to\infty} \frac{1}{d} \normtwo{\wh{M}^{t'} \wh{v}^t + \wh{e}^{t,t'}}^2 . 
    \label{eqn:RHS_with_errors}
\end{align}
We claim that 
\begin{align}
\lim_{t'\to\infty} \lim_{t\to\infty} \plim_{d\to\infty} \frac{1}{d} \normtwo{\wh{e}^{t,t'}}^2 &= 0 , \label{eqn:error_t_t'}
\end{align}
which implies, by the triangle inequality, that
\Cref{eqn:RHS_with_errors} is equal to 
\begin{align}
\lim_{t'\to\infty} \lim_{t\to\infty} \plim_{d\to\infty} \frac{1}{d} \normtwo{\wh{M}^{t'} \wh{v}^t}^2 . \label{eqn:RHS_todo}
\end{align}
The proof of \Cref{eqn:error_t_t'} requires the technical analysis of various error terms, and it is deferred to \Cref{sec:pf_bound_error_t_t'}. 
The quantity in \Cref{eqn:RHS_todo} can be decomposed as
\begin{align}
\frac{1}{d} \normtwo{\wh{M}^{t'} \wh{v}^t}^2 &= \frac{1}{d} \normtwo{\wh{M}^{t'} (\Pi + \Pi^\perp) \wh{v}^t}^2
\notag\\
&= \frac{1}{d} \normtwo{\wh{M}^{t'} \Pi \wh{v}^t}^2 + \frac{1}{d} \normtwo{\wh{M}^{t'} \Pi^\perp \wh{v}^t}^2 + \frac{2}{d} \inprod{\wh{M}^{t'} \Pi \wh{v}^t}{\wh{M}^{t'} \Pi^\perp \wh{v}^t} , \label{eqn:proj_first_second}
\end{align}
where $\Pi \coloneqq v_1(D) v_1(D)^\top$ and $\Pi^\perp \coloneqq I_d - \Pi$.
Note that the eigendecomposition of $\wh{M}^{t'}$ is 
\begin{align}
    \wh{M}^{t'} &= \sum_{i = 1}^d \lambda_i(\wh{M}^{t'}) v_i(\wh{M}^{t'}) v_i(\wh{M}^{t'})^\top
    = \sum_{i = 1}^d \lambda_i(\wh{M})^{t'} v_i(D) v_i(D)^\top , \notag 
\end{align}
since for any univariate polynomial $ P $ with real coefficients and any matrix $ K\in\bbR^{d\times d} $, $ P(K) $ shares the same eigenspace with $ K $ and its eigenvalues are $ \{P(\lambda_i(K))\}_{i\in\{1, \ldots, d\}} $. 
Therefore, the first term on the right-hand side of \Cref{eqn:proj_first_second} equals
\begin{align}
\frac{1}{d} \normtwo{\wh{M}^{t'} \Pi \wh{v}^t}^2 &= \frac{1}{d} \normtwo{\sum_{i=1}^d \lambda_i(\wh{M})^{t'} v_i(D) v_i(D)^\top \Pi \wh{v}^t}^2 %
\notag \\ &
= \frac{1}{d} \normtwo{\lambda_1(\wh{M})^{t'} v_1(D) v_1(D)^\top \wh{v}^t}^2 
= \lambda_1(\wh{M})^{2t'} \frac{\inprod{v_1(D)}{\wh{v}^t}^2}{d} . \label{eqn:RHS_first} 
\end{align}
The third term on the right-hand side of \Cref{eqn:proj_first_second} vanishes: 
\begin{align}
    \frac{1}{d} & \inprod{\wh{M}^{t'} \Pi \wh{v}^t}{\wh{M}^{t'} \Pi^\perp \wh{v}^t}\notag\\
    &= \frac{1}{d} \inprod{ \lambda_1(\wh{M})^{t'} \inprod{v_1(D)}{\wh{v}^t} \, v_1(D)}{ \sum_{i = 2}^d \lambda_i(\wh{M})^{t'} \inprod{v_i(D)}{\wh{v}^t} v_i(D) }
    = 0 . \label{eqn:RHS_third}
\end{align}
To analyze the second term on the right-hand side of \Cref{eqn:proj_first_second}, we  define the matrix
\begin{align}
    \wt{M} &\coloneqq  \wh{M} \Pi^\perp  = \sum_{i = 2}^d \lambda_i(\wh{M}) v_i(D) v_i(D)^\top. 
    \notag
\end{align}
We then have
\begin{align}
\frac{1}{d} \normtwo{\wh{M}^{t'} \Pi^\perp \wh{v}^t}^2 &= \frac{1}{d} \normtwo{\sum_{i=2}^d \lambda_i(\wh{M})^{t'} v_i(D) v_i(D)^\top \wh{v}^t}^2 \notag \\&
= \frac{1}{d} \normtwo{\wt{M}^{t'} \wh{v}^t}^2 
\le \frac{\normtwo{\wh{v}^t}^2}{d} \max_{v\in\bbS^{d-1}} \normtwo{\wt{M}^{t'} v}^2 \notag \\
&= \frac{\normtwo{\wh{v}^t}^2}{d} \sigma_1(\wt{M}^{t'})^2 \notag \\ &
= \frac{\normtwo{\wh{v}^t}^2}{d} \lambda_1(\wt{M}^{t'})^2 = \frac{\normtwo{\wh{v}^t}^2}{d} \lambda_1(\wt{M})^{2t'} = \frac{\normtwo{\wh{v}^t}^2}{d} \lambda_2(\wh{M})^{2t'} , \notag 
\end{align}
where the passages in the second line follow from  the positive definiteness of $ \wt{M} $. 

We have proved in \Cref{sec:bulk} (see \Cref{lem:bulk,lem:bulk_nonpositive}) that  almost surely
\begin{align}
\lim_{d\to\infty} \lambda_2(D) &= \lambda_2 \coloneqq a^\circ \gamma^\circ . \notag 
\end{align}
Recalling from \Cref{eqn:def-phi-psi,eqn:def-zeta} the definitions of $ \psi, \zeta $, we can alternatively write $ \lambda_2 = \psi(a^\circ) = \zeta(a^\circ) $ as in \Cref{eqn:def_lambda1_lambda2}. 
Also recall from \Cref{eqn:pow_itr} that $ \lambda_1 = a^* \gamma^* = \psi(a^*) $. 
Under the condition $ a^* > a^\circ $, we further have $ \lambda_1 = \zeta(a^*) $ as in \Cref{eqn:def_lambda1_lambda2}. 
Thus, by the monotonicity of $\psi$ (see \Cref{lem:a0=sup}), we obtain the strict inequality $ \lambda_1 > \lambda_2 $ in the second part of \Cref{eqn:overlap_vhat_v1}.
 
In words, the limiting value of $ \lambda_2(D) $ is strictly less than $ \lambda_1 $. 
In view of \Cref{eqn:def-Mhat-ehat}, this gives that $\lim_{d\to\infty} \lambda_2(\wh{M})<1$, which implies
\begin{align}
\lim_{t'\to\infty} \lim_{t\to\infty} \plimsup_{d\to\infty} \frac{1}{d} \normtwo{\wh{M}^{t'} \Pi^\perp \wh{v}^t}^2
&\le \lim_{t'\to\infty} \lim_{t\to\infty} \plimsup_{d\to\infty} \frac{\normtwo{\wh{v}^t}^2}{d} \lambda_2(\wh{M})^{2t'} \notag \\
&\le \lim_{t'\to\infty} \paren{ \lim_{t\to\infty} \plim_{d\to\infty} \frac{\normtwo{\wh{v}^t}^2}{d} } \paren{ \lim_{d\to\infty} \lambda_2(\wh{M})^{2t'} }
= 0 . \label{eqn:RHS_second} 
\end{align}
The last equality holds since the limit in the first parentheses is finite (see \Cref{eqn:LHS_result}). 

Combining \Cref{eqn:RHS_first,eqn:RHS_second,eqn:RHS_third}, we obtain that the quantity in \Cref{eqn:RHS_todo} equals:
\begin{align}
    \lim_{t'\to\infty} \lim_{t\to\infty} \plim_{d\to\infty} \frac{1}{d} \normtwo{\wh{M}^{t'} \wh{v}^t}^2
    &= \lim_{t'\to\infty} \lim_{t\to\infty} \plim_{d\to\infty} \lambda_1(\wh{M})^{2t'} \frac{\inprod{v_1(D)}{\wh{v}^t}^2}{d} \notag \\
    &= \lim_{t'\to\infty} \lim_{t\to\infty} \paren{ \plim_{d\to\infty} \lambda_1(\wh{M})^{2t'} } \paren{ \plim_{d\to\infty} \frac{\inprod{v_1(D)}{\wh{v}^t}^2}{d} } \notag \\
    &= \paren{ \lim_{t'\to\infty} \plim_{d\to\infty} \lambda_1(\wh{M})^{2t'} } \paren{ \lim_{t\to\infty} \plim_{d\to\infty} \frac{\inprod{v_1(D)}{\wh{v}^t}^2}{d} } . \label{eqn:RHS_result} 
\end{align}
Now, putting \Cref{eqn:LHS_result,eqn:RHS_result} together, we arrive at the following relation:
\begin{align}
\nu^2 &= \paren{ \lim_{t'\to\infty} \plim_{d\to\infty} \lambda_1(\wh{M})^{2t'} } \paren{ \lim_{t\to\infty} \plim_{d\to\infty} \frac{\inprod{v_1(D)}{\wh{v}^t}^2}{d} } . \notag 
\end{align}
By \Cref{eqn:LHS_result}, this is equivalent to
\begin{align}
    1 &= \paren{ \lim_{t'\to\infty} \plim_{d\to\infty} \lambda_1(\wh{M})^{2t'} } \paren{ \lim_{t\to\infty} \plim_{d\to\infty} \frac{\inprod{v_1(D)}{\wh{v}^t}^2}{\normtwo{\wh{v}^t}^2} } . \label{eqn:LHS_vs_RHS} 
\end{align}
This allows us to conclude:
\begin{align}
\plim_{d\to\infty} \lambda_1(\wh{M}) &= 1 , \quad 
\lim_{t\to\infty} \plim_{d\to\infty} \frac{\inprod{v_1(D)}{\wh{v}^t}^2}{\normtwo{\wh{v}^t}^2} = 1 .  \label{eqn:eigval_eigvec_almost} 
\end{align}
Indeed, otherwise if the limit of $ \lambda_1(\wh{M})^2 $ is different from $1$, the right-hand side of \Cref{eqn:LHS_vs_RHS} will either be $ 0 $ (if $ \plim\limits_{d\to\infty} \lambda_1(\wh{M})^2 \in[0, 1) $) or $ \infty $ (if $ \plim\limits_{d\to\infty} \lambda_1(\wh{M})^2 \in (1, \infty) $) once the limit with respect to $ t'\to\infty $ is taken. 
However, this contradicts the left-hand side of \Cref{eqn:LHS_vs_RHS}. 
Since $ \wh{M} $ is positive definite, $ \lambda_1(\wh{M}) $ must converge to $1$ (instead of $-1$). 
Finally, note that by \Cref{eqn:def-Mhat-ehat}, the first identity in \Cref{eqn:eigval_eigvec_almost} gives that $\plim_{d\to\infty} \lambda_1(D) = \lambda_1$
and the second equation says that $ \wh{v}^t $ is asymptotically aligned with $ v_1(D) $. 
This concludes the proof of \Cref{eqn:overlap_vhat_v1}.  
\end{proof}

\paragraph{Proof of \Cref{eqn:main_thm_overlap}.}
Since $ \wh{v}^t $ is asymptotically aligned with $ v_1(D) $ by \Cref{eqn:overlap_vhat_v1}, the overlap between $ v_1(D) $ and $ \beta^* $ is the same as that between $ \wh{v}^t $ and $ \beta^* $ in the large $t$ limit. 
Specifically, 
\begin{align}
    \frac{\inprod{v_1(D)}{\beta^*}^2}{\normtwo{\beta^*}^2}
    &= \inprod{\frac{\wh{v}^t}{\normtwo{\wh{v}^t}}}{\frac{\beta^*}{\sqrt{d}}}^2 
    + \inprod{{v_1(D)} - \frac{\wh{v}^t}{\normtwo{\wh{v}^t}}}{\frac{\beta^*}{\sqrt{d}}}^2 \notag \\
    &\phantom{=}+ 2 \inprod{\frac{\wh{v}^t}{\normtwo{\wh{v}^t}}}{\frac{\beta^*}{\sqrt{d}}} \inprod{v_1(D)- \frac{\wh{v}^t}{\normtwo{\wh{v}^t}} }{\frac{\beta^*}{\sqrt{d}}} . \label{eqn:overlap_change} 
\end{align}
Note that \Cref{eqn:overlap_vhat_v1} implies
\begin{align}
    \lim_{t\to\infty} \plim_{d\to\infty} \normtwo{\frac{\wh{v}^t}{\normtwo{\wh{v}^t}} - v_1(D)}^2 &= 0 . \notag 
\end{align}
Therefore, we have
\begin{align}
    0 &\le \lim_{t\to\infty} \plim_{d\to\infty} \inprod{v_1(D) - \frac{\wh{v}^t}{\normtwo{\wh{v}^t}}}{\frac{\beta^*}{\sqrt{d}}}^2 
    \le \lim_{t\to\infty} \plim_{d\to\infty}  \normtwo{v_1(D) - \frac{\wh{v}^t}{\normtwo{\wh{v}^t}}}^2
    = 0 , \notag 
\end{align}
and 
\begin{align}
    0 &\le \lim_{t\to\infty} \plim_{d\to\infty} \abs{ \inprod{\frac{\wh{v}^t}{\normtwo{\wh{v}^t}}}{\frac{\beta^*}{\sqrt{d}}} \inprod{v_1(D) - \frac{\wh{v}^t}{\normtwo{\wh{v}^t}} }{\frac{\beta^*}{\sqrt{d}}} } \notag\\
  &  \le \lim_{t\to\infty} \plim_{d\to\infty} \normtwo{v_1(D) - \frac{\wh{v}^t}{\normtwo{\wh{v}^t}}}
    = 0 . \notag 
\end{align}
Then, taking the limit with respect to $ d $ and $t$ on both sides of \Cref{eqn:overlap_change}, we obtain
\begin{align}
    \plim_{d\to\infty} \frac{\inprod{v_1(D)}{\beta^*}^2}{\normtwo{\beta^*}^2}
    &= \lim_{t\to\infty} \plim_{d\to\infty} \frac{\inprod{\wh{v}^t}{\beta^*}^2}{\normtwo{\wh{v}^t}^2 \cdot d}
    = \frac{\lim\limits_{t\to\infty} \plim\limits_{d\to\infty} \frac{1}{d^2} \inprod{\wh{v}^t}{\beta^*}^2}{\lim\limits_{t\to\infty} \plim\limits_{d\to\infty} \frac{1}{d} \normtwo{\wh{v}^t}^2} , \notag 
\end{align}
the right-hand side of which we compute below. 

Note that the denominator has already been computed in \Cref{eqn:LHS_result} and equals $ \nu^2 $. 
The numerator can be computed in a similar way using state evolution. 
Recalling from \Cref{eqn:def_vhat,eqn:def-B-tilde} that $ \wh{v}^t = \Sigma^{-1/2} B v^t $, we have 
\begin{align}
    \lim_{t\to\infty} \plim_{d\to\infty} \frac{\inprod{\wh{v}^t}{\beta^*}^2}{d^2}
    &= \lim\limits_{t\to\infty} \lim\limits_{d\to\infty} \frac{1}{d^2} \expt{{\mathfrak{B}^*}^\top \Sigma^{-1/2} B V_t}^2 \notag \\
    &= \lim\limits_{t\to\infty} \chi_t^2 \lim\limits_{d\to\infty} \frac{1}{d^2} \expt{{\mathfrak{B}^*}^\top \Sigma^{-1/2} B \wt{\mathfrak{B}}^*}^2 \notag \\
    &= \paren{ \lim\limits_{t\to\infty} \chi_t^2 } \paren{ \lim\limits_{d\to\infty} \frac{1}{d^2} \expt{{\mathfrak{B}^*}^\top \Sigma^{-1/2} (\gamma^* I_d - c \Sigma)^{-1} \Sigma \Sigma^{1/2} \mathfrak{B}^*}^2 } \notag \\
    &= \chi^2 \expt{\frac{\ol{\Sigma}}{\gamma^* - c \ol{\Sigma}}}^2 . \notag 
\end{align}
Finally, recalling the expressions of $ \chi, \sigma_V $ in \Cref{eqn:def_chi_sigmaV}, we obtain
\begin{align}
    \plim_{d\to\infty} \frac{\inprod{v_1(D)}{\beta^*}^2}{\normtwo{\beta^*}^2}
    &= \frac{\chi^2 \expt{\frac{\ol{\Sigma}}{\gamma^* - c \ol{\Sigma}}}^2}{\nu^2} \notag \\
    &= \frac{\chi^2 \expt{\frac{\ol{\Sigma}}{\gamma^* - c \ol{\Sigma}}}^2}{ \expt{\frac{\ol{\Sigma}^2}{(\gamma^* - c \ol{\Sigma})^2}} \chi^2 + \expt{\frac{\ol{\Sigma}}{(\gamma^* - c \ol{\Sigma})^2}} \sigma_V^2 } \notag \\
    &= \frac{(1 - w_2) \expt{\frac{\ol{\Sigma}}{\gamma^* - c \ol{\Sigma}}}^2}{(1 - w_2) \expt{\frac{\ol{\Sigma}^2}{(\gamma^* - c \ol{\Sigma})^2}} + w_1 \expt{\frac{\ol{\Sigma}}{(\gamma^* - c \ol{\Sigma})^2}}} = \eta^2 , \notag 
\end{align}
as defined in \Cref{eqn:def-eta}.

\section{Discussion}
\label{sec:discuss}

\paragraph{Information-theoretic limits.} 

In some settings (e.g., phase retrieval), spectral estimators 
saturate information-theoretic limits when the design matrix is either i.i.d.\ Gaussian \cite{mondelli-montanari-2018-fundamental} or obtained from a uniformly random orthogonal matrix \cite{dudeja-2020-information}. That is, below the optimal spectral threshold, \emph{no estimator} can achieve weak recovery, i.e., strictly positive asymptotic overlap 
with $\beta^*$.
Thus, it is natural to ask whether the spectral threshold in \Cref{eqn:opt_thr_fp} is information-theoretically optimal for weak recovery in problems such as phase retrieval with correlated design.
Positive evidence in this regard comes from the comparison with \cite{maillard2020phase} which heuristically derives the information-theoretic weak recovery threshold for general right rotationally invariant designs. As mentioned in \Cref{subsec:optrot}, by taking a Gaussian prior on $\beta^*$, the model in \Cref{eqn:model-main} is equivalent to one in which $X$ is right rotationally invariant, and the threshold derived in \cite{maillard2020phase} in fact coincides with the expression in \Cref{eqn:opt_thr_fp} (see \Cref{rk:equiv_threshold}). 
An interesting future direction  would be to establish whether (and under what conditions) spectral estimators achieve the information-theoretic weak recovery threshold, or conversely to provide evidence of the existence of a statistical-to-computational gap.


\paragraph{Optimal covariance design.} Since our results characterize the performance of spectral estimators for a Gaussian design with any covariance $\Sigma$,  a natural question is to characterize the 
 $ \Sigma $ that induces the maximal 
overlap. A similar problem is considered in \cite{Ma_spectrum} which studies the impact of the spectrum of a bi-rotationally invariant design matrix on the performance of a family of algorithms known as expectation propagation. 
In contrast, we consider spectral estimators, and our general Gaussian design is only left rotationally invariant. 
In our context, given the characterization of the limiting overlap $ \eta=\eta(\delta, \ol{\Sigma}, \cT)$ in \Cref{eqn:def-eta} and the expression for the optimal preprocessing $\cT^*$ in \Cref{eqn:opt_T_main}, the problem can be formulated as maximizing $ \eta(\ol{\Sigma}, \cT^*, \delta) $ over $ \ol{\Sigma} $, 
for any fixed $ \delta $. 
Remarkably, \Cref{fig:toeplitz_vs_circulant_vs_identity} in \Cref{sec:experiments} shows that picking $\Sigma=I_d$ may not be optimal for the phase retrieval problem. This is in contrast with 
\cite{Ma_spectrum}, where it is proved that  ``spikier'' spectra are better for phase retrieval. 

\paragraph{{Unknown link function.}}

The optimal preprocessing function $ \cT^* $ in \mbox{\Cref{eqn:opt_T_main}} depends on the link function $q$. 
We now discuss the scenario where $q$ is not exactly known. 

In the special case where the link function is parametrized by $\theta$ (of fixed dimension) that can be obtained from the moments of the random variable $\ol{Y}$, we can consistently estimate $\theta$ with $o(n)$ samples using  the empirical moments of the observation vector $y$. This is, for example, the case when the observations have additive Gaussian noise of unknown variance. One can then apply our spectral estimator using the remaining $ n - o(n) $ samples with $ \cT^* $ constructed from the consistent estimate of $q$ above. 
By a simple matrix perturbation argument, the same recovery guarantees of the paper continue to hold under the same asymptotic aspect ratio $ \delta $.

If $q$ belongs to a nonparametric function class or if the parameters $\theta$ cannot be estimated from moments of $\ol{Y}$, then one can still construct $ \wt{\cT} $ (without knowing $q$) such that the spectral estimator achieves positive asymptotic overlap with $ \beta^* $  when $ \delta $ is sufficiently large, see \cite[Corollary 4.4]{Damian_etal}. We note that the analysis of  \cite{Damian_etal} considers identity covariance, but as the focus is not on obtaining a tight result in terms of $\delta$, we expect the same to hold if the covariance is well-conditioned. 
Importantly, this comes at the price that $ \wt{\cT} $ no longer achieves the optimal spectral threshold as $ \cT^* $ does in our \Cref{prop:opt_thr}. 

Generalized linear models with unknown link function, also known as single-index models, have been studied in the high-dimensional regime under various assumptions on the link \cite{alquier2013sparse, radchenko2015high, ganti2015learningsingleindexmodels,eftekhari2021inference,Pananjady21, sawaya2024high}.  Sawaya et al.\ \cite{sawaya2024high} recently studied single-index models under assumptions similar to our paper (Gaussian covariates, and the proportional high-dimensional regime $n/d \to \delta$). Their paper suggests the following three-step procedure to estimate both the signal $\beta^*$ and the non-parametric link $q$: \emph{(i)} obtain  a pilot estimate of $\beta^*$, e.g.\ by using ridge regression, \emph{(ii)} use the pilot estimate to obtain an estimate $\hat{q}$ of the link function, and \emph{(iii)} use $\hat{q}$ to obtain an improved final estimate $\hat{\beta}$. However, the  theoretical guarantees in  \cite{sawaya2024high}  rely on sample splitting which is suboptimal and reduces statistical efficiency. (In sample splitting, the $n$ samples are split into two sets, with the first  set  used for steps \emph{(i)}-\emph{(ii)} and the second set used for obtaining the final estimate in step \emph{(iii)}.) If the link-function is parametrized by a low-dimensional $\theta$ of fixed dimension $k$, e.g., the class of cubic B-spline functions \cite{radchenko2015high}, the number of samples required for estimating the link (e.g., via gradient descent) could be relatively small, but we expect this would still be a constant fraction of $n$ since a reasonably accurate pilot estimate is required. 
In summary, developing sample-efficient spectral estimators for GLMs with unknown link function is an important open question, which we leave to future work.

\paragraph{{Discovering spikes in random matrices via AMP}} 

Our proof strategy 
offers a new, general methodology for analyzing large spiked random matrices. 
We expect this strategy to be useful in a variety of statistical inference problems beyond GLMs with correlated Gaussian designs, including rotationally invariant designs \cite{maillard2020construction}, mixtures of GLMs \cite{mixed-zmv-arxiv}, principal component analysis with inhomogeneous noise \cite{pak2023optimal}, and the universality of spiked random matrices \cite{Dudeja_Univ,WangZhongFan_Universality}. 
For many models, the ``null'' setting in which no information 
is present can be understood using tools from random matrix theory. 
When statistically informative components emerge as 
spectral outliers, our proof recipe can be carried out -- as long as an AMP iteration can be designed to simulate the desired 
power iteration. 
Suitably combining the analysis for AMP with the random matrix theory arguments for the bulk then allows one to determine the exact outlier locations and estimation accuracy. 

\newpage
\appendix 

\paragraph{Organization of the appendix.} \Cref{sec:pf-AMP-all} contains the proofs of a number of intermediate results useful to show \Cref{thm:main}. \Cref{sec:opt_thr} contains the proof of \Cref{prop:opt_thr}. \Cref{sec:remove_asmp_sigma_tech} shows how to remove \Cref{asmp:sigma_technical,asmp:preprocessor_technical}. \Cref{sec:properties} states and proves a few useful properties of auxiliary functions and parameters. \Cref{sec:bulk_pf}
contains the proof of \Cref{lem:bulk_nonpositive}. \Cref{sec:spec_known} establishes the performance of the whitened spectral estimator. \Cref{sec:aux_lem} presents some useful auxiliary results.

\section{Details of the proof of Theorem \ref{thm:main}}
\label{sec:pf-AMP-all}


\subsection{Proof of \Cref{prop:SE}}
\label{sec:pf_SE}

We start by defining the state evolution random vectors $(U_t, V_{t+1})_{t\ge 0}$ in a different, but equivalent form.  
Let 
 $ U_0\in\bbR^n $ be a Gaussian random vector whose joint distribution with $G$ is given by 
$
\matrix{G \\ U_0} \sim \cN(0_{2n}, \Omega_0 \ot I_n)$,
where $ \Omega_0\in\bbR^{2\times2} $ is defined as 
\begin{align}
\Omega_0 &=  \matrix{\frac{1}{\delta}\expt{\ol{\Sigma}} & \displaystyle \lim_{n\to\infty} \frac{1}{n} \expt{\inprod{\wt{\mathfrak{B}}^*}{f_0(\wt{\mathfrak{B}}^*)}} \\ 
\displaystyle \lim_{n\to\infty} \frac{1}{n} \expt{\inprod{\wt{\mathfrak{B}}^*}{f_0(\wt{\mathfrak{B}}^*)}} & \displaystyle \frac{1}{\delta} \paren{ \plim_{d\to\infty} \frac{1}{\sqrt{d}} \normtwo{\wt{v}^0} }^2 } . \label{eqn:Omega0} 
\end{align}
For each $ t\ge0 $, 
define the random vectors $ U_t\in\bbR^n $ and $ V_{t+1}\in\bbR^d $ such that
\begin{align}
\matrix{G \\ U_t} &\sim \cN\paren{0_{2n} , \Omega_t \ot I_n} , \quad 
V_{t+1} = \chi_{t+1} \wt{\mathfrak{B}}^* + \sigma_{V, t+1} W_{V, t+1} , \label{eqn:def-Ut-Vt+1_alt}
\end{align}
where $ W_{V,t+1}\sim\cN(0_d, I_d) $ is independent of $ \wt{\mathfrak{B}}^* $ and $ \Omega_t\in\bbR^{2\times2} , \chi_{t+1}\in\bbR, \sigma_{V, t+1}\in\bbR $ are defined recursively as
\begin{align}
\Omega_t 
&= \matrix{\displaystyle \frac{1}{\delta} \expt{\ol{\Sigma}} & \displaystyle \lim_{n\to\infty}\frac{1}{n} \expt{\inprod{\wt{\mathfrak{B}}^*}{f_t(V_t)}} \\
\displaystyle \lim_{n\to\infty}\frac{1}{n} \expt{\inprod{\wt{\mathfrak{B}}^*}{f_t(V_t)}} & \displaystyle \lim_{n\to\infty}\frac{1}{n} \expt{\inprod{f_t(V_t)}{f_t(V_t)}} } , \label{eqn:Omegat} \\
\chi_{t+1} &= \lim_{n\to\infty} \frac{1}{n} \expt{\div_G \wt{g}_t(U_t, G, \eps)} , \quad
\sigma_{V, t+1}^2 = \lim_{n\to\infty} \frac{1}{n} \expt{\inprod{g_t(U_t; Y)}{g_t(U_t; Y)}} . \label{eqn:chi-sigmaV} 
\end{align}
Here the function $ \wt{g}_t\colon(\bbR^n)^3\to\bbR^n $ is given by $ \wt{g}_t(U_t, G, \eps) = g_t(U_t; q(G, \eps)) $. 

We now show that the alternative representations of $ U_t $ and $ \chi_{t+1} $ in \Cref{eqn:Omega0,eqn:def-Ut-Vt+1_alt,eqn:Omegat,eqn:chi-sigmaV} are equivalent to \Cref{eqn:def-Ut-Vt+1,eqn:mu0,eqn:chi-alt,eqn:mut}. 
%
\begin{proposition}
\label{prop:Ut_chit+1_alt}
The random vectors $ (G, U_t) $ defined in \Cref{eqn:def-Ut-Vt+1_alt} can be alternatively written as 
\begin{align}
U_t &= \mu_t G + \sigma_{U, t} W_{U, t} , \label{eqn:def-Ut-alt}
\end{align}
where $ (G, W_{U,t}) \sim \cN\paren{0_n, \frac{\expt{\ol{\Sigma}}}{\delta} I_n} \ot \cN(0_n,I_n) $; for $t=0$, 
\begin{align}
\mu_0 &= \frac{\delta}{\expt{\ol{\Sigma}}} \lim_{n\to\infty} \frac{1}{n} \expt{\inprod{\wt{\mathfrak{B}}^*}{f_0(\wt{\mathfrak{B}}^*)}} , \quad \sigma_{U,0}^2 = \plim_{n\to\infty} \frac{1}{n} \inprod{\wt{v}^0}{\wt{v}^0} - \frac{\expt{\ol{\Sigma}}}{\delta} \mu_0^2  \label{eqn:mu0_alt}
\end{align}
and for $ t\ge1 $, 
\begin{align}
\mu_t &= \frac{\delta}{\expt{\ol{\Sigma}}} \lim_{n\to\infty}\frac{1}{n} \expt{\inprod{\wt{\mathfrak{B}}^*}{f_t(V_t)}} , \quad \sigma_{U, t}^2 = \lim_{n\to\infty}\frac{1}{n} \expt{\inprod{f_t(V_t)}{f_t(V_t)}} - \frac{\expt{\ol{\Sigma}}}{\delta} \mu_t^2. \label{eqn:mut_alt} 
\end{align}
Furthermore, the scalar $ \chi_{t+1} $ defined in \Cref{eqn:chi-sigmaV} can be alternatively written as
\begin{align}
    \chi_{t+1} &= \frac{\delta}{\expt{\ol{\Sigma}}} \lim_{n\to\infty} \frac{1}{n} \expt{\inprod{G}{g_t(U_t; Y)}} - \mu_t \lim_{n\to\infty} \frac{1}{n} \expt{\div_{U_t} g_t(U_t; Y)} . \label{eqn:chi}
\end{align}
\end{proposition}

\begin{proof}
The decomposition of $U_t$ in \Cref{eqn:def-Ut-alt} and the expressions of $ \mu_t,\sigma_{U,t} $ in \Cref{eqn:mu0_alt,eqn:mut_alt} can be easily obtained from \Cref{eqn:Omega0,eqn:Omegat} using the following elementary proprty of Gaussian random variables. 
If 
\begin{align}
    (G_1, G_2) \sim \cN\paren{ 0_2, \matrix{\sigma_{1,1} & \sigma_{1,2} \\ \sigma_{1,2} & \sigma_{2,2}} } , \notag 
\end{align}
then their joint law can be realized as
\begin{align}
    (G_1, G_2) \eqqlaw \paren{ G_1, \frac{\sigma_{1,2}}{\sigma_{1,1}} \,G_1 + \sqrt{\sigma_{2,2} - \frac{\sigma_{1,2}^2}{\sigma_{1,1}}} \,W } , \label{eqn:G1G2} 
\end{align}
where $ W\sim\cN(0,1) $ is independent of $G_1$. 

To show \Cref{eqn:chi}, we use the chain rule and Stein's lemma. We have:
\begin{align}
\chi_{t+1} &= \lim_{n\to\infty} \frac{1}{n} \sum_{i = 1}^n \expt{\frac{\partial}{\partial G_i} \wt{g}_t(U_t, G, \eps)_i} \notag \\
&= \lim_{n\to\infty} \frac{1}{n} \sum_{i = 1}^n \expt{\frac{\partial}{\partial G_i} g_t(U_t; q(G, \eps))_i} \notag \\
&= \lim_{n\to\infty} \frac{1}{n} \sum_{i = 1}^n \paren{ \expt{\frac{\partial}{\partial G_i} g_t(\mu_t G + \sigma_{U, t} W_{U, t}; q(G, \eps))_i} - \mu_t \expt{\frac{\partial}{\partial U_{t, i}} g_t(U_t; Y)_i} } \label{eqn:use-chain} \\
&= \lim_{n\to\infty} \frac{1}{n} \sum_{i = 1}^n \biggl( \frac{\delta}{\expt{\ol{\Sigma}}} \expt{G_i g_t(\mu_t G + \sigma_{U, t} W_{U, t}; q(G, \eps))_i} \notag\\&\hspace{15em}- \mu_t \expt{\frac{\partial}{\partial U_{t, i}} g_t(U_t; Y)_i} \biggr) \label{eqn:use-stein} \\
&= \frac{\delta}{\expt{\ol{\Sigma}}} \lim_{n\to\infty} \frac{1}{n} \expt{\inprod{G}{g_t(U_t; Y)}} - \mu_t \lim_{n\to\infty} \frac{1}{n} \expt{\div_{U_t} g_t(U_t; Y)} . \notag
\end{align}
\Cref{eqn:use-chain} follows from the chain rule of derivatives:
\begin{align}
\frac{\partial}{\partial G_i} g_t(U_t; q(G, \eps))_i &= \frac{\partial}{\partial G_i} g_t(\mu_t G + \sigma_{U, t} W_{U, t}; q(G, \eps))_i \notag \\
&= \mu_t \frac{\partial}{\partial U_{t, i}} g_t(U_t; Y)_i + \frac{\partial}{\partial G_i} g_t(U_t; q(G, \eps))_i . \notag 
\end{align}
\Cref{eqn:use-stein} is by Stein's lemma, noting that $G \sim \cN\paren{0_n, \frac{\expt{\ol{\Sigma}}}{\delta} I_n}$. 
\end{proof}

Next, we show the desired state evolution result.

\begin{proof}[Proof of \Cref{prop:SE}]
Define the rescaled version of $ \wt{X} $ as $ \wc{X} \coloneqq \sqrt{\frac{n}{n + d}} \,\wt{X} \in\bbR^{n\times d} $. 
Note that each entry of $ \wc{X} $ is i.i.d.\ according to $ \cN(0,1/(n+d)) $ and that  
 $ g =X\beta^*= \wt{X} \wt{\beta}^* $. 
Consider a pair of matrix-valued iterates $ p^t \in\bbR^{n\times2} $ and $ q^t\in\bbR^{d\times2} $ defined as
\begin{align}
    p^t &= \matrix{ \wc{u}^t & g } \in\bbR^{n\times2} , \quad 
    q^t = \matrix{ \wc{v}^t - \wc{\chi}_{t-1} \wt{\beta}^* & \, 0_d } \in\bbR^{d\times2} , \label{eqn:itr_abstract} 
\end{align}
where $ (\wc{u}^t, \wc{v}^t, \wc{\chi}_{t-1})_{t\ge0} \subset \bbR^{n + d + 1} $ will be specified later in \Cref{eqn:change_time_index}. 
For $ (i,j)\in\{1, \ldots, n\}\times\{1, 2\}$, we use $ p_j^t\in\bbR^n $ and $ p_{i,j}^t\in\bbR $ to denote the $j$-th column and the $(i,j)$-th entry of the matrix $ p^t $, respectively. 
Similar notation is used for other matrix-valued iterates. 
Consider also a pair of denoising functions $ \pi_t \colon \bbR^{d\times3} \to \bbR^{d\times2} $ and $ \rho_t \colon \bbR^{n\times3} \to \bbR^{n\times2} $ defined as
\begin{align}
\begin{aligned}
    \pi_t(q^t; \wt{\beta}^*) &= \sqrt{\frac{n + d}{n}} \matrix{ \wc{f}_t(q_1^t + \wc{\chi}_{t-1} \wt{\beta}^*) & \, \wt{\beta}^* } \in \bbR^{d\times2} , \\ 
    \rho_t(p^t; \eps) &= \matrix{ \sqrt{\frac{n + d}{n}} \,\wc{g}_t(p_1^t; q(p_2^t; \eps)) & \, 0_n } \in \bbR^{n\times2} , 
\end{aligned}    
\label{eqn:denoiser_abstract}
\end{align}
where $ (\wc{f}_t, \wc{g}_t)_{t\ge0} $ will be specified later in \Cref{eqn:change_time_index}. 
We claim that the iteration
\begin{align}
\begin{aligned}
    p^{t+1} &\hspace{-.25em}=\hspace{-.25em} \wc{X} \wt{q}^t \hspace{-.25em}-\hspace{-.25em} \wt{p}^{t-1} \ell_t^\top , \,\, 
    \wt{p}^t \hspace{-.25em}=\hspace{-.25em} \rho_t(p^t; \eps) , \,\, 
    \ell_t \hspace{-.25em}=\hspace{-.25em} \frac{1}{n+d} \sum_{i = 1}^d \matrix{
    \frac{\partial \pi_t(q^t; \wt{\beta}^*)_{i, 1}}{\partial q^t_{i,1}} & \frac{\partial \pi_t(q^t; \wt{\beta}^*)_{i, 1}}{\partial q^t_{i,2}} \\
    \frac{\partial \pi_t(q^t; \wt{\beta}^*)_{i, 2}}{\partial q^t_{i,1}} & \frac{\partial \pi_t(q^t; \wt{\beta}^*)_{i, 2}}{\partial q^t_{i,2}} 
    } , \\
    q^{t+1} &\hspace{-.25em}=\hspace{-.25em} \wc{X}^\top \wt{p}^t \hspace{-.25em}-\hspace{-.25em} \wt{q}^{t-1} m_t^\top , \,\, 
    \wt{q}^t \hspace{-.25em}=\hspace{-.25em} \pi_t(q^t; \wt{\beta}^*) , \,\, 
    m_t \hspace{-.25em}=\hspace{-.25em} \frac{1}{n+d} \sum_{i = 1}^n \matrix{
    \frac{\partial \rho_t(p^t; \eps)_{i, 1}}{\partial p^t_{i,1}} & \frac{\partial \rho_t(p^t; \eps)_{i, 1}}{\partial p^t_{i,2}} \\
    \frac{\partial \rho_t(p^t; \eps)_{i, 2}}{\partial p^t_{i,1}} & \frac{\partial \rho_t(p^t; \eps)_{i, 2}}{\partial p^t_{i,2}} 
    } , 
\end{aligned}
\label{eqn:GraphAMP}
\end{align}
initialized with $ \pi_{-1} = 0, \rho_{-1} = 0 $ and $ p^0 = \matrix{ \wc{u}^0 & g }, q^0 = \matrix{ \wc{v}^0 & 0_d } $ (for some $ \wc{u}^0\in\bbR^n, \wc{v}^0\in\bbR^d $ to be specified later in \Cref{eqn:change_time_index}), 
is equivalent to the following iteration:
\begin{align}
\begin{aligned}
    \wc{u}^{t+1} &= \wt{X} \wc{f}_t(\wc{v}^t) - \wc{b}_t \wc{g}_{t-1}(\wc{u}^{t-1}; y) , \quad 
    \wc{b}_t = \frac{1}{n} \sum_{i = 1}^d \frac{\partial \wc{f}_t(\wc{v}^t)_i}{\partial \wc{v}_i^t} , \\
    \wc{v}^{t+1} &= \wt{X}^\top \wc{g}_t(\wc{u}^t; y) - \wc{c}_t \wc{f}_{t-1}(\wc{v}^{t-1}) , \quad 
    \wc{c}_t = \frac{1}{n} \sum_{i = 1}^n \frac{\partial \wc{g}_t(\wc{u}^t; y)_i}{\partial \wc{u}_i^t} , 
\end{aligned}
\label{eqn:GAMP_aux}
\end{align}
initialized with $ \wc{f}_{-1} = 0, \wc{g}_{-1} = 0 $ and $ \wc{u}^0\in\bbR^n, \wc{v}^0\in\bbR^d $.

Let us verify the equivalence. 
By the design of the matrix-valued iterates in \Cref{eqn:itr_abstract} and the matrix-valued denoisers in \Cref{eqn:denoiser_abstract}, we have
\begin{align}
    \wt{p}^t &= \rho_t\paren{ \matrix{ \wc{u}^t & g }; \eps }
    = \matrix{ \sqrt{\frac{n + d}{n}} \,\wc{g}_t(\wc{u}^t; q(g; \eps)) & 0_n }
    = \matrix{ \sqrt{\frac{n + d}{n}} \,\wc{g}_t(\wc{u}^t; y) & 0_n } , \notag \\
    \wt{q}^t &= \pi_t\paren{ \matrix{ \wc{v}^t - \wc{\chi}_{t-1} \wt{\beta}^* & 0_d }; \wt{\beta}^* }
    = \sqrt{\frac{n + d}{n}} \matrix{ \wc{f}_t(\wc{v}^t) & \wt{\beta}^* } . \notag 
\end{align}
Furthermore, by chain rule of derivatives, the matrices $ \ell_t, m_t $ specialize to 
\begin{align}
    \ell_t &= \frac{1}{n+d} \sum_{i = 1}^d \matrix{
    \sqrt{\frac{n+d}{n}} \frac{\partial \wc{f}_t(\wc{v}^t)_i}{\partial \wc{v}_i^t} & 0 \\
    0 & 0
    } = \sqrt{\frac{n}{n+d}} \matrix{
    \frac{1}{n} \sum_{i = 1}^d \frac{\partial \wc{f}_t(\wc{v}^t)_i}{\partial \wc{v}_i^t} & 0 \\
    0 & 0
    } \notag\\&\hspace{20em}= \sqrt{\frac{n}{n+d}} \matrix{
    \wc{b}_t & 0 \\
    0 & 0
    } , \notag \\
    m_t &= \frac{1}{n+d} \sum_{i = 1}^n \matrix{
    \sqrt{\frac{n+d}{n}} \frac{\partial \wc{g}_t(\wc{u}^t; y)_i}{\partial \wc{u}_i^t} & \sqrt{\frac{n+d}{n}} \frac{\partial \wc{g}_t(\wc{u}^t; q(g; \eps))_i}{\partial g_i} \\
    0 & 0
    } = \sqrt{\frac{n}{n+d}} \matrix{
    \wc{c}_t & \wc{\chi}_t \\
    0 & 0
    } . \notag 
\end{align}
Using these expressions, we write the iteration in \Cref{eqn:GraphAMP} as 
\begin{align}
    &\matrix{\wc{u}^{t+1} & g}\notag\\ 
    &= \sqrt{\frac{n + d}{n}} \wc{X} \matrix{ \wc{f}_t(\wc{v}^t) & \wt{\beta}^* }
    - \matrix{ \sqrt{\frac{n + d}{n}} \,\wc{g}_{t-1}(\wc{u}^{t-1}; y) & 0_n } 
    \sqrt{\frac{n}{n+d}} \matrix{
    \wc{b}_t & 0 \\
    0 & 0
    } , \notag \\
  &  \matrix{ \wc{v}^{t+1} - \wc{\chi}_t \wt{\beta}^* & 0_d }\notag\\ &= \wc{X}^\top \matrix{ \sqrt{\frac{n + d}{n}} \,\wc{g}_t(\wc{u}^t; y) & 0_n }
    - \sqrt{\frac{n + d}{n}} \matrix{ \wc{f}_{t-1}(\wc{v}^{t-1}) & \wt{\beta}^* }
    \sqrt{\frac{n}{n+d}} \matrix{
    \wc{c}_t & 0 \\
    \wc{\chi}_t & 0
    } . \notag 
\end{align}
Expanding the above equations into vector form and using the relation between $ \wt{X} $ and $ \wc{X} $, we obtain:
\begin{align}
    \wc{u}^{t+1} &= \wt{X} \wc{f}_t(\wc{v}^t) - \wc{b}_t \wc{g}_{t-1}(\wc{u}^{t-1}; y) , \quad
    g = \wt{X} \wt{\beta}^* , \notag \\
    \wc{v}^{t+1} &= \wt{X}^\top \wc{g}_t(\wc{u}^t; y) - \wc{c}_t \wc{f}_{t-1}(\wc{v}^{t-1}) , \notag 
\end{align}
which matches \Cref{eqn:GAMP_aux} and the definition of $g$. 

The iteration in \Cref{eqn:GraphAMP} is an instance of the abstract graph-based AMP iteration proposed in \cite{GraphAMP}. To see this, consider a simple graph on two vertices $ v_1, v_2 $ with two directed edges $ \vec{e} = (v_1, v_2) $ to $ \cev{e} = (v_2, v_1) $ between them. 
The tuple $ (\wt{X}, p^t, \pi_t) $ is associated with the edge $ \vec{e} $ and the tuple $ (\wt{X}^\top, q^t, \rho_t) $ is associated with $ \cev{e} $. 
We record below the state evolution results in \cite[Section 3.3]{GraphAMP} for our special case of \Cref{eqn:GraphAMP}, and then translate them to \Cref{eqn:GAMP_aux}. 
For each $t\ge1$, define two sequences of random matrices 
\begin{align}
    (P_0, P_1, \cdots, P_t) &\sim \cN(0_{2n(t+1)}, \Theta_t \ot I_n) , \quad 
    (Q_0, Q_1, \cdots, Q_t) \sim \cN(0_{2d(t+1)}, \Xi_t \ot I_d) , 
    \label{eqn:graph_SE_gauss}
\end{align}
where $ P_r\in\bbR^{n\times2}, Q_r\in\bbR^{d\times2} $ ($0\le r\le t$), 
and the entries of the covariance matrices $ \Theta_t, \Xi_t \in \bbR^{2(t+1) \times 2(t+1)} $ are specified recursively as follows: 
for $ 0\le r, s\le t $, 
\begin{align}
    (\Theta_t)_{r+1, s+1} &= \lim_{n\to\infty} \frac{1}{n+d} \expt{ \pi_r(Q_r; \wt{\mathfrak{B}}^*)^\top \pi_s(Q_s\wt{\mathfrak{B}}^*) } \in\bbR^{2\times2} , \notag \\
    (\Xi_t)_{r+1, s+1} &= \lim_{n\to\infty} \frac{1}{n+d} \expt{ \rho_r(P_r; \eps)^\top \rho_s(P_s; \eps) } \in\bbR^{2\times2} . \notag 
\end{align}
The notation $ (P_0, P_1, \cdots, P_t) \in (\bbR^{n\times2})^{t+1} $ should be interpreted as a $ 2n(t+1) $-dimensional vector given by
\begin{align}
    \matrix{
    (P_0)_1 \\ (P_0)_2 \\ 
    \vdots \\ 
    (P_t)_1 \\ (P_t)_2 
    } \notag 
\end{align}
where $ (P_r)_{j} $ ($0\le r\le t, j\in\{1,2\}$) denotes the $j$-th column of $ P_r\in\bbR^{n\times2} $. 
The notation $ (Q_0, Q_1, \cdots, Q_t) \in (\bbR^{d\times2})^{t+1} $ should be interpreted in a similar way. 
Accordingly, $ \Theta_t, \Xi_t \in \bbR^{2(t+1) \times 2(t+1)} $ are block matrices whose $(r+1,s+1)$-st ($0\le r,s\le t$) block has size $2\times2$. 

The state evolution result in \cite[Theorem 1 and Section 3.3]{GraphAMP} asserts that for any uniformly pseudo-Lipschitz functions $ h_1\colon\bbR^{2n(t+1)} \to\bbR, h_2\colon\bbR^{2d(t+1)}\to\bbR $  of finite order, 
\begin{align}
\begin{aligned}
    \plim_{n\to\infty} h_1(p^0, p^1, \cdots, p^t) - \expt{h_1(P_0, P_1, \cdots, P_t)} &= 0 , \\
    \plim_{d\to\infty} h_2(q^0, q^1, \cdots, q^t) - \expt{h_2(Q_0, Q_1, \cdots, Q_t)} &= 0 . 
\end{aligned}
\label{eqn:SE_p_q}
\end{align}
With the reduction in \Cref{eqn:itr_abstract,eqn:denoiser_abstract}, the state evolution iterates become
\begin{align}
    P_t &= \matrix{ \wc{U}_t & G } , \quad 
    Q_t = \matrix{ \wc{V}_t - \wc{\chi}_{t-1} \wt{\mathfrak{B}}^* & 0_d } , \notag 
\end{align}
whose covariance structure specializes to 
\begin{align}
    (\Theta_t)_{r+1, s+1} &= \lim_{n\to\infty} \frac{1}{n+d} \expt{
    \frac{n+d}{n}
    \matrix{ \wc{f}_r(\wc{V}_r) & \wt{\beta}^* }^\top
    \matrix{ \wc{f}_s(\wc{V}_s) & \wt{\beta}^* }
    } \notag \\
    &= \matrix{
    \lim\limits_{n\to\infty} \frac{1}{n} \expt{ \wc{f}_r(\wc{V}_r)^\top \wc{f}_s(\wc{V}_s) } & 
    \lim\limits_{n\to\infty} \frac{1}{n} \expt{ \wc{f}_r(\wc{V}_r)^\top \wt{\mathfrak{B}}^* } \\
    \lim\limits_{n\to\infty} \frac{1}{n} \expt{ \wc{f}_s(\wc{V}_s)^\top \wt{\mathfrak{B}}^* } & 
    \lim\limits_{n\to\infty} \frac{1}{n} \expt{ (\wt{\mathfrak{B}}^*)^\top \wt{\mathfrak{B}}^* } 
    } , \label{eqn:Thetat} \\
    (\Xi_t)_{r+1, s+1} &= \lim_{n\to\infty} \frac{1}{n+d} \expt{
    \matrix{ \sqrt{\frac{n+d}{n}} \wc{g}_r(\wc{U}_r; Y) & 0_n }^\top 
    \matrix{ \sqrt{\frac{n+d}{n}} \wc{g}_s(\wc{U}_s; Y) & 0_n }
    } \notag \\
    &= \matrix{
    \lim\limits_{n\to\infty} \frac{1}{n} \expt{ \wc{g}_r(\wc{U}_r; Y)^\top \wc{g}_s(\wc{U}_s; Y) } & 0 \\
    0 & 0
    } . \label{eqn:Xit} 
\end{align}
Reorganizing the elements of $ P_t, Q_t $ and $ \Theta_t, \Xi_t $, we obtain
\begin{align}
    & (G, \wc{U}_0, \cdots, \wc{U}_t) \sim \cN(0_{n(t+2)}, \wc{\Theta}_t \ot I_n) , \notag \\
    & (\wc{V}_0 - \wc{\chi}_{-1} \wt{\mathfrak{B}}^*, \cdots, \wc{V}_t - \wc{\chi}_{t-1} \wt{\mathfrak{B}}^*) \sim \cN(0_{d(t+1)}, \wc{\Xi}_t \ot I_d) , \label{eqn:check_V-chi_X} 
\end{align}
where the entries of $ \wc{\Theta}_t\in\bbR^{(t+2)\times(t+2)} $ and $ \wc{\Xi}_t \in \bbR^{(t+1)\times(t+1)} $ are obtained as follows from $ \Theta_t $ and $ \Xi_t$.  Recalling that each entry $ (\Theta_t)_{r,s}, (\Xi_t)_{r,s} $ of $ \Theta_t, \Xi_t $, respectively, is itself a $2\times2$ matrix, we use $ ((\Theta_t)_{r,s})_{i,j}, ((\Xi_t)_{r,s})_{i,j} $ to denote the $(i,j)$-th ($i,j\in\{1,2\}$) entry of $ (\Theta_t)_{r,s}, (\Xi_t)_{r,s} $, respectively: 
\begin{align}
    & (\wc{\Theta}_t)_{1,1} = ((\Theta_t)_{1,1})_{2,2} , \quad 
    (\wc{\Theta}_t)_{1, s} = ((\Theta_t)_{s-1,s-1})_{1,2} , \quad 2\le s\le t+2 , \notag \\
    & (\wc{\Theta}_t)_{r, s} = (\wc{\Theta}_t)_{s, r} = ((\Theta_t)_{r-1,s-1})_{1,1} , \quad 2\le r\le s\le t+2 , \notag \\
    & (\wc{\Xi}_t)_{r, s} = (\wc{\Xi}_t)_{s, r} = ((\Xi_t)_{r,s})_{1,1} , \quad 1\le r \le s \le t+1 . \notag 
\end{align}

We further transform $\wc{\Theta}_t$ by introducing $ \wc{\Omega}_t\in\bbR^{2\times2}, \wc{\Phi}_t\in\bbR^{(t+1)\times(t+1)} $. 
First, we have $ (G, \wc{U}_t) \sim \cN(0_2, \wc{\Omega}_t) $ where 
\begin{align}
    \wc{\Omega}_t & = \matrix{
    (\wc{\Theta}_t)_{1,1} & (\wc{\Theta}_t)_{1,t+2} \\
    (\wc{\Theta}_t)_{1,t+2} & (\wc{\Theta}_t)_{t+2,t+2}
    } \in\bbR^{2\times2} . \label{eqn:wc_Omegat} 
\end{align}

Next, applying the representation in \Cref{eqn:def-Ut-alt} to $ (G, \wc{U}_t) $, we write $ \wc{U}_t = \wc{\mu}_t G + \wc{\sigma}_{U,t} \wc{W}_{U,t} $. 
Here $ \wc{\mu}_t $ can be derived in a way similar to \Cref{prop:Ut_chit+1_alt}: 
\begin{align}
    \wc{\mu}_t &= \frac{(\wc{\Theta}_t)_{1,t+2}}{(\wc{\Theta}_t)_{1,1}}
    = \frac{((\Theta_t)_{t+1,t+1})_{1,2}}{(\wc{\Theta}_t)_{1,1}} 
    = \frac{\delta}{\expt{\ol{\Sigma}}} \lim\limits_{n\to\infty} \frac{1}{n} \expt{ \wc{f}_t(\wc{V}_t)^\top \wt{\mathfrak{B}}^* } , \label{eqn:check_mut} 
\end{align}
where the last equality is obtained by recalling \Cref{eqn:Thetat}. 
Moreover, $ (\wc{\sigma}_{U,0} \wc{W}_{U,0}, \cdots, \wc{\sigma}_{U,t} \wc{W}_{U,t}) \sim \cN(0_{n(t+1)}, \wc{\Phi}_t \ot I_n) $ are jointly Gaussian whose covariance can be derived from $ \wc{\Theta}_t $. 
For any $ 0\le r,s\le t $, 
\begin{align}
    (\wc{\Theta}_t)_{r+2,s+2}
    = \frac{1}{n} \expt{\inprod{\wc{U}_r}{\wc{U}_s}}
    &= \wc{\mu}_r \wc{\mu}_s (\wc{\Theta}_t)_{1,1}
    + \frac{1}{n} \expt{\inprod{\wc{\sigma}_{U,r} \wc{W}_{U,r}}{\wc{\sigma}_{U,s} \wc{W}_{U,s}}} , \notag 
\end{align}
from which we obtain
\begin{align}
    (\wc{\Phi}_t)_{r+1,s+1} &= \frac{1}{n} \expt{\inprod{\wc{\sigma}_{U,r} \wc{W}_{U,r}}{\wc{\sigma}_{U,s} \wc{W}_{U,s}}}
    = (\wc{\Theta}_t)_{r+2,s+2} - \wc{\mu}_r \wc{\mu}_s (\wc{\Theta}_t)_{1,1} \notag \\
    &= ((\Theta_t)_{r+1,s+1})_{1,1} - \frac{((\Theta_t)_{r+1,r+1})_{1,2} ((\Theta_t)_{s+1,s+1})_{1,2}}{((\Theta_t)_{1,1})_{2,2}} . \label{eqn:check_Phit_entry} 
\end{align}
We claim that the the above expression equals 
\begin{align}
    \lim_{n\to\infty} \frac{1}{n} \expt{ \inprod{\wc{f}_r(\wc{V}_r) - \wc{\mu}_r \wt{\mathfrak{B}}^*}{\wc{f}_s(\wc{V}_s) - \wc{\mu}_s \wt{\mathfrak{B}}^*} } . 
    \label{eqn:checkWUr_checkWUs}
\end{align}
Indeed, 
\begin{align}
    &\phantom{=}~ \lim_{n\to\infty} \frac{1}{n} \expt{ \inprod{\wc{f}_r(\wc{V}_r) - \wc{\mu}_r \wt{\mathfrak{B}}^*}{\wc{f}_s(\wc{V}_s) - \wc{\mu}_s \wt{\mathfrak{B}}^*} } \notag \\
    &= \lim_{n\to\infty} \frac{1}{n} \bigg( 
    \expt{ \inprod{\wc{f}_r(\wc{V}_r)}{\wc{f}_s(\wc{V}_s)} } 
    - \wc{\mu}_s \expt{ \inprod{\wc{f}_r(\wc{V}_r)}{\wt{\mathfrak{B}}^*} } \notag \\
    &\hspace{6em}
    - \wc{\mu}_r \expt{ \inprod{\wc{f}_s(\wc{V}_s)}{\wt{\mathfrak{B}}^*} } 
    + \wc{\mu}_r \wc{\mu}_s \expt{ \inprod{\wt{\mathfrak{B}}^*}{\wt{\mathfrak{B}}^*} }
    \bigg) \notag \\
    &= \lim_{n\to\infty} \frac{1}{n} \expt{ \inprod{\wc{f}_r(\wc{V}_r)}{\wc{f}_s(\wc{V}_s)} } \notag \\
    &\hspace{6em} - \frac{\delta}{\expt{\ol{\Sigma}}} 
    \paren{ \lim_{n\to\infty} \frac{1}{n} \expt{ \inprod{\wc{f}_r(\wc{V}_r)}{\wt{\mathfrak{B}}^*} } } 
    \paren{ \lim_{n\to\infty} \frac{1}{n} \expt{ \inprod{\wc{f}_s(\wc{V}_s)}{\wt{\mathfrak{B}}^*} } } , \notag 
\end{align}
which agrees with \Cref{eqn:check_Phit_entry}. 
In the last equality, we use \Cref{eqn:check_mut}. 

Finally, for $ t\ge0 $, let $ \wc{\sigma}_{V,t} \wc{W}_{V,t} \coloneqq \wc{V}_t - \wc{\chi}_{t-1} \wt{\mathfrak{B}}^* $ where $ \wc{W}_{V,t}\sim\cN(0,1) $ is independent of $ \wt{\mathfrak{B}}^* $. 
From \Cref{eqn:check_V-chi_X}, we have $ (\wc{\sigma}_{V,0} \wc{W}_{V,0}, \cdots, \wc{\sigma}_{V,t} \wc{W}_{V,t}) \sim \cN(0_{d(t+1)}, \wc{\Xi}_t \ot I_d) $ where $ \wc{\Xi}_t $ has entries
\begin{align}
    (\wc{\Xi}_t)_{r+1,s+1} &= ((\Xi_t)_{r+1,s+1})_{1,1}
    = \lim_{n\to\infty} \frac{1}{n} \expt{\inprod{\wc{g}_r(\wc{U}_r; Y)}{\wc{g}_s(\wc{U}_s; Y)}} . \label{eqn:check_Xit} 
\end{align}

With $ (\wc{\mu}_t, \wc{\sigma}_{U,t}) $ (or equivalently $ \wc{\Omega}_t $), $ \wc{\Phi}_t, \wc{\chi}_{t-1}, \wc{\Xi}_t $ at hand, \Cref{eqn:SE_p_q} naturally translates to the following state evolution result.  
For any uniformly pseudo-Lipschitz functions $ h_1\colon\bbR^{n(t+2)}\to\bbR, h_2\colon\bbR^{d(t+2)}\to\bbR $ of finite order, 
\begin{align}
\begin{aligned}
    \plim_{n\to\infty} h_1(g, \wc{u}_0, \cdots, \wc{u}_t) - \expt{h_1(G, \wc{U}_0, \cdots, \wc{U}_t)} &= 0 , \\
    \plim_{d\to\infty} h_2(\wt{\beta}^*, \wc{v}_0, \cdots, \wc{v}_t) - \expt{h_2(\wt{\mathfrak{B}}^*, \wc{V}_0, \cdots, \wc{V}_t)} &= 0 .  
\end{aligned}
\label{eqn:SE_change}
\end{align}

Note that the AMP iteration in \Cref{eqn:GAMP_aux} is almost the same as that in \Cref{eqn:GAMP_nonsep} albeit with a difference in time indices. 
Indeed, the following relabeling maps \Cref{eqn:GAMP_aux} to \Cref{eqn:GAMP_nonsep} precisely: 
\begin{equation}
\begin{gathered}
    \wc{u}^{2t - 1} = u^{t-1} , \quad 
    \wc{v}^{2t} = v^t , \quad 
    t\ge 1 , \\
    \wc{u}^0 = 0_n , \quad 
    \wc{f}_0(\wc{v}^0) = \wt{v}^0 , \\
    \wc{g}_{2t-1} = g_t , \quad 
    \wc{g}_{2t} = 0 , \quad 
    \wc{f}_{2t-1} = 0 , \quad 
    \wc{f}_{2t} = f_t , \quad 
    \wc{\chi}_{2t-2} = 0 , \quad 
    \wc{\chi}_{2t-1} = \chi_t , \quad 
    t\ge 1 , \\
    \wc{g}_0 = 0 , \quad 
    \wc{f}_0 = 0 , \quad 
    \wc{\chi}_{-1} = 0 . 
\end{gathered}
\label{eqn:change_time_index}
\end{equation}
The change of indices above is similar to that presented in \cite[Appendix A]{GraphAMP}.

The change of time index in \Cref{eqn:change_time_index} also maps respectively $ (\wc{\mu}_{2t-1}, \wc{\sigma}_{U,2t-1}) $ (or equivalently $ \wc{\Omega}_{2t-1} $), $ \wc{\Phi}_{2t-1}, \wc{\chi}_{2t-1}, \wc{\Xi}_{2t} $ in \Cref{eqn:wc_Omegat,eqn:check_mut,eqn:check_Phit_entry,eqn:check_Xit} to $ (\mu_t, \sigma_{U,t}) $ (or equivalently $ \Omega_t $), $ \Phi_t, \chi_t, \Psi_t $ in \Cref{eqn:Omegat,eqn:mut,eqn:Phit1,eqn:Phit2,eqn:Psit}. 
Thus, the convergence result in \Cref{eqn:SE_change} translates to \Cref{eqn:SE_GAMP_nonsep}, which completes the proof.  
\end{proof}

%
%

\subsection{Proof of \Cref{lem:FP_SE}}
\label{subsec:pflem:FP_SE}

We start by simplifying the recursion in \Cref{eqn:SE_before} using the distributional properties of various random variables/vectors in \Cref{eqn:rand-var,eqn:rand-vec,eqn:def-Ut-Vt+1}. 
First, 
\begin{align}
\mu_t &= \frac{\delta}{\expt{\ol{\Sigma}}} \lim_{n\to\infty}\frac{1}{n} \expt{(\wt{\mathfrak{B}}^*)^\top (\gamma_t I_d - c \Sigma)^{-1} \Sigma (\chi_t \wt{\mathfrak{B}}^* + \sigma_{V, t} W_{V, t})} \label{eqn:mu1} \\
&= \chi_t \frac{\delta}{\expt{\ol{\Sigma}}} \lim_{n\to\infty}\frac{1}{n} \expt{(\wt{\mathfrak{B}}^*)^\top (\gamma_t I_d - c \Sigma)^{-1} \Sigma \wt{\mathfrak{B}}^* } \label{eqn:mu2} \\
&= \chi_t \frac{\delta}{\expt{\ol{\Sigma}}} \lim_{n\to\infty}\frac{1}{n} \expt{{\mathfrak{B}^*}^\top \Sigma^{1/2} (\gamma_t I_d - c \Sigma)^{-1} \Sigma \Sigma^{1/2} \mathfrak{B}^* } \label{eqn:mu3} \\
&= \frac{1}{\expt{\ol{\Sigma}}} \expt{\frac{\ol{\Sigma}^2}{\gamma_t - \expt{\cF_{a^*}(\ol{Y})} \ol{\Sigma}}} \chi_t . \label{eqn:mu-last}
\end{align}
\Cref{eqn:mu1} is by the definition of $ B_t $ (see \Cref{eqn:def-mtx-B}) and $ V_t $ (see \Cref{eqn:def-Ut-Vt+1}). 
\Cref{eqn:mu2} holds since $ W_{V,t} $ is independent of $ \wt{\mathfrak{B}}^* $. 
\Cref{eqn:mu3} is by the definition of $ \wt{\mathfrak{B}}^* $ (see \Cref{eqn:rand-vec}). 
In \Cref{eqn:mu-last} we use \Cref{prop:trace}, the distribution of $ \mathfrak{B}^* $ (see \Cref{eqn:rand-vec}) and the assumption $ d/n\to1/\delta $.  

Second, 
\begin{align}
\sigma_{U, t}^2 &= \lim_{n\to\infty}\frac{1}{n} \expt{(\chi_t \wt{\mathfrak{B}}^* + \sigma_{V, t} W_{V, t})^\top \Sigma (\gamma_t I_d - c \Sigma)^{-2} \Sigma (\chi_t \wt{\mathfrak{B}}^* + \sigma_{V, t} W_{V, t})} - \frac{\expt{\ol{\Sigma}}}{\delta} \mu_t^2 \notag \\
&= \chi_t^2 \lim_{n\to\infty}\frac{1}{n} \expt{( \mathfrak{B}^* )^\top \Sigma^{1/2} \Sigma (\gamma_t I_d - c \Sigma)^{-2} \Sigma \Sigma^{1/2} \mathfrak{B}^* } \notag \\
&\phantom{=}~+ \sigma_{V,t}^2 \lim_{n\to\infty}\frac{1}{n} \expt{W_{V, t}^\top \Sigma (\gamma_t I_d - c \Sigma)^{-2} \Sigma W_{V, t}} - \frac{\expt{\ol{\Sigma}}}{\delta} \mu_t^2 \notag \\
&= \frac{1}{\delta} \expt{\frac{\ol{\Sigma}^3}{(\gamma_t - \expt{\cF_{a^*}(\ol{Y})} \ol{\Sigma})^2}} \chi_t^2
+ \frac{1}{\delta} \expt{\frac{\ol{\Sigma}^2}{(\gamma_t - \expt{\cF_{a^*}(\ol{Y})} \ol{\Sigma})^2}} \sigma_{V,t}^2 - \frac{1}{\delta} \expt{\ol{\Sigma}} \mu_t^2 \label{eqn:uselater2} \\
&= \frac{1}{\delta} \paren{\expt{\frac{\ol{\Sigma}^3}{(\gamma_t - \expt{\cF_{a^*}(\ol{Y})} \ol{\Sigma})^2}} - \frac{1}{\expt{\ol{\Sigma}}} \expt{\frac{\ol{\Sigma}^2}{\gamma_t - \expt{\cF_{a^*}(\ol{Y})} \ol{\Sigma}}}^2}\, \chi_t^2 \notag \\
&\phantom{=}~+ \frac{1}{\delta} \expt{\frac{\ol{\Sigma}^2}{(\gamma_t - \expt{\cF_{a^*}(\ol{Y})} \ol{\Sigma})^2}} \sigma_{V,t}^2 , \label{eqn:sigmaU-last}
\end{align}
where we use \Cref{eqn:mu-last} in \Cref{eqn:sigmaU-last}. 

Third, 
\begin{align}
\chi_{t+1} 
&= \frac{\delta}{\expt{\ol{\Sigma}}} \lim_{n\to\infty} \frac{1}{n} \expt{G^\top \diag(\cF_{a^*}(Y)) (\mu_t G + \sigma_{U, t} W_{U, t})} - \mu_t \expt{\cF_{a^*}(\ol{Y})} \label{eqn:chi1} \\
&= \frac{\delta}{\expt{\ol{\Sigma}}} \lim_{n\to\infty} \frac{1}{n} \expt{G^\top \diag(\cF_{a^*}(Y)) G} \mu_t - \mu_t \expt{\cF_{a^*}(\ol{Y})} \label{eqn:chi2} \\
&= \expt{\paren{\frac{\delta}{\expt{\ol{\Sigma}}} \ol{G}^2 - 1} \cF_{a^*}(\ol{Y})} \mu_t \label{eqn:chi3} \\
&= \frac{1}{\expt{\ol{\Sigma}}} \expt{\paren{\frac{\delta}{\expt{\ol{\Sigma}}} \ol{G}^2 - 1} \cF_{a^*}(\ol{Y})} \expt{\frac{\ol{\Sigma}^2}{\gamma_t - \expt{\cF_{a^*}(\ol{Y})} \ol{\Sigma}}} \chi_t . \label{eqn:chi-last} 
\end{align}
\Cref{eqn:chi1} is by the definition of $U_t$ (see \Cref{eqn:def-Ut-alt}). 
\Cref{eqn:chi2} holds since $ W_{U,t} $ is independent of $ G $ and hence also independent of $ Y $. 
\Cref{eqn:chi3} follows since each entry of $G$ and $\cF_{a^*}(Y)$ is i.i.d.\ and hence
\begin{align}
    \lim_{n\to\infty} \frac{1}{n} \expt{G^\top \diag(\cF_{a^*}(Y)) G}
    = \lim_{n\to\infty} \frac{1}{n} \sum_{i = 1}^n \expt{G_i^2 \cF_{a^*}(Y_i)}
    = \expt{\ol{G}^2 \cF_{a^*}(\ol{Y})} . \notag
\end{align}
\Cref{eqn:chi-last} follows from \Cref{eqn:mu-last}. 

Fourth, 
\begin{align}
\sigma_{V, t+1}^2 &= \lim_{n\to\infty} \frac{1}{n} \expt{(\mu_t G + \sigma_{U, t} W_{U, t})^\top \diag(\cF_{a^*}(Y))^2 (\mu_t G + \sigma_{U, t} W_{U, t})} \notag \\
&= \mu_t^2 \lim_{n\to\infty} \frac{1}{n} \expt{G^\top \diag(\cF_{a^*}(Y))^2 G} \notag\\
&\hspace{10em}+ \sigma_{U,t}^2 \lim_{n\to\infty} \frac{1}{n} \expt{W_{U,t}^\top \diag(\cF_{a^*}(Y))^2 W_{U,t}} \notag \\
&= \expt{\ol{G}^2 \cF_{a^*}(\ol{Y})^2} \mu_t^2 + \expt{\cF_{a^*}(\ol{Y})^2} \sigma_{U,t}^2 \label{eqn:uselater1} \\
&= \frac{1}{\expt{\ol{\Sigma}}^2} \expt{\ol{G}^2 \cF_{a^*}(\ol{Y})^2} \expt{\frac{\ol{\Sigma}^2}{\gamma_t - \expt{\cF_{a^*}(\ol{Y})} \ol{\Sigma}}}^2 \chi_t^2 \notag \\
&+\frac{\expt{\cF_{a^*}(\ol{Y})^2}}{\delta} \paren{\expt{\frac{\ol{\Sigma}^3}{(\gamma_t - \expt{\cF_{a^*}(\ol{Y})} \ol{\Sigma})^2}} - \frac{1}{\expt{\ol{\Sigma}}} \expt{\frac{\ol{\Sigma}^2}{\gamma_t - \expt{\cF_{a^*}(\ol{Y})} \ol{\Sigma}}}^2}\, \chi_t^2 \notag \\
&+ \frac{\expt{\cF_{a^*}(\ol{Y})^2}}{\delta} \expt{\frac{\ol{\Sigma}^2}{(\gamma_t - \expt{\cF_{a^*}(\ol{Y})} \ol{\Sigma})^2}} \sigma_{V,t}^2 \label{eqn:sigmaV1} \\
&= \frac{1}{\delta} \Bigg(\frac{1}{\expt{\ol{\Sigma}}} \expt{\paren{\frac{\delta}{\expt{\ol{\Sigma}}} \ol{G}^2 - 1} \cF_{a^*}(\ol{Y})^2} \expt{\frac{\ol{\Sigma}^2}{\gamma_t - \expt{\cF_{a^*}(\ol{Y})} \ol{\Sigma}}}^2 \notag \\
&\phantom{=}~+ \expt{\cF_{a^*}(\ol{Y})^2} \expt{\frac{\ol{\Sigma}^3}{(\gamma_t - \expt{\cF_{a^*}(\ol{Y})} \ol{\Sigma})^2}} \Bigg)\, \chi_t^2 \notag \\
&\phantom{=}~+ \frac{\expt{\cF_{a^*}(\ol{Y})^2}}{\delta} \expt{\frac{\ol{\Sigma}^2}{(\gamma_t - \expt{\cF_{a^*}(\ol{Y})} \ol{\Sigma})^2}} \sigma_{V,t}^2 . \label{eqn:sigmaV-last} 
\end{align}
\Cref{eqn:sigmaV1} is by \Cref{eqn:mu-last,eqn:sigmaU-last}. 

Furthermore, the 
right-hand side of \Cref{eqn:SE_before_gamma} equals:
\begin{align}
&\phantom{=}~\lim_{d\to\infty} \frac{1}{d} \expt{V_{t+1}^\top B_{t+1}^\top B_{t+1} V_{t+1}} \notag \\
&= \lim_{d\to\infty} \frac{1}{d} \mathbb E\bigl[(\chi_{t+1} \wt{\mathfrak{B}}^* + \sigma_{V, t+1} W_{V, t+1})^\top \notag\\&\hspace{5em}\Sigma (\gamma_{t+1} I_d - c \Sigma)^{-2} \Sigma (\chi_{t+1} \wt{\mathfrak{B}}^* + \sigma_{V, t+1} W_{V, t+1})\bigr] \notag \\
&= \chi_{t+1}^2 \lim_{d\to\infty} \frac{1}{d} \expt{{\mathfrak{B}^*}^\top \Sigma^{3/2} (\gamma_{t+1} I_d - c \Sigma)^{-2} \Sigma^{3/2} \mathfrak{B}^*} \notag \\
&\hspace{4em}
+ \sigma_{V, t+1}^2 \lim_{d\to\infty} \frac{1}{d} \expt{W_{V, t+1}^\top \Sigma (\gamma_{t+1} I_d - c \Sigma)^{-2} \Sigma W_{V, t+1}} \notag \\
&= \chi_{t+1}^2 \expt{\frac{\ol{\Sigma}^3}{(\gamma_{t+1} - \expt{\cF_{a^*}(\ol{Y})}\ol{\Sigma})^2}} + \sigma_{V,t+1}^2 \expt{\frac{\ol{\Sigma}^2}{(\gamma_{t+1} - \expt{\cF_{a^*}(\ol{Y})}\ol{\Sigma})^2}} . \notag
\end{align}
We therefore obtain the following more transparent expression for $ \gamma_{t+1} $ (cf.\ \Cref{eqn:SE_before_gamma}):
\begin{align}
    1 &= \chi_{t+1}^2 \expt{\frac{\ol{\Sigma}^3}{(\gamma_{t+1} - \expt{\cF_{a^*}(\ol{Y})}\ol{\Sigma})^2}} + \sigma_{V,t+1}^2 \expt{\frac{\ol{\Sigma}^2}{(\gamma_{t+1} - \expt{\cF_{a^*}(\ol{Y})}\ol{\Sigma})^2}} , \label{eqn:def-gamma}
\end{align}
where $ \chi_{t+1}, \sigma_{V,t+1} $ are computed via \Cref{eqn:chi-last,eqn:sigmaV-last}. 
Again, using a similar monotonicity argument as that following \Cref{eqn:def-gamma-fn}, we readily have that the solution to the above equation must exist in $ (s(a^*), \infty) $ and is unique (where we use \emph{(b)} and \emph{(c)} in \Cref{eqn:asmp_gamma_lim}), and therefore $ \gamma_{t+1} $ is well-defined.

Next, we solve the fixed points of the above state evolution recursion. 
Suppose the state evolution parameters $ \mu_t, \sigma_{U,t}, \chi_{t+1}, \sigma_{V,t+1}, \gamma_{t+1} $ converge to $ \mu, \sigma_U, \chi, \sigma_V, \gamma $, respectively, as $ t\to\infty $. 
Then the latter quantities satisfy the following set of equations which are obtained by removing the time indices in \Cref{eqn:mu-last,eqn:sigmaU-last,eqn:chi-last,eqn:sigmaV-last,eqn:def-gamma}: 
\begin{align}
\mu &= \frac{1}{\expt{\ol{\Sigma}}} \expt{\frac{\ol{\Sigma}^2}{\gamma - \expt{\cF_{a^*}(\ol{Y})} \ol{\Sigma}}} \chi , 
\label{eqn:fp-mu} \\
\sigma_U^2 &= \frac{1}{\delta} \paren{\expt{\frac{\ol{\Sigma}^3}{(\gamma - \expt{\cF_{a^*}(\ol{Y})} \ol{\Sigma})^2}} - \frac{1}{\expt{\ol{\Sigma}}} \expt{\frac{\ol{\Sigma}^2}{\gamma - \expt{\cF_{a^*}(\ol{Y})} \ol{\Sigma}}}^2}\, \chi^2 \notag \\
&+ \frac{1}{\delta} \expt{\frac{\ol{\Sigma}^2}{(\gamma - \expt{\cF_{a^*}(\ol{Y})} \ol{\Sigma})^2}} \sigma_V^2 , 
\label{eqn:fp-sigmaU} \\
\chi &= \frac{1}{\expt{\ol{\Sigma}}} \expt{\paren{\frac{\delta}{\expt{\ol{\Sigma}}} \ol{G}^2 - 1} \cF_{a^*}(\ol{Y})} \expt{\frac{\ol{\Sigma}^2}{\gamma - \expt{\cF_{a^*}(\ol{Y})} \ol{\Sigma}}} \chi , 
\label{eqn:fp-chi} \\
\sigma_V^2 &= \frac{1}{\delta} \Bigg(\frac{1}{\expt{\ol{\Sigma}}} \expt{\paren{\frac{\delta}{\expt{\ol{\Sigma}}} \ol{G}^2 - 1} \cF_{a^*}(\ol{Y})^2} \expt{\frac{\ol{\Sigma}^2}{\gamma - \expt{\cF_{a^*}(\ol{Y})} \ol{\Sigma}}}^2 \notag \\
&+ \expt{\cF_{a^*}(\ol{Y})^2} \expt{\frac{\ol{\Sigma}^3}{(\gamma - \expt{\cF_{a^*}(\ol{Y})} \ol{\Sigma})^2}} \Bigg)\, \chi^2 \notag\\
&+ \frac{\expt{\cF_{a^*}(\ol{Y})^2}}{\delta} \expt{\frac{\ol{\Sigma}^2}{(\gamma - \expt{\cF_{a^*}(\ol{Y})} \ol{\Sigma})^2}} \sigma_V^2 , 
\label{eqn:fp-sigmaV} \\
1 &= \expt{\frac{\ol{\Sigma}^3}{(\gamma - \expt{\cF_{a^*}(\ol{Y})}\ol{\Sigma})^2}} \chi^2 + \expt{\frac{\ol{\Sigma}^2}{(\gamma - \expt{\cF_{a^*}(\ol{Y})}\ol{\Sigma})^2}} \sigma_V^2 . 
\label{eqn:fp-gamma} 
\end{align}

We observe from \Cref{eqn:fp-chi} that a trivial fixed point of $\chi$ is $\chi = 0$. 
This implies, via \Cref{eqn:fp-mu}, that $ \mu = 0 $. 
\Cref{eqn:fp-sigmaV,eqn:fp-gamma} then become
\begin{align}
    \sigma_V^2 &= \frac{\expt{\cF_{a^*}(\ol{Y})^2}}{\delta} \expt{\frac{\ol{\Sigma}^2}{(\gamma - \expt{\cF_{a^*}(\ol{Y})} \ol{\Sigma})^2}} \sigma_V^2 , \quad 
    1 = \expt{\frac{\ol{\Sigma}^2}{(\gamma - \expt{\cF_{a^*}(\ol{Y})}\ol{\Sigma})^2}} \sigma_V^2 , \notag
\end{align}
from which $ \gamma $ and $ \sigma_V^2 $ can be solved. 
Specifically, $ \gamma $ is the unique solution in $ (s(a^*), \infty) $ to:
\begin{align}
    1 &= \frac{\expt{\cF_{a^*}(\ol{Y})^2}}{\delta} \expt{\frac{\ol{\Sigma}^2}{(\gamma - \expt{\cF_{a^*}(\ol{Y})} \ol{\Sigma})^2}} , \notag
\end{align}
and $ \sigma_V^2 $ is given by
\begin{align}
    \sigma_V^2 &= \frac{1}{\expt{\frac{\ol{\Sigma}^2}{(\gamma - \expt{\cF_{a^*}(\ol{Y})}\ol{\Sigma})^2}}} . \notag 
\end{align}
Finally, $ \sigma_U^2 $ can be solved using \Cref{eqn:fp-sigmaU}: $ \sigma_U^2 = \frac{1}{\delta} $. 

Now assume $ \chi\ne0 $. 
\Cref{eqn:fp-chi} implies 
\begin{align}
1 &= \frac{1}{\expt{\ol{\Sigma}}} \expt{\paren{\frac{\delta}{\expt{\ol{\Sigma}}} \ol{G}^2 - 1} \cF_{a^*}(\ol{Y})} \expt{\frac{\ol{\Sigma}^2}{\gamma - \expt{\cF_{a^*}(\ol{Y})} \ol{\Sigma}}} , \label{eqn:def-gamma-fp}
\end{align}
from which $ \gamma $ can be solved: $\gamma = \gamma^*$. 
Recall that $ \gamma^* $ (together with $ a^* $) is well-defined through \Cref{eqn:fp-a-gamma} and $ a^* $ is taken to be the largest solution. 

Given $ \gamma $, \Cref{eqn:fp-mu,eqn:fp-sigmaU,eqn:fp-sigmaV,eqn:fp-gamma}
form a linear system with unknowns $ \mu^2, \sigma_U^2, \chi^2, \sigma_V^2 $. 
Combining \Cref{eqn:fp-gamma,eqn:fp-sigmaV} and using the definitions of $ w_1,w_2,z_1,z_2 $ in \Cref{eqn:def_x1_main,eqn:def_x2_main,eqn:def_z1_z2_main}, we obtain
\begin{align}
\chi^2 &= \frac{1-w_2}{(1-w_2) z_1 + w_1 z_2} , \quad 
\sigma_V^2 = \frac{w_1}{(1-w_2) z_1 + w_1 z_2} . \label{eqn:FP_sol_chi_sigmaV}
\end{align}
Note that the above solution is valid since $ 1-w_2, w_1, z_1, z_2 $ are all positive, provided $ a^* > a^\circ $ (see \Cref{itm:threshold4} in \Cref{prop:equiv-threshold} and \Cref{prop:x1>0}). 
According to \Cref{eqn:fp-mu,eqn:fp-sigmaU}, this immediately implies 
\begin{align}
    \mu^2 &= \frac{1}{\expt{\ol{\Sigma}}^2} \expt{\frac{\ol{\Sigma}^2}{\gamma^* - \expt{\cF_{a^*}(\ol{Y})} \ol{\Sigma}}}^2 \frac{1-w_2}{(1-w_2) z_1 + w_1 z_2} , \label{eqn:FP_sol_mu} \\
    \sigma_U^2 &= \frac{1}{\delta} \paren{\expt{\frac{\ol{\Sigma}^3}{(\gamma^* - \expt{\cF_{a^*}(\ol{Y})} \ol{\Sigma})^2}} - \frac{1}{\expt{\ol{\Sigma}}} \expt{\frac{\ol{\Sigma}^2}{\gamma^* - \expt{\cF_{a^*}(\ol{Y})} \ol{\Sigma}}}^2} \frac{1-w_2}{(1-w_2) z_1 + w_1 z_2} \notag \\
    &\phantom{=}~+ \frac{1}{\delta} \expt{\frac{\ol{\Sigma}^2}{(\gamma^* - \expt{\cF_{a^*}(\ol{Y})} \ol{\Sigma})^2}} \frac{w_1}{(1-w_2) z_1 + w_1 z_2} \notag \\
    &= \frac{1/\delta}{(1-w_2) z_1 + w_1 z_2} \left( \expt{\frac{\ol{\Sigma}^3}{(\gamma^* - \expt{\cF_{a^*}(\ol{Y})} \ol{\Sigma})^2}} - \frac{1}{\expt{\ol{\Sigma}}} \expt{\frac{\ol{\Sigma}^2}{\gamma^* - \expt{\cF_{a^*}(\ol{Y})} \ol{\Sigma}}}^2 \right. \notag \\
    &\phantom{=}~+ \left. \frac{1}{\expt{\ol{\Sigma}}^2} \expt{\frac{\ol{\Sigma}^2}{(\gamma^* - \expt{\cF_{a^*}(\ol{Y})}\ol{\Sigma})^2}} \expt{ \ol{G}^2 \cF_{a^*}(\ol{Y})^2} \expt{\frac{\ol{\Sigma}^2}{\gamma^* - \expt{\cF_{a^*}(\ol{Y})} \ol{\Sigma}}}^2 \right) , \label{eqn:FP_sol_sigmaU} 
\end{align}
where the last equality follows from the definitions of $ w_1, w_2 $. This concludes the proof.

\subsection{Proof of \Cref{lem:SE_stay}}
\label{subsec:pflem:SE_stay}
For each $ t\ge0 $, the next value of $ (\mu_{t+1}, \sigma_{U, t+1}, \chi_{t+2}, \sigma_{V, t+2}, \gamma_{t+2}) $ only depends on the current value of $ (\mu_t, \sigma_{U, t}, \chi_{t+1}, \sigma_{V, t+1}, \gamma_{t+1}) $. Hence, to show that the state evolution parameters do not change, it suffices to check that $ (\mu_0, \sigma_{U, 0}, \chi_1, \sigma_{V, 1}, \gamma_1) $ coincides with the fixed point $ (\mu, \sigma_U, \chi, \sigma_V, \gamma^*) $. 

By the construction of the AMP initializer $ (\wt{u}^{-1}, \wt{v}^0)\in\bbR^n\times\bbR^d $, we have $ \mu_0 = \mu $ (see \Cref{eqn:SE_init}).
It is easy to verify that $ \sigma_{U, 0} $ given by \Cref{eqn:SE_init} coincides with $ \sigma_U $ derived in \Cref{eqn:FP_sol_sigmaU}. 
Indeed, 
\begin{align}
    \sigma_{U, 0}^2 &= \frac{1}{\delta} \paren{1 - \expt{\ol{\Sigma}} \mu^2} \notag \\
    &= \frac{1}{\delta} \paren{1 - \frac{1}{\expt{\ol{\Sigma}}} \expt{\frac{\ol{\Sigma}^2}{\gamma^* - \expt{\cF_{a^*}(\ol{Y})} \ol{\Sigma}}}^2 \frac{1-w_2}{(1-w_2) z_1 + w_1 z_2}} \label{eqn:sigmaU_stay1} \\
    &= \frac{1/\delta}{(1-w_2)z_1 + w_1z_2} \biggl( (1-w_2)z_1 + w_1z_2 \notag\\
    &\hspace{10em}- \frac{1}{\expt{\ol{\Sigma}}} \expt{\frac{\ol{\Sigma}^2}{\gamma^* - \expt{\cF_{a^*}(\ol{Y})} \ol{\Sigma}}}^2 (1-w_2) \biggr) \notag \\
    &= \frac{1/\delta}{(1-w_2) z_1 + w_1 z_2} \left( \expt{\frac{\ol{\Sigma}^3}{(\gamma^* - \expt{\cF_{a^*}(\ol{Y})} \ol{\Sigma})^2}}\right. \notag\\
    &\phantom{=}~- \frac{1}{\expt{\ol{\Sigma}}} \expt{\frac{\ol{\Sigma}^2}{\gamma^* - \expt{\cF_{a^*}(\ol{Y})} \ol{\Sigma}}}^2  \notag \\
    &\phantom{=}~+ \left. \frac{1}{\expt{\ol{\Sigma}}^2} \expt{\frac{\ol{\Sigma}^2}{(\gamma^* - \expt{\cF_{a^*}(\ol{Y})}\ol{\Sigma})^2}} \expt{ \ol{G}^2 \cF_{a^*}(\ol{Y})^2} \expt{\frac{\ol{\Sigma}^2}{\gamma^* - \expt{\cF_{a^*}(\ol{Y})} \ol{\Sigma}}}^2 \right) \label{eqn:sigmaU_stay2} \\
    &= \sigma_U^2 . \label{eqn:init_positive}
\end{align}
We use the expression of $ \mu $ (see \Cref{eqn:FP_sol_mu}) in \Cref{eqn:sigmaU_stay1} and the expressions of $w_1,w_2,z_1,z_2$ (see \Cref{eqn:def_x1_main,eqn:def_x2_main,eqn:def_z1_z2_main}) in \Cref{eqn:sigmaU_stay2}.  

We then verify $ \chi_1 = \chi $. 
By \Cref{eqn:chi3}, 
\begin{align}
    \chi_1 &= \expt{\paren{\frac{\delta}{\expt{\ol{\Sigma}}} \ol{G}^2 - 1} \cF_{a^*}(\ol{Y})} \mu_0 \notag \\
    &= \expt{\paren{\frac{\delta}{\expt{\ol{\Sigma}}} \ol{G}^2 - 1} \cF_{a^*}(\ol{Y})} \frac{1}{\expt{\ol{\Sigma}}} \expt{\frac{\ol{\Sigma}^2}{\gamma^* - \expt{\cF_{a^*}(\ol{Y})} \ol{\Sigma}}} \sqrt{\frac{1-w_2}{(1-w_2) z_1 + w_1 z_2}} . \notag 
\end{align}
Comparing the above expression with $\chi$ in \Cref{eqn:FP_sol_chi_sigmaV}, we see that it suffices to verify 
\begin{align}
    \expt{\paren{\frac{\delta}{\expt{\ol{\Sigma}}} \ol{G}^2 - 1} \cF_{a^*}(\ol{Y})} \frac{1}{\expt{\ol{\Sigma}}} \expt{\frac{\ol{\Sigma}^2}{\gamma^* - \expt{\cF_{a^*}(\ol{Y})} \ol{\Sigma}}} &= 1 , \notag 
\end{align}
which is true since the fixed point $ \gamma = \gamma^* $ satisfies \Cref{eqn:def-gamma-fp}.  

Next, we show $ \sigma_{V,1} = \sigma_V $. 
Using \Cref{eqn:uselater1}, we have
\begin{align}
    \sigma_{V, 1}^2 &= \expt{\ol{G}^2 \cF_{a^*}(\ol{Y})^2} \mu_0^2 + \expt{\cF_{a^*}(\ol{Y})^2} \sigma_{U,0}^2 \notag \\
    &= \expt{\ol{G}^2 \cF_{a^*}(\ol{Y})^2} \mu^2 + \frac{\expt{\cF_{a^*}(\ol{Y})^2}}{\delta} \paren{1 - \expt{\ol{\Sigma}} \mu^2} \notag \\
    &= \frac{\expt{\ol{\Sigma}}}{\delta} \expt{\paren{ \frac{\delta}{\expt{\ol{\Sigma}}} \ol{G}^2 - 1 } \cF_{a^*}(\ol{Y})^2} \mu^2 + \frac{\expt{\cF_{a^*}(\ol{Y})^2}}{\delta} \notag \\
    &= \frac{1}{\delta \expt{\ol{\Sigma}}} \expt{\paren{ \frac{\delta}{\expt{\ol{\Sigma}}} \ol{G}^2 - 1 } \cF_{a^*}(\ol{Y})^2} \expt{\frac{\ol{\Sigma}^2}{\gamma^* - \expt{\cF_{a^*}(\ol{Y})} \ol{\Sigma}}}^2 \frac{1-w_2}{(1-w_2) z_1 + w_1 z_2} \notag \\
    &\phantom{=}~+ \frac{\expt{\cF_{a^*}(\ol{Y})^2}}{\delta} \notag \\
    &= \frac{1}{(1-w_2) z_1 + w_1 z_2} \Bigg( \paren{w_1 - \frac{\expt{\cF_{a^*}(\ol{Y})^2}}{\delta} z_1} (1-w_2) \notag \\
    &\phantom{=}~+ \frac{\expt{\cF_{a^*}(\ol{Y})^2}}{\delta} ((1-w_2) z_1 + w_1 z_2) \Bigg) \label{eqn:sigmaV_stay1} \\
    &= \frac{1}{(1-w_2) z_1 + w_1 z_2} \paren{w_1 - w_1w_2 + \frac{\expt{\cF_{a^*}(\ol{Y})^2}}{\delta} w_1z_2 } \notag \\
    &= \frac{w_1}{(1-w_2) z_1 + w_1 z_2} \label{eqn:sigmaV_stay2} \\
    &= \sigma_V^2 . \notag 
\end{align}
\Cref{eqn:sigmaV_stay1} is by the definitions of $ w_1, z_1 $. 
\Cref{eqn:sigmaV_stay2} is by the definitions of $ w_2, z_2 $, in particular, $ w_2 = \frac{\expt{\cF_{a^*}(\ol{Y})^2}}{\delta} z_2 $. 

Finally, it remains to verify $ \gamma_1 = \gamma^* $. 
By \Cref{eqn:def-gamma}, $ \gamma_1 $ is the unique solution to 
\begin{align}
    1 &= \chi_1^2 \expt{\frac{\ol{\Sigma}^3}{(\gamma_1 - \expt{\cF_{a^*}(\ol{Y})}\ol{\Sigma})^2}} + \sigma_{V,1}^2 \expt{\frac{\ol{\Sigma}^2}{(\gamma_1 - \expt{\cF_{a^*}(\ol{Y})}\ol{\Sigma})^2}} \notag \\
    &= \chi^2 \expt{\frac{\ol{\Sigma}^3}{(\gamma_1 - \expt{\cF_{a^*}(\ol{Y})}\ol{\Sigma})^2}} + \sigma_V^2 \expt{\frac{\ol{\Sigma}^2}{(\gamma_1 - \expt{\cF_{a^*}(\ol{Y})}\ol{\Sigma})^2}} \notag \\
    &= \frac{1}{(1-w_2) z_1 + w_1 z_2} \biggl( (1-w_2) \expt{\frac{\ol{\Sigma}^3}{(\gamma_1 - \expt{\cF_{a^*}(\ol{Y})}\ol{\Sigma})^2}} \notag\\
    &\hspace{15em}+ w_1 \expt{\frac{\ol{\Sigma}^2}{(\gamma_1 - \expt{\cF_{a^*}(\ol{Y})}\ol{\Sigma})^2}} \biggr) . \notag 
\end{align}
Rearranging terms, we have
\begin{align}
    0 &= (1-w_2) \paren{z_1 - \expt{\frac{\ol{\Sigma}^3}{(\gamma_1 - \expt{\cF_{a^*}(\ol{Y})}\ol{\Sigma})^2}}} \\
    &\hspace{10em}+ w_1 \paren{z_2 - \expt{\frac{\ol{\Sigma}^2}{(\gamma_1 - \expt{\cF_{a^*}(\ol{Y})}\ol{\Sigma})^2}}} . \label{eqn:gamma1}
\end{align}
We argue that $ \gamma_1 $ has to equal $ \gamma^* $ for the above equation to hold. 
Note that both $ (1-w_2) $ and $ w_1 $ are strictly positive (provided $a^* > a^\circ$; see \Cref{itm:threshold4} in \Cref{prop:equiv-threshold} and \Cref{prop:x1>0}). 
If $ \gamma_1 < \gamma^* $, then by the definitions of $ z_1, z_2 $, 
\begin{align}
    z_1 &< \expt{\frac{\ol{\Sigma}^3}{(\gamma_1 - \expt{\cF_{a^*}(\ol{Y})}\ol{\Sigma})^2}} , \quad
    z_2 < \expt{\frac{\ol{\Sigma}^2}{(\gamma_1 - \expt{\cF_{a^*}(\ol{Y})}\ol{\Sigma})^2}} , \notag 
\end{align}
and hence the right-hand side of \Cref{eqn:gamma1} is strictly positive, which is a contradiction. 
A similar contradiction can be derived if $ \gamma_1 > \gamma^* $. 
Thus, $ \gamma_1 = \gamma^* $. This concludes the proof.

\subsection{Proof of \Cref{lem:bulk}}
\label{sec:pf_bulk}

\begin{lemma}
\label{lem:sandwich}
Consider the matrix $D$ in \Cref{eqn:def_D_intro}. 
Define another matrix $ \brv{D} $ as
\begin{align}
    \brv{D} &= \Sigma^{1/2} \brv{X}^\top \brv{T} \brv{X} \Sigma^{1/2} \in \bbR^{d \times d} , \notag 
\end{align}
where $ \brv{T} \in \bbR^{(n-1)\times(n-1)} $ is a diagonal matrix satisfying: 
\begin{align}
    \lambda_1(T) &\ge \lambda_1(\brv{T}) 
    \ge \lambda_2(T) \ge \lambda_2(\brv{T}) \ge \cdots 
    \ge \lambda_{n-1}(T) \ge \lambda_{n-1}(\brv{T}) 
    \ge \lambda_n(T) , \notag
\end{align}
and $ \brv{X} \in \bbR^{(n-1)\times d} $ consists of i.i.d.\ $ \cN(0,1/n) $ entries, independent of $ \brv{T} $. 
Then for every $ n,d\ge1 $, it holds almost surely that
\begin{align}
    \lambda_3(\brv{D}) \le \lambda_2(D) &\le \lambda_1(\brv{D}) . \label{eqn:sandwich_lambda123}
\end{align}
\end{lemma}

\begin{proof}
Recall $ g = \wt{X} \wt{\beta}^* $ and
\begin{align}
D &= \Sigma^{1/2} \wt{X}^\top \diag(\cT(q(\wt{X} \Sigma^{1/2} \beta^*, \eps))) \wt{X} \Sigma^{1/2} 
= \Sigma^{1/2} \wt{X}^\top \diag(\cT(q(\wt{X} \wt{\beta}^*, \eps))) \wt{X} \Sigma^{1/2}  . \notag 
\end{align}
We can decompose $\wt{X}$ into the sum of two pieces: one along the direction of $g$ and the other perpendicular to $g$. 
Furthermore, by isotropy of Gaussians (see \cite[Lemma 3.1]{MontWu_AdvEg}, \cite[Lemma 2.1]{WuZhou_TenDecomp}), the distribution of $\wt{X}$ remains unchanged if the perpendicular part is replaced with an i.i.d.\ copy. 
Specifically, 
\begin{align}
\wt{X} &\eqqlaw \Pi_g \wt{X} + \Pi_g^\perp \wh{X} , \notag 
\end{align}
where
\begin{align}
\Pi_g &\coloneqq \frac{1}{\normtwo{g}^2} gg^\top , \quad 
\Pi_g^\perp \coloneqq I_n - \Pi_g , \notag 
\end{align}
and $ \wh{X}\in\bbR^{n\times d} $ is an i.i.d.\ copy of $\wt{X}$. 
Using the variational representation of eigenvalues, we can bound the second eigenvalue of $D$ by the first eigenvalue of a related matrix in which $ T $ and $ 
\wt{X} $ are ``decoupled''. 
Indeed, 
\begin{align}
&    \lambda_2(D) = \min_{\substack{\cV\subset\bbR^d \\ \dim(\cV) = d-1}} \max_{v\in\cV\cap\bbS^{d-1}} v^\top \Sigma^{1/2} \wt{X}^\top T \wt{X} \Sigma^{1/2} v \label{eqn:courant_fischer} \\
    &\eqqlaw \min_{\substack{\cV\subset\bbR^d \\ \dim(\cV) = d-1}} \max_{v\in\cV\cap\bbS^{d-1}} v^\top \Sigma^{1/2} \paren{ \Pi_g \wt{X} + \Pi_g^\perp \wh{X} }^\top T \paren{ \Pi_g \wt{X} + \Pi_g^\perp \wh{X} } \Sigma^{1/2} v \notag \\
    &= \hspace{-1em}\min_{\substack{\cV\subset\bbR^d \\ \dim(\cV) = d-1}}\hspace{-1em} \max_{v\in\cV\cap\bbS^{d-1}} \hspace{-.5em}v^\top \paren{ \frac{\Sigma^{1/2} \wt{X}^\top g}{\normtwo{g}} \frac{g^\top}{\normtwo{g}} + \Sigma^{1/2} \wh{X}^\top \Pi_g^\perp } T \paren{ \frac{g}{\normtwo{g}} \frac{g^\top \wt{X} \Sigma^{1/2}}{\normtwo{g}} + \Pi_g^\perp \wh{X} \Sigma^{1/2} } v \notag \\
    &\le \hspace{-1em}\max_{\substack{v\in\bbS^{d-1} \\ \inprod{v}{\Sigma^{1/2} \wt{X}^\top g / \normtwo{g}} = 0}} \hspace{-1em}v^\top \paren{ \frac{\Sigma^{1/2} \wt{X}^\top g}{\normtwo{g}} \frac{g^\top}{\normtwo{g}} + \Sigma^{1/2} \wh{X}^\top \Pi_g^\perp } T \paren{ \frac{g}{\normtwo{g}} \frac{g^\top \wt{X} \Sigma^{1/2}}{\normtwo{g}} + \Pi_g^\perp \wh{X} \Sigma^{1/2} } v \label{eqn:special_subspace} \\
    &\le \max_{v\in\bbS^{d-1}} v^\top \paren{ \Sigma^{1/2} \wh{X}^\top \Pi_g^\perp } T \paren{ \Pi_g^\perp \wh{X} \Sigma^{1/2} } v \notag \\
    &= \lambda_1\paren{ \Sigma^{1/2} \wh{X}^\top \Pi_g^\perp T \Pi_g^\perp \wh{X} \Sigma^{1/2} } . \notag 
\end{align}
In \Cref{eqn:courant_fischer} and subsequent steps, the minimization is over all $ (d-1) $-dimensional subspaces $\cV\subset\bbR^d$. 
In \Cref{eqn:special_subspace}, instead of minimizing over all $(d-1)$-dimensional subspaces, we take a particular one:
\begin{align}
    \cV_0 = \brace{v\in\bbR^d : \inprod{v}{\frac{\Sigma^{1/2} \wt{X}^\top g}{\normtwo{g}}} = 0} \in\bbR^d \notag 
\end{align}
Writing the eigendecomposition of $\Pi_g^\perp$ as $ \Pi_g^\perp = Q (I_n - e_n e_n^\top) Q^\top $ for some $ Q\in\bbO(n) $ and using the left rotational invariance of $ \wh{X} $, we continue as follows:
\begin{align}
\lambda_1(\Sigma^{1/2} \wh{X}^\top \Pi_g^\perp T \Pi_g^\perp \wh{X} \Sigma^{1/2})
&= \lambda_1(\Sigma^{1/2} \wh{X}^\top Q (I_n - e_n e_n^\top) Q^\top T Q (I_n - e_n e_n^\top) Q^\top \wh{X} \Sigma^{1/2}) \notag \\
&\eqqlaw \lambda_1(\Sigma^{1/2} \wh{X}^\top (I_n - e_n e_n^\top) Q^\top T Q (I_n - e_n e_n^\top) \wh{X} \Sigma^{1/2}) \notag \\
&= \lambda_1(\Sigma^{1/2} \wh{X}^\top (I_n - e_n e_n^\top) \wt{T} (I_n - e_n e_n^\top) \wh{X} \Sigma^{1/2}) , \label{eqn:T-tilde} 
\end{align}
where in \Cref{eqn:T-tilde} we define $ \wt{T} \coloneqq Q^\top T Q $. 
Although $ \wt{T} $ is no longer diagonal, we note that it has the same eigenvalues as $T$, i.e., $ \{\cT(y_1), \cdots, \cT(y_n)\} $. 

For convenience of the proceeding calculations, let us write $ \wh{X} $ and $ \wt{T} $ in block forms:
\begin{align}
\wh{X} &= \matrix{\wh{X}_{-n} \\ x_n^\top} , \quad 
\wt{T} = \matrix{\wt{T}_{-n} & s \\ s^\top & \wt{t}_n} , \notag 
\end{align}
where $ \wh{X}_{-n}\in\bbR^{(n-1)\times d} $ consist of the first $n-1$ rows of $\wh{X}$; $ \wt{T}_{-n}\in\bbR^{(n-1)\times(n-1)} $ is the top-left $(n-1)\times(n-1)$-submatrix of $\wt{T}$ and $ \wt{t}_n\in\bbR $ is the bottom-right element of $\wt{T}$. 
Note that by the Cauchy interlacing theorem, the eigenvalues of $ \wt{T} $ (i.e., the diagonal elements of $ T $) are interlaced with those of $ \wt{T}_{-n} $, i.e., 
\begin{align}
\lambda_1(\wt{T}) &\ge \lambda_1(\wt{T}_{-n}) 
\ge \lambda_2(\wt{T}) \ge \lambda_2(\wt{T}_{-n}) \ge \cdots 
\ge \lambda_{n-1}(\wt{T}) \ge \lambda_{n-1}(\wt{T}_{-n}) 
\ge \lambda_n(\wt{T}) . \label{eqn:interlace} 
\end{align}

Now, returning to bounding $ \lambda_2(D) $:
\begin{align}
&\phantom{=}~ \lambda_1(\Sigma^{1/2} \wh{X}^\top (I_n - e_n e_n^\top) \wt{T} (I_n - e_n e_n^\top) \wh{X} \Sigma^{1/2}) \notag \\
&= \lambda_1\paren{\Sigma^{1/2} \wh{X}^\top \matrix{\wt{T}_{-n} & 0_{n-1} \\ 0_{n-1}^\top & 0} \wh{X} \Sigma^{1/2}} \notag \\
&= \lambda_1\paren{\Sigma^{1/2} \matrix{\wh{X}_{-n}^\top & x_n} \matrix{\wt{T}_{-n} & 0_{n-1} \\ 0_{n-1}^\top & 0} \matrix{\wh{X}_{-n} \\ x_n^\top} \Sigma^{1/2}} \notag \\
&= \lambda_1(\Sigma^{1/2} \wh{X}_{-n}^\top \wt{T}_{-n} \wh{X}_{-n} \Sigma^{1/2}) \notag \\
&\eqqlaw \lambda_1(\Sigma^{1/2} \wh{X}_{-n}^\top \diag(\lambda_1(\wt{T}_{-n}), \cdots, \lambda_{n-1}(\wt{T}_{-n})) \wh{X}_{-n} \Sigma^{1/2}) . \notag 
\end{align}
The last step follows from the left rotational invariance of $\wh{X}_{-n}$. 
Denoting $ \brv{X} \coloneqq \wh{X}_{-n} \in \bbR^{(n-1)\times d} $ and $ \brv{T} \coloneqq \diag(\lambda_1(\wt{T}_{-n}), \cdots, \lambda_{n-1}(\wt{T}_{-n})) \in\bbR^{(n-1)\times(n-1)} $, we obtain the upper bound in \Cref{eqn:sandwich_lambda123}. 

We then prove a lower bound on $ \lambda_2(D) $, again using the Courant--Fischer theorem. 
Recall 
\begin{align}
    \lambda_2(D) &\eqqlaw \hspace{-1em}\min_{\substack{\cV\subset\bbR^d \\ \dim(\cV) = d-1}}\hspace{-1em} \max_{v\in\cV\cap\bbS^{d-1}}\hspace{-.5em} v^\top \paren{ \frac{\Sigma^{1/2} \wt{X}^\top g}{\normtwo{g}} \frac{g^\top}{\normtwo{g}} + \Sigma^{1/2} \wh{X}^\top \Pi_g^\perp } \notag\\
    &\hspace{10em}T \paren{ \frac{g}{\normtwo{g}} \frac{g^\top \wt{X} \Sigma^{1/2}}{\normtwo{g}} + \Pi_g^\perp \wh{X} \Sigma^{1/2} } v . \notag 
\end{align}
Let $ \cV^*\subset\bbR^d $ be a minimizer. 
Since $ \dim(\cV^*) = d - 1 $, it can be written as $ \cV^* = \brace{v\in\bbR^d : \inprod{v}{v^*} = 0} $ for a vector $ v^*\in\bbS^{d-1} $. 
We proceed as follows
\begin{align}
    \lambda_2(D) &\eqqlaw \max_{\substack{v\in\bbS^{d-1} \\ \inprod{v}{v^*} = 0}} v^\top \paren{ \frac{\Sigma^{1/2} \wt{X}^\top g}{\normtwo{g}} \frac{g^\top}{\normtwo{g}} + \Sigma^{1/2} \wh{X}^\top \Pi_g^\perp } T \paren{ \frac{g}{\normtwo{g}} \frac{g^\top \wt{X} \Sigma^{1/2}}{\normtwo{g}} + \Pi_g^\perp \wh{X} \Sigma^{1/2} } v \notag \\
    &\ge \hspace{-1.5em}\max_{\substack{v\in\bbS^{d-1} \\ \inprod{v}{v^*} = 0 \\ \inprod{v}{\Sigma^{1/2} \wt{X}^\top g / \normtwo{g}} = 0}} \hspace{-1.5em}v^\top \paren{ \frac{\Sigma^{1/2} \wt{X}^\top g}{\normtwo{g}} \frac{g^\top}{\normtwo{g}} + \Sigma^{1/2} \wh{X}^\top \Pi_g^\perp } T \paren{ \frac{g}{\normtwo{g}} \frac{g^\top \wt{X} \Sigma^{1/2}}{\normtwo{g}} + \Pi_g^\perp \wh{X} \Sigma^{1/2} } v \notag \\
    &= \max_{\substack{v\in\bbS^{d-1} \\ \inprod{v}{v^*} = 0 \\ \inprod{v}{\Sigma^{1/2} \wt{X}^\top g / \normtwo{g}} = 0}} v^\top \paren{ \Sigma^{1/2} \wh{X}^\top \Pi_g^\perp } T \paren{ \Pi_g^\perp \wh{X} \Sigma^{1/2} } v \notag \\
    &= \max_{v\in\cU_0 \cap \bbS^{d-1}} v^\top \paren{ \Sigma^{1/2} \wh{X}^\top \Pi_g^\perp } T \paren{ \Pi_g^\perp \wh{X} \Sigma^{1/2} } v \label{eqn:cU_0} \\
    &\ge \min_{\substack{\cU\subset\bbR^d \\ \dim(\cU) = d-2}} \max_{v\in\cU \cap \bbS^{d-1}} v^\top \paren{ \Sigma^{1/2} \wh{X}^\top \Pi_g^\perp } T \paren{ \Pi_g^\perp \wh{X} \Sigma^{1/2} } v \notag \\
    &= \lambda_3\paren{ \Sigma^{1/2} \wh{X}^\top \Pi_g^\perp T \Pi_g^\perp \wh{X} \Sigma^{1/2} } . \notag 
\end{align}
In \Cref{eqn:cU_0}, we let 
\begin{align}
    \cU_0 \coloneqq \brace{v\in\bbR^d : \inprod{v}{v^*} = \inprod{v}{\frac{\Sigma^{1/2} \wt{X}^\top g}{\normtwo{g}}} = 0} \subset\bbR^d . \notag 
\end{align}
If $ v^* $ and $ \Sigma^{1/2} \wt{X}^\top g / \normtwo{g} $ happen to be collinear, then introduce an additional constraint $ \inprod{v}{u} = 0 $ for an arbitrary vector $ u\in\bbS^{d-1} $ orthogonal to $v^*$ and the `$=$' in \Cref{eqn:cU_0} becomes `$\ge$'. 
Furthermore, we have $ \dim(\cU_0) = d-2 $. 

Finally, by the same reasoning as for the upper bound (in particular \Cref{eqn:interlace}), 
\begin{align}
    \lambda_3(\Sigma^{1/2} \wh{X}^\top \Pi_g^\perp T \Pi_g^\perp \wh{X} \Sigma^{1/2})
    &\eqqlaw \lambda_3(\Sigma^{1/2} \wh{X}_{-n}^\top \diag(\lambda_1(T), \cdots, \lambda_{n-1}(T)) \wh{X}_{-n} \Sigma^{1/2}) , \notag 
\end{align}
where $ \wh{X}_{-n}\in\bbR^{(n-1)\times d} $ has i.i.d.\ $ \cN(0,1/n) $ entries and is independent of everything else. 
This concludes the proof of \Cref{lem:sandwich}. 
\end{proof}


Note that \Cref{eqn:interlace} in the above proof implies that $ \brv{T} $ has the same limiting spectral distribution as $ T $ which is in turn given by $ \law(\cT(\ol{Y})) $. 
Now the only difference between the bound in \Cref{lem:sandwich} and the one in \Cref{lem:bulk} is that $n$ in the latter is replaced with $n-1$ in the former. 
However, this is immaterial asymptotically as $n,d\to\infty$ with $ n/d\to\delta $.

To prove \Cref{lem:bulk}, it then remains to show that both the upper and lower bounds in \Cref{lem:sandwich} converge to the same limit $ \sup\supp(\ol{\mu}_{\wh{D}}) $. 
It suffices to consider $ \lambda_{1,3}(\wh{D}) $ (instead of $ \lambda_{1,3}(\brv{D}) $). 

Since the following result may be of independent interest, we isolate the required assumptions and state it in a self-contained manner. 
\begin{enumerate}[label=(A\arabic*)]
    \setcounter{enumi}{\value{asmpctr}}

    \item[\ref{asmp:proportional}] $ n,d\to\infty $ with $n/d\to\delta$.

    \item \label[asmp]{asmp:bulk_bdd} $\normtwo{\Sigma}$ and $\normtwo{T}$ are uniformly bounded over $n$. 

    \item \label[asmp]{asmp:bulk_esd} The empirical spectral distributions $\mu_{T}$ and $\mu_{\Sigma}$ of $T$ and $\Sigma$ converge respectively to $\ol{\mu}_{T}$ and $\ol{\mu}_{\Sigma}$, with $\ol{\mu}_{T},\ol{\mu}_{\Sigma}\neq\delta_{0}$.
    Furthermore, for all $\varsigma>0$ there exists $n_{0}\in \N$ such that whenever $n\geq n_{0}$ we have 
    \beq\begin{aligned}
        \supp\mu_{T}\subset \supp\ol{\mu}_{T}+[-\varsigma,\varsigma],
        \qquad 
        \supp\mu_{\Sigma}\subset \supp\ol{\mu}_{\Sigma}+[-\varsigma,\varsigma]. 
        \label{eqn:supp_Sigma_T} 
    \end{aligned}\eeq

    \item \label[asmp]{asmp:bulk_sup_supp} The support of $\ol{\mu}_{T}$ intersects with $(0,\infty)$, i.e.,
    \beq\label{eq:T}
    \sup\supp\ol{\mu}_{T}>0.
\eeq

\setcounter{asmpctr}{\value{enumi}}
\end{enumerate}

The uniform boundedness of $\normtwo{\Sigma}$ has been assumed in \Cref{asmp:sigma}. 
The uniform boundedness of $\normtwo{T}$ follows from the boundedness of $\cT$ in \Cref{asmp:preprocessor}.
In \Cref{asmp:bulk_esd}, the convergence of $\mu_{T} = \frac{1}{n}\sum_{i = 1}^n \delta_{\cT(q(\langle x_i, \beta^*\rangle, \eps_i))}$ and the first part of \Cref{eqn:supp_Sigma_T} follows from the law of large numbers; the convergence of $\mu_{\Sigma}$ has been assumed in \Cref{asmp:sigma} and the second part of \Cref{eqn:supp_Sigma_T} is the same as \Cref{eqn:sigma_no_outlier}. 
Neither $\ol{\mu}_{T}$ nor $\ol{\mu}_{\Sigma}$ can be $\delta_0$ since $\cT$ is not constantly $0$ by \Cref{eqn:asmp_T}, and $\ol{\Sigma}$ is strictly positive. 
\Cref{asmp:bulk_sup_supp} is implied by $\sup\limits_{y\in\supp(\ol{Y})} \cT(y) > 0$ in \Cref{asmp:preprocessor}. 
\begin{lemma}[$\lambda_1(\wh{D})$ converges to right edge, {\cite[Theorem 4.3]{FanSunWang-mixed}}]\label{lem:lambda1}
Suppose that \Cref{asmp:proportional,asmp:bulk_bdd,asmp:bulk_esd,asmp:bulk_sup_supp} hold true.
Consider the matrix $ \wh{D} $ in \Cref{eqn:def_D_hat} and let $ \mu_{\wh{D}} $ denote its empirical spectral distribution. 
Then, almost surely, $\mu_{\wh{D}}$ converges to a deterministic probability measure $\ol{\mu}_{\wh{D}}$ on $\R$ and 
\beq
  \lim_{d\to\infty} \lambda_{1}(\wh{D}) = \sup\supp(\ol{\mu}_{\wh{D}}). \notag 
\eeq
\end{lemma}

\begin{lemma}[$\lambda_3(\wh{D})$ converges to right edge]
\label{lem:lambda3}
Suppose that \Cref{asmp:proportional,asmp:bulk_bdd,asmp:bulk_esd,asmp:bulk_sup_supp} hold true.
Then 
\begin{align}
    \lim_{d\to\infty} \lambda_3(\wh{D}) &= \sup\supp(\ol{\mu}_{\wh{D}}) , \quad \text{almost surely} . \notag 
\end{align}
\end{lemma}

\begin{proof}
To derive the limit, we show a pair of matching upper and lower bounds. 
Denote $ \lambda^\circ = \sup\supp(\ol{\mu}_{\wh{D}}) $. 
The upper bound is straightforward:
\begin{align}
    \lim_{d\to\infty} \lambda_3(\wh{D}) 
    &\le \lim_{d\to\infty} \lambda_1(\wh{D})
    = \sup\supp(\ol{\mu}_{\wh{D}}) , \notag 
\end{align}
where the equality is by \Cref{lem:lambda1}.

As for the lower bound, we would like to show: for any $ \lambda < \lambda^\circ $, $ \lim\limits_{d\to\infty} \lambda_3(\wh{D}) \ge \lambda $ almost surely. 
By the choice of $\lambda$, there exists a constant $c>0$ such that $ \ol{\mu}_{\wh{D}}(\lambda, \infty) \ge 2c $.
Recall that by \cite[Theorem 1.2.1]{Zhang_Thesis_RMT}, almost surely $ \mu_{\wh{D}} $ weakly converges to $ \ol{\mu}_{\wh{D}} $. 
Therefore, with probability $1$, for every sufficiently large $d$, $ \mu_{\wh{D}}(\lambda, \infty) \ge c \ge 3/d $. 
This means
\begin{align}
    \frac{1}{d} \card{\brace{ i\in\{1,\ldots, d\} : \lambda_i(\wh{D}) \ge \lambda }} &\ge \frac{3}{d} , \notag 
\end{align}
that is, $ \lambda_3(\wh{D}) \ge \lambda $, which completes the proof of the lower bound and hence the lemma. 
\end{proof}

\subsection{Proof of \Cref{eqn:error_t_t'}}
\label{sec:pf_bound_error_t_t'}

Recall from \Cref{eqn:def-Mhat-ehat,eqn:def_et} the definition of $ \wh{e}^t $.
We will first provide a suite of auxiliary bounds on the spectral norms of various matrices in \Cref{sec:pf_bound_mtx_norm}. 
They will prove useful in the sequel. 
We then show in \Cref{sec:bound_e1t_e2t} that 
\begin{align}
\lim_{t\to\infty} \plim_{n\to\infty} \frac{1}{\sqrt{n}} \normtwo{e_1^t} &= 0 , \quad
\lim_{t\to\infty} \plim_{d\to\infty} \frac{1}{\sqrt{d}} \normtwo{e_2^t} = 0 . \label{eqn:bound_e1t_e2t}
\end{align}
Next, using this, we show in \Cref{sec:bound_hat_et} that  
\begin{align}
\lim_{t\to\infty} \plim_{d\to\infty} \frac{1}{\sqrt{d}} \normtwo{\wh{e}^t} &= 0 . \label{eqn:bound_hat_et}
\end{align}
Finally, in \Cref{sec:bound_hat_ett'} we prove \Cref{eqn:error_t_t'}, i.e., 
\begin{align}
\lim_{t'\to\infty} \lim_{t\to\infty} \plim_{d\to\infty} \frac{1}{\sqrt{d}} \normtwo{\wh{e}^{t,t'}} = 0. \notag 
\end{align}

\subsubsection{Bounding the norms of various matrices}
\label{sec:pf_bound_mtx_norm}

We first recall the following elementary facts regarding the spectral norm, singular values and eigenvalues of a matrix. 
For any matrix $K\in\bbR^{n\times d}$, 
\begin{align}
    \normtwo{K} &= \sigma_1(K)
    = \sqrt{\lambda_1(K^\top K)}
    = \sqrt{\lambda_1(K K^\top)} . \notag 
\end{align}
If $K$ is symmetric ($n = d$), this is further equal to 
\begin{align}
    \normtwo{K} &= \sqrt{\lambda_1(K^2)}
    = \max\brace{|\lambda_1(K)|, |\lambda_n(K)|} . \notag 
\end{align}
If $K$ is PSD, then singular values coincide with eigenvalues and hence $ \normtwo{K} = \lambda_1(K) $. 

Using these facts, we have 
\begin{align}
    \lim_{d\to\infty} \normtwo{\Sigma} &= \lim_{d\to\infty} \lambda_1(\Sigma) 
    = \sup\supp(\ol{\Sigma})
    \eqqcolon \const_{\Sigma} , \label{eqn:bound_Sigma} \\
    \lim_{n\to\infty} \normtwo{T} &= \lim_{d\to\infty} \max_{i} |\cT(y_i)|
    = \max\brace{ \abs{\inf\supp(\cT(\ol{Y}))} , \abs{\sup\supp(\cT(\ol{Y}))} }
    \eqqcolon \const_T , \label{eqn:bound_T} \\
    \lim_{d\to\infty} \normtwo{\wt{X}} &= \lim_{d\to\infty} \sqrt{\lambda_1(\wt{X}^\top \wt{X})}
    = 1 + 1/\sqrt{\delta}
    \eqqcolon \const_{\wt{X}} , \label{eqn:bound_A} 
\end{align}
where the last two lines hold almost surely. 
Note that $ \const_{\Sigma} < \infty $ since $ \normtwo{\Sigma} $ is uniformly bounded (see \Cref{asmp:sigma}) and $ \const_{T} < \infty $ since $ \cT $ is bounded (see \Cref{asmp:preprocessor}). 
The last line follows since $ \wt{X}^\top \wt{X} $ is a Wishart matrix and its top eigenvalue converges almost surely to the right edge $ (1 + 1/\sqrt{\delta})^2 $ of the support of its limiting spectral distribution, the Marchenko--Pastur law \cite{YinBaiKrishnaiah}. 
Additionally, note that $ \normtwo{\Sigma^k} = \const_{\Sigma}^k $ for any $k\in\bbR$, since $ \Sigma $ is PSD. Using the sub-multiplicativity of matrix norms, we then have the following bound for $D$: 
\begin{align}
    \lim_{d\to\infty} \normtwo{D}
    &= \lim_{d\to\infty} \normtwo{\Sigma^{1/2} \wt{X}^\top T \wt{X} \Sigma^{1/2}} 
    \le \lim_{d\to\infty} \normtwo{\Sigma^{1/2}}^2 \normtwo{\wt{X}}^2 \normtwo{T} 
    = \const_{\Sigma} \const_{\wt{X}}^2 \const_T 
    \eqqcolon \const_D . \label{eqn:bound_||D||} 
\end{align}
Since $D$ is a symmetric matrix, $ \normtwo{D} = \max\brace{ |\lambda_1(D)| , |\lambda_d(D)| } $ and therefore for every sufficiently large $d$, it holds almost surely that 
\begin{align}
    -(\const_D + 1) &\le \lambda_d(D) \le \lambda_1(D) \le \const_D + 1 . \label{eqn:bound_eigval_D} 
\end{align}
The extra $ +1 $ term is to exclude fluctuation when $ d\le d_0 $ for some constant $ d_0 $. 

Recall that $ a^* > \sup\supp(\cT(\ol{Y})) $ and denote 
\begin{align}
    \wc{\const}_T &\coloneqq \abs{\inf\supp(\cT(\ol{Y}))} , \quad 
    \wh{\const}_T \coloneqq \sup\supp(\cT(\ol{Y})) > 0 . \notag 
\end{align}
Then, we have the following bound for $F$:
\begin{align}
    \lim_{n\to\infty} \normtwo{F} &= \lim_{n\to\infty} \max_{i} \frac{\abs{\cT(y_i)}}{a^* - \cT(y_i)} 
    \le \lim_{n\to\infty} \frac{\max_{i} \abs{\cT(y_i)}}{a^* - \max_{i} \cT(y_i)} 
    \le \frac{\const_T}{a^* - \wh{\const}_T}
    \eqqcolon \const_F . \label{eqn:bound_F} 
\end{align}

Recall 
\begin{align}
    B &= \paren{ \gamma^* I_d - \expt{\cF_{a^*}(\ol{Y})} \Sigma }^{-1} \Sigma , \notag 
\end{align}
and $ \gamma^* > s(a^*) $. 
Therefore $ \gamma^* I_d - \expt{\cF_{a^*}(\ol{Y})} \Sigma $ is positive definite. 
We can then bound the spectral norm of $B$ as follows: 
\begin{align}
    \lim_{d\to\infty} \normtwo{B} 
    &\le \lim_{d\to\infty} \normtwo{\gamma^* I_d - \expt{\cF_{a^*}(\ol{Y})} \Sigma}^{-1} \normtwo{\Sigma} 
    \le \frac{\const_{\Sigma}}{\gamma^* - s(a^*)}
    \eqqcolon \const_B . \label{eqn:bound_B} 
\end{align}

Recalling $ \wt{B} = \Sigma^{-1/2} B $ and using \Cref{eqn:bound_Sigma,eqn:bound_B}, we have 
\begin{align}
    \lim_{d\to\infty} \normtwo{\wt{B}} &\le \lim_{d\to\infty} \normtwo{\Sigma}^{-1/2} \normtwo{B} 
    \le \frac{\const_B}{\sqrt{\inf\supp(\ol{\Sigma})}} 
    \eqqcolon \const_{\wt{B}} . \label{eqn:bound_B_tilde} 
\end{align}
Note that $ \const_{\wt{B}} < \infty $ since $ \ol{\Sigma} > 0 $ (see \Cref{asmp:sigma}). 
Recalling $ \wh{M} = \frac{D + \ell I_d}{\lambda_1 + \ell} $ and using \Cref{eqn:bound_||D||}, we have
\begin{align}
    \lim_{d\to\infty} \normtwo{\wh{M}} 
    &\le \lim_{d\to\infty} \frac{\normtwo{D} + \abs{\ell}}{\abs{\lambda_1 + \ell}} 
    \le \frac{\const_D + \abs{\ell}}{\abs{\lambda_1 + \ell}} 
    \eqqcolon \const_{\wh{M}} . \label{eqn:bound_M_hat} 
\end{align}

\subsubsection{Bounding $ e_1^t, e_2^t $}
\label{sec:bound_e1t_e2t}
To prove \Cref{eqn:bound_e1t_e2t}, or equivalently, 
\begin{align}
    \lim_{t\to\infty} \plim_{n\to\infty} \frac{1}{n} \normtwo{e_1^t}^2 &= 0 , \quad
    \lim_{t\to\infty} \plim_{d\to\infty} \frac{1}{d} \normtwo{e_2^t}^2 = 0 , \notag 
\end{align}
we follow the proof strategy of \cite[Lemma 5.3]{mondelli2021optimalcombination}. 
The idea is to express these quantities as state evolution parameters and show that they converge to the desired fixed points. 
Writing  
\begin{align}
\frac{1}{n} \normtwo{e_1^t}^2 &= \frac{1}{n} \normtwo{u^t - u^{t-1}}^2 = \frac{1}{n} \normtwo{u^t}^2 + \frac{1}{n} \normtwo{u^{t-1}}^2 - \frac{2}{n}\inprod{u^t}{u^{t-1}} , \notag \\
\frac{1}{d} \normtwo{e_2^t}^2 &= \frac{1}{d} \normtwo{v^{t+1} - v^{t}}^2 = \frac{1}{d} \normtwo{v^{t+1}}^2 + \frac{1}{d} \normtwo{v^{t}}^2 - \frac{2}{d}\inprod{v^{t+1}}{v^{t}} , \notag 
\end{align}
and using the state evolution result in \Cref{prop:SE}, we have 
\begin{align}
\plim_{n\to\infty} \frac{1}{n} \normtwo{e_1^t}^2
&= \lim_{n\to\infty} \frac{1}{n} \expt{\inprod{U_t}{U_t}} + \lim_{n\to\infty} \frac{1}{n} \expt{\inprod{U_{t-1}}{U_{t-1}}} - 2 \lim_{n\to\infty} \frac{1}{n} \expt{\inprod{U_t}{U_{t-1}}} \notag \\
&= \frac{\expt{\ol{\Sigma}}}{\delta} \mu_t^2 + \sigma_{U,t}^2 + \frac{\expt{\ol{\Sigma}}}{\delta} \mu_{t-1}^2 + \sigma_{U,t-1}^2 \notag \\
&\phantom{=}~- 2 \paren{\frac{\expt{\ol{\Sigma}}}{\delta} \mu_t \mu_{t-1} + \lim_{n\to\infty} \frac{1}{n} \expt{\inprod{\sigma_{U,t} W_{U,t}}{\sigma_{U,t-1} W_{U,t-1}}}} , \notag 
\end{align}
and 
\begin{align}
\plim_{d\to\infty} \frac{1}{d} \normtwo{e_2^t}^2 
&= \lim_{d\to\infty} \frac{1}{d} \expt{\inprod{V_{t+1}}{V_{t+1}}} + \lim_{d\to\infty} \frac{1}{d} \expt{\inprod{V_{t}}{V_{t}}} - 2 \lim_{d\to\infty} \frac{1}{d} \expt{\inprod{V_{t+1}}{V_{t}}} \notag \\
&= \expt{\ol{\Sigma}} \chi_{t+1}^2 + \sigma_{V,t+1}^2
+ \expt{\ol{\Sigma}} \chi_{t}^2 + \sigma_{V,t}^2 \notag \\
&\phantom{=}~-2 \paren{\expt{\ol{\Sigma}} \chi_{t+1} \chi_t  + \lim_{d\to\infty} \frac{1}{d} \expt{\inprod{\sigma_{V,t+1} W_{V,t+1}}{\sigma_{V,t} W_{V,t}}}} . \notag 
\end{align}
By \Cref{lem:SE_stay}, the values of $ \mu_t, \sigma_{U,t}, \chi_{t+1}, \sigma_{V,t+1} $ do not change with $t$ and are equal to $ \mu, \sigma_U, \chi, \sigma_V $. 
Therefore, to show \Cref{eqn:bound_e1t_e2t}, it suffices to show 
\begin{align}
\lim_{t\to\infty} \lim_{n\to\infty} \frac{1}{n} \expt{\inprod{\sigma_{U,t} W_{U,t}}{\sigma_{U,t-1} W_{U,t-1}}} &= \sigma_U^2 , \notag \\
\lim_{t\to\infty} \lim_{d\to\infty} \frac{1}{d} \expt{\inprod{\sigma_{V,t+1} W_{V,t+1}}{\sigma_{V,t} W_{V,t}}} &= \sigma_V^2 . \notag
\end{align}
From the state evolution, we have
\begin{align}
(\Phi_t)_{t+1,t} 
&= \lim_{n\to\infty} \frac{1}{n} \expt{\inprod{\sigma_{U,t} W_{U,t}}{\sigma_{U,t-1} W_{U,t-1}}} \notag \\
&= \lim_{n\to\infty} \frac{1}{n} \expt{\inprod{f_t(V_t) - \mu_t \wt{\mathfrak{B}}^*}{f_{t-1}(V_{t-1}) - \mu_{t-1} \wt{\mathfrak{B}}^*}} \notag \\
&= \lim_{n\to\infty} \frac{1}{n} \expt{\inprod{f_t(V_t)}{f_{t-1}(V_{t-1})}} 
- \mu_t \lim_{n\to\infty} \frac{1}{n} \expt{\inprod{f_{t-1}(V_{t-1})}{\wt{\mathfrak{B}}^*}} \notag \\
&\phantom{=}~- \mu_{t-1} \lim_{n\to\infty} \frac{1}{n} \expt{\inprod{f_t(V_t)}{\wt{\mathfrak{B}}^*}} 
+ \mu_t \mu_{t-1} \lim_{n\to\infty} \frac{1}{n} \expt{\inprod{\wt{\mathfrak{B}}^*}{\wt{\mathfrak{B}}^*}} \notag \\
&= \lim_{n\to\infty} \frac{1}{n} \expt{\inprod{f_t(V_t)}{f_{t-1}(V_{t-1})}} 
- \frac{\expt{\ol{\Sigma}}}{\delta} \mu_t \mu_{t-1} , \label{eqn:Phi_t+1_t} 
\end{align}
where the last equality is by \Cref{eqn:mut}; 
and 
\begin{align}
(\Psi_t)_{t+1,t} &= \lim_{d\to\infty} \frac{1}{d} \expt{\inprod{\sigma_{V,t+1} W_{V,t+1}}{\sigma_{V,t} W_{V,t}}} \notag \\
&= \lim_{n\to\infty} \frac{1}{n} \expt{\inprod{g_t(U_t; Y)}{g_{t-1}(U_{t-1}; Y)}} . \label{eqn:Psi_t+1_t}
\end{align}
Recall from \Cref{eqn:my-denoisers-2} that $ g_t(U_t; Y) = FU_t $ and $ f_{t+1}(V_{t+1}) = B_{t+1} V_{t+1} $. 
Therefore we have
\begin{align}
&\phantom{=}~ \lim_{n\to\infty} \frac{1}{n} \expt{\inprod{f_t(V_t)}{f_{t-1}(V_{t-1})}} \notag \\
&= \lim_{n\to\infty} \frac{1}{n} \expt{(\chi_t \wt{\mathfrak{B}}^* + \sigma_{V,t} W_{V,t})^\top B_t^\top B_{t-1} (\chi_{t-1} \wt{\mathfrak{B}}^* + \sigma_{V,t-1} W_{V,t-1})} \notag \\
&= \chi_t \chi_{t-1} \lim_{n\to\infty} \frac{1}{n} \expt{{\mathfrak{B}^*}^\top \Sigma^{1/2} B_t^\top B_{t-1} \Sigma^{1/2} \mathfrak{B}^*} \notag\\
&\hspace{10em}+ \lim_{n\to\infty} \frac{1}{n} \expt{(\sigma_{V,t} W_{V,t})^\top B_t^\top B_{t-1} (\sigma_{V,t-1} W_{V,t-1})} \notag \\
&= \chi_t \chi_{t-1} \frac{1}{\delta} \expt{\frac{\ol{\Sigma}^3}{(\gamma_t - c \ol{\Sigma}) (\gamma_{t-1} - c \ol{\Sigma})}} \notag \\
&\phantom{=}~+ \frac{1}{\delta} \expt{\frac{\ol{\Sigma}^2}{(\gamma_t - c \ol{\Sigma}) (\gamma_{t-1} - c \ol{\Sigma})}} \lim_{d\to\infty} \frac{1}{d} \expt{\inprod{\sigma_{V,t} W_{V,t}}{\sigma_{V,t-1} W_{V,t-1}}} , \label{eqn:corr_ft_ft-1} 
\end{align}
where we use \Cref{prop:trace_corr} in the last step. Similarly, we have 
\begin{align}
&\phantom{=}~ \lim_{n\to\infty} \frac{1}{n} \expt{\inprod{g_t(U_t; Y)}{g_{t-1}(U_{t-1}; Y)}} \notag \\
&= \lim_{n\to\infty} \frac{1}{n} \expt{(\mu_t G + \sigma_{U,t} W_{U,t})^\top F^2 (\mu_{t-1} G + \sigma_{U,t-1} W_{U,t-1})} \notag \\
&= \mu_t \mu_{t-1} \lim_{n\to\infty} \frac{1}{n} \expt{G^\top F^2 G}
+ \lim_{n\to\infty} \frac{1}{n} \expt{(\sigma_{U,t} W_{U,t})^\top F^2 (\sigma_{U,t-1} W_{U,t-1})} \notag \\
&= \mu_t \mu_{t-1} \expt{\ol{G}^2 \cF_{a^*}(\ol{Y})^2} 
+ \expt{\cF_{a^*}(\ol{Y})^2} \lim_{n\to\infty} \frac{1}{n} \expt{\inprod{\sigma_{U,t} W_{U,t}}{\sigma_{U,t-1} W_{U,t-1}}} .\label{eqn:corr_gt_gt-1}
\end{align}

Letting 
\begin{align}
\tau_t &\coloneqq \lim_{n\to\infty} \frac{1}{n} \expt{\inprod{\sigma_{U,t} W_{U,t}}{\sigma_{U,t-1} W_{U,t-1}}} ,\notag\\ 
\omega_t &\coloneqq \lim_{d\to\infty} \frac{1}{d} \expt{\inprod{\sigma_{V,t} W_{V,t}}{\sigma_{V,t-1} W_{V,t-1}}} \notag 
\end{align}
and using \Cref{eqn:corr_ft_ft-1,eqn:corr_gt_gt-1} in \Cref{eqn:Phi_t+1_t,eqn:Psi_t+1_t}, 
we obtain a pair of recursions for $ \tau_t, \omega_t $: 
\begin{align}
\tau_t &= \chi_t \chi_{t-1} \frac{1}{\delta} \expt{\frac{\ol{\Sigma}^3}{(\gamma_t - c \ol{\Sigma}) (\gamma_{t-1} - c \ol{\Sigma})}} \notag \\
&\hspace{1em}- \frac{\expt{\ol{\Sigma}}}{\delta} \mu_t \mu_{t-1}
+ \frac{1}{\delta} \expt{\frac{\ol{\Sigma}^2}{(\gamma_t - c \ol{\Sigma}) (\gamma_{t-1} - c \ol{\Sigma})}} \omega_t , \label{eqn:alphat_recursion} \\
\omega_{t+1} &= \mu_t \mu_{t-1} \expt{\ol{G}^2 \cF_{a^*}(\ol{Y})^2} 
+ \expt{\cF_{a^*}(\ol{Y})^2} \tau_t . \label{eqn:betat_recursion}
\end{align}
Using \Cref{eqn:alphat_recursion} in \Cref{eqn:betat_recursion}, we further obtain
\begin{align}
\omega_{t+1} &= \frac{\expt{\ol{\Sigma}}}{\delta} \expt{\paren{\frac{\delta}{\expt{\ol{\Sigma}}} \ol{G}^2 - 1} \cF_{a^*}(\ol{Y})^2} \mu_t \mu_{t-1} \notag \\
&\phantom{=}~+\chi_t \chi_{t-1} \frac{\expt{\cF_{a^*}(\ol{Y})^2}}{\delta} \expt{\frac{\ol{\Sigma}^3}{(\gamma_t - c \ol{\Sigma}) (\gamma_{t-1} - c \ol{\Sigma})}} \notag \\
&\phantom{=}~+ \frac{\expt{\cF_{a^*}(\ol{Y})^2}}{\delta} \expt{\frac{\ol{\Sigma}^2}{(\gamma_t - c \ol{\Sigma}) (\gamma_{t-1} - c \ol{\Sigma})}} \omega_t . \notag 
\end{align}
We would like to show 
\begin{align}
\lim_{t\to\infty} \omega_{t+1} &= \sigma_V^2 . \label{eqn:beta_liminf_limsup}
\end{align}
To this end, we will upper bound the $\limsup$ and lower bound the $\liminf$ both by $ \sigma_V^2 $. 
Let
\begin{align}
p_t &\coloneqq \frac{\expt{\cF_{a^*}(\ol{Y})^2}}{\delta} \expt{\frac{\ol{\Sigma}^2}{(\gamma_t - c \ol{\Sigma}) (\gamma_{t-1} - c \ol{\Sigma})}} , \notag \\
q_t &\coloneqq \frac{\expt{\ol{\Sigma}}}{\delta} \expt{\paren{\frac{\delta}{\expt{\ol{\Sigma}}} \ol{G}^2 - 1} \cF_{a^*}(\ol{Y})^2} \mu_t \mu_{t-1} \notag \\
&\phantom{\coloneqq}~+\chi_t \chi_{t-1} \frac{\expt{\cF_{a^*}(\ol{Y})^2}}{\delta} \expt{\frac{\ol{\Sigma}^3}{(\gamma_t - c \ol{\Sigma}) (\gamma_{t-1} - c \ol{\Sigma})}} , \notag 
\end{align}
and 
\begin{align}
\ul{\omega} &= \liminf_{t\to\infty} \omega_{t+1} , \quad
\ol{\omega} = \limsup_{t\to\infty} \omega_{t+1} . \notag
\end{align}
Then by subadditivity of $\limsup$, 
\begin{align}
\ol{\omega}
&= \limsup_{t\to\infty} q_t + p_t \omega_t \notag \\
&\le \lim_{t\to\infty} q_t + \paren{\lim_{t\to\infty} p_t} \paren{\limsup_{t\to\infty} \omega_t} = \frac{\expt{\ol{\Sigma}}}{\delta} \expt{\paren{\frac{\delta}{\expt{\ol{\Sigma}}} \ol{G}^2 - 1} \cF_{a^*}(\ol{Y})^2} \mu^2 \notag \\
&\hspace{5em}+ \frac{\expt{\cF_{a^*}(\ol{Y})^2}}{\delta} \expt{\frac{\ol{\Sigma}^3}{(\gamma^* - c \ol{\Sigma})^2}} \chi^2 
+ \frac{\expt{\cF_{a^*}(\ol{Y})^2}}{\delta} \expt{\frac{\ol{\Sigma}^2}{(\gamma^* - c \ol{\Sigma})^2}} \ol{\omega} , \notag 
\end{align}
where the inequality holds since $ \lim\limits_{t\to\infty} p_t \ge0 $. 
Rearranging terms on both sides gives
\begin{align}
\ol{\omega} &\le \paren{1 - \frac{\expt{\cF_{a^*}(\ol{Y})^2}}{\delta} \expt{\frac{\ol{\Sigma}^2}{(\gamma^* - c \ol{\Sigma})^2}}}^{-1} \Bigg( \frac{\expt{\ol{\Sigma}}}{\delta} \expt{\paren{\frac{\delta}{\expt{\ol{\Sigma}}} \ol{G}^2 - 1} \cF_{a^*}(\ol{Y})^2} \mu^2 \notag \\
&\hspace{18em}+ \frac{\expt{\cF_{a^*}(\ol{Y})^2}}{\delta} \expt{\frac{\ol{\Sigma}^3}{(\gamma^* - c \ol{\Sigma})^2}} \chi^2 \Bigg) . \notag 
\end{align}
Note that the term in the first parentheses is positive since it is nothing but $ 1 - w_2 $ which is positive whenever $ a^* > a^\circ $. 
We claim that the right-hand side is equal to $ \sigma_V^2 $. 
This can be seen from the fixed point equations of the state evolution recursion. 
Indeed, from \Cref{eqn:uselater1,eqn:uselater2}, we have the following identity for $ \sigma_V^2 $:
\begin{align}
\sigma_V^2 &= \expt{\ol{G}^2 \cF_{a^*}(\ol{Y})^2} \mu^2 + \expt{\cF_{a^*}(\ol{Y})^2} \sigma_{U}^2 \notag \\
&= \expt{\ol{G}^2 \cF_{a^*}(\ol{Y})^2} \mu^2
+ \frac{\expt{\cF_{a^*}(\ol{Y})^2}}{\delta} \expt{\frac{\ol{\Sigma}^3}{(\gamma^* - \expt{\cF_{a^*}(\ol{Y})} \ol{\Sigma})^2}} \chi^2 \notag \\
&\phantom{=}~+ \frac{\expt{\cF_{a^*}(\ol{Y})^2}}{\delta} \expt{\frac{\ol{\Sigma}^2}{(\gamma^* - \expt{\cF_{a^*}(\ol{Y})} \ol{\Sigma})^2}} \sigma_{V}^2 
- \frac{\expt{\cF_{a^*}(\ol{Y})^2}}{\delta} \expt{\ol{\Sigma}} \mu^2 . \label{eqn:alt-sigmaV} 
\end{align}
Solving for $ \sigma_V^2 $, we obtain exactly the upper bound on $ \ol{\omega} $. 

Analogously, a lower bound on $ \ul{\omega} $ can be derived using superadditivity of $ \liminf $:
\begin{align}
\ul{\omega}
&= \liminf_{t\to\infty} q_t + p_t \omega_t \notag \\
&\ge \lim_{t\to\infty} q_t + \paren{\lim_{t\to\infty} p_t} \paren{\liminf_{t\to\infty} \omega_t} = \frac{\expt{\ol{\Sigma}}}{\delta} \expt{\paren{\frac{\delta}{\expt{\ol{\Sigma}}} \ol{G}^2 - 1} \cF_{a^*}(\ol{Y})^2} \mu^2 \notag \\
&\hspace{4em}+ \frac{\expt{\cF_{a^*}(\ol{Y})^2}}{\delta} \expt{\frac{\ol{\Sigma}^3}{(\gamma^* - c \ol{\Sigma})^2}} \chi^2 
+ \frac{\expt{\cF_{a^*}(\ol{Y})^2}}{\delta} \expt{\frac{\ol{\Sigma}^2}{(\gamma^* - c \ol{\Sigma})^2}} \ul{\omega} . \notag 
\end{align}
Rearranging and using \Cref{eqn:alt-sigmaV} gives $ \ul{\omega} \ge \sigma_V^2 $. 
This establishes \Cref{eqn:beta_liminf_limsup}. 

Next, using \Cref{eqn:beta_liminf_limsup} in \Cref{eqn:alphat_recursion}, we get
\begin{align}
    \lim_{t\to\infty} \tau_t 
    &= \frac{1}{\delta} \expt{\frac{\ol{\Sigma}^3}{(\gamma^* - c \ol{\Sigma})^2}} \chi^2
    - \frac{\expt{\ol{\Sigma}}}{\delta} \mu^2 
    + \frac{1}{\delta} \expt{\frac{\ol{\Sigma}^2}{(\gamma^* - c \ol{\Sigma})^2}} \sigma_V^2 . \notag 
\end{align}
By \Cref{eqn:uselater2}, the right-hand side is precisely $ \sigma_U^2 $. 
Therefore, we conclude 
\begin{align}
    \lim_{t\to\infty} \tau_t &= \sigma_U^2 , \notag 
\end{align}
which, together with \Cref{eqn:beta_liminf_limsup}, completes the proof of \Cref{eqn:bound_e1t_e2t}.

\subsubsection{Bounding $ \wh{e}^t $}
\label{sec:bound_hat_et}
Let us now prove \Cref{eqn:bound_hat_et}. 
Recall from \Cref{eqn:def-Mhat-ehat,eqn:def_et} that $ \wh{e}^t $ comprises the following terms:
\begin{align}
\wh{e}^t &= \wh{e}^t_1 + \wh{e}^t_2 + \wh{e}^t_3 + \wh{e}^t_4 + \wh{e}^t_5 + \wh{e}^t_6 , \notag 
\end{align}
where 
\begin{align}
\wh{e}^t_1 &= \frac{\ell}{\lambda_1 + \ell} \wt{B} e_2^t + \frac{a^* c}{\lambda_1 + \ell} \Sigma^{1/2} B e_2^t , \notag \\
\wh{e}^t_2 &= \frac{a^*}{\lambda_1 + \ell} (b - b_t) \Sigma^{1/2} \wt{X}^\top F (b_t F + I_n)^{-1} (b F + I_n)^{-1} \wt{X} B_t v^t , \notag \\
\wh{e}^t_3 &= \frac{a^*}{\lambda_1 + \ell} (\gamma^* - \gamma_t) \Sigma^{1/2} \wt{X}^\top F (b F + I_n)^{-1} \wt{X} (\gamma_t I_d - c \Sigma)^{-1} (\gamma^* I_d - c \Sigma)^{-1} \Sigma v^t , \notag \\
\wt{e}^t_4 &\coloneqq \frac{a^*(c_t - c)}{\lambda_1 + \ell} B_t v^t , \notag \\
\wh{e}^t_5 &= \frac{a^* c}{\lambda_1 + \ell} (\gamma_t - \gamma^*) \Sigma^{1/2} (\gamma_t I_d - c \Sigma)^{-1} (\gamma^* I_d - c \Sigma)^{-1} \Sigma v^t , \notag \\
\wh{e}^t_6 &= \frac{a^* b_t}{\lambda_1 + \ell} \Sigma^{1/2} \wt{X}^\top F^2 (b_t F + I_n)^{-1} e_1^t . \notag 
\end{align}
Since the AMP is initialized so that the state evolution parameters stay fixed (see \Cref{lem:SE_stay}), for every $t\ge1$, $ \gamma_t = \gamma^* $ and we immediately get 
\begin{align}
\wh{e}_3^t = \wh{e}_5^t = 0_d . \label{eqn:bound_hat_e35t} 
\end{align}
By convergence of the empirical spectral distribution of $\Sigma$ (see \Cref{asmp:sigma}), for every $t\ge1$, 
\begin{align}
    \lim_{d\to\infty} b_t &= \lim_{d\to\infty} \frac{d}{n} \tr((\gamma_t I_d - c \Sigma)^{-1} \Sigma)
    = \frac{1}{\delta} \expt{\frac{\ol{\Sigma}}{\gamma_t - c \ol{\Sigma}}}
    = b , \notag 
\end{align}
and consequently
\begin{align}
    \plim_{d\to\infty} \frac{1}{\sqrt{d}} \normtwo{\wh{e}^t_2} &= 0 . 
    \label{eqn:bound_hat_e2t} 
\end{align}
By convergence of the noise sequence $ \eps = (\eps_1, \cdots, \eps_n) $ (see \Cref{asmp:noise}) and independence of covariate vectors $ 
(x_1, \cdots, x_n) $ (see \Cref{asmp:sigma}), 
\begin{align}
    \plim_{n\to\infty} c_t &= \plim_{n\to\infty} \frac{1}{n} \tr(F)
    = \expt{\cF_{a^*}(\ol{Y})} = c , \notag 
\end{align}
and consequently, 
\begin{align}
    \plim_{d\to\infty} \frac{1}{\sqrt{d}} \normtwo{\wh{e}^t_4} &= 0 . 
    \label{eqn:bound_hat_e4t} 
\end{align}

We use the bounds developed in the previous sections to bound $ \wh{e}_1^t $ and $ \wh{e}_6^t $. 
Specifically, 
\begin{align}
    \lim_{t\to\infty} \plim_{d\to\infty} \frac{1}{\sqrt{d}} \normtwo{\wh{e}_1^t} 
    &\le \lim_{t\to\infty} \plim_{d\to\infty} \abs{\frac{\ell}{\lambda_1 + \ell}} \normtwo{\wt{B}} \frac{\normtwo{e_2^t}}{\sqrt{d}} 
    + \abs{\frac{a^* c}{\lambda_1 + \ell}} \normtwo{\Sigma}^{1/2} \normtwo{B} \frac{\normtwo{e_2^t}}{\sqrt{d}} \notag \\
    &\le \paren{ \abs{\frac{\ell}{\lambda_1 + \ell}} \const_{\wt{B}} + \abs{\frac{a^* c}{\lambda_1 + \ell}} \sqrt{\const_{\Sigma}} \const_B } \lim_{t\to\infty} \plim_{d\to\infty} \frac{\normtwo{e_2^t}}{\sqrt{d}}
    = 0 , \label{eqn:bound_hat_e1t} \\
    \lim_{t\to\infty} \plim_{d\to\infty} \frac{1}{\sqrt{d}} \normtwo{\wh{e}_6^t} 
    &\le \lim_{t\to\infty} \plim_{d\to\infty} \abs{\frac{a^* b_t}{\lambda_1 + \ell}}
    \normtwo{\Sigma}^{1/2} \normtwo{\wt{X}} \normtwo{F}^2 \normtwo{(b_t F + I_n)^{-1}} \frac{\normtwo{e_1^t}}{\sqrt{d}} \notag \\
    &= \lim_{t\to\infty} \plim_{d\to\infty} \abs{\frac{a^* b_t}{\lambda_1 + \ell}}
    \normtwo{\Sigma}^{1/2} \normtwo{\wt{X}} \normtwo{F}^2 \normtwo{I_n - \frac{T}{a^*}} \frac{\normtwo{e_1^t}}{\sqrt{d}} \label{eqn:hat_e5t} \\
    &= \frac{\sqrt{\const_{\Sigma}} \const_{\wt{X}} \const_F^2 (a^* - \wc{\const}_T)}{\abs{\lambda_1 + \ell}}
    \lim_{t\to\infty} \plim_{d\to\infty} \frac{\normtwo{e_1^t}}{\sqrt{d}} 
    = 0 . \label{eqn:bound_hat_e6t} 
\end{align}
To obtain \Cref{eqn:hat_e5t}, it is useful to recall $ F = T(a^* I_n - T)^{-1} $ (see \Cref{eqn:my-denoisers-1}) and observe from \Cref{eqn:my_Onsager_coeff,eqn:fp-a-gamma} that $ b_t = 1 $ for every $ t\ge1 $ (where we use $ \gamma_t = \gamma^* $ for every $t\ge1$ from \Cref{lem:SE_stay}). 

Combining \Cref{eqn:bound_hat_e35t,eqn:bound_hat_e1t,eqn:bound_hat_e6t,eqn:bound_hat_e2t,eqn:bound_hat_e4t} yields \Cref{eqn:bound_hat_et}, as required. 

\subsubsection{Bounding $\wh{e}^{t,t'}$}
\label{sec:bound_hat_ett'}

Finally, we prove \Cref{eqn:error_t_t'}. 
Recalling the definition of $ \wh{e}^{t,t'} $ in \Cref{eqn:def_hat_e_t_t'} and using the triangle inequality and the sub-multiplicativity of norms, we have 
\begin{align}
\lim_{t'\to\infty} \lim_{t\to\infty} \plim_{d\to\infty} \frac{1}{\sqrt{d}} \normtwo{\wh{e}^{t,t'}} 
&\le \lim_{t'\to\infty} \lim_{t\to\infty} \plim_{d\to\infty} \frac{1}{\sqrt{d}} \sum_{s = 1}^{t'} \normtwo{\wh{M}}^{t' - s} \normtwo{\wh{e}^{t + s - 1}} \notag \\
&= \lim_{t'\to\infty} \lim_{t\to\infty} \sum_{s = 1}^{t'} \paren{ \lim_{d\to\infty} \normtwo{\wh{M}}^{t' - s} } \paren{ \plim_{d\to\infty} \frac{1}{\sqrt{d}} \normtwo{\wh{e}^{t + s - 1}} } \notag \\
&\le \lim_{t'\to\infty} \sum_{s = 1}^{t'} \const_{\wh{M}}^{t' - s} \paren{ \lim_{t\to\infty} \plim_{d\to\infty} \frac{1}{\sqrt{d}} \normtwo{\wh{e}^{t + s - 1}} } = 0 , \notag 
\end{align}
which implies \Cref{eqn:error_t_t'}. 
The inequality in the penultimate line is by \Cref{eqn:bound_M_hat} and the last equality is by \Cref{eqn:bound_hat_et}.


\section{Proof of Theorem \ref{prop:opt_thr}}
\label{sec:opt_thr}

We first prove \Cref{itm:opt_thr_2} of \Cref{prop:opt_thr}. 
Suppose that the condition $ a^* > a^\circ $ holds for some $\cT\in\sT$. 
If $ \phi $ is strictly decreasing on $ (\sup\supp(\cT(\ol{Y})), \infty) $, this condition is equivalent to the following one 
\begin{align}
    1 &< \frac{1}{\expt{\ol{\Sigma}}} \expt{\paren{\frac{\delta}{\expt{\ol{\Sigma}}} \ol{G}^2 - 1} \cF_{a^\circ}(\ol{Y})} \expt{\frac{\ol{\Sigma}^2}{\gamma^\circ - \expt{\cF_{a^\circ}(\ol{Y})} \ol{\Sigma}}} , \label{eqn:thr_to_opt} 
\end{align}
by \Cref{itm:threshold2} of \Cref{prop:equiv-threshold}. 
We assume $ a^\circ = 1 $. 
This assumption is without loss of generality due to scaling invariance. 
Indeed, the threshold condition for $\delta$ (i.e., \Cref{eqn:thr_to_opt} above) and the self-consistent equations for $ (a^\circ, \gamma^\circ) $ (see \Cref{eqn:bulk-selfcons-new} and \Cref{lem:bulk-a0-gamma0}) only depend on $ (a^\circ, \cT) $ through $ \cF_{a^\circ}(\ol{Y}) $. 
Therefore, they continue to hold if $ (a^\circ, \cT) $ is replaced\footnotemark{} with $ (1, \cT/a^\circ) $. 
\footnotetext{Note that $ a^\circ > \sup\supp(\cT(\ol{Y})) > 0 $. }
Let $ \cJ(y) = \frac{\cT(y)}{1 - \cT(y)} $ for notational convenience. 
The definition of $ (a^\circ, \gamma^\circ) $ in \Cref{eqn:bulk-selfcons-new} can then be written as 
\begin{align}
    1 &= \frac{1}{\delta} \expt{\cJ(\ol{Y})^2} \expt{\paren{\frac{\ol{\Sigma}}{\gamma^\circ - \expt{\cJ(\ol{Y})} \ol{\Sigma}}}^2} , \quad 
    1 = \frac{1}{\delta} \expt{\frac{\ol{\Sigma}}{\gamma^\circ - \expt{\cJ(\ol{Y})} \ol{\Sigma}}} . \label{eqn:a0-gamma0-optthr} 
\end{align}

 Let $ p_{\ol{G}} $ denote the density of $ \ol{G} \sim \cN(0, \expt{\ol{\Sigma}} / \delta) $, and $ p(\cdot \,|\, g) $  the conditional density of $ y = q(g, \eps) \in\bbR $ given $ g\in\bbR $ where $ \eps\sim P_\eps $.  Then, using the Cauchy--Schwarz inequality, the second factor on the right-hand side of \Cref{eqn:thr_to_opt}  can be bounded as follows:
\begin{align}
    &\phantom{=}~ \expt{\paren{\frac{\delta}{\expt{\ol{\Sigma}}} \ol{G}^2 - 1} \cF_{a^\circ}(\ol{Y})} 
    = \expt{ \paren{\frac{\delta}{\expt{\ol{\Sigma}}} \ol{G}^2 - 1} \cJ(\ol{Y}) } \notag \\
    &= \int_{\supp(\ol{Y})} \int_{\bbR} p_{\ol{G}}(g) p(y \,|\, g) \paren{\frac{\delta}{\expt{\ol{\Sigma}}} g^2 - 1} \cJ(y) \,\diff g \,\diff y \notag \\
    &= \int_{\supp(\ol{Y})} \expt{ p(y \,|\, \ol{G}) \paren{\frac{\delta}{\expt{\ol{\Sigma}}} \ol{G}^2 - 1} } \cJ(y) \,\diff y \,\notag \\
    &= \int_{\supp(\ol{Y})} \frac{\expt{ p(y \,|\, \ol{G}) \paren{\frac{\delta}{\expt{\ol{\Sigma}}} \ol{G}^2 - 1} }}{\sqrt{\expt{p(y\,|\,\ol{G})}}} \cdot \sqrt{\expt{p(y\,|\,\ol{G})}} \, \cJ(y) \,\diff y \notag \\
    &\le \paren{ \int_{\supp(\ol{Y})} \frac{\expt{ p(y \,|\, \ol{G}) \paren{\frac{\delta}{\expt{\ol{\Sigma}}} \ol{G}^2 - 1} }^2}{\expt{p(y\,|\,\ol{G})}} \,\diff y }^{1/2} \paren{ \int_{\supp(\ol{Y})} \expt{p(y\,|\,\ol{G})} \cJ(y)^2 \,\diff y }^{1/2} \notag \\
    &= \paren{ \int_{\supp(\ol{Y})} \frac{\expt{ p(y \,|\, \ol{G}) \paren{\frac{\delta}{\expt{\ol{\Sigma}}} \ol{G}^2 - 1} }^2}{\expt{p(y\,|\,\ol{G})}} \,\diff y }^{1/2} \expt{\cJ(\ol{Y})^2}^{1/2} . \label{eqn:cs1} 
\end{align}
Applying the Cauchy--Schwarz inequality to the third factor on the right-hand side of \Cref{eqn:thr_to_opt}, we obtain
\begin{align}
    \frac{1}{\expt{\ol{\Sigma}}} \expt{\frac{\ol{\Sigma}^2}{\gamma^\circ - \expt{\cF_{a^\circ}(\ol{Y})} \ol{\Sigma}}} 
    &= \frac{1}{\expt{\ol{\Sigma}}} \expt{\frac{\ol{\Sigma}^2}{\gamma^\circ - \expt{\cJ(\ol{Y})} \ol{\Sigma}}} \notag \\
    &= \frac{1}{\expt{\ol{\Sigma}}} \expt{\frac{\ol{\Sigma}}{\gamma^\circ - \expt{\cJ(\ol{Y})} \ol{\Sigma}} \cdot \ol{\Sigma}} \notag \\
    &\le \frac{\expt{\ol{\Sigma}^2}^{1/2}}{\expt{\ol{\Sigma}}} \expt{ \paren{\frac{\ol{\Sigma}}{\gamma^\circ - \expt{\cJ(\ol{Y})} \ol{\Sigma}}}^2 }^{1/2} . \label{eqn:cs2}
\end{align}
Combining \Cref{eqn:cs1,eqn:cs2}, we have that the right-hand side of \Cref{eqn:thr_to_opt} is bounded from above by
\begin{multline}
    \frac{\expt{\ol{\Sigma}^2}^{1/2}}{\expt{\ol{\Sigma}}} 
    \paren{ \int_{\supp(\ol{Y})} \frac{\expt{ p(y \,|\, \ol{G}) \paren{\frac{\delta}{\expt{\ol{\Sigma}}} \ol{G}^2 - 1} }^2}{\expt{p(y\,|\,\ol{G})}} \,\diff y }^{1/2} \notag \\
    \times\expt{\cJ(\ol{Y})^2}^{1/2} 
    \expt{ \paren{\frac{\ol{\Sigma}}{\gamma^\circ - \expt{\cJ(\ol{Y})} \ol{\Sigma}}}^2 }^{1/2} \notag \\
    = \frac{\expt{\ol{\Sigma}^2}^{1/2}}{\expt{\ol{\Sigma}}} 
    \paren{ \int_{\supp(\ol{Y})} \frac{\expt{ p(y \,|\, \ol{G}) \paren{\frac{\delta}{\expt{\ol{\Sigma}}} \ol{G}^2 - 1} }^2}{\expt{p(y\,|\,\ol{G})}} \,\diff y }^{1/2} 
    \sqrt{\delta} , \notag 
\end{multline}
where the equality follows from the first identity in \Cref{eqn:a0-gamma0-optthr}. 
Using this in \Cref{eqn:thr_to_opt}, we have 
\begin{align}
    \delta &> \frac{\expt{\ol{\Sigma}}^2}{\expt{\ol{\Sigma}^2}}
    \paren{ \int_{\supp(\ol{Y})} \frac{\expt{ p(y \,|\, \ol{G}) \paren{\frac{\delta}{\expt{\ol{\Sigma}}} \ol{G}^2 - 1} }^2}{\expt{p(y\,|\,\ol{G})}} \,\diff y }^{-1} . \label{eqn:opt_thr_ineq} 
\end{align}
In words, the condition above (which is independent of the choice of $ \cT $) holds for any $ \cT $ that satisfies \Cref{eqn:thr_to_opt} and therefore achieves a positive overlap. 

In the following, we show that the condition above is tight by proving \Cref{itm:opt_thr_1} of \Cref{prop:opt_thr}. 
Specifically, whenever \Cref{eqn:opt_thr_ineq} holds, we exhibit a preprocessing function $ \cT^*\colon\bbR\to\bbR $ that meets \Cref{eqn:thr_to_opt} and therefore must induce a positive overlap.  

Suppose that \Cref{eqn:opt_thr_ineq} holds. 
As before, we choose the scaling such that $ a^\circ = 1 $. Constructing $ \cT^*(y) $ is equivalent to constructing 
\begin{align}
    \cJ^*(y) = \frac{\cT^*(y)}{1 - \cT^*(y)} . 
    \label{eqn:Jstar_Tstar}
\end{align}
We require the following notation. 
Denote the right-hand side of \Cref{eqn:opt_thr_ineq} by $ \Delta(\delta) $. 
Moreover, 
\begin{align}
    m_0(y) &\coloneqq \expt{p(y\,|\,\ol{G})} , \quad 
    m_2(y) \coloneqq \expt{p(y\,|\,\ol{G}) \cdot \frac{\delta}{\expt{\ol{\Sigma}}} \ol{G}^2} , \label{eqn:def_m0_m2}
\end{align}
Before presenting the construction of $ \cJ^* $, we first observe that the integrals of both $ m_0 $ and $ m_2 $ are equal to $1$. 
\begin{align}
\begin{aligned}
    \int_{\supp(\ol{Y})} m_0(y) \,\diff y &= \expt{\int_{\supp(\ol{Y})} p(y\,|\,\ol{G}) \,\diff y} = 1 , \\
    \int_{\supp(\ol{Y})} m_2(y) \,\diff y &= \expt{\paren{ \int_{\supp(\ol{Y})} p(y\,|\,\ol{G}) \,\diff y } \frac{\delta}{\expt{\ol{\Sigma}}} \ol{G}^2}
    = \expt{\frac{\delta}{\expt{\ol{\Sigma}}} \ol{G}^2} = 1 . 
\end{aligned}
\label{eqn:int_m0_m2_1}
\end{align}

Now, consider
\begin{align}
    \cJ^*(y) &\coloneqq \sqrt{\frac{\Delta(\delta)}{\delta}} \paren{\frac{m_2(y)}{m_0(y)} - 1} . \label{eqn:Jstar} 
\end{align}
We claim that $ \cJ^* $ satisfies \Cref{eqn:thr_to_opt,eqn:a0-gamma0-optthr} and therefore attains positive overlap. 
In fact, we claim that $ \cJ^* $ satisfies a stronger condition than \Cref{eqn:thr_to_opt} which is displayed below in conjunction with \Cref{eqn:a0-gamma0-optthr}: 
\begin{align}
\begin{aligned}
    \sqrt{\frac{\delta}{\Delta(\delta)}} &= \frac{1}{\expt{\ol{\Sigma}}} \expt{\paren{\frac{\delta}{\expt{\ol{\Sigma}}} \ol{G}^2 - 1} \cJ^*(\ol{Y})} \expt{\frac{\ol{\Sigma}^2}{\gamma^\circ - \expt{\cJ^*(\ol{Y})} \ol{\Sigma}}} , \\
    1 &= \frac{1}{\delta} \expt{\cJ^*(\ol{Y})^2} \expt{\paren{\frac{\ol{\Sigma}}{\gamma^\circ - \expt{\cJ^*(\ol{Y})} \ol{\Sigma}}}^2} , 
\end{aligned}
\label{eqn:optthr_constr}
\end{align}
where
\begin{align}
    1 = \frac{1}{\delta} \expt{\frac{\ol{\Sigma}}{\gamma^\circ - \expt{\cJ^*(\ol{Y})} \ol{\Sigma}}} . 
    \label{eqn:gamma0-opt}
\end{align}
Note that the first identity in \Cref{eqn:optthr_constr} implies \Cref{eqn:thr_to_opt} since $ \delta > \Delta(\delta) $ by \Cref{eqn:opt_thr_ineq}. 

Let us verify the validity of \Cref{eqn:optthr_constr}. 
By the construction of $ \cJ^* $ (see \Cref{eqn:Jstar}), 
\begin{align}
    \expt{\cJ^*(\ol{Y})}
    &= \int_{\supp(\ol{Y})} m_0(y) \cJ^*(y) \,\diff y
    = \sqrt{\frac{\Delta(\delta)}{\delta}} \int_{\supp(\ol{Y})} m_2(y) - m_0(y) \,\diff y
    = 0 , \label{eqn:E[J^*]=0} 
\end{align}
where the last equality follows from \Cref{eqn:int_m0_m2_1}. 
Using this in \Cref{eqn:gamma0-opt}, we can solve $ \gamma^\circ $ explicitly: 
\begin{align}
    \gamma^\circ &= \frac{\expt{\ol{\Sigma}}}{\delta} . \label{eqn:optthr-gamma0} 
\end{align}
Consequently, the first two identities of \Cref{eqn:optthr_constr} can be simplified as follows. 
First look at the first identity of \Cref{eqn:optthr_constr}. 
The right-hand side equals
\begin{align}
    &\phantom{=}~\frac{1}{\expt{\ol{\Sigma}}} \expt{\paren{\frac{\delta}{\expt{\ol{\Sigma}}} \ol{G}^2 - 1} \cJ^*(\ol{Y})} \expt{\frac{\ol{\Sigma}^2}{\gamma^\circ - \expt{\cJ^*(\ol{Y})} \ol{\Sigma}}} \notag \\
    &= \frac{\delta \expt{\ol{\Sigma}^2}}{\expt{\ol{\Sigma}}^2} \expt{\paren{\frac{\delta}{\expt{\ol{\Sigma}}} \ol{G}^2 - 1} \cJ^*(\ol{Y})} \label{eqn:optthr_eq1_1} \\
    &= \frac{\delta \expt{\ol{\Sigma}^2}}{\expt{\ol{\Sigma}}^2} \int_{\supp(\ol{Y})} (m_2(y) - m_0(y)) \cJ^*(y) \,\diff y \notag \\
    &= \sqrt{\Delta(\delta) \delta} \cdot \frac{\expt{\ol{\Sigma}^2}}{\expt{\ol{\Sigma}}^2} \int_{\supp(\ol{Y})} \frac{(m_2(y) - m_0(y))^2}{m_0(y)} \,\diff y . \label{eqn:optthr_eq1_2}
\end{align}
\Cref{eqn:optthr_eq1_1} is by \Cref{eqn:E[J^*]=0,eqn:optthr-gamma0}. 
\Cref{eqn:optthr_eq1_2} is by \Cref{eqn:Jstar}. 
Therefore, the first identity of \Cref{eqn:optthr_constr} is equivalent to: 
\begin{align}
    \Delta(\delta) &= \frac{\expt{\ol{\Sigma}}^2}{\expt{\ol{\Sigma}^2}} \paren{ \int_{\supp(\ol{Y})} \frac{(m_2(y) - m_0(y))^2}{m_0(y)} \,\diff y }^{-1} . \notag 
\end{align}
The right-hand side is the same as that of \Cref{eqn:opt_thr_ineq}, hence the first identity of \Cref{eqn:optthr_constr} indeed holds by the definition of $ \Delta(\delta) $. 

Next, we move to the second identity of \Cref{eqn:optthr_constr}.
Using \Cref{eqn:E[J^*]=0,eqn:optthr-gamma0} again, the right-hand side equals:
\begin{align}
    & \frac{1}{\delta} \expt{\cJ^*(\ol{Y})^2} \expt{\paren{\frac{\ol{\Sigma}}{\gamma^\circ - \expt{\cJ^*(\ol{Y})} \ol{\Sigma}}}^2} 
    = \frac{1}{\delta} \expt{\cJ^*(\ol{Y})^2} \frac{\expt{\ol{\Sigma}^2}}{(\gamma^\circ)^2} \notag \\
    &= \delta \frac{\expt{\ol{\Sigma}^2}}{\expt{\ol{\Sigma}}^2} \expt{\cJ^*(\ol{Y})^2} 
    = \Delta(\delta) \frac{\expt{\ol{\Sigma}^2}}{\expt{\ol{\Sigma}}^2} \expt{\paren{ \frac{m_2(\ol{Y})}{m_0(\ol{Y})} - 1 }^2} \notag \\
    &= \Delta(\delta) \frac{\expt{\ol{\Sigma}^2}}{\expt{\ol{\Sigma}}^2} \int_{\supp(\ol{Y})} m_0(y) \paren{\frac{m_2(y)}{m_0(y)} - 1}^2 \diff y 
    = 1 , \notag  
\end{align}
which verifies the second identity of \Cref{eqn:optthr_constr}. 
The second line uses the definition of $\cJ^*$ in \Cref{eqn:Jstar} and the last equality is by the definition of $ \Delta(\delta) $ (see the right-hand side of \Cref{eqn:opt_thr_ineq}). 

To complete the proof, it remains to verify that $ \cT^* $ satisfies \Cref{asmp:preprocessor}. 
Recalling \Cref{eqn:Jstar_Tstar,eqn:Jstar}, we have
\begin{align}
    \cT^*(y) &= \frac{\cJ^*(y)}{1 + \cJ^*(y)}
    = \frac{\sqrt{\frac{\Delta(\delta)}{\delta}} \paren{\frac{m_2(y)}{m_0(y)} - 1}}{1 + \sqrt{\frac{\Delta(\delta)}{\delta}} \paren{\frac{m_2(y)}{m_0(y)} - 1}}
    = 1 - \frac{1}{\sqrt{\frac{\Delta(\delta)}{\delta}} \frac{m_2(y)}{m_0(y)} + 1 - \sqrt{\frac{\Delta(\delta)}{\delta}}} . \label{eqn:opt_T_pf} 
\end{align}
By definitions, both $ m_2 $ and $ m_0 $ are non-negative functions. 
Therefore
\begin{align}
    \inf_{y\in\supp(\ol{Y})} \cT^*(y) &\ge 1 - \frac{1}{1 - \sqrt{\frac{\Delta(\delta)}{\delta}}}
    > -\infty , \label{eqn:inf_verify} 
\end{align}
where the last inequality holds since $ \delta > \Delta(\delta) $ by the assumption in \Cref{eqn:opt_thr_ineq}. 
Also, it trivially holds that 
\begin{align}
    \sup_{y\in\supp(\ol{Y})} \cT^*(y) &\le 1 < \infty . \label{eqn:sup_lb_verify}  
\end{align}

It is easy to see that $ \cT^*(y) > 0 $ if and only if $ m_2(y) > m_0(y) $.
We first claim that $ m_2 $ and $ m_0 $ are not identically equal. 
Otherwise, $ \Delta(\delta) $ (i.e., the right-hand side of \Cref{eqn:opt_thr_ineq}) is infinity and $ \delta $ satisfying \Cref{eqn:opt_thr_ineq} is also infinity, 
violating \Cref{asmp:proportional}. 
Moreover, by \Cref{eqn:int_m0_m2_1}, 
\begin{align}
    \int_{\supp(\ol{Y})} m_2(y) - m_0(y) \diff y &= 0 . \notag
\end{align}
It follows from the mean value theorem for definite integrals that there exists $ y\in\supp(\ol{Y}) $ such that $ m_2(y) > m_0(y) $ which implies 
\begin{align}
    \sup_{y\in\supp(\ol{Y})} \cT^*(y) &> 0 . \label{eqn:sup_ub_verify} 
\end{align}
Since $\cT^*$ is assumed to be pseudo-Lipschitz of finite order, putting \Cref{eqn:inf_verify,eqn:sup_lb_verify,eqn:sup_ub_verify} together verifies \Cref{asmp:preprocessor}.

Note that, by the arguments in \Cref{sec:remove_asmp_sigma_tech}, $ \cT^* $ does not need to satisfy \Cref{asmp:preprocessor_technical} to have positive limiting overlap. In fact, if \Cref{eqn:opt_thr_fp} holds and $ \cT^* $ does not have a point mass at the boundaries of its support (otherwise \Cref{asmp:preprocessor_technical} automatically holds), we can create such point masses via a perturbation. Now, the perturbed function satisfies \Cref{asmp:preprocessor_technical} and it has positive limiting overlap for all sufficiently small perturbations. Then, an application of  
the Davis--Kahan theorem shows that we can set the perturbation to $0$, and obtain the desired result for $\cT^*$.
This concludes the proof.

\section{Removing Assumptions \ref{asmp:sigma_technical} and \ref{asmp:preprocessor_technical}}
\label{sec:remove_asmp_sigma_tech}

We show that the conclusions of \Cref{thm:main} remain valid even if $\Sigma$ and/or $ \cT $ fail to satisfy \Cref{asmp:sigma_technical} and/or \ref{asmp:preprocessor_technical}. 
To do so, we create $\wt{\Sigma}, \wt{\cT} $ that closely approximate $\Sigma, \cT $ and satisfy \Cref{asmp:sigma_technical,asmp:preprocessor_technical}. 
\Cref{thm:main} then applies to $\wt{\Sigma}, \wt{\cT}$. We then show using a perturbation analysis that the same characterizations also hold for $\Sigma, \cT$ once the perturbation is sent to zero. 
The detailed proof is presented below where we assume that both \Cref{asmp:sigma_technical,asmp:preprocessor_technical} are violated. 
The proof when only one of them holds is analogous and is omitted. 

We first construct $ \wt{\Sigma} $. 
Note that if 
\begin{align}
    \prob{ \ol{\Sigma} = \inf\supp(\ol{\Sigma}) } > 0 , \quad
    \prob{ \ol{\Sigma} = \sup\supp(\ol{\Sigma}) } > 0 , \label{eqn:mass_at_edge_sigma}
\end{align}
then \Cref{asmp:sigma_technical} is automatically satisfied and one can take $ \wt{\Sigma} = \Sigma $. 
In what follows, we assume that both probabilities in \Cref{eqn:mass_at_edge_sigma} are zero.
(Again, the case where exactly one of the probabilities is zero can be handled verbatim and the details are omitted.)
Write the eigendecomposition of $\Sigma$ as
$    \Sigma = \sum_{i = 1}^d \lambda_i(\Sigma) v_i(\Sigma) v_i(\Sigma)^\top$ . 
By the convergence of the empirical spectral distribution of $\Sigma$ (see \Cref{asmp:sigma}), we have that for any sufficiently small $ \varsigma>0 $, there exists $ \xi>0 $ (depending on $\varsigma$) such that for every sufficiently large $d$, 
\begin{align}
    \frac{1}{d} \card{\brace{ i\in\{1,\ldots,  d\} : \lambda_i(\Sigma) \ge \paren{ \sqrt{\lambda_1(\Sigma)} - \xi }^2 }} &\in [\varsigma/2, \varsigma] , \notag \\
    \frac{1}{d} \card{\brace{ i\in\{1,\ldots,  d\} : \lambda_i(\Sigma) \le \paren{ \sqrt{\lambda_d(\Sigma)} + \xi }^2 }} &\in [\varsigma/2, \varsigma] . \notag 
\end{align}
Let $ \wt{\Sigma}\in\bbR^{d\times d} $ be the matrix obtained by truncating the spectrum of $\Sigma$:
\begin{align}
    \wt{\Sigma} &= \sum_{i = 1}^d \lambda_i(\wt{\Sigma}) v_i(\Sigma) v_i(\Sigma)^\top , \notag 
\end{align}
where
\begin{align}
    \lambda_i(\wt{\Sigma}) &= \begin{cases}
        \paren{ \sqrt{\lambda_1(\Sigma)} - \xi }^2 , & \lambda_i(\Sigma) \ge \paren{ \sqrt{\lambda_1(\Sigma)} - \xi }^2 \\
        \paren{ \sqrt{\lambda_d(\Sigma)} + \xi }^2 , & \lambda_i(\Sigma) \le \paren{ \sqrt{\lambda_d(\Sigma)} + \xi }^2 \\
        \lambda_i(\Sigma) , & \ow
    \end{cases} . \notag 
\end{align}
It is easy to check that $ \wt{\Sigma} $ still satisfies \Cref{asmp:sigma} if $\Sigma$ does. 
Moreover, upon truncation, the limiting spectral distribution of $ \wt{\Sigma} $ has positive mass on both the left and right edges and hence obviously satisfies \Cref{asmp:sigma_technical}. 

Let us then construct $ \wt{\cT} $. 
Clearly, if 
\begin{align}
    \prob{\cT(\ol{Y}) = \sup\supp(\cT(\ol{Y}))} &> 0 , \label{eqn:mass_at_edge} 
\end{align}
then \Cref{eqn:asmp_a_lim} is satisfied. 
We therefore assume that the above equation holds with equality. 
In this case, we truncate $ \cT $ slightly below its supremum to create $\wt{\cT}$ which satisfies \Cref{eqn:asmp_a_lim}. 
Specifically, for any $ \varsigma>0 $, there exists $ \xi>0 $ (depending on $ \varsigma $) such that
\begin{align}
    \prob{\cT(\ol{Y}) \in [\sup\supp(\cT(\ol{Y})) - \xi, \sup\supp(\cT(\ol{Y}))]} &\in [\varsigma/2, \varsigma] . \notag 
\end{align}
Define $ \wt{\cT} $ as
\begin{align}
    \wt{\cT}(y) &\coloneqq \min\brace{ \cT(y), \sup\supp(\cT(\ol{Y})) - \xi } . \label{eqn:opt_T_trunc} 
\end{align}
Note that $ \wt{\cT} $ depends on $ \varsigma $. 
Also, it satisfies \Cref{eqn:mass_at_edge} and therefore \Cref{eqn:asmp_a_lim}. 
It is easy to see that \Cref{asmp:preprocessor} will not be violated after the truncation.

Now the conclusions of \Cref{thm:main} hold for $ \wt{\Sigma}, \wt{\cT} $. 
In particular, $ \wt{a}^*, \wt{a}^\circ $ can be defined using \Cref{eqn:def_a_circ,eqn:def_a*} but with $ \wt{\cT} $ and the limiting spectral distribution of $ \wt{\Sigma} $. 
It then suffices to show that as long as $ \wt{a}^* > \wt{a}^\circ $, the difference between the spectral statistics under $ \Sigma,\cT $ and those under $ \wt{\Sigma}, \wt{\cT} $ is vanishing as $\varsigma\to0$. 
Let 
\begin{align}
    D \coloneqq \Sigma^{1/2} \wt{X}^\top T \wt{X} \Sigma^{1/2} , \quad 
    \wt{D} \coloneqq \wt{\Sigma}^{1/2} \wt{X}^\top \wt{T} \wt{X} \wt{\Sigma}^{1/2} , \notag 
\end{align}
where
\begin{align}
    T &\coloneqq \diag(\cT(y)) , \quad 
    \wt{T} \coloneqq \diag(\wt{\cT}(y)) . \notag
\end{align}
Then 
\begin{align}
    \normtwo{ D - \wt{D} } &= \normtwo{ \Sigma^{1/2} \wt{X}^\top T \wt{X} \Sigma^{1/2} - \wt{\Sigma}^{1/2} \wt{X}^\top \wt{T} \wt{X} \wt{\Sigma}^{1/2} } \notag \\
    &\le \normtwo{ \Sigma^{1/2} \wt{X}^\top T \wt{X} \Sigma^{1/2} - \wt{\Sigma}^{1/2} \wt{X}^\top T \wt{X} \Sigma^{1/2} } + \normtwo{ \wt{\Sigma}^{1/2} \wt{X}^\top T \wt{X} \Sigma^{1/2} - \wt{\Sigma}^{1/2} \wt{X}^\top \wt{T} \wt{X} \Sigma^{1/2} } \notag \\
    &\phantom{=}~ + \normtwo{ \wt{\Sigma}^{1/2} \wt{X}^\top \wt{T} \wt{X} \Sigma^{1/2} - \wt{\Sigma}^{1/2} \wt{X}^\top \wt{T} \wt{X} \wt{\Sigma}^{1/2} } \notag \\
    &\le \normtwo{\Sigma^{1/2} - \wt{\Sigma}^{1/2}} \normtwo{\wt{X}}^2 \normtwo{T} \normtwo{\Sigma^{1/2}} + \normtwo{\wt{\Sigma}^{1/2}} \normtwo{\wt{X}}^2 \normtwo{T - \wt{T}} \normtwo{\Sigma^{1/2}} \notag \\
    &\phantom{=}~ + \normtwo{\wt{\Sigma}^{1/2}} \normtwo{\wt{X}}^2 \normtwo{\wt{T}} \normtwo{\Sigma^{1/2} - \wt{\Sigma}^{1/2}} \notag \\
    &\le 2 \normtwo{\Sigma^{1/2} - \wt{\Sigma}^{1/2}} \normtwo{\wt{X}}^2 \normtwo{T} \normtwo{\Sigma^{1/2}} 
    + \normtwo{\Sigma^{1/2}}^2 \normtwo{\wt{X}}^2 \normtwo{T - \wt{T}} \notag \\
    &\le 2 \xi \paren{ 1 + 1/\sqrt{\delta} + 0.01 }^2 \paren{ \sup\supp(\cT(\ol{Y})) + 0.01 } \paren{ \supp(\ol{\Sigma}) + 0.01 } \notag \\
    &\phantom{=}~+ \paren{ \supp(\ol{\Sigma}) + 0.01 } \paren{ 1 + 1/\sqrt{\delta} + 0.01 }^2 \xi \notag \\
    &\le c_1 \xi , \label{eqn:DK_op_norm} 
\end{align}
where the bound on the penultimate line holds almost surely for every sufficiently large $d$, and $ c_1>0 $ in the last line is a constant independent of $d$. 
The $+0.01$ terms are to exclude deviations for small $d$. 
Furthermore, if $ \wt{a}^* > \wt{a}^\circ $, \Cref{thm:main} guarantees that there exists a constant $ c_2>0 $ such that for every sufficiently large $d$, with probability $1$, 
\begin{align}
    \lambda_1(\wt{D}) - \lambda_2(\wt{D}) \ge c_2 . 
    \label{eqn:DK_spec_gap}
\end{align}
Using \Cref{eqn:DK_op_norm,eqn:DK_spec_gap} in the Davis--Kahan theorem (\Cref{prop:davis_kahan}), we obtain
\begin{align}
    \min\brace{ \normtwo{v_1(D) - v_1(\wt{D})}, \normtwo{v_1(D) + v_1(\wt{D})} } &\le \frac{4\normtwo{ D - \wt{D} }}{\lambda_1(D) - \lambda_2(D)} 
    \le 4 c_1 \xi / c_2 , \notag 
\end{align}
which implies
\begin{align}
    \abs{ \abs{\inprod{v_1(D)}{\frac{\beta^*}{\sqrt{d}}}} - \abs{\inprod{v_1(\wt{D})}{\frac{\beta^*}{\sqrt{d}}}} } &= \min_{\sigma\in\{-1,1\}} \abs{ \inprod{v_1(D) - \sigma v_1(\wt{D})}{\frac{\beta^*}{\sqrt{d}}} } \notag \\
    &\le \min_{\sigma\in\{-1,1\}} \normtwo{ v_1(D) - v_1(\wt{D}) }
    \le 4 c_1 \xi / c_2 . \label{eqn:v1_v1_tilde}
\end{align}
By \Cref{thm:main}, the condition $ \wt{a}^* > \wt{a}^\circ $ also implies that the overlap between $ v_1(\wt{D}) $ and $ \beta^* $ converges in probability to $ \eta>0 $. 
Since $ \varsigma>0 $ (and therefore $ \xi $) can be made arbitrarily small, \Cref{eqn:v1_v1_tilde} then allows us to conclude that the overlap between $ v_1(D) $ and $ \beta^* $ also converges to $\eta$. 
This proves \Cref{eqn:main_thm_overlap} for $ D $. 

Using \Cref{eqn:DK_op_norm} and Weyl's inequality, we have for any $ i\in\{1, \ldots, d\} $, 
\begin{align}
    \abs{ \lambda_i(D) - \lambda_i(\wt{D}) }
    &\le \normtwo{D - \wt{D}} \le c_1 \xi , \notag 
\end{align}
which in particular establishes \Cref{eqn:main_thm_eigval} for $D$. 
This completes the proof.

\section{Properties of auxiliary functions and parameters}
\label{sec:properties}

\subsection{Existence and uniqueness of $ a^* $}
\label{sec:pf_properties_phi_psi}

Recall the functions $ \phi, \psi \colon (\sup\supp(\cT(\ol{Y})), \infty)\to\bbR $ defined in \Cref{eqn:def-phi-psi}.

\begin{proposition}[Existence of $a^*$]
\label{prop:exist_a*}
Let \Cref{asmp:preprocessor_technical} hold. 
Then, the equation $ \phi(a^*) = \zeta(a^*) $ has at least one solution in $ (\sup\supp(\cT(\ol{Y})), \infty) $. 
\end{proposition}

\begin{proof}
Recall that both $\phi$ and $\zeta$ are defined on $ (\sup\supp(\cT(\ol{Y})), \infty) $. 
It is not hard to see from \Cref{eqn:def-gamma-fn} that $ \gamma $ is a continuous function. 
Therefore $ \phi, \psi, \zeta $ are also continuous. 
We will show
\begin{align}
    \lim_{a\searrow\sup\supp(\cT(\ol{Y}))} \phi(a) 
    > \lim_{a\searrow\sup\supp(\cT(\ol{Y}))} \zeta(a) , \qquad 
    \lim_{a\nearrow\infty} \phi(a) 
    < \lim_{a\nearrow\infty} \zeta(a) . \label{eqn:left_right_lim}
\end{align}
Then by the intermediate value theorem, this immediately implies the result. 

We will explicitly evaluate the four limits. 
To this end, let us first study the limiting values of $\gamma(a)$ defined through \Cref{eqn:def-gamma-fn}. 

\paragraph{Limiting values of $\gamma$.}
By inspecting the defining equation, it is clear that 
\begin{align}
    \lim_{a\to\infty} \frac{1}{\delta} \expt{\frac{\ol{\Sigma}}{\gamma - \expt{\cF_a(\ol{Y})} \ol{\Sigma}}}
    &= \frac{\expt{\ol{\Sigma}}}{\delta\gamma} , \notag 
\end{align}
and hence
\begin{align}
    \lim_{a\to\infty} \gamma(a) &= \frac{\expt{\ol{\Sigma}}}{\delta} , \label{eqn:gamma-right} 
\end{align}
which is positive and finite. 
We also claim that 
\begin{align}
    \lim_{a\searrow\sup\supp(\cT(\ol{Y})) } \gamma(a) &= \infty . \label{eqn:gamma-left} 
\end{align}
Otherwise, for any finite $ \gamma $, by \emph{(d)} in \Cref{eqn:asmp_a_lim}, 
\begin{align}
    \lim_{a\searrow\sup\supp(\cT(\ol{Y})) } \frac{1}{\delta} \expt{\frac{\ol{\Sigma}}{\gamma - \expt{\cF_a(\ol{Y})} \ol{\Sigma}}} &= 0 , \notag 
\end{align}
which violates \Cref{eqn:def-gamma-fn}. 
The possibility of $ \lim\limits_{a\searrow\sup\supp(\cT(\ol{Y})) } \gamma(a) = -\infty $ can be similarly excluded. 

\paragraph{Limiting values of $ \phi $.}
We claim that
\begin{align}
    \lim_{a\searrow\sup\supp(\cT(\ol{Y})) } \phi(a) &= \infty , \quad 
    \lim_{a\to\infty} \phi(a) = \delta \expt{\ol{G}^2 \cT(\ol{Y})} \frac{\expt{\ol{\Sigma}^2}}{\expt{\ol{\Sigma}}^2} < \infty . \label{eqn:phi_boundary} 
\end{align}
The limit towards the right boundary of the domain is easy to verify: 
\begin{align}
    \lim_{a\to\infty} \phi(a) 
    &= \lim_{a\to\infty} \frac{1}{\expt{\ol{\Sigma}}} \expt{\ol{G}^2 \frac{\cT(\ol{Y})}{1 - \cT(\ol{Y}) / a}} \expt{\frac{\ol{\Sigma}^2}{\gamma(a) - \expt{\cF_a(\ol{Y})} \ol{\Sigma}}} \notag \\
    &= \frac{1}{\expt{\ol{\Sigma}}} \expt{\ol{G}^2 \cT(\ol{Y})} \expt{\frac{\ol{\Sigma}^2}{\expt{\ol{\Sigma}} / \delta}} \notag \\
    &= \delta \expt{\ol{G}^2 \cT(\ol{Y})} \frac{\expt{\ol{\Sigma}^2}}{\expt{\ol{\Sigma}}^2} , \notag 
\end{align}
where we use \Cref{eqn:gamma-left} in the second equality. 
To show the first equality in \Cref{eqn:phi_boundary}, let us start by observing that for any $ a > \sup\supp(\cT(\ol{Y})) $, 
\begin{align}
    0 < \expt{\frac{1}{\gamma(a) - \expt{\cF_a(\ol{Y})} \ol{\Sigma}}}
    &\le \frac{1}{\inf\supp(\ol{\Sigma})} \expt{\frac{\ol{\Sigma}}{\gamma(a) - \expt{\cF_a(\ol{Y})} \ol{\Sigma}}}
    = \frac{\delta}{\inf\supp(\ol{\Sigma})} . \label{eqn:const_ub} 
\end{align}
The second inequality is valid since $ \inf\supp(\ol{\Sigma}) > 0 $ by \Cref{asmp:sigma} and hence $ \frac{\ol{\Sigma}}{\inf\supp(\ol{\Sigma})} \ge 1 $ almost surely. 
The last equality is by the definition of $ \gamma(\cdot) $ (see \Cref{eqn:def-gamma-fn}). 
On the other hand, a simple application of the Cauchy--Schwarz inequality yields: 
\begin{align}
    \delta^2 &= \expt{\frac{\ol{\Sigma}}{\gamma(a) - \expt{\cF_a(\ol{Y})} \ol{\Sigma}}}^2 \notag\\
    &\le \expt{\frac{\ol{\Sigma}}{\paren{ \gamma(a) - \expt{\cF_a(\ol{Y})} \ol{\Sigma} }^{1/2}} \cdot \frac{1}{\paren{ \gamma(a) - \expt{\cF_a(\ol{Y})} \ol{\Sigma} }^{1/2}}}^2 \notag \\
    &\le \expt{\frac{\ol{\Sigma}^2}{\gamma(a) - \expt{\cF_a(\ol{Y})} \ol{\Sigma}}} \expt{\frac{1}{\gamma(a) - \expt{\cF_a(\ol{Y})} \ol{\Sigma}}} . \notag 
\end{align}
Rearranging and using \Cref{eqn:const_ub} gives:
\begin{align}
    \expt{\frac{\ol{\Sigma}^2}{\gamma(a) - \expt{\cF_a(\ol{Y})} \ol{\Sigma}}}
    &\ge \frac{\delta^2}{\expt{\frac{1}{\gamma(a) - \expt{\cF_a(\ol{Y})} \ol{\Sigma}}}}
    \ge \delta \cdot \inf\supp(\ol{\Sigma}) , \notag 
\end{align}
the right-hand side of which is a strictly positive lower bound independent of $a$. 
From here, we conclude
\begin{align}
    \lim_{a\searrow\sup\supp(\cT(\ol{Y})) } \phi(a)
    &= \lim_{a\searrow\sup\supp(\cT(\ol{Y})) } \frac{a}{\expt{\ol{\Sigma}} } \expt{ \ol{G}^2 \cF_a(\ol{Y})} \expt{\frac{\ol{\Sigma}^2}{\gamma(a) - \expt{\cF_a(\ol{Y})} \ol{\Sigma}}}
    = \infty , \notag 
\end{align}
since the middle term converges to $\infty$ by \emph{(e)} in \Cref{eqn:asmp_a_lim} and the remaining terms are lower bounded by some positive constant as $a\searrow\sup\supp(\cT(\ol{Y}))$. 

\paragraph{Limiting values of $\zeta$.}
By definition, 
\begin{align}
    \lim_{a\searrow\sup\supp(\cT(\ol{Y}))} \zeta(a)
    &= \zeta(a^\circ) = \psi(a^\circ) < \infty . \label{eqn:psi-left} 
\end{align}
Using \Cref{eqn:gamma-right}, we obtain
\begin{align}
    \lim_{a\to\infty} \zeta(a)
    = \lim_{a\to\infty} \psi(a)
    = \lim_{a\to\infty} a \gamma(a)
    = \infty . \label{eqn:psi-right} 
\end{align}

Finally, combining \Cref{eqn:phi_boundary,eqn:psi-left,eqn:psi-right} gives \Cref{eqn:left_right_lim} which completes the proof of the proposition. 
\end{proof}

\begin{proposition}[Monotonicity of $ \phi $]
\label{prop:properties_phi_psi}
Let \Cref{asmp:preprocessor} hold. 
    Suppose 
    \begin{align}
        \inf_{y\in\supp(\ol{Y})} \cT(y) \ge 0 . \label{eqn:positive_T} 
    \end{align}
    Then, the function $ \phi $ is strictly decreasing. 

\end{proposition}

\begin{proof}
We show that $ \phi $ is strictly decreasing by proving $ \phi'<0 $. 
Let us start by computing $ \phi' $. 
Recall
\begin{align}
    \expt{\ol{\Sigma}} \phi(a) &= \expt{ \ol{G}^2 a \cF_a(\ol{Y})} \expt{\frac{\ol{\Sigma}^2}{\gamma(a) - \expt{\cF_a(\ol{Y})} \ol{\Sigma}}} . \notag 
\end{align}
Using the chain rule, we obtain:
\begin{align}
   & \expt{\ol{\Sigma}} \phi'(a)
    = - \expt{\ol{G}^2 \cF_a(\ol{Y})^2} \expt{\frac{\ol{\Sigma}^2}{\gamma(a) - \expt{\cF_a(\ol{Y})} \ol{\Sigma}}} \notag \\
    &\phantom{=}~- \expt{ \ol{G}^2 a \cF_a(\ol{Y})} \expt{\frac{\ol{\Sigma}^2}{\paren{\gamma(a) - \expt{\cF_a(\ol{Y})} \ol{\Sigma}}^2} \paren{\gamma'(a) + \expt{\frac{\cT(\ol{Y})}{\paren{a - \cT(\ol{Y})}^2}} \ol{\Sigma}}} . \label{eqn:phi'_TBC} 
\end{align}
The derivative of $\gamma$ can be accessed via the implicit function theorem. 
Let
\begin{align}
    H(a, \gamma) &= \frac{1}{\delta} \expt{\frac{\ol{\Sigma}}{\gamma - \expt{\cF_a(\ol{Y})}\ol{\Sigma}}} - 1 . \notag
\end{align}
Recalling \Cref{eqn:def-gamma-fn}, we see that $ \gamma(a) $ is the solution $\gamma$ to  $ H(a, \gamma) = 0 $. 
We have
\begin{align}
    \frac{\partial}{\partial a} H(a, \gamma) 
    &= \frac{1}{\delta} \expt{\frac{-\ol{\Sigma}}{\paren{ \gamma - \expt{\cF_a(\ol{Y})} \ol{\Sigma} }^2} \cdot (-\ol{\Sigma}) \cdot \expt{\frac{-\cT(\ol{Y})}{(a - \cT(\ol{Y}))^2}} } \notag \\
    &= -\frac{1}{\delta} \expt{\frac{\cT(\ol{Y})}{(a - \cT(\ol{Y}))^2}} \expt{\frac{\ol{\Sigma}^2}{\paren{ \gamma - \expt{\cF_a(\ol{Y})} \ol{\Sigma} }^2}} , \notag 
\end{align}
and 
\begin{align}
    \frac{\partial}{\partial \gamma} H(a, \gamma) 
    &= - \frac{1}{\delta} \expt{\frac{\ol{\Sigma}}{\paren{ \gamma - \expt{\cF_a(\ol{Y})}\ol{\Sigma} }^2}} . \notag 
\end{align}
By the implicit function theorem, 
\begin{align}
    \frac{\diff}{\diff a} \gamma(a) &= - \frac{\frac{\partial}{\partial a} H(a, \gamma(a))}{\frac{\partial}{\partial \gamma} H(a, \gamma(a))}
    = - \frac{\expt{\frac{\cT(\ol{Y})}{(a - \cT(\ol{Y}))^2}} \expt{\frac{\ol{\Sigma}^2}{\paren{ \gamma(a) - \expt{\cF_a(\ol{Y})} \ol{\Sigma} }^2}}}{\expt{\frac{\ol{\Sigma}}{\paren{ \gamma(a) - \expt{\cF_a(\ol{Y})}\ol{\Sigma} }^2}}} . \label{eqn:gamma'} 
\end{align}
Using this, we simplify the second term of \Cref{eqn:phi'_TBC}:
\begin{align}
    &\phantom{=}~ - \expt{ \ol{G}^2 a \cF_a(\ol{Y})} \expt{\frac{\ol{\Sigma}^2}{\paren{\gamma(a) - \expt{\cF_a(\ol{Y})} \ol{\Sigma}}^2} \paren{\gamma'(a) + \expt{\frac{\cT(\ol{Y})}{\paren{a - \cT(\ol{Y})}^2}} \ol{\Sigma}}} \notag \\
    &= - \expt{ \ol{G}^2 a \cF_a(\ol{Y})} \expt{\frac{\ol{\Sigma}^2}{\paren{\gamma(a) - \expt{\cF_a(\ol{Y})} \ol{\Sigma}}^2} } \gamma'(a) \notag \\
    &\phantom{=}~ - \expt{ \ol{G}^2 a \cF_a(\ol{Y})} \expt{\frac{\cT(\ol{Y})}{\paren{a - \cT(\ol{Y})}^2}} \expt{\frac{\ol{\Sigma}^3}{\paren{\gamma(a) - \expt{\cF_a(\ol{Y})} \ol{\Sigma}}^2} } \notag \\
    &= \expt{ \ol{G}^2 a \cF_a(\ol{Y})} \expt{\frac{\ol{\Sigma}^2}{\paren{\gamma(a) - \expt{\cF_a(\ol{Y})} \ol{\Sigma}}^2} }^2 \frac{\expt{\frac{\cT(\ol{Y})}{(a - \cT(\ol{Y}))^2}} }{\expt{\frac{\ol{\Sigma}}{\paren{ \gamma(a) - \expt{\cF_a(\ol{Y})}\ol{\Sigma} }^2}}} \notag \\
    &\phantom{=}~- \expt{ \ol{G}^2 a \cF_a(\ol{Y})} \expt{\frac{\cT(\ol{Y})}{\paren{a - \cT(\ol{Y})}^2}} \expt{\frac{\ol{\Sigma}^3}{\paren{\gamma(a) - \expt{\cF_a(\ol{Y})} \ol{\Sigma}}^2} } . \label{eqn:phi'_second}
\end{align}
Let us argue that the right-hand side is negative. 
First note that since \emph{(i)} $ a > \sup\supp(\cT(\ol{Y})) > 0 $, \emph{(ii)} $ \inf\supp(\cT(\ol{Y})) \ge 0 $ by \Cref{eqn:positive_T}, \emph{(iii)} $\cT(\ol{Y})$ is not almost surely zero by \Cref{asmp:preprocessor}, the common factors are positive: 
\begin{align}
    \expt{ \ol{G}^2 a \cF_a(\ol{Y})} \expt{\frac{\cT(\ol{Y})}{(a - \cT(\ol{Y}))^2}} > 0 . \label{eqn:phi'-second-common-positive} 
\end{align}
Then we apply the Cauchy--Schwarz inequality to obtain:
\begin{align}
    \expt{\frac{\ol{\Sigma}^2}{\paren{\gamma(a) - \expt{\cF_a(\ol{Y})} \ol{\Sigma}}^2} }^2
    &= \expt{\frac{\ol{\Sigma}^{1/2}}{\gamma(a) - \expt{\cF_a(\ol{Y})} \ol{\Sigma}} \cdot \frac{\ol{\Sigma}^{3/2}}{\gamma(a) - \expt{\cF_a(\ol{Y})} \ol{\Sigma}} }^2 \label{eqn:cs} \\
    &\hspace{-4em}\le \expt{\frac{\ol{\Sigma}}{\paren{ \gamma(a) - \expt{\cF_a(\ol{Y})}\ol{\Sigma} }^2}}
    \expt{\frac{\ol{\Sigma}^3}{\paren{\gamma(a) - \expt{\cF_a(\ol{Y})} \ol{\Sigma}}^2} } . \label{eqn:phi'-second-cs}
\end{align}
\Cref{eqn:cs} is valid since $ \ol{\Sigma} $ is positive and $ \gamma(a) > s(a) $. 
\Cref{eqn:phi'-second-common-positive,eqn:phi'-second-cs} jointly imply that the right-hand side of \Cref{eqn:phi'_second}, i.e., the second term of \Cref{eqn:phi'_TBC}, is non-positive, as claimed. 
Moreover, the first term of \Cref{eqn:phi'_TBC} is strictly negative. 
We therefore conclude that $ \phi'(a)<0 $ for any $ a > \sup\supp(\cT(\ol{Y})) $. 
\end{proof}

\begin{remark}[Monotonicity of $\phi$]
\label{rk:mono-phi}
The monotonicity property of $\phi$ relies on the non-negativity of $\cT$ in \Cref{eqn:positive_T}. 
We believe that this assumption can be relaxed. 
In fact, numerical evidence suggests that $\phi$ is monotone: we report in \Cref{fig:phi} that in the setting of noiseless phase retrieval $q(g, \eps) = \abs{g}$ with optimal preprocessing function $ \cT(y) = \max\brace{ 1 - \frac{1}{\delta y^2} , -10 } $ (where $ \delta = 0.1 $), the function $\phi$ is strictly decreasing and convex in $ (1, \infty) $ (note that $ \sup\supp(\cT(\ol{Y})) = 1 $) when $\ol{\Sigma}$ is Toeplitz with $\rho = 0.9$ or circulant with $ c_0 = 1, c_1 = 0.1, \ell = 17 $. 
Note that the function $\cT$ here is not everywhere non-negative. 
\begin{figure}[htbp]
    \centering
    \begin{subfigure}{0.4\linewidth}
        \centering
        \includegraphics[width = \linewidth]{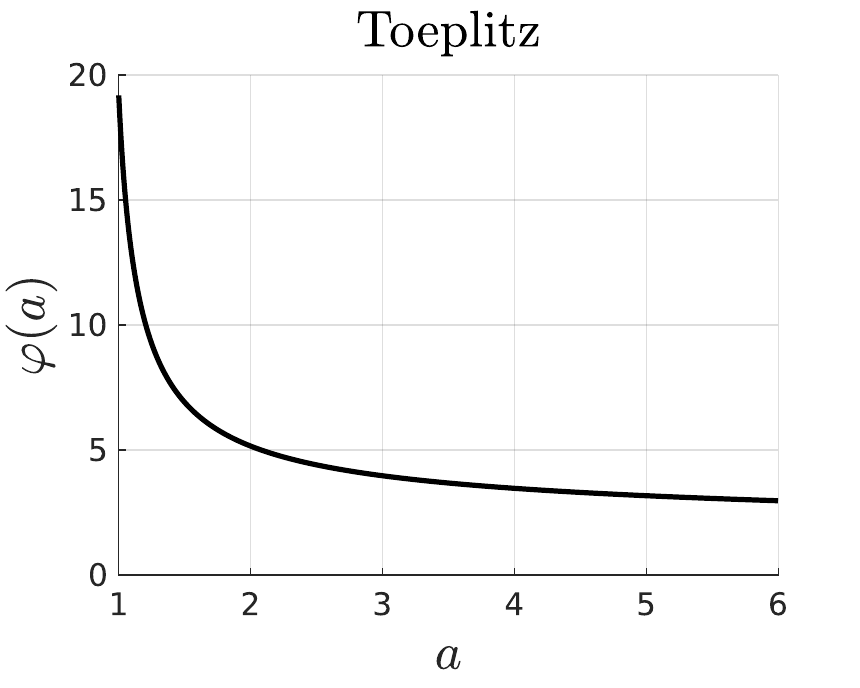}
        \label{fig:phi_toeplitz}
    \end{subfigure}
    \begin{subfigure}{0.4\linewidth}
        \centering
        \includegraphics[width = \linewidth]{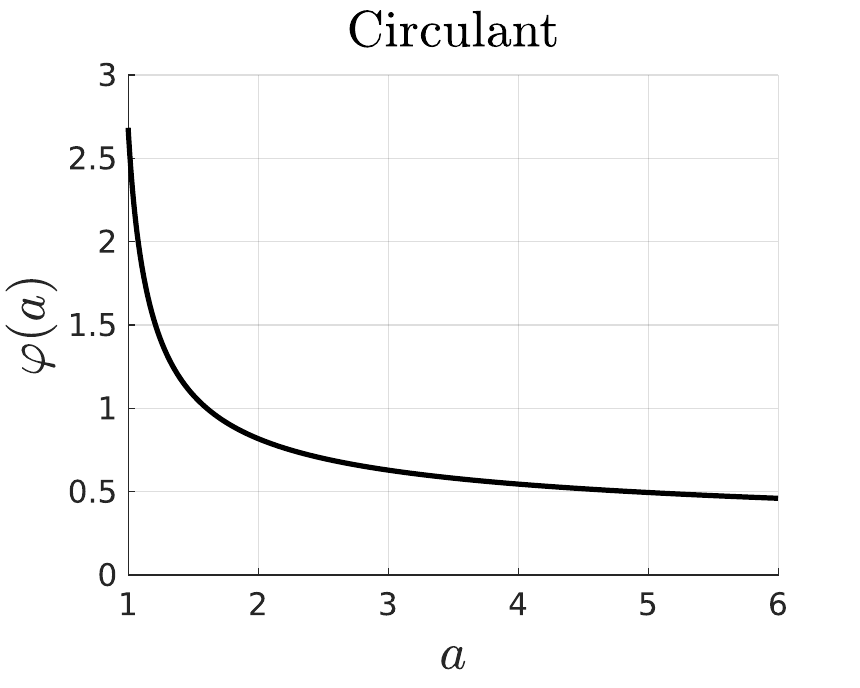}
        \label{fig:phi_circulant}
    \end{subfigure}
    \caption{Plots of the function $\phi$ defined in \Cref{eqn:def-phi-psi} with parameters specified in \Cref{rk:mono-phi}.
   }
    \label{fig:phi}
\end{figure}
\end{remark}

\begin{proposition}[Uniqueness of $ a^* $]
\label{prop:unique}
Let \Cref{asmp:preprocessor} hold. 
    Suppose that $\phi$ is strictly decreasing.
    Then, $ \phi(a^*) = \zeta(a^*) $ has a unique solution in $ (\sup\supp(\cT(\ol{Y})), \infty) $. 
\end{proposition}

\begin{proof}
The uniqueness of $ a^* $ follows from several  properties that have been proved for  $ \phi $ and $ \zeta $. 
Recall the assumption that $ \phi $ is strictly decreasing and that $ \zeta $ is non-decreasing by \Cref{lem:a0=sup}. 
Furthermore, from the proof of \Cref{prop:exist_a*} (in particular \Cref{eqn:phi_boundary,eqn:psi-left,eqn:psi-right}), we know that in the interval $ (\sup\supp(\cT(\ol{Y})), \infty) $, $\phi$ strictly decreases from $\infty$ to a finite constant, whereas $ \zeta $ increases from a finite constant to $\infty$. 
By the intermediate value theorem, the solution to $ \phi(a^*) = \zeta(a^*) $ must exist and is unique. 
\end{proof}

\subsection{Equivalent definitions of $ a^\circ, a^* $ and equivalent description of $\sup\supp(\ol{\mu}_{\wh{D}})$}
\label{sec:pf-equiv-def-a-gamma}

Let $\domain\subset\bbR^2$ be the domain on which the potential solutions to various self-consistent equations of interest are to be considered: 
\begin{align}
    \domain &\coloneqq \brace{(a, \gamma) : a > \sup\supp(\cT(\ol{Y})),\, \gamma > s(a) } , \notag 
\end{align}
where $s(a)$ is defined in \Cref{eqn:def-s(a)}. 

\begin{proposition}[Equivalent definitions of $ a^\circ, a^* $]
\label{prop:equiv-def-a-gamma}
\leavevmode
\begin{itemize}
    \item In the domain $ \domain $, the unique solution $ (a^\circ, \gamma^\circ) $ to 
    \begin{align}
    \begin{aligned}
    1 &= \frac{1}{\delta} \expt{\cF_{a^\circ}(\ol{Y})^2} \expt{\paren{\frac{\ol{\Sigma}}{\gamma^\circ - \expt{\cF_{a^\circ}(\ol{Y})} \ol{\Sigma}}}^2}, 
    \quad 
  1  = \frac{1}{\delta} \expt{\frac{\ol{\Sigma}}{\gamma^\circ - \expt{\cF_{a^\circ}(\ol{Y})} \ol{\Sigma}}} 
    \end{aligned}
    \label{eqn:bulk-selfcons-new} 
    \end{align}
    is the same as the unique solution to the following equations:
    \begin{align}
        \psi'(a^\circ) &= 0 , \quad 
        \gamma^\circ = \gamma(a^\circ) . \label{eqn:a0-gamma0-alt-def}
    \end{align} 
    \item 
    Let $ (a^*, \gamma^*) $ be the solution in $\domain$ to 
    \begin{align}
        \zeta(a^*) &= \phi(a^*) , \quad 
        \gamma^* = \gamma(a^*) , \label{eqn:a-gamma-alt-def} 
    \end{align}
    such that $a^*$ is the largest among all solutions. 
    If $ a^* > a^\circ $, then $ (a^*, \gamma^*) $ is also a solution to \Cref{eqn:fp-a-gamma}. 
\end{itemize}
\end{proposition}

\begin{proof}
\leavevmode
We start by showing the equivalence between \Cref{eqn:a0-gamma0-alt-def,eqn:bulk-selfcons-new}.
We will argue that $\psi'(a) = 0$ if and only if \Cref{eqn:bulk-selfcons-new} holds. 
The derivative of $\psi'$ is
\begin{align}
    \psi'(a) &= \gamma(a) + a \gamma'(a) 
    = \gamma(a) - a \cdot \frac{\expt{\frac{\cT(\ol{Y})}{(a - \cT(\ol{Y}))^2}} \expt{\frac{\ol{\Sigma}^2}{\paren{ \gamma(a) - \expt{\cF_a(\ol{Y})} \ol{\Sigma} }^2}}}{\expt{\frac{\ol{\Sigma}}{\paren{ \gamma(a) - \expt{\cF_a(\ol{Y})}\ol{\Sigma} }^2}}} , \label{eqn:psi'} 
\end{align}
where the formula for $ \gamma' $ has been derived in \Cref{eqn:gamma'}. 
Using the above expression and rearranging terms, we can write the equation $\psi'(a) = 0$ as
\begin{align}
    \expt{\frac{\gamma(a) \ol{\Sigma}}{\paren{ \gamma(a) - \expt{\cF_a(\ol{Y})}\ol{\Sigma} }^2}} &= \expt{\frac{a \cT(\ol{Y})}{(a - \cT(\ol{Y}))^2}} \expt{\frac{\ol{\Sigma}^2}{\paren{ \gamma(a) - \expt{\cF_a(\ol{Y})} \ol{\Sigma} }^2}} . \label{eqn:a0-gamma0-todo}
\end{align}
We rewrite the first two terms in the above equation in the following way:
\begin{align}
\begin{aligned}
    \expt{\frac{\gamma(a) \ol{\Sigma}}{\paren{ \gamma(a) - \expt{\cF_a(\ol{Y})}\ol{\Sigma} }^2}}
    &= \expt{\frac{ \ol{\Sigma}}{ \gamma(a) - \expt{\cF_a(\ol{Y})} \ol{\Sigma} }} \\
    &\phantom{=}~+ \expt{\frac{ \ol{\Sigma}^2}{\paren{ \gamma(a) - \expt{\cF_a(\ol{Y})}\ol{\Sigma} }^2}} \expt{\cF_a(\ol{Y})} , \\
    \expt{\frac{a \cT(\ol{Y})}{(a - \cT(\ol{Y}))^2}}
    &= \expt{\cF_a(\ol{Y})} + \expt{\cF_a(\ol{Y})^2} .  
\end{aligned}
\label{eqn:rewrite-fraction}
\end{align}
Using the right-hand sides above in place of the left-hand sides in \Cref{eqn:a0-gamma0-todo}, we see that the term $ \expt{\frac{ \ol{\Sigma}^2}{\paren{ \gamma(a) - \expt{\cF_a(\ol{Y})}\ol{\Sigma} }^2}} \expt{\cF_a(\ol{Y})} $ cancels on both sides and \Cref{eqn:a0-gamma0-todo} becomes
\begin{align}
    \expt{\frac{ \ol{\Sigma}}{ \gamma(a) - \expt{\cF_a(\ol{Y})} \ol{\Sigma} }}
    &= \expt{\cF_a(\ol{Y})^2} \expt{\frac{\ol{\Sigma}^2}{\paren{ \gamma(a) - \expt{\cF_a(\ol{Y})} \ol{\Sigma} }^2}} . \notag 
\end{align}
The left-hand side equals $\delta$ since $ \gamma(a) $ satisfies \Cref{eqn:def-gamma-fn}. 
Therefore the above equation matches \Cref{eqn:bulk-selfcons-new}. 

Next, assuming that \Cref{eqn:a-gamma-alt-def} holds, we verify \Cref{eqn:fp-a-gamma}. 
For any $a > a^\circ$, $ \zeta(a) = \psi(a) $, hence \Cref{eqn:a-gamma-alt-def} can be written as 
\begin{align}
    \frac{1}{\expt{\ol{\Sigma}} } \expt{ \ol{G}^2 \cF_a(\ol{Y})} \expt{\frac{\ol{\Sigma}^2}{\gamma(a) - \expt{\cF_a(\ol{Y})} \ol{\Sigma}}}
    &= \gamma(a) , \notag
\end{align}
or equivalently, 
\begin{align}
    \frac{1}{\expt{\ol{\Sigma}}} \expt{ \frac{\delta}{\expt{\ol{\Sigma}}} \ol{G}^2 \cF_a(\ol{Y})} \expt{\frac{\ol{\Sigma}^2}{\gamma(a) - \expt{\cF_a(\ol{Y})} \ol{\Sigma}}}
    &= \frac{\delta \gamma(a)}{\expt{\ol{\Sigma}}} . \notag 
\end{align}
To show that the above equation is the same as \Cref{eqn:fp-a-gamma}, it suffices to verify
\begin{align}
    \frac{\delta \gamma(a)}{\expt{\ol{\Sigma}}}
    &= \frac{1}{\expt{\ol{\Sigma}}} \expt{\cF_a(\ol{Y})} \expt{\frac{\ol{\Sigma}^2}{\gamma(a) - \expt{\cF_a(\ol{Y})} \ol{\Sigma}}}
    + 1 . \label{eqn:a-alt-todo}  
\end{align}
We rewrite the first term on the right-hand side as
\begin{align}
    &\phantom{=}~ \frac{1}{\expt{\ol{\Sigma}}} \expt{\cF_a(\ol{Y})} \expt{\frac{\ol{\Sigma}^2}{\gamma(a) - \expt{\cF_a(\ol{Y})} \ol{\Sigma}}} \notag \\
    &= \frac{1}{\expt{\ol{\Sigma}}} \frac{1}{\expt{\cF_a(\ol{Y})}} \paren{ \expt{\frac{\expt{\cF_a(\ol{Y})}^2 \ol{\Sigma}^2 - \gamma(a) \expt{\cF_a(\ol{Y})} \ol{\Sigma} }{\gamma(a) - \expt{\cF_a(\ol{Y})} \ol{\Sigma}}} + \expt{\frac{\gamma(a) \expt{\cF_a(\ol{Y})} \ol{\Sigma}}{\gamma(a) - \expt{\cF_a(\ol{Y})} \ol{\Sigma}}} } \notag \\
    &= \frac{1}{\expt{\ol{\Sigma}}} \frac{1}{\expt{\cF_a(\ol{Y})}} \paren{ - \expt{\expt{\cF_a(\ol{Y})} \ol{\Sigma}} + \gamma(a) \expt{\cF_a(\ol{Y})} \expt{\frac{ \ol{\Sigma}}{\gamma(a) - \expt{\cF_a(\ol{Y})} \ol{\Sigma}}} } \notag \\
    &= \frac{1}{\expt{\ol{\Sigma}}} \paren{ \gamma(a) \expt{\frac{ \ol{\Sigma}}{\gamma(a) - \expt{\cF_a(\ol{Y})} \ol{\Sigma}}} - \expt{\ol{\Sigma}} } \notag \\
    &= \frac{\gamma(a)}{\expt{\ol{\Sigma}}} \expt{\frac{ \ol{\Sigma}}{\gamma(a) - \expt{\cF_a(\ol{Y})} \ol{\Sigma}}} - 1 . \notag 
\end{align}
Noting that $ \gamma(a) $ satisfies \Cref{eqn:def-gamma-fn}, we further obtain
\begin{align}
    \frac{1}{\expt{\ol{\Sigma}}} \expt{\cF_a(\ol{Y})} \expt{\frac{\ol{\Sigma}^2}{\gamma(a) - \expt{\cF_a(\ol{Y})} \ol{\Sigma}}} 
    &= \frac{\delta \gamma(a)}{\expt{\ol{\Sigma}}} - 1 . \notag 
\end{align}
This then implies \Cref{eqn:a-alt-todo} and hence \Cref{eqn:fp-a-gamma}. 
\end{proof}

Finally, we derive an alternative form of \Cref{eqn:bulk-selfcons} in terms of $ a^\circ, \gamma^\circ $ defined through a pair of self-consistent equations. The proof follows from verifying that $ \psi'(a^\circ) = 0 $ is algebraically equivalent to \Cref{eqn:bulk-selfcons-new}, as shown in \Cref{prop:equiv-def-a-gamma} above. 

\begin{lemma}
\label{lem:bulk-a0-gamma0}
The description of $ \sup\supp(\ol{\mu}_{\wh{D}}) $ in \Cref{lem:bulk_nonpositive} is equivalent to
    $ \sup\supp(\ol{\mu}_{\wh{D}}) = a^\circ \gamma^\circ $ where $ (a^\circ, \gamma^\circ) \in \domain $ 
    solves \Cref{eqn:bulk-selfcons-new}, 
    and $ a^\circ $ is the largest among all such solutions.
\end{lemma}

\subsection{Alternative formulations of $ a^* > a^\circ $}
\label{sec:alt-above-threshold}

The following proposition is a direct consequence of the monotonicity properties of $ \psi,\phi $ (see \Cref{prop:properties_phi_psi} and \Cref{lem:a0=sup}). 

\begin{proposition}
\label{prop:equiv-threshold}
The following conditions are equivalent. 
\begin{enumerate}
    \item \label{itm:threshold1} $ a^* > a^\circ $; 

    \item \label{itm:threshold3} $ \zeta(a^*) > \zeta(a^\circ) $; 

    \item \label{itm:threshold4} $ \psi'(a^*) > 0 $, or more explicitly 
    \begin{align}
        1 &> \frac{1}{\delta} \expt{\cF_{a^*}(\ol{Y})^2} \expt{\frac{\ol{\Sigma}^2}{\paren{ \gamma^* - \expt{\cF_{a^*}(\ol{Y})} \ol{\Sigma} }^2}} , \label{eqn:threshold_1>x2} 
    \end{align} 
    i.e., $ 1>w_2 $ by recalling the definition of $ w_2 $ in \Cref{eqn:def_x2_main};

    \item \label{itm:threshold2} If the function $ \phi\colon(\sup\supp(\cT(\ol{Y})), \infty) \to \bbR $ defined in \Cref{eqn:def-phi-psi} is strictly decreasing, the above conditions are further equivalent to $ \psi(a^\circ) < \phi(a^\circ) $, or more explicitly 
    \begin{align}
        1 &< \frac{1}{\expt{\ol{\Sigma}}} \expt{\paren{\frac{\delta}{\expt{\ol{\Sigma}}} \ol{G}^2 - 1} \cF_{a^\circ}(\ol{Y})} \expt{\frac{\ol{\Sigma}^2}{\gamma^\circ - \expt{\cF_{a^\circ}(\ol{Y})} \ol{\Sigma}}} . \label{eqn:threshold}
    \end{align} 
\end{enumerate}
\end{proposition}


\section{Proof of Lemma \ref{lem:bulk_nonpositive}}
\label{sec:bulk_pf}
Recall from \Cref{eqn:def_D_hat} the definition of $\wh{D}\in\R^{d\times d}$:
\begin{equation}
	\wh{D} =\Sigma^{1/2} \wh{X}^\top T \wh{X} \Sigma^{1/2}.\notag
\end{equation}
We already know that both $\lambda_{1}(\wh{D})$ and $\lambda_{3}(\wh{D})$ converge to the upper edge $\lambda^{\circ}=\sup\supp\ol{\mu}_{\wh{D}}$ of the limiting spectrum (see \Cref{lem:lambda1,lem:lambda3}). 
The main goal of this section is to prove the characterization of the upper edge $\lambda^{\circ}$ in \Cref{lem:bulk_nonpositive}. We deduce \Cref{lem:bulk_nonpositive} from the following lemma. We present the proofs of \Cref{lem:bulk_nonpositive,lem:a0=sup} at the end of this appendix.
\begin{lemma}\label{lem:a0=sup}
	Let $a\in(\sup\supp\ol{\mu}_{T},\infty)$. Then, the following holds:
	\begin{enumerate}
		\item\label{itm:a0=sup-1} If $\psi(\wt{a})>\lambda^{\circ}$ for all $\wt{a}\geq a$, then $\psi'(a)>0$;
		\item\label{itm:a0=sup-2} If $\psi'(a)>0$, then $\psi(a)\notin\supp\ol{\mu}_{\wh{D}}$.
	\end{enumerate}
\end{lemma}
We will see in \Cref{lem:a0} that $a^{\circ}$ is indeed well-defined. More precisely, $\psi$ is an analytic function with at least one critical point, and $\psi'(a)$ converges to a positive number as $a\to\infty$.

\subsection{Properties of $\psi$}\label{sec:psi}
Recall that  $\psi:(\sup\supp\ol{\mu}_{T},\infty)\to\R$ is defined by $\psi(a)= a\gamma(a)$. With a slight modification to the definition of $\gamma(a)$, we have the following result.
\begin{lemma}\label{lem:a}\
	\begin{enumerate}
		\item\label{itm:finite} The sets $\caS,\caS'\subset\R$ defined by 
            \begin{align}
            \caS&\deq \brace{a>\sup\supp\ol{\mu}_{T}:\expt{ \cF_a(\ol{Y}) }=0},\notag \\
            \caS'&\deq \brace{a>\sup\supp\ol{\mu}_{T}:-\expt{ \cF_a(\ol{Y}) }=\frac{1}{\delta}}  \notag     
            \end{align}
		 are finite.
		
		\item\label{itm:omega1} For each $a\in (\sup\supp\ol{\mu}_{T},\infty)\setminus\caS$, there exists a unique $\omega\equiv\omega(a)\in\R\setminus(\inf\supp\ol{\mu}_{\Sigma},\allowbreak\sup\supp\ol{\mu}_{\Sigma})$ such that
		\beq\label{eq:omega}
		\delta\int_{\R}\frac{t}{t-a}\dd\ol{\mu}_{T}(t)=\int_{\R}\frac{s}{s-\omega}\dd\ol{\mu}_{\Sigma}(s).
		\eeq
		
		\item \label{itm:omega} The map $\omega:(\sup\supp\ol{\mu}_{T},\infty)\setminus\caS\to \R$ defined in \Cref{itm:omega1} extends meromorphically to an open set in $\C$ containing $(\sup\supp\ol{\mu}_{T},\infty)$. The extension is analytic at each $a\in (\sup\supp\ol{\mu}_{T},\infty)\setminus\caS$, has a pole at each $a\in\caS$ and a zero at each $a\in\caS'$.
		
		\item\label{itm:psi} The function $\psi:(\sup\supp\ol{\mu}_{T},\infty)\to\R$ defined by $\psi(a)=a\gamma(a)$ satisfies
		\beq\label{eq:psi}
		\psi(a)=-\frac{a}{\delta}\int_{\R}\frac{s\omega(a)}{s-\omega(a)}\dd\ol{\mu}_{\Sigma}(s),\qquad \forall a\in(\sup\supp\ol{\mu}_{T},\infty)\setminus\caS.
		\eeq
		Furthermore, $\psi$ extends analytically to an open set in $\C$ containing $(\sup\supp\ol{\mu}_{T},\infty)$, and has zeros precisely at $\caS'$.
	\end{enumerate}
\end{lemma}
\begin{proof}
	Note that the function $a\mapsto -\expt{\cF_a(\ol{Y})}$ is analytic in $(\sup\supp\ol{\mu}_{T},\infty)$, so both $\caS$ and $\caS'$ cannot have accumulating points in $(\sup\supp\ol{\mu}_{T},\infty)$. Thus, in order to prove \Cref{itm:finite}, it suffices to prove that $\caS,\caS'$ are contained in a compact subset of $(\sup\supp\ol{\mu}_{T},\infty)$. By the assumptions on $\cT$ (\emph{(d)} in \Cref{eqn:asmp_a_lim}) we have
	\beq\notag
		\lim_{a\searrow\sup\supp\ol{\mu}_{T}}-\expt{\cF_a(\ol{Y})}=-\infty,
	\eeq
	hence $\caS$ and $\caS'$ are contained in $[x,\infty)$ for some $x>\sup\supp\ol{\mu}_{T}$. Also, we have the series expansion
	\beq\notag
		-\expt{\cF_a(\ol{Y})}=-\frac{\expt{\cT(\ol{Y})}}{a}-\frac{\expt{\cT(\ol{Y})^{2}}}{a^{2}}+\cO(a^{-3}),\qquad \text{as }a\to\infty,
	\eeq
        where $ \expt{\cT(\ol{Y})^2}>0 $ by the assumption in \Cref{eqn:asmp_T}. 
	This already proves that $\caS'$ is bounded, as $-\expt{\cF_a(\ol{Y})}$ converges to $0$ as $a\to\infty$. Similarly, the same expansion implies that for large enough $x>\sup\supp\ol{\mu}_{T}$ we have
	\beq\notag
		-\expt{\cF_a(\ol{Y})}\in \begin{cases}
			(0,\infty), & \text{if }\expt{\cT(\ol{Y})}<0, \\
			(-\infty,0), & \text{if }\expt{\cT(\ol{Y})}\geq 0,
		\end{cases}
	\qquad \forall a>x.
	\eeq
	Thus, $\caS\cap [x,\infty)=\emptyset$. This concludes \Cref{itm:finite}.
	
	For \Cref{itm:omega1}, we only need to notice that the right-hand side of \Cref{eq:omega} is a bijection between $\R\setminus(\inf\supp\ol{\mu}_{\Sigma},\sup\supp\ol{\mu}_{\Sigma})$ and $\R\setminus\{0\}$. Notice further that the right-hand side is analytic in $\omega$ with strictly positive derivative whenever $\omega$ is well-defined;
	\beq\notag
	\frac{\dd}{\dd \omega}\int_{\R}\frac{s}{s-\omega}\dd\ol{\mu}_{\Sigma}=\int_{\R}\frac{s}{(s-\omega)^{2}}\dd\ol{\mu}_{\Sigma}.
	\eeq
	
    We now turn to \Cref{itm:omega}. Since the left-hand side of \Cref{eq:omega} is an analytic function of $a$, it immediately follows from analytic inverse function theorem that $\omega$ extends analytically to a neighborhood of $(\sup\supp\ol{\mu}_{T},\infty)\setminus\caS$. Similarly, for each $a>\sup\supp\cT(\ol{Y})$ with $a\notin \caS\cup\caS'$, we find that $\wt{\omega}(a)\deq 1/\omega(a)$ solves
	\beq
		\delta\int_{\R}\frac{t}{t-a}\dd\ol{\mu}_{T}(t)=-\wt{\omega}(a)\int_{\R}\frac{s}{1-s\wt{\omega}(a)}\dd\ol{\mu}_{\Sigma}.\notag
	\eeq
	Defining $\wt{\omega}(a)\equiv0$ for $a\in\caS$ and following the same reasoning as for $\omega$, one easily finds that $\wt{\omega}$ extends analytically to a neighborhood of $(\sup\supp\ol{\mu}_{T},\infty)\setminus\caS'$. By analytic continuation, $\omega$ extends to a meromorphic function on a neighborhood of $(\sup\supp\ol{\mu}_{T},\infty)$ with poles at $\caS$. From \Cref{eq:omega} we immediately find that the zeros of $\omega$ are exactly at $\caS'$.
	
	Finally, for \Cref{itm:psi}, note that by a trivial rescaling we have
	\beq\notag
		-\omega(a)\int_{\R}\frac{t}{t-a}\dd\ol{\mu}_{T}(t)=\gamma(a),
	\eeq
	which implies
	\beq\label{eq:psi1}
		\psi(a)=- a\omega(a)\int_{\R}\frac{t}{t-a}\dd\ol{\mu}_{T}(t),\qquad a\notin \caS.
	\eeq
	Using the definition of $\omega$, we immediately have \Cref{eq:psi} from \Cref{eq:psi1}. Also, \Cref{eq:psi1} already shows that $\psi$ is a meromorphic function on a neighborhood of $(\sup\supp\ol{\mu}_{T},\infty)$ by \Cref{itm:omega1}, with possible poles at $\caS$. Hence we only need to check that each $a\in\caS$ is a removable singularity for $\psi$. Recall that $\omega(z)\to\infty$ as $z\to a\in\caS$, so that by dominated convergence
        \begin{align}
		\psi(z)&=-\frac{z}{\delta}\int_{\R}\frac{s\omega(z)}{s-\omega(z)}\dd\ol{\mu}_{\Sigma}(s)=-\frac{z}{\delta}\int_{\R}\frac{s}{s/\omega(z)-1}\dd\ol{\mu}_{\Sigma}(s)\to \frac{a}{\delta}\expt{\ol{\Sigma}}.
            \notag 
	\end{align}%
\end{proof}
\begin{lemma}\label{lem:a0}
	We have
	\beq\label{eq:psi'_lim}
	\lim_{a\to\infty}\psi'(a)=\frac{\expt{\ol{\Sigma}}}{\delta}=\lim_{\re a\to\infty}\frac{\im\psi(a)}{\im a},
	\eeq
	where we identified $\psi$ with its analytic extension. We also have
	\beq\label{eq:psi_lim}
	\lim_{a\to\infty}\psi(a)=\infty=\lim_{a\searrow\sup\supp\ol{\mu}_{T}}\psi(a).
	\eeq
	In particular, the set of critical points of $\psi$ is nonempty and bounded from above (as a subset of $\bbR$).
\end{lemma}
\begin{proof}
	We compute the derivative of $\psi$ as
        \begin{align}
            \begin{aligned}
                \delta\psi'(a)&=-\int_{\R}\frac{s\omega(a)}{s-\omega(a)}\dd\ol{\mu}_{\Sigma}(s)-a\omega'(a)\int_{\R}\frac{s^{2}}{(s-\omega(a))^{2}}\dd\ol{\mu}_{\Sigma}(s) \\
		      &=-\int_{\R}\frac{s\omega(a)}{s-\omega(a)}\dd\ol{\mu}_{\Sigma}(s) \\
                &\phantom{=}~-a\delta\left(\int_{\R}\frac{s}{(s-\omega(a))^{2}}\dd\ol{\mu}_{\Sigma}(s)\right)^{-1}\int_{\R}\frac{t}{(t-a)^{2}}\dd\ol{\mu}_{T}(t)\int_{\R}\frac{s^{2}}{(s-\omega(a))^{2}}\dd\ol{\mu}_{\Sigma}(s).
            \end{aligned}
            \label{eq:psi'}
        \end{align}
	Furthermore,  notice from \Cref{itm:omega1} of \Cref{lem:a} that $\absv{\omega(a)}\to\infty$ as $a\to\infty$, so that the second term in \Cref{eq:psi'} satisfies
        \begin{align}
            &\lim_{a\to\infty}\left(\int_{\R}\frac{s\omega(a)^{2}}{(s-\omega(a))^{2}}\dd\ol{\mu}_{\Sigma}(s)\right)^{-1}\int_{\R}\frac{s^{2}\omega(a)^{2}}{(s-\omega(a))^{2}}\dd\ol{\mu}_{\Sigma}(s)\cdot \int_{\R}\frac{ta}{(t-a)^{2}}\dd\ol{\mu}_{T}(t) \notag \\
		&=\frac{\expt{\ol{\Sigma}^{2}}}{\expt{\ol{\Sigma}}}\lim_{a\to\infty}\int_{\R}\frac{ta}{(t-a)^{2}}\dd\ol{\mu}_{T}(t)=0. \notag 
        \end{align}
	Therefore, we conclude that the first equality in \Cref{eq:psi'_lim} holds as
	\beq
		\lim_{a\to\infty}\psi'(a)=\frac{1}{\delta}\lim_{a\to\infty}\int_{\R}\frac{-s\omega(a)}{s-\omega(a)}\dd\ol{\mu}_{\Sigma}(s)=\frac{1}{\delta}\expt{\ol{\Sigma}}.\notag
	\eeq
	The second equality can be proved analogously, except that the following identity replaces \Cref{eq:psi'}: 
	\beq\begin{aligned}
&		\delta\frac{\im\psi(a)}{\im a}=-\re\left[ \int_{\R}\frac{s\omega(a)}{s-\omega(a)}\dd\ol{\mu}_{\Sigma}(s)\right]\\
		&\phantom{=}~-\delta\re[a]\left(\int_{\R}\frac{s}{\absv{s-\omega(a)}^{2}}\dd\ol{\mu}_{\Sigma}(s)\right)^{-1}\int_{\R}\frac{t}{\absv{t-a}^{2}}\dd\ol{\mu}_{T}(t)\int_{\R}\frac{s^{2}}{\absv{s-\omega(a)}^{2}}\dd\ol{\mu}_{\Sigma}(s),
	\end{aligned}\notag\eeq
 where we used
 \beq\notag
    \frac{\im\omega(a)}{\im a}=\delta\left(\int_{\R}\frac{s}{\absv{s-\omega(a)}^{2}}\dd\ol{\mu}_{\Sigma}(s)\right)^{-1}\int_{\R}\frac{t}{\absv{t-a}^{2}}\dd\ol{\mu}_{T}(t),
 \eeq
 from \Cref{eq:omega}.

	Notice that the first equality in \Cref{eq:psi_lim} follows from the first equality in \Cref{eq:psi'_lim}. For the second equality in \Cref{eq:psi_lim}, recall from the assumption \emph{(d)} in \Cref{eqn:asmp_a_lim} that
	\beq\notag
		\lim_{a\searrow\sup\supp\ol{\mu}_{T}}\int_{\R}\frac{t}{t-a}\dd\ol{\mu}_{T}(t)=-\infty,
	\eeq
	which implies $\lim_{a\searrow\sup\supp\ol{\mu}_{T}}\omega(a)=\sup\supp\ol{\mu}_{\Sigma}$ via \Cref{itm:omega1} of \Cref{lem:a}. Plugging these in the definition of $\psi$ in \Cref{eq:psi} and using $\sup\supp\ol{\mu}_{T}>0$ prove $\psi(a)\to\infty$.
\end{proof}

\subsection{Complex analytic characterization of $\ol{\mu}_{\wh{D}}$}
\begin{lemma}[{\cite[Theorem 1.2.1]{Zhang_Thesis_RMT}}]
	Let $m_{\ol{\mu}_{\wh{D}}}$ denote the Stieltjes transform of the limiting eigenvalue distribution $\ol{\mu}_{\wh{D}}$. 
	For each $z\in\bbH\coloneqq \brace{z\in\bbC : \Im(z) > 0}$, $m = m_{\ol{\mu}_{\wh{D}}}(z)$ is characterized as the unique solution $(m,m_{1},m_{2})$ of the following system of equations:
	\begin{equation}\label{eq:subor_m}
		\begin{cases}\displaystyle
			-zm=(1-\delta)+\delta\int_{\mathbb{R}}\frac{1}{1+m_{1}t}\mathrm{d}\ol{\mu}_{T}(t), \\\displaystyle
			-zm=\int_{\mathbb{R}}\frac{1}{1+m_{2}s}\mathrm{d}\ol{\mu}_{\Sigma}(s),\\\displaystyle
			-zm=1+\delta zm_{1}m_{2},
		\end{cases}
	\end{equation}
	subject to the constraint $m,m_{1},zm_{2}\in\bbH$. All of $m,m_{1},m_{2}$ are analytic in $\bbH$ as a function of $z$.
\end{lemma}
We adopt the notation $m(\ol{z})=\ol{m}(z)$ and $m_{i}(\ol{z})=\ol{m_{i}(z)}$ ($ i\in\{1,2\} $).
The major difference from the case of positive $\cT$ is that $m_{2}$ might not be in $\bbH$; still the second equation in \Cref{eq:subor_m} is well-defined as $m_{2}(z)\in \{z^{-1}w:w\in\bbH\}\subset\mathbb{C}\setminus(-\infty,0]$. (Cf., when $\cT$ is positive then $m_{i}\in\bbH$ and $zm_{i}\in\bbH$ for both $i\in\{1,2\}$.)
Alternatively, using the last equation in \Cref{eq:subor_m} to substitute $m$ in the first two equations, we may write the system of two equations for $m_{1},m_{2}$:
\begin{equation}\label{eq:subor}
	\begin{cases}\displaystyle
		-zm_{1}=\frac{1}{\delta}\int_{\mathbb{R}}\frac{s}{1+m_{2}s}\dd\ol{\mu}_{\Sigma}(s), \\\displaystyle
		-zm_{2}=\int_{\mathbb{R}}\frac{t}{1+m_{1}t}\mathrm{d}\ol{\mu}_{T}(t).
	\end{cases}
\end{equation}
For later purposes, we define for all $z,w\in\bbC\setminus\R$, 
\beq\label{eq:I_def}\begin{aligned}
		I_{1}(z,w)&\deq \int_{\R}\frac{t^{2}}{(1+m_{1}(z)t)(1+m_{1}(w)t)}\dd\ol{\mu}_{T}(t),	\\
		I_{2}(z,w)&\deq \int_{\R}\frac{s^{2}}{(1+m_{2}(z)s)(1+m_{2}(w)s)}\dd\ol{\mu}_{\Sigma}(s),	\\
	\end{aligned}\eeq
	so that $I_{1}(z,\ol{z})$ and $I_{2}(z,\ol{z})$ are positive since $m_{i}(\ol{z})=\ol{m_{i}(z)}$. Note also that 
 \beq\label{eqn:zm_vs_I}\begin{aligned}
    \absv{zm_{1}(z)}\leq \delta^{-1}I_{2}(z,\ol{z})^{1/2},\qquad \absv{zm_{2}(z)}\leq I_{1}(z,\ol{z})^{1/2},
 \end{aligned}\eeq
 by Cauchy--Schwarz.
	
	\begin{lemma}
		For all $z\in\bbH$,
		\beq\label{eq:stab}
			\frac{1}{\delta\absv{z}^{2}}I_{1}(z,\ol{z})I_{2}(z,\ol{z})<1.
		\eeq
		Consequently, 
		\beq\begin{aligned}\label{eq:stab_1}
			\absv{m_{1}(z)}^{2}I_{1}(z,\ol{z})<\frac{1}{\delta}, \qquad
			\absv{m_{2}(z)}^{2}I_{2}(z,\ol{z})<\delta.
		\end{aligned}\eeq
	\end{lemma}
\begin{proof}
	Dividing the first line of \Cref{eq:subor} by $z$ and then taking imaginary parts, we get
\begin{equation}\label{eq:imm}
	\Im m_{1}(z)=\frac{1}{\delta}\Im \int_{\mathbb{R}}\frac{s}{-z(1+m_{2}(z)s)}\mathrm{d}\ol{\mu}_{\Sigma}(s)
	=\frac{1}{\delta}\int_{\mathbb{R}}\frac{s\Im z+s^{2}\Im zm_{2}(z)}{\absv{z}^{2}\absv{1+m_{2}(z)s}^{2}}\mathrm{d}\ol{\mu}_{\Sigma}(t).
\end{equation}
Similarly taking the imaginary part of the second line of \Cref{eq:subor} gives 
\begin{equation}\label{eq:imm2}
	\Im zm_{2}(z) =-\Im\int_{\mathbb{R}}\frac{t}{1+m_{1}(z)t}\dd\ol{\mu}_{T}(t)
	=\int_{\mathbb{R}}\frac{t^{2}\Im m_{1}(z)}{\absv{1+m_{1}(z)t}^{2}}\dd\ol{\mu}_{T}(t).
\end{equation}
Combining \Cref{eq:imm,eq:imm2}, we obtain
\begin{align}
    \begin{aligned}
        \delta\Im m_{1}(z)&=\int_{\mathbb{R}}\frac{s\Im z}{\absv{z}^{2}\absv{1+m_{2}(z)s}^{2}}\dd\ol{\mu}_{\Sigma}(s)\\
	&\phantom{=}~+\frac{ \Im m_{1}(z)}{\absv{z}^{2}}\left(\int_{\mathbb{R}}\frac{t^{2}}{\absv{1+m_{1}(z)t}^{2}}\dd\ol{\mu}_{T}(t)\right)\left(\int_{\mathbb{R}}\frac{s^{2}}{\absv{1+m_{2}(z)s}^{2}}\dd\ol{\mu}_{\Sigma}(s)\right).
    \end{aligned}
    \label{eq:imm22}
\end{align}
Since $\im m_{1}(z)$ and the first term on the right-hand side of \Cref{eq:imm22} are positive for all $z\in\bbH$, we have proved \Cref{eq:stab}: \begin{equation}
	\frac{1}{\delta\absv{z}^{2}}\left(\int_{\mathbb{R}}\frac{t^{2}}{\absv{1+m_{1}(z)t}^{2}}\dd\ol{\mu}_{T}(t)\right)\left(\int_{\mathbb{R}}\frac{s^{2}}{\absv{1+m_{2}(z)s}^{2}}\dd\ol{\mu}_{\Sigma}(s)\right)< 1,\qquad \forall z\in\bbH.\notag
\end{equation}

For \Cref{eq:stab_1}, we only need to notice from \Cref{eq:stab,eqn:zm_vs_I} that 
\beq
	\absv{m_{1}}^{2}I_{1}(z,\ol{z})\leq \frac{1}{\delta^{2}\absv{z}^{2}}I_{1}(z,\ol{z})I_{2}(z,\ol{z})<\frac{1}{\delta},\notag
\eeq
and the second line in \Cref{eq:stab_1} follows similarly.
\end{proof}
Note also that \Cref{eq:stab} implies for all $z\in\bbH$ that
\begin{align}
	\absv{zm(z)+1} \leq \frac{\delta}{\absv{z}}\absv{zm_{1}(z)}\absv{zm_{2}(z)}
	&\leq \frac{1}{\absv{z}}\sqrt{I_{1}(z,\ol{z})I_{2}(z,\ol{z})}\leq \sqrt{\delta}, \label{eq:m_bdd}
\end{align}
where we used the third line of \Cref{eq:subor_m} in the first, \Cref{eqn:zm_vs_I} in the second, and \Cref{eq:stab} in the last inequality.

\begin{lemma}\label{lem:m_bdd}
	Let $\cD\subset\bbH$ be bounded. Then, there exists a constant $K>0$ depending only on $\cD$, $\ol{\mu}_{\Sigma}$, and $\ol{\mu}_{T}$ such that
	\beq\notag
	\absv{zm_{1}(z)}\leq K,\qquad \absv{zm_{2}(z)}\leq K,\qquad \forall z\in \cD.
	\eeq
\end{lemma}
\begin{proof}
	We only consider $\absv{zm_{1}(z)}$, and the same argument applies to $\absv{zm_{2}(z)}$. The proof is by contradiction. Suppose that there exists a sequence $z_{k}$ in $\cD$ such that $\absv{z_{k}m_{1}(z_{k})}\to\infty$. Then by combining \Cref{eq:m_bdd} with the third equation in \Cref{eq:subor_m}, we have $\absv{m_{2}(z_{k})}\to0$. Therefore by dominated convergence (together with $\sup\supp\ol{\mu}_{\Sigma}<\infty)$ we have
	\beq
	-\delta\lim_{k\to\infty}z_{k}m_{1}(z_{k})=\lim_{k\to\infty}\int_{\R}\frac{s}{1+m_{2}(z_{k})s}\dd\ol{\mu}_{\Sigma}(s)=\int_{\R}s\dd\ol{\mu}_{\Sigma}(s)\in \R, \notag 
	\eeq
	which gives a contradiction to $\absv{z_{k}m_{1}(z_{k})}\to\infty$.
\end{proof}

\begin{lemma}\label{lem:imm_compa}
	For all $z\in\bbH$, we have
	\beq\label{eq:imm_compa}
		0<(\inf \supp\ol{\mu}_{\Sigma})\leq \delta\frac{\im m_{1}(z)}{\im m(z)}\leq (\sup\supp\ol{\mu}_{\Sigma}).
	\eeq
	For each bounded $\caD\subset\bbH$, there exists a constant $K_{1}$ depending only on $\caD,\ol{\mu}_{\Sigma}$, and $\ol{\mu}_{T}$ such that
	\beq\label{eq:imm_compa_2}
		\im (zm_{2}(z))\leq K_{1}\im m_{1}(z),\qquad z\in\caD.
	\eeq
\end{lemma}
\begin{proof}
	To see \Cref{eq:imm_compa}, note that the second line of \Cref{eq:subor_m} implies
	\beq\label{eq:imm_compa_prf}
		\im m(z)=\int_{\R}\im\left[\frac{1}{-z(1+m_{2}(z)s)}\right]\dd\ol{\mu}_{\Sigma}(s).
	\eeq
	Comparing \Cref{eq:imm_compa_prf} with \Cref{eq:imm} proves \Cref{eq:imm_compa}.
	
	For \Cref{eq:imm_compa_2}, we recall from \Cref{eq:imm2,eq:stab} that
	\beq
		\im zm_{2}(z)=\im m_{1}(z) \cdot I_{1}(z,\ol{z})\leq \im m_{1}(z)\cdot \frac{\delta\absv{z}^{2}}{I_{2}(z,\ol{z})}.\notag
	\eeq
	By definition of $I_{2}(z,\ol{z})$, we have
        \begin{align}
                \frac{\absv{z}^{2}}{I_{2}(z,\ol{z})}&=\left(\int_{\R}\frac{s^{2}}{\absv{z+zm_{2}(z)s}^{2}}\dd\ol{\mu}_{\Sigma}(s)\right)^{-1} \notag \\
                &\leq 2 \left(\brack{ (\sup\supp\ol{\mu}_{\Sigma}) \cdot \absv{zm_{2}(z)} }^{2}+\absv{z}^{2}\right)\left(\int_{\R} s^{2}\dd\ol{\mu}_{\Sigma}(s)\right)^{-1}.
            \label{eq:imm_compa_prf_2}
        \end{align}
Since $\caD$ is bounded, the right-hand side of \Cref{eq:imm_compa_prf_2} is bounded by a constant for all $z\in\caD$. This proves \Cref{eq:imm_compa_2}.
\end{proof}

\begin{proposition}\label{prop:Pick}\
	\begin{enumerate}
		\item \label{itm:Pick} There exist two finite measures $\nu_{1},\nu_{2}$ on $\R$ such that the following holds; for all $z\in\bbH$ we have
		\beq\begin{aligned}\label{eq:Pick}
			\int_{\R}\frac{1}{x-z}\dd\nu_{1}(x)&=m_{1}(z), &
			\nu_{1}(\R)&=\frac{\expt{\ol{\Sigma}}}{\delta},\\ 
			\int_{\R}\frac{1}{x-z}\dd\nu_{2}(x)&=zm_{2}(z)+\int_{\R}t\dd\ol{\mu}_{T}(t), &
			\nu_{2}(\R)&=\frac{\expt{\ol{\Sigma}}}{\delta}\expt{\cT(\ol{Y})^{2}}.
		\end{aligned}\eeq
		Consequently we have 
		\beq\label{eq:Pick_supp}
			\supp\nu_{1}=\supp\ol{\mu}_{\wh{D}},\qquad \supp\nu_{2}\subset\supp\ol{\mu}_{\wh{D}},
		\eeq
		so that $m_{1}$ and $m_{2}$ are respectively analytic and meromorphic functions on $\R\setminus\supp\ol{\mu}_{\wh{D}}$.
		\item\label{itm:stab} For all $x>\lambda^{\circ}$, we have
		\beq\label{eq:stab1}\begin{aligned} 
			&-\frac{1}{m_{1}(x)}\in (\sup\supp\ol{\mu}_{T},\infty),\\
			& 
			-\frac{1}{m_{2}(x)}\in (\R\cup\{\infty\})\setminus (\inf\supp\ol{\mu}_{\Sigma},\sup\supp\ol{\mu}_{\Sigma}),\\
			&\limsup_{z\to x,z\in\bbH}\frac{1}{\delta\absv{z}^{2}}I_{1}(z,\ol{z})I_{2}(z,\ol{z})<1,
		\end{aligned}\eeq
	where we used the convention $1/0=\infty$ in the second assertion.
	\end{enumerate}
\end{proposition}
\begin{proof}
	We start with the proof of \Cref{itm:Pick}. First, notice that once \Cref{eq:Pick} is proved, \Cref{eq:Pick_supp} immediately follows from \Cref{lem:imm_compa} and Stieltjes inversion. 
    In order to prove the first identity in \Cref{eq:Pick}, since $m_{1}$ is an analytic self-map of $\bbH$, by Nevanlinna--Pick representation theorem it suffices to check
	\beq\label{eq:nu1}
	\limsup_{\eta\to\infty}\eta \absv{m_{1}(\ii\eta)}<\infty. 
	\eeq
	Suppose the contrary, so that there exists a sequence $\eta_{k}\to\infty$ with $\eta_{k}\absv{m_{1}(\ii\eta_{k})}\to\infty$. Then by \Cref{eq:m_bdd} we find that $\absv{m_{2}(\ii\eta_{k})}\to0$. On the other hand by \Cref{eq:subor}, we have
	\beq\label{eq:nu1R}
	-\ii\eta m_{1}(\ii\eta)=\frac{1}{\delta}\int_{\R}\frac{s}{1+m_{2}(\ii\eta)s}\dd\ol{\mu}_{\Sigma}(s),
	\eeq
	so that the dominated convergence theorem (with $\norm{2}{\Sigma}=\cO(1)$) leads to a contradiction as
	\beq
	\lim_{k\to\infty}\eta_{k}\absv{m_{1}(\ii\eta_{k})}=\frac{1}{\delta}\lim_{k\to\infty}\Absv{\int_{\R}\frac{s}{1+m_{2}(\ii\eta_{k})s}\dd\ol{\mu}_{\Sigma}(s)}=\frac{1}{\delta}\int_{\R}s\dd\ol{\mu}_{\Sigma}(s). \notag 
	\eeq
	Thus we have proved the first line of \Cref{eq:Pick}.
	
	Next, we prove the corresponding representation for $zm_{2}(z)$, the second line of \Cref{eq:Pick}. As before, it suffices to prove
	\beq\label{eq:Pick_crit}
	\limsup_{\eta\to\infty}\eta\Absv{\ii\eta m_{2}(\ii\eta)+\int_{\R}t\dd\ol{\mu}_{T}(t)}<\infty.
	\eeq
	To this end, we use \Cref{eq:subor} to write
	\beq\label{eq:nu2R}
	z\left(zm_{2}(z)+\int_{\R}t\dd\ol{\mu}_{T}(t)\right)= zm_{1}(z)\int_{\R}\frac{t^{2}}{1+m_{1}(z)t}\dd\ol{\mu}_{T}(t).
	\eeq
	Taking the limit along $z=\ii\eta\to\ii\infty$, by \Cref{eq:Pick} we have $m_{1}(z)\to0$ and $zm_{1}(z)\to-\nu_{1}(\R)$ (note that $\nu_{1}(\R)$ is finite due to \Cref{eq:nu1}), so that 
	\beq
	\lim_{\eta\to\infty}\ii\eta\left(\ii\eta m_{2}(\ii\eta)+\int_{\R}t\dd\ol{\mu}_{T}(t)\right)=-\nu_{1}(\R)\int_{\R}t^{2}\dd\ol{\mu}_{T}(t). \notag 
	\eeq
	
	Finally, given the two representations in \Cref{eq:Pick}, we have $m_{1}(\ii\eta),m_{2}(\ii\eta)\to 0$ as $\eta\to\infty$. Then $\nu_{1}(\R)$ and $\nu_{2}(\R)$ can be computed by taking the limits of \Cref{eq:nu1R,eq:nu2R} as $z=\ii\eta\to\ii\infty$. This completes the proof of \Cref{itm:Pick}.  
	
	Now we prove \Cref{itm:stab}. Notice that $m_{1}$ is analytic, negative-valued, and increasing on $(\lambda^{\circ},\infty)$, and that $\lim_{x\to\infty}m_{1}(x)=0$. Therefore the image of the half line $(\lambda^{\circ},\infty)$ under $x\mapsto -1/m_{1}(x)$ is again an half-line $(y_{0},\infty)$ for some $y_{0}>0$. Next, notice from \Cref{eq:stab_1} that for all $x\in\R$, 
	\beq\label{eq:stab_m1}
		\limsup_{z\to x,z\in \bbH}\absv{m_{1}(z)}^{2}I_{1}(z,\ol{z})=\limsup_{z\to x,z\in\bbH}\int_{\R}\frac{t^{2}}{\absv{t-(-1/m_{1}(z))}^{2}}\dd\ol{\mu}_{T}(t) < \frac{1}{\delta}.
	\eeq
	On the other hand, by the assumptions on $\cT$ (see \emph{(d)} in \Cref{eqn:asmp_a_lim}) and Cauchy--Schwarz, there exists an $\epsilon>0$ so that
	\beq\label{eq:stab_m11}
		\lim_{w\to y,w\in\bbH}\int_{\R}\frac{t^{2}}{\absv{t-w}^{2}}\dd\ol{\mu}_{T}(t)=\int_{\R}\frac{t^{2}}{\absv{t-y}^{2}}\dd\ol{\mu}_{T}(t)>\frac{1}{\delta},\quad \forall y\in (\sup\supp \ol{\mu}_{T},\sup\supp\ol{\mu}_{T}+\epsilon).
	\eeq
	Combining \Cref{eq:stab_m1,eq:stab_m11}, we conclude that $(y_{0},\infty)$ does not intersect with $(\sup\supp\ol{\mu}_{T} , \allowbreak \sup\supp\ol{\mu}_{T}+\epsilon)$, so that $y_{0}\geq \sup\supp\ol{\mu}_{T}+\epsilon$. This proves the first assertion of \Cref{itm:stab}.
	
	The proof of the second assertion in \Cref{itm:stab} follows similar lines, except that we view $x\mapsto -1/m_{2}(x)$ as an analytic (instead of meromorphic) function mapping into the Riemann sphere $\C\cup\{\infty\}$. Consequently, the closure of the image of $(\lambda^{\circ},\infty)$ under $z\mapsto -1/m_{2}(z)$ is a connected real interval in the Riemann sphere; or equivalently, it is the image of a closed connected arc in the unit circle under stereographic projection. Next, notice from the assumptions on $\ol{\Sigma}$ (see \emph{(b)} in \Cref{eqn:asmp_gamma_lim}) that there exists an $\epsilon>0$ so that
	\beq\notag
		\lim_{w\to y,w\in\bbH} \int_{\R}\frac{s^{2}}{\absv{s-w}^{2}}\dd\ol{\mu}_{\Sigma}(s)>\delta,
	\eeq
	for all $y\in (\inf\supp\ol{\mu}_{\Sigma}-\epsilon,\inf\supp\ol{\mu}_{\Sigma})\cup (\sup\supp\ol{\mu}_{\Sigma},\sup\supp\ol{\mu}_{\Sigma}+\epsilon)$.
	Therefore \Cref{eq:stab_1} implies that the image of $(\lambda^{\circ},\infty)$ under $x\mapsto -1/m_{2}(x)$ does not intersect with the two segments of length $\epsilon$, while containing $\infty$ in its closure since $m_{2}(x)\to0$ as $x\to\infty$. This proves the second assertion of \Cref{itm:stab}.
	
	For the final assertion of \Cref{itm:stab}, recall from \Cref{eq:imm22} that for all $z\in\bbH$, 
	\beq\notag\begin{aligned}
		1-\frac{1}{\delta\absv{z}^{2}}I_{1}(z,\ol{z})I_{2}(z,\ol{z})&=\frac{\im z}{\delta\im m_{1}(z)}\int_{\R}\frac{s}{\absv{z}^{2}\absv{1+m_{2}(z)s}^{2}}\dd\ol{\mu}_{\Sigma}(s)	\\
		&=\left(\delta\int_{\R}\frac{1}{\absv{y-z}^{2}}\dd\nu_{1}(y)\right)^{-1}\int_{\R}\frac{s}{\absv{z}^{2}\absv{1+m_{2}(z)s}^{2}}\dd\ol{\mu}_{\Sigma}(s),
	\end{aligned}\eeq
	where we used \Cref{eq:Pick} in the second equality. Taking the limit $z\to x>\lambda^{\circ}$, we have
        \begin{align}
            1-\limsup_{z\to x,z\in\bbH}\frac{1}{\delta\absv{z}^{2}}I_{1}(z,\ol{z})I_{2}(z,\ol{z})
		&=\left(\delta\int_{\R}\frac{1}{\absv{y-x}^{2}}\dd\nu_{1}(y)\right)^{-1} \notag \\
            &\phantom{=}~\cdot\int_{\R}s\left(x^{2}+s^{2}\limsup_{z\to x,z\in\bbH}\absv{zm_{2}(z)}^{2}\right)^{-1}\dd\ol{\mu}_{\Sigma}(s)>0,
            \notag 
        \end{align}
	where we used Fatou's lemma in the first equality and \Cref{lem:m_bdd} in the last inequality. This concludes the proof of \Cref{prop:Pick}.
\end{proof}

\subsection{Proof of \Cref{lem:bulk_nonpositive,lem:a0=sup}}
\begin{proof}[Proof of \Cref{lem:bulk_nonpositive} given \Cref{lem:a0=sup}]
	Notice that since $a^{\circ}$ is the largest critical point of $\psi$ and $\lim_{a\to\infty}\psi'(a)>0$, we find that $\psi'(a)>0$ for all $a\in (a^{\circ},\infty)$, i.e. $\psi$ is strictly increasing on $[a^{\circ},\infty)$.
	
    Next, we prove $\psi(a^{\circ})\leq \lambda^{\circ}$. Note from the contrapositive of \Cref{itm:a0=sup-1} of \Cref{lem:a0=sup} that if $a>\sup\supp\ol{\mu}_{\caT}$ and $\psi'(a)\leq 0$, then there exists an $\wt{a}\geq a$ such that $\psi(\wt{a})\leq \lambda^{\circ}$. We may apply this to the largest critical point $a^{\circ}$ since $\psi'(a^{\circ})=0$, so that $\psi(\wt{a})\leq\lambda^{\circ}$ for some $\wt{a}\geq a^{\circ}$. As $\psi$ is increasing in $[a^{\circ},\infty)$, we conclude $\psi(a^{\circ})\leq \psi(\wt{a})\leq\lambda^{\circ}$
    
    Conversely, \Cref{itm:a0=sup-2} of \Cref{lem:a0=sup} implies $(\psi(a^{\circ}),\infty)\cap\supp\ol{\mu}_{\wh{D}}=\emptyset$, so that $\lambda^{\circ}\leq \psi(a^{\circ})$. Therefore we have $\psi(a^{\circ})=\lambda^{\circ}$.
\end{proof}
\begin{proof}[Proof of \Cref{itm:a0=sup-1} of \Cref{lem:a0=sup}]
	Let $a\in(\sup\supp\ol{\mu}_{T},\infty)$ satisfy the assumption of \Cref{itm:a0=sup-1} of \Cref{lem:a0=sup}, that is, $\psi(\wt{a})>\lambda^{\circ}$ for all $\wt{a}\geq a$. First of all, we prove that there exists a complex neighborhood $U$ of $[a,\infty)$ such that
	\beq\label{eq:psi=minv}
		w=-1/m_{1}(\psi(w)),\qquad \omega(w)=-1/m_{2}(\psi(w)),	\qquad \forall w\in U.
	\eeq
	Here we remark that $\psi(a)>\lambda^{\circ}$ by assumption, so that all four functions of $w$ in \Cref{eq:psi=minv} are well-defined by \Cref{prop:Pick}; those in the first and second equalities are analytic and meromorphic, respectively.
	
	Recall from \Cref{lem:a0} that for large enough $\wt{a}>a$, there exists a neighborhood $V$ of $\wt{a}$ so that $\im \psi(w)/\im w>0$ for every $ w\in V $. Then, it also follows that, for each $w\in V\cap\bbH$, 
	\beq\label{eq:omega_san}
		\im\left[-\frac{\psi(w)}{\omega(w)}\right]=\im\left[ \int_{\R}\frac{tw}{t-w}\dd\ol{\mu}_{T}(t)\right]=\im w\int_{\R}\frac{t^{2}}{\absv{t-w}^{2}}\dd\ol{\mu}_{T}(t)>0.
	\eeq
	Also notice that the triple $(\psi(w),-1/w,-1/\omega(w))$ satisfies the same system of equations as in \Cref{eq:subor}: 
	\beq\label{eq:subor_psi}\begin{aligned}
		-\frac{\psi(w)}{w}=\frac{1}{\delta}\int_{\R}\frac{s\omega(w)}{s-\omega(w)}\dd\ol{\mu}_{\Sigma}(s)=-\frac{1}{\delta}\int_{\R}\frac{s}{1+s\cdot(-\omega(w))^{-1}}\dd\ol{\mu}_{\Sigma}(s),\\
		\frac{\psi(w)}{\omega(w)}=-\int_{\R}\frac{tw}{t-w}\dd\ol{\mu}_{T}(t)=-\int_{\R}\frac{t}{1+t\cdot(-w^{-1})}\dd\ol{\mu}_{T}(t).
	\end{aligned}\eeq
	Therefore, by the uniqueness of the solution of \Cref{eq:subor}, we conclude
	\beq\label{eq:psi=minv_1}
		(\psi(w),-1/w,-1/\omega(w))=(\psi(w),m_{1}(\psi(w)),m_{2}(\psi(w))),\qquad w\in V\cap\bbH.
	\eeq
	By \Cref{prop:Pick} and the assumption of \Cref{itm:a0=sup-1}, in both sides of \Cref{eq:psi=minv_1} are meromorphic functions defined on a neighborhood of $[a,\infty)$, so that the identity holds in the whole (connected) neighborhood. 
	
	We now prove $\psi'(a)>0$, provided $a\notin \caS\cup\caS'$. Recall from \Cref{eq:psi'} that
	\beq\label{eq:psi'_expa}\begin{aligned}
		\delta\psi'(a)&=\left(\int_{\R}\frac{s}{(s-\omega(a))^{2}}\dd\ol{\mu}_{\Sigma}(s)\right)^{-1}	\\
		&\phantom{=}~\cdot \left[-\delta\int_{\R}\int_{\R}\frac{t}{t-a}\frac{s\omega(a)}{(s-\omega(a))^{2}}+\frac{ta}{(t-a)^{2}}\frac{s^{2}}{(s-\omega(a))^{2}}\dd\ol{\mu}_{\Sigma}(s)\dd\ol{\mu}_{T}(t)\right].
	\end{aligned}\eeq
	Note that the second line in \Cref{eq:psi'_expa} can be written as 
        \begin{align}
            \begin{aligned}
                &-\delta\int_{\R}\int_{\R}\left[-\frac{t}{t-a}\frac{s}{s-\omega(a)}+\frac{t^{2}}{(t-a)^{2}}\frac{s^{2}}{(s-\omega(a))^{2}}\right]\dd\ol{\mu}_{\Sigma}(s)\dd\ol{\mu}_{T}(t)	\\
		      &=\frac{\delta^{2}\psi(a)^{2}}{a^{2}\omega(a)^{2}}-\delta\int_{\R}\int_{\R}\frac{t^{2}}{(t-a)^{2}}\frac{s^{2}}{(s-\omega(a))^{2}}\dd\ol{\mu}_{\Sigma}(s)\dd\ol{\mu}_{T}(t).
            \end{aligned}
            \label{eq:psi'_expa1}
        \end{align}
	Then, we use \Cref{eq:psi=minv} for $w=a$ to substitute $a$ and $\omega(a)$ in \Cref{eq:psi'_expa} to obtain
	\beq\label{eq:psi'_1}\begin{aligned}
		\psi'(a)&=\frac{\delta\psi(a)^{2}}{a^{2}\omega(a)^{2}}\left(\int_{\R}\frac{s}{(s-\omega(a))^{2}}\dd\ol{\mu}_{\Sigma}(s)\right)^{-1}	\\
		&\phantom{=}~\cdot \left(1-\frac{1}{\delta\psi(a)^{2}}I_{1}(\psi(a),\psi(a))I_{2}(\psi(a),\psi(a))\right)>0,
	\end{aligned}\eeq
	where we used $0<\absv{m_{2}(\psi(a))},\absv{\psi(a)}<\infty$ for $a\neq \caS\cup\caS'$ and \Cref{eq:stab1}.
	
	It only remains to prove $\psi'(a)>0$ for $a\in\caS\cup\caS'$. Since $\caS$ and $\caS'$ are both finite, we may consider a sequence $\wt{a}_{k}>a$ such that $\wt{a}_{k}\notin \caS\cup\caS'$ and $\wt{a}_{k}\to a$. Since $\psi$ is analytic at $a$ and the second line of \Cref{eq:psi'_1} is strictly positive by \Cref{prop:Pick}, is suffices to prove
	\beq
		\lim_{k\to\infty}\frac{\psi(\wt{a}_{k})^{2}}{\omega(\wt{a}_{k})^{2}}\left(\int_{\R}\frac{s}{(s-\omega(\wt{a}_{k}))^{2}}\dd\ol{\mu}_{\Sigma}(s)\right)^{-1}>0. \notag
	\eeq
	If $a\in\caS$ so that $\omega(\wt{a}_{k})\to \infty$, we have
	\begin{align}
		\lim_{k\to\infty}&\frac{\psi(\wt{a}_{k})^{2}}{\omega(\wt{a}_{k})^{2}}\left(\int_{\R}\frac{s}{(s-\omega(\wt{a}_{k}))^{2}}\dd\ol{\mu}_{\Sigma}(s)\right)^{-1}\\
		&=\psi(a)^{2}\lim_{k\to\infty}\left(\int_{\R}s\left(\frac{\omega(\wt{a}_{k})}{s-\omega(\wt{a}_{k})}\right)^{2}\dd\ol{\mu}_{\Sigma}(s)\right)^{-1}\notag \\
		&=\frac{\psi(a)^{2}}{\expt{\ol{\Sigma}}}>0,\notag
	\end{align}
	where in the last inequality we used $a\in\caS$ implies $a\notin\caS'$, which in turn gives $\psi(a)\neq0$. Finally for $a\in\caS'$, we use $\omega(\wt{a}_{k})\to 0$ to write
	\beq\begin{aligned}
		&\lim_{k\to\infty}\frac{\psi(\wt{a}_{k})^{2}}{\omega(\wt{a}_{k})^{2}}\left(\int_{\R}\frac{s}{(s-\omega(\wt{a}_{k}))^{2}}\dd\ol{\mu}_{\Sigma}(s)\right)^{-1}
		=\frac{1}{\expt{\ol{\Sigma}^{-1}}}\lim_{k\to\infty}\frac{\psi(\wt{a}_{k})^{2}}{\omega(\wt{a}_{k})^{2}}\\
		&=\frac{a^{2}}{\delta^{2}\expt{\ol{\Sigma}^{-1}}}\lim_{k\to\infty}\left(\int_{\R}\frac{s}{s-\omega(\wt{a}_{k})}\dd\ol{\mu}_{\Sigma}(s)\right)^{2}=\frac{a^{2}}{\delta^{2}\expt{\ol{\Sigma}^{-1}}}>0,
	\end{aligned}\notag\eeq
where we used the definition of $\psi$ in the second equality and $\inf\supp\ol{\mu}_{\Sigma}>0$ in the last inequality. This concludes the proof of \Cref{itm:a0=sup-1} of \Cref{lem:a0=sup}.
\end{proof}

\begin{proof}[Proof \Cref{itm:a0=sup-2} of \Cref{lem:a0=sup}]
	Since $\psi'(a)>0$, there exist small neighborhoods $U$ and $V$ respectively of $a$ and $\psi(a)$ and an analytic inverse function $\psi^{-1}:V\to U$ of $\psi$. We first prove that 
	\beq\label{eq:psi_minv_2}
		(z,-1/\psi^{-1}(z),-1/\omega(\psi^{-1}(z)))=(z,m_{1}(z),m_{2}(z)),
	\eeq
	for all $z\in V\cap \bbH$. Following \Cref{eq:subor_psi}, we easily find that $(z,-1/\psi^{-1}(z),-1/\omega(\psi^{-1}(z)))$ satisfies \Cref{eq:subor}. Also, there is an open subset $V'\subset V\cap\bbH$ so that $\im \psi^{-1}(z)>0$ for all $z\in V'$; to see this, we write
	\begin{equation*}
	\begin{split}
		\im \psi^{-1}(z)&=\im\left[(\psi^{-1})'(\psi(a))\cdot(z-\psi(a))\right]+\cO(\absv{z-\psi(a)}^{2})
	\\ &	=\frac{1}{\psi'(a)}\im z+\cO(\absv{z-\psi(a)}^{2}).\notag
	\end{split}
	\end{equation*}
	Hence, it suffices to take $V'=\{z:\absv{z-\psi(a)}<2\im z<r\}$ with small enough $r>0$ in order to have $\psi^{-1}(V')\subset\bbH$. Then, by \Cref{eq:omega_san} it also follows that $\im[-z/\omega(\psi^{-1}(z))]>0$. As in the proof of \Cref{itm:a0=sup-1} of \Cref{lem:a0=sup}, the uniqueness of the solution of \Cref{eq:subor} implies \Cref{eq:psi_minv_2} for $z\in V'$. Finally the conclusion extends to $V\cap\bbH$ by analytic continuation.
	
	Since $\psi$ maps $(\sup\supp\ol{\mu}_{T},\infty)$ to $\R$, its inverse function $\psi^{-1}$ is real-valued on $V\cap\R$. Hence it follows
	\beq
		\lim_{\eta\to0}\im m_{1}(x+\ii\eta)=\lim_{\eta\to 0} \im\left[-\frac{1}{\psi^{-1}(x+\ii\eta)}\right]=0,\qquad x\in V\cap\R.\notag
	\eeq
	Then, applying Stieltjes inversion to \Cref{eq:Pick}, we have $\supp\nu_{1}\cap V=\emptyset$. Finally by \Cref{eq:Pick} we conclude $\supp\ol{\mu}_{\wh{D}}\cap V=\emptyset$, so that $\psi(a)\notin\supp\ol{\mu}_{\wh{D}}$. This completes the proof of \Cref{itm:a0=sup-2} in \Cref{lem:a0=sup}.
\end{proof}

\section{Performance of the whitened spectral estimator}
\label{sec:spec_known}

In this section, we characterize the limiting overlap of the whitened spectral estimator, whose definition we recall from \Cref{eqn:whiten_spec_main}:
\begin{align}
    \beta^{\spec}_\known(y, X, \Sigma) &\coloneqq \Sigma^{-1/2} v_1(D_\known) , \label{eqn:def-xspec-known} 
\end{align}
where
\begin{align}
    D_\known &\coloneqq \sum_{i = 1}^n (\Sigma^{-1/2} x_i)(\Sigma^{-1/2} x_i)^\top \cT(y_i)
    = \Sigma^{-1/2} X^\top T X \Sigma^{-1/2}
    = \wt{X}^\top T \wt{X} 
    = \Sigma^{-1/2} D \Sigma^{-1/2}. \notag 
\end{align}
As discussed in \Cref{sec:experiments}, one can think of $ \Sigma^{1/2} \beta^* $ as an auxiliary parameter in the model $ y = q(\wt{X} \Sigma^{1/2} \beta^*, \eps) $ with design matrix $ \wt{X} $. 
Therefore, the top eigenvector of $ D_\known = \wt{X}^\top \diag(\cT(q(\wt{X} \Sigma^{1/2} \beta^*, \eps))) \wt{X} $ estimates $ \Sigma^{1/2} \beta^* $ and $ \Sigma^{-1/2} v_1(D_\known) $ estimates $ \beta^* $. We highlight that computing this spectral estimator requires knowledge of $\Sigma$.

As before, our results concerning $ \beta^{\spec}_\known $ are expressed in terms of a few functions and parameters. 
Define $ \phi_\known, \psi_\known, \zeta_\known \colon (\sup\supp(\cT(\ol{Y})), \infty) \to\bbR, a^\circ_\known\in(\sup\supp(\cT(\ol{Y})), \infty) $ as
\begin{align}
    \phi_\known(a) &= \frac{a \delta}{\expt{\ol{\Sigma}}} \expt{\ol{G}^2 \cF_a(\ol{Y})} , \quad 
    \psi_\known(a) = a \paren{\frac{1}{\delta} + \expt{\cF_a(\ol{Y})}} , \notag \\ 
    a^\circ_\known &= \argmin_{a\in(\sup\supp(\cT(\ol{Y})), \infty)} \psi_\known(a) , \quad 
    \zeta_\known(a) = \psi_\known(\max\{a, a^\circ_\known\}) , \notag 
\end{align}
where $ \cF_a $ is given in \Cref{eqn:cF_a}, and $ a^*_\known\in(\sup\supp(\cT(\ol{Y})), \infty) $ as the unique solution to 
\begin{align}
    \zeta_\known(a^*_\known) = \phi_\known(a^*_\known) . \notag 
\end{align}
 Both $ a^\circ_\known $ and $ a^*_\known $ are uniquely defined, as shown in \cite[Item 1 of Theorem 2.1]{LuLi} and \cite[Item 1 of Lemma 2]{mondelli-montanari-2018-fundamental}. In fact, $ \sqrt{\frac{\delta}{\expt{\ol{\Sigma}}}} \, \ol{G} \sim \cN(0,1) $, so our functions $ \phi_\known, \psi_\known, \zeta_\known $ match $\phi, \psi, \zeta$ in \cite{LuLi} by taking $ \kappa $ in \cite{LuLi} to be $ \sqrt{\frac{\expt{\ol{\Sigma}}}{\delta}} $. 
The formula of the asymptotic overlap $ \eta_\known $ is: 
\begin{align}
    \eta_\known &\coloneqq \paren{ \frac{1 - \delta \expt{\cF_{a^*_\known}(\ol{Y})^2}}{1 + \delta \expt{\paren{\expt{\frac{\delta}{\ol{\Sigma}}} \ol{G}^2 - 1} \cF_{a^*_\known}(\ol{Y})^2}} }^{1/2} . \notag
\end{align}
\begin{theorem}[Whitened spectral estimator]
\label{thm:spec_known}
Consider the above setting and let \Cref{asmp:signal_prior,asmp:sigma,asmp:noise,asmp:proportional,asmp:preprocessor} hold. 
Suppose $ a^*_\known > a^\circ_\known $. 
Then, the top two eigenvalues $ \lambda_1(D), \lambda_2(D) $ of $D$ satisfy
\begin{align}
    \plim_{d\to\infty} \lambda_1(D) &= \zeta(a^*_\known) , \qquad 
    \lim_{d\to\infty} \lambda_2(D) = \zeta(a^\circ_\known) \quad \text{almost surely} , \notag 
\end{align}
and $ \zeta(a^*_\known) > \zeta(a^\circ_\known) $. 
Furthermore, the limiting overlap between the spectral estimator $ \beta^{\spec}_\known = \Sigma^{-1/2} v_1(D_\known) $ and $ \beta^* $ equals
\begin{align}
    \plim_{d\to\infty} \frac{\abs{\inprod{\beta^{\spec}_\known}{\beta^*}}}{\normtwo{\beta^{\spec}_\known} \normtwo{\beta^*}} &= \eta_\known > 0 . \notag 
\end{align}
\end{theorem}

We emphasize that, even if the spectral estimator is now computed with respect to $\wt{X}$ whose rows have identity covariance, the observation $y$ still depends on $ \Sigma $ through $ y = q(\wt{X} \Sigma^{1/2} \beta^*, \eps) $ and there is no easy way to further invert out $\Sigma^{1/2}$ therein. 
Thus, we cannot reduce to the $\Sigma = I_d$ case studied in \cite{LuLi,mondelli-montanari-2018-fundamental}, and we follow a strategy similar to that described in \Cref{sec:heuristics} to prove \Cref{thm:main}. 

\begin{proof}[Proof of \Cref{thm:spec_known}]
Let us consider the generic GAMP iteration in \Cref{eqn:GAMP_nonsep}. 
Let $ \cF_\known\colon\bbR\to\bbR $ be an auxiliary preprocessing function to be chosen later. 
Set
\begin{align}
    f_{t+1}(v^{t+1}) &= \frac{v^{t+1}}{\beta_{t+1}} , 
    \quad t\ge0 , \label{eqn:ft+1_known} 
\end{align}
for a sequence $ (\beta_{t+1})_{t\ge0} $ to be specified later via state evolution. 
One should think of the normalization $\beta_{t+1}>0$ as 
$    \beta_{t+1} = \lim_{d\to\infty} \normtwo{v^{t+1}}/\sqrt{d}$ ,
so that $$
    \lim_{d\to\infty} \normtwo{f_{t+1}(v^{t+1})}/\sqrt{d} = 1,$$
as in \Cref{eqn:vtilde_1}. 
Furthermore, we set 
\begin{align}
    g_t(u^t; y) &= F_\known u^t , \quad t\ge0, \label{eqn:gt_known} 
\end{align}
where $ F_\known = \diag(\cF_\known(y)) \in\bbR^{n\times n} $ and $ \cF_\known(y)\in\bbR^n $ is obtained by applying $ \cF_\known $ to each entry of $y$. 
The coefficients $b_{t+1}, c_t$ specialize to
\begin{align}
    b_{t+1} &= \frac{1}{\delta \beta_{t+1}} , \quad 
    c_t = \expt{\cF_\known(\ol{Y})} \eqqcolon c . \notag 
\end{align}

Following the argument of \Cref{sec:pf-AMP}, we can show that $ u^t, v^{t+1}, \beta_{t+1} $ converge respectively to $ u\in\bbR^n, v\in\bbR^d, \beta\in\bbR $ in the following sense
\begin{align}
    \lim_{t\to\infty} \lim_{n\to\infty} \frac{1}{\sqrt{n}} \normtwo{u^t - u} &= 0 , \quad 
    \lim_{t\to\infty} \lim_{d\to\infty} \frac{1}{\sqrt{d}} \normtwo{v^{t+1} - v} = 0 , \quad 
    \lim_{t\to\infty} \abs{\beta_{t+1} - \beta} = 0 . \notag 
\end{align}
Then in the $t\to\infty$ limit, the GAMP iteration becomes
\begin{align}
    u &= \frac{1}{\beta} \wt{X} v - b F_\known u , \quad 
    v = \wt{X}^\top F_\known u - \frac{1}{\beta} c v , \notag 
\end{align}
where $ b = \frac{1}{\delta\beta} $ is the limit of $b_{t+1}$ as $t\to\infty$. 
Solving $u$ in terms of $v$ from the first equation, we get
\begin{align}
    u &= \frac{1}{\beta} (I_n + b F_\known)^{-1} \wt{X} v . \notag 
\end{align}
We then use this to eliminate $u$ from the equation for $v$ and obtain: 
\begin{align}
    (\beta + c) v &= \wt{X}^\top F_\known (I_n + b F_\known)^{-1} \wt{X} v . \label{eqn:known_heuristic_eigeqn}
\end{align}
Our aim is to choose $\cF_\known$ judiciously to turn the above equation into an eigenequation for $D_\known = \wt{X}^\top T \wt{X}$. 
First, to simplify the derivation, we require $ b = 1 $ which will be the case if $ \beta = \frac{1}{\delta} $. 
Next, we choose 
\begin{align}
    \cF_\known(\cdot) &= \cF_{a^*_\known}(\cdot) , \label{eqn:wt_F_fn} 
\end{align}
where the right-hand side is defined in \Cref{eqn:cF_a} and $ a^*_\known $ is to be specified later. 
With these choices, \Cref{eqn:known_heuristic_eigeqn} becomes
\begin{align}
    \paren{\frac{1}{\delta} + c} v &= \frac{1}{a^*_\known} \wt{X}^\top T \wt{X} v
    = \frac{1}{a^*_\known} D_\known v , \notag 
\end{align}
which, upon multiplying by $a^*_\known$ on both sides, is an eigenequation of $D_\known$ with respect to the eigenvalue 
\begin{align}
    a^*_\known \paren{ \frac{1}{\delta} + c }
    = a^*_\known \paren{ \frac{1}{\delta} + \expt{\cF_{a^*_\known}(\ol{Y})} } , \notag 
\end{align}
and the corresponding eigenvector (up to scaling) $ v $. The value of $ a^*_\known $ is fixed when we enforce $ \beta = \frac{1}{\delta} $ which in turn enforces $b = 1$. 
From the state evolution analysis presented below, $\beta$ can be derived and therefore $a^*_\known$ is defined as the solution to 
\begin{align}
    \beta = \lim_{t\to\infty} \beta_{t+1} 
    = \expt{\paren{ \frac{\delta}{\expt{\ol{\Sigma}}} \ol{G}^2 - 1 } \cF_{a^*_\known}(\ol{Y})} = \frac{1}{\delta} . \label{eqn:def_wt_a^*} 
\end{align}

Consider the unique solution $ a^*_\known $ to \Cref{eqn:def_wt_a^*} in $ \paren{ \sup\supp(\cT(\ol{Y})), \infty } $ and let $\cF_{a^*_\known}\colon\bbR\to\bbR$ be defined in \Cref{eqn:wt_F_fn}. 
Set the denoisers $ (f_{t+1}, g_t)_{t\ge0} $ in \Cref{eqn:GAMP_nonsep} to those given in \Cref{eqn:ft+1_known,eqn:gt_known} and initialize the GAMP iteration with
\begin{align}
    \wt{u}^{-1} &= 0_n , \quad 
    \wt{v}^0 = \mu \wt{\beta}^* + \sqrt{1 - \mu^2 \expt{\ol{\Sigma}}} w \in\bbR^d , 
    \label{eqn:init_known}
\end{align}
where $ w\sim\cN(0_d, I_d) $ is independent of everything else and $ \mu $ is given in \Cref{eqn:mu_known} below. 
Given all these configurations, the state evolution recursion specializes to 
\begin{align}
    \mu_t &= \frac{\delta}{\expt{\ol{\Sigma}}} \lim_{n\to\infty} \frac{1}{n} \expt{(\wt{\mathfrak{B}}^*)^\top V_t / \beta_t}
    = \frac{\delta}{\expt{\ol{\Sigma}}} \lim_{n\to\infty} \frac{1}{n} \expt{ (\wt{\mathfrak{B}}^*)^\top \wt{\mathfrak{B}}^* } \chi_t / \beta_t
    = \chi_t / \beta_t , \notag \\
    \sigma_{U, t}^2 &= \lim_{n\to\infty} \frac{1}{n} \expt{V_t^\top V_t / \beta_t^2} - \frac{\expt{\ol{\Sigma}}}{\delta} \mu_t^2 \notag \\
    &= \frac{1}{\beta_t^2} \lim_{n\to\infty} \frac{1}{n} \expt{(\wt{\mathfrak{B}}^*)^\top \wt{\mathfrak{B}}^*} \chi_t^2
    + \frac{1}{\beta_t^2} \lim_{n\to\infty} \frac{1}{n} \expt{W_{V,t}^\top W_{V,t}} \sigma_{V,t}^2
    - \frac{\expt{\ol{\Sigma}}}{\delta} \mu_t^2 \notag \\
    &= \frac{1}{\beta_t^2} \frac{\expt{\ol{\Sigma}}}{\delta} \chi_t^2 
    + \frac{1}{\beta_t^2} \frac{1}{\delta} \sigma_{V,t}^2
    - \frac{\expt{\ol{\Sigma}}}{\delta} \mu_t^2 
    = \frac{\sigma_{V,t}^2}{\delta \beta_t^2} , \notag \\
    \chi_{t+1} &= \frac{\delta}{\expt{\ol{\Sigma}}} \lim_{n\to\infty} \frac{1}{n} \expt{G^\top F_\known U_t} - \mu_t \expt{\cF_{a^*_\known}(\ol{Y})} \notag \\
    &= \frac{\delta}{\expt{\ol{\Sigma}}} \lim_{n\to\infty} \frac{1}{n} \expt{G^\top F_\known G} \mu_t - \mu_t \expt{\cF_{a^*_\known}(\ol{Y})} \notag \\
    &= \expt{\paren{\frac{\delta}{\expt{\ol{\Sigma}}} \ol{G}^2 - 1} \cF_{a^*_\known}(\ol{Y})} \mu_t 
    = \expt{\paren{\frac{\delta}{\expt{\ol{\Sigma}}} \ol{G}^2 - 1} \cF_{a^*_\known}(\ol{Y})} \frac{\chi_t}{\beta_t} , \notag \\
    \sigma_{V,t+1}^2 &= \lim_{n\to\infty} \frac{1}{n} \expt{U_t^\top F_\known^2 U_t} \notag \\
  &  = \lim_{n\to\infty} \frac{1}{n} \expt{G^\top F_\known^2 G} \mu_t^2 
    + \lim_{n\to\infty} \frac{1}{n} \expt{W_{U,t}^\top F_\known^2 W_{U,t}} \sigma_{U,t}^2 \notag \\
    &= \expt{\ol{G}^2 \cF_{a^*_\known}(\ol{Y})^2} \mu_t^2 + \expt{\cF_{a^*_\known}(\ol{Y})^2} \sigma_{U,t}^2 \notag \\
&    = \expt{\ol{G}^2 \cF_{a^*_\known}(\ol{Y})^2} \frac{\chi_t^2}{\beta_t^2} + \expt{\cF_{a^*_\known}(\ol{Y})^2} \frac{\sigma_{V,t}^2}{\delta \beta_t^2} , \notag \\
    \beta_{t+1}^2 &= \lim_{d\to\infty} \frac{1}{d} \expt{V_{t+1}^\top V_{t+1}} \notag \\
   & = \lim_{d\to\infty} \frac{1}{d} \expt{(\wt{\mathfrak{B}}^*)^\top \wt{\mathfrak{B}}^*} \chi_{t+1}^2 
    + \lim_{d\to\infty} \frac{1}{d} \expt{W_{V,t+1}^\top W_{V,t+1}} \sigma_{V,t+1}^2 \notag \\
    &= \expt{\ol{\Sigma}} \chi_{t+1}^2 + \sigma_{V,t+1}^2 . \notag 
\end{align}
There are $3$ fixed points of $ (\mu_t, \sigma_{U,t}, \chi_{t+1}, \sigma_{V,t+1}, \beta_{t+1}) $:
\begin{align}
    \mathsf{FP}_+ &= (\mu, \sigma_U, \chi, \sigma_V, \beta) , \quad 
    \mathsf{FP}_- = (-\mu, \sigma_U, -\chi, \sigma_V, \beta) , \notag \\
    \mathsf{FP}_0 &= \paren{ 0, \frac{1}{\sqrt{\delta}}, 0, \frac{1}{\sqrt{\delta}} \expt{\cF_{a^*_\known}(\ol{Y})^2}^{1/2}, \frac{1}{\sqrt{\delta}} \expt{\cF_{a^*_\known}(\ol{Y})^2}^{1/2} } , \notag 
\end{align}
where $ \mu, \sigma_U, \chi, \sigma_V, \beta $ are given by
\begin{align}
    \beta &= \expt{\paren{ \frac{\delta}{\expt{\ol{\Sigma}}} \ol{G}^2 - 1 } \cF_{a^*_\known}(\ol{Y})} = \frac{1}{\delta} , \notag \\
    \chi &= \paren{ \frac{\beta^2 - \frac{1}{\delta} \expt{\cF_{a^*_\known}(\ol{Y})^2}}{\expt{\ol{\Sigma}} - \frac{\expt{\ol{\Sigma}}}{\delta\beta^2} \expt{\cF_{a^*_\known}(\ol{Y})^2} + \frac{1}{\beta^2} \expt{\ol{G}^2 \cF_{a^*_\known}(\ol{Y})^2}} }^{1/2} \notag \\
    &= \paren{ \frac{\frac{1}{\delta^2} - \frac{1}{\delta} \expt{\cF_{a^*_\known}(\ol{Y})^2}}{\expt{\ol{\Sigma}} - \delta \expt{\ol{\Sigma}} \expt{\cF_{a^*_\known}(\ol{Y})^2} + \delta^2 \expt{\ol{G}^2 \cF_{a^*_\known}(\ol{Y})^2}} }^{1/2} , \notag \\
    \sigma_V &= \paren{ \frac{\delta^2 \expt{\ol{G}^2 \cF_{a^*_\known}(\ol{Y})^2} \chi^2}{1 - \frac{1}{\delta\beta^2} \expt{\cF_{a^*_\known}(\ol{Y})^2}} }^{1/2} \notag \\
    &
    = \paren{ \frac{\expt{\ol{G}^2 \cF_{a^*_\known}(\ol{Y})^2}}{\expt{\ol{\Sigma}} - \delta \expt{\ol{\Sigma}} \expt{\cF_{a^*_\known}(\ol{Y})^2} + \delta^2 \expt{\ol{G}^2 \cF_{a^*_\known}(\ol{Y})^2}} }^{1/2} , \notag \\
    \mu &= \frac{\chi}{\beta}
    = \paren{ \frac{1 - \delta \expt{\cF_{a^*_\known}(\ol{Y})^2}}{\expt{\ol{\Sigma}} - \delta \expt{\ol{\Sigma}} \expt{\cF_{a^*_\known}(\ol{Y})^2} + \delta^2 \expt{\ol{G}^2 \cF_{a^*_\known}(\ol{Y})^2}} }^{1/2} , \label{eqn:mu_known} \\
    \sigma_U &= \frac{\sigma_V}{\sqrt{\delta} \beta}
    = \paren{ \frac{\delta \expt{\ol{G}^2 \cF_{a^*_\known}(\ol{Y})^2}}{\expt{\ol{\Sigma}} - \delta \expt{\ol{\Sigma}} \expt{\cF_{a^*_\known}(\ol{Y})^2} + \delta^2 \expt{\ol{G}^2 \cF_{a^*_\known}(\ol{Y})^2}} }^{1/2} . \notag 
\end{align}
Furthermore, the initialization scheme in \Cref{eqn:init_known} guarantees that $ (\mu_t, \sigma_{U,t}, \chi_{t+1}, \sigma_{V,t+1}, \beta_{t+1}) $ stays at $ \mathsf{FP}_+ $ for every $ t\ge0 $. 

Executing similar arguments in the proofs of \Cref{lem:bulk} and of \Cref{eqn:overlap_vhat_v1} gives
\begin{align}
    \lim_{t\to\infty} \plim_{d\to\infty} \frac{\inprod{v^{t+1}}{v_1(D_\known)}^2}{\normtwo{v^{t+1}}^2 \normtwo{v_1(D_\known)}^2} &= 1 , \quad 
    \plim_{d\to\infty} \lambda_1(D_\known) = \zeta(a^*_\known) >
    \zeta(a^\circ_\known) = \lim_{d\to\infty} \lambda_2(D_\known) . \label{eqn:align_and_bulk_known} 
\end{align}
Recall from \Cref{eqn:def-xspec-known} that the whitened spectral estimator is defined as $ \beta_\known^{\spec} = \Sigma^{-1/2} v_1(D_\known) $. 
Given the result in \Cref{eqn:align_and_bulk_known}, the overlap between $ \beta_\known^{\spec} $ and $ \beta^* $ is asymptotically the same as that between $ \Sigma^{-1/2} v^{t+1} $ and $ \beta^* $ which we compute below:
\begin{align}
    \lim_{t\to\infty} \plim_{d\to\infty} \frac{\inprod{\Sigma^{-1/2} v^{t+1}}{\beta^*}^2}{\normtwo{\Sigma^{-1/2} v^{t+1}}^2 \normtwo{\beta^*}^2}
    &= \frac{\lim\limits_{t\to\infty} \plim\limits_{d\to\infty} \frac{1}{d^2} \inprod{\Sigma^{-1/2} v^{t+1}}{\beta^*}^2}{\lim\limits_{t\to\infty} \plim\limits_{d\to\infty} \frac{1}{d} \normtwo{\Sigma^{-1/2} v^{t+1}}^2} , \notag 
\end{align}
the numerator and denominator of which are given respectively as follows:
\begin{align}
    &\lim_{t\to\infty} \plim_{d\to\infty} \frac{1}{d^2} \inprod{\Sigma^{-1/2} v^{t+1}}{\beta^*}^2
    = \lim_{t\to\infty} \lim_{d\to\infty} \frac{1}{d^2} \expt{(\wt{\mathfrak{B}}^*)^\top \Sigma^{-1/2} \mathfrak{B}^*}^2 \chi_{t+1}^2 
    = \chi^2 , \notag \\
	& \lim_{t\to\infty} \plim_{d\to\infty} \frac{1}{d} \normtwo{\Sigma^{-1/2} v^{t+1}}^2
    = \lim_{t\to\infty} \lim_{d\to\infty} \frac{1}{d} \expt{ (\wt{\mathfrak{B}}^*)^\top \Sigma^{-1} \wt{\mathfrak{B}}^* } \chi_{t+1}^2 \notag \\
    &\hspace{10.6em}+ \frac{1}{d} \expt{ W_{V,t+1}^\top \Sigma^{-1} W_{V,t+1} } \sigma_{V,t+1}^2  
    = \chi^2 + \expt{\frac{1}{\ol{\Sigma}}} \sigma_V^2 . \notag 
\end{align}
Using the expressions of $ \chi,\sigma_V $, we obtain
\begin{align}
    \plim_{d\to\infty} \frac{\inprod{\beta^{\spec}_\known}{\beta^*}^2}{\normtwo{\beta^{\spec}_\known}^2 \normtwo{\beta^*}^2}
    &= \frac{\chi^2}{\chi^2 + \expt{\frac{1}{\ol{\Sigma}}} \sigma_V^2}
    = \frac{\frac{1}{\delta^2} - \frac{1}{\delta} \expt{\cF_{a^*_\known}(\ol{Y})^2}}{\frac{1}{\delta^2} - \frac{1}{\delta} \expt{\cF_{a^*_\known}(\ol{Y})^2} + \expt{\frac{1}{\ol{\Sigma}}} \expt{\ol{G}^2 \cF_{a^*_\known}(\ol{Y})^2}} \notag \\
    &= \frac{1 - \delta \expt{\cF_{a^*_\known}(\ol{Y})^2}}{1 + \delta \expt{\paren{\expt{\frac{\delta}{\ol{\Sigma}}} \ol{G}^2 - 1} \cF_{a^*_\known}(\ol{Y})^2}} =\eta_\known, \notag 
\end{align}
which concludes the proof. 
\end{proof}

\section{Auxiliary results}
\label{sec:aux_lem}

\begin{proposition}[$w_1>0$]
\label{prop:x1>0}
Let $ w_1 $ be defined in \Cref{eqn:def_x1_main}. 
Then $ w_1 > 0 $. 
\end{proposition}

\begin{proof}
By definition, we have
\begin{align}
    w_1 &= \frac{1}{ \expt{\ol{\Sigma}}^2 } \expt{ \ol{G}^2 \cF_{a^*}(\ol{Y})^2} \expt{\frac{\ol{\Sigma}^2}{\gamma^* - \expt{ \cF_{a^*}(\ol{Y}) } \ol{\Sigma}}}^2 \notag \\
    &\phantom{=}~- \frac{1}{\delta \expt{\ol{\Sigma}}} \expt{ \cF_{a^*}(\ol{Y})^2} \expt{\frac{\ol{\Sigma}^2}{\gamma^* - \expt{ \cF_{a^*}(\ol{Y}) } \ol{\Sigma}}}^2 \notag \\
    &\phantom{=}~+ \frac{1}{\delta} \expt{\cF_{a^*}(\ol{Y})^2} \expt{\frac{\ol{\Sigma}^3}{\paren{ \gamma^* - \expt{ \cF_{a^*}(\ol{Y}) } \ol{\Sigma} }^2}} . \notag 
\end{align}
The first term is strictly positive. 
It suffices to show that the sum of the last two terms is non-negative. 
This follows from the Cauchy--Schwarz inequality: 
\begin{align}
    \expt{ \frac{\ol{\Sigma}^2}{\gamma^* - \expt{ \cF_{a^*}(\ol{Y}) } \ol{\Sigma}} }^2
    &= \expt{ \ol{\Sigma}^{1/2} \cdot \frac{\ol{\Sigma}^{3/2}}{\gamma^* - \expt{ \cF_{a^*}(\ol{Y}) } \ol{\Sigma}} }^2 \notag \\
&\le \expt{\ol{\Sigma}} \expt{ \frac{\ol{\Sigma}^3}{ \paren{\gamma^* - \expt{ \cF_{a^*}(\ol{Y}) } \ol{\Sigma}}^2 } } . \label{eqn:cs_x1>0} 
\end{align}
Rearranging terms and noting that the common factor $ \frac{1}{\delta} \expt{\cF_{a^*}(\ol{Y})^2} $ in the last two terms is positive, the proof is complete. 
\end{proof}


\begin{proposition}
\label{prop:trace}
Let $ W\sim P^{\ot d} $ where $ P $ is a distribution on $ \bbR $ with mean $0$ and variance $ \sigma^2 $. 
Let $ B\in\bbR^{d\times d} $ denote a sequence of deterministic matrices such that the empirical spectral distribution of $ \frac{1}{d} B $ converges to the law of a random variable $ \ol{\Sigma} $. 
Then 
\begin{align}
\lim_{d\to\infty} \frac{1}{d} \expt{W^\top B W} &= \sigma^2 \expt{\ol{\Sigma}} . \notag
\end{align}
\end{proposition}

\begin{proof}
The proof follows from a straightforward calculation:
\begin{align}
\lim_{d\to\infty} \frac{1}{d} \expt{W^\top B W} &= \lim_{d\to\infty} \frac{1}{d} \sum_{i,j} \expt{B_{i,j} W_i W_j} \notag \\
&= \lim_{d\to\infty} \frac{1}{d} \sum_{i} \expt{W_i^2} B_{i,i}
= \lim_{d\to\infty} \frac{\sigma^2}{d} \tr(B)
= \sigma^2 \expt{\ol{\Sigma}} . \notag
\end{align}
\end{proof}

\begin{proposition}
\label{prop:trace_corr}
Let 
$ (G,H)  \sim \cN\paren{\matrix{0_d \\ 0_d} , \matrix{\sigma^2 & \rho \\ \rho & \tau^2} \ot I_d}$. 
Let $ B\in\bbR^{d\times d} $ denote a sequence of deterministic matrices such that the empirical spectral distribution of $ \frac{1}{d} B $ converges to the law of a random variable $ \ol{\Sigma} $. 
Then 
\begin{align}
\lim_{d\to\infty} \, \frac{1}{d} \expt{G^\top B H} &= \rho \expt{\ol{\Sigma}} . \notag 
\end{align}
\end{proposition}

\begin{proof}
The proof follows from a straightforward calculation:
\begin{align}
\lim_{d\to\infty} \frac{1}{d} \expt{G^\top B H} 
&= \lim_{d\to\infty} \frac{1}{d} \sum_{i,j} \expt{B_{i,j} G_i H_j} \notag \\
&= \lim_{d\to\infty} \frac{1}{d} \sum_{i} \expt{G_i H_i} B_{i,i}
= \lim_{d\to\infty} \frac{\rho}{d} \tr(B)
= \rho \expt{\ol{\Sigma}} . \notag
\end{align}
\end{proof}


\begin{proposition}[Davis--Kahan \cite{Davis_Kahan}]
\label{prop:davis_kahan}
Let $ A, B \in \bbR^{d\times d} $ be symmetric matrices. 
Then
\begin{align}
    \min\brace{ \normtwo{v_1(A) - v_1(B)}, \normtwo{v_1(A) + v_1(B)} }
    &\le \frac{4\normtwo{A - B}}{\max\brace{ \lambda_1(A) - \lambda_2(A), \lambda_1(B) - \lambda_2(B) }} . \notag 
\end{align}
\end{proposition}
Note that the minimum on the left-hand side is to resolve the sign ambiguity since $ v $ is an eigenvector if and only if $-v$ is.

\begin{remark}[Spectral threshold with right rotationally invariant designs]  
\label{rk:equiv_threshold}
The optimal spectral threshold for phase retrieval with right rotationally invariant designs  was derived by Maillard et al.\ in  \cite[Equation (11)]{maillard2020phase}, and this expression coincides with \Cref{eqn:opt_thr_fp}.
To see this, note that \Cref{eqn:opt_thr_fp} involves the limiting spectral distribution of $\Sigma$ only through its first two moments. 
One can then express the same result using the limiting spectral distribution $ \ol{\mu}_{X^\top X} $ of $ X^\top X = \Sigma^{1/2} \wt{X}^\top \wt{X} \Sigma^{1/2} $, which equals the free multiplicative convolution between the Marchenko--Pastur law $ \mathsf{MP}_\lambda $ (with $\lambda = 1/\delta$) 
and $ \law(\ol{\Sigma}) $.
In particular, let $ \ol{\Lambda} $ be the random variable with law $ \ol{\mu}_{X^\top X} $. By using the moment-cumulant relation \cite[Section 2.5]{Novak_RMT_LecNote} and an identity relating the square free cumulants of $ \law(\ol{\Sigma}) $ to the rectangular free cumulants of $ \law(\ol{\Sigma}) \boxtimes \mathsf{MP}_{1/\delta} $ \cite[Remark 2]{Benaych-Georges_surprising}, we have that 
    $\expt{\ol{\Lambda}} = \expt{\ol{\Sigma}}$ and 
    $\expt{\ol{\Lambda}^2} = \expt{\ol{\Sigma}^2} + \frac{1}{\delta} \expt{\ol{\Sigma}}^2$. 
Using these identities to write \Cref{eqn:opt_thr_fp} in terms of the first two moments of  $\ol{\Lambda}$, we readily obtain that this expression coincides with Equation (11) in \cite{maillard2020phase}.
\end{remark}


\begin{ack}
This work was done when Y.Z.\ and H.C.J.\ were at the Institute of Science and Technology Austria. 
Y.Z.\ thanks Hugo Latourelle-Vigeant for bringing \cite{liao2021hessian} to the authors' attention. 
\end{ack}

\begin{funding}
Y.Z.\ and M.M.\ are partially supported by the 2019 Lopez-Loreta Prize and by the Interdisciplinary Projects Committee (IPC) at ISTA.
H.C.J.\ is supported by the ERC Advanced Grant ``RMTBeyond'' No.\ 101020331.
\end{funding}


\bibliographystyle{alpha}
\bibliography{ref}









\end{document}